\documentclass[numbers=enddot,12pt,final,onecolumn,notitlepage]{scrartcl}%
\usepackage[headsepline,footsepline,manualmark]{scrlayer-scrpage}
\usepackage[all,cmtip]{xy}
\usepackage{amssymb}
\usepackage{amsmath}
\usepackage{amsthm}
\usepackage{framed}
\usepackage{comment}
\usepackage{color}
\usepackage{hyperref}
\usepackage[sc]{mathpazo}
\usepackage[T1]{fontenc}
\usepackage{needspace}
\usepackage{tabls}
\providecommand{\U}[1]{\protect\rule{.1in}{.1in}}
\theoremstyle{definition}
\newtheorem{theo}{Theorem}[section]
\newenvironment{theorem}[1][]
{\begin{theo}[#1]\begin{leftbar}}
{\end{leftbar}\end{theo}}
\newtheorem{lem}[theo]{Lemma}
\newenvironment{lemma}[1][]
{\begin{lem}[#1]\begin{leftbar}}
{\end{leftbar}\end{lem}}
\newtheorem{prop}[theo]{Proposition}
\newenvironment{proposition}[1][]
{\begin{prop}[#1]\begin{leftbar}}
{\end{leftbar}\end{prop}}
\newtheorem{defi}[theo]{Definition}
\newenvironment{definition}[1][]
{\begin{defi}[#1]\begin{leftbar}}
{\end{leftbar}\end{defi}}
\newtheorem{remk}[theo]{Remark}
\newenvironment{remark}[1][]
{\begin{remk}[#1]\begin{leftbar}}
{\end{leftbar}\end{remk}}
\newtheorem{coro}[theo]{Corollary}
\newenvironment{corollary}[1][]
{\begin{coro}[#1]\begin{leftbar}}
{\end{leftbar}\end{coro}}
\newtheorem{quest}[theo]{Question}

\newtheorem{warn}[theo]{Warning}

\newtheorem{conj}[theo]{Conjecture}

\newtheorem{exam}[theo]{Example}
\newenvironment{example}[1][]
{\begin{exam}[#1]\begin{leftbar}}
{\end{leftbar}\end{exam}}
\newenvironment{statement}{\begin{quote}}{\end{quote}}
\let\sumnonlimits\sum
\let\prodnonlimits\prod
\let\cupnonlimits\bigcup
\let\capnonlimits\bigcap
\renewcommand{\sum}{\sumnonlimits\limits}
\renewcommand{\prod}{\prodnonlimits\limits}
\renewcommand{\bigcup}{\cupnonlimits\limits}
\renewcommand{\bigcap}{\capnonlimits\limits}
\newenvironment{verlong}{}{}
\newenvironment{vershort}{}{}
\newenvironment{noncompile}{}{}
\excludecomment{verlong}
\includecomment{vershort}
\excludecomment{noncompile}

\ihead{The trace Cayley-Hamilton theorem}
\ohead{page \thepage}
\begin{document}

\title{The trace Cayley-Hamilton theorem}
\author{Darij Grinberg}
\date{April 14, 2026}
\maketitle

\begin{abstract}
\textbf{Abstract.} In this expository paper, various properties of matrix
traces, determinants and adjugate matrices are proved, including the
\textit{trace Cayley-Hamilton theorem}, which says that%
\[
kc_{k}+\sum_{i=1}^{k}\operatorname*{Tr}\left(  A^{i}\right)  c_{k-i}%
=0\ \ \ \ \ \ \ \ \ \ \text{for every }k\in\mathbb{N}%
\]
whenever $A$ is an $n\times n$-matrix with characteristic polynomial
$\det\left(  tI_{n}-A\right)  =\sum_{i=0}^{n}c_{n-i}t^{i}$ over a commutative
ring $\mathbb{K}$. While the results are not new, some of the proofs are. The
proofs illustrate some general techniques in linear algebra over commutative rings.

\end{abstract}
\tableofcontents

\section{\label{sec.intro}Introduction}

Let $\mathbb{K}$ be a commutative ring. The famous Cayley-Hamilton theorem
says that if $\chi_{A}=\det\left(  tI_{n}-A\right)  \in\mathbb{K}\left[
t\right]  $ is the characteristic polynomial of an $n\times n$-matrix
$A\in\mathbb{K}^{n\times n}$, then $\chi_{A}\left(  A\right)  =0$. Speaking
more explicitly, it means that if we write this polynomial $\chi_{A}$ in the
form $\chi_{A}=\sum_{i=0}^{n}c_{n-i}t^{i}$ (with $c_{n-i}\in\mathbb{K}$), then
$\sum_{i=0}^{n}c_{n-i}A^{i}=0$. Various proofs of this theorem are well-known
(we will present one in this paper, but it could not be any farther from being
new). A less standard fact, which I call the \textit{trace Cayley-Hamilton
theorem}, states that
\begin{equation}
kc_{k}+\sum_{i=1}^{k}\operatorname*{Tr}\left(  A^{i}\right)  c_{k-i}%
=0\ \ \ \ \ \ \ \ \ \ \text{for every }k\in\mathbb{N} \label{eq.trach.intro}%
\end{equation}
(where $\sum_{i=0}^{n}c_{n-i}t^{i}$ is $\chi_{A}$ as before, and where we set
$c_{n-i}=0$ for every $i<0$). In the case of $k\geq n$, this can easily be
obtained from the Cayley-Hamilton theorem $\sum_{i=0}^{n}c_{n-i}A^{i}=0$ by
multiplying by $A^{k-n}$ and taking traces\footnote{The details are left to
the interested reader. The $kc_{k}$ term on the left hand side appears off,
but it actually is harmless: In the $k=n$ case, it can be rewritten as
$\operatorname*{Tr}\left(  A^{0}\right)  c_{n}$ and incorporated into the sum,
whereas in the $k>n$ case, it simply vanishes.}; no such simple proof exists
in the general case, however. The result itself is not new (the $k\leq n$
case, for example, is \cite[Chapter III, Exercise 14]{LomQui16}), and is
well-known e.g. to algebraic combinatorialists and representation theorists
(indeed, the underlying idea goes back to Frobenius); however, it is hard to
find an expository treatment.

When the ground ring $\mathbb{K}$ is a field, it is possible to prove the
trace Cayley-Hamilton theorem by expressing both $\operatorname*{Tr}\left(
A^{i}\right)  $ and the $c_{j}$ through the eigenvalues of $A$ (indeed,
$\operatorname*{Tr}\left(  A^{i}\right)  $ is the sum of the $i$-th powers of
these eigenvalues, whereas $c_{j}$ is $\left(  -1\right)  ^{j}$ times their
$j$-th elementary symmetric function); the identity (\ref{eq.trach.intro})
then boils down to the Newton identities for said eigenvalues. This proof goes
back to Gantmacher \cite[\S IV.5, (44)]{Gantma77}. However, of course, the use
of eigenvalues in this proof requires $\mathbb{K}$ to be a field. There are
ways to adapt this proof to the case when $\mathbb{K}$ is a commutative ring.
One is to apply the \textquotedblleft method of universal
identities\textquotedblright\ (see, e.g., \cite[Chapter III, Exercise
14]{LomQui16}; the method is also explained in \cite{Conrad09}) to reduce the
general case to the case when $\mathbb{K}$ is a field\footnote{This relies on
the observation that (\ref{eq.trach.intro}) (for a given $k$) is a polynomial
identity in the entries of $A$.}. Another is to build up the theory of
eigenvalues for square matrices over an arbitrary commutative ring
$\mathbb{K}$; this is not as simple as for fields, but doable (see
\cite{Laksov13}).

In this paper (specifically, in Section \ref{sec.proofs}), I shall prove both
the Cayley-Hamilton and the trace Cayley-Hamilton theorems without the use of
eigenvalues. Instead, I will use a trick that involves expanding the adjugate
matrix $\operatorname*{adj}\left(  tI_{n}-A\right)  $ as $D_{0}t^{0}%
+D_{1}t^{1}+\cdots+D_{n-1}t^{n-1}$ for some $n$ matrices $D_{0},D_{1}%
,\ldots,D_{n-1}\in\mathbb{K}^{n\times n}$. This way of proving the
Cayley-Hamilton theorem is well-known, and is in fact the oldest known proof
of the theorem, proposed by Buchheim in 1884 \cite{Buchhe84}. What is not
well-known, however, is that the same trick can be used for the trace
Cayley-Hamilton theorem, although it requires more work; in particular, an
intermediate step is necessary, establishing that the derivative of the
characteristic polynomial $\chi_{A}=\det\left(  tI_{n}-A\right)  $ is
$\operatorname*{Tr}\left(  \operatorname*{adj}\left(  tI_{n}-A\right)
\right)  $. I hope that this writeup will have two uses: making the trace
Cayley-Hamilton theorem more accessible, and demonstrating that the trick just
mentioned can serve more than one purpose. Next, I shall show an application
of the trace Cayley-Hamilton theorem, answering a question from
\cite{m.se1798703} (Section \ref{sec.nilp}).

In the second half of this paper (Section \ref{sec.adj}), I shall discuss
several other properties of the adjugate matrix as well as further
applications of polynomial matrices in proving determinant identities. The
proofs of these properties rely on the \textquotedblleft$tI_{n}+A$
trick\textquotedblright, which consists in replacing the matrix $A\in
\mathbb{K}^{n\times n}$ by the polynomial matrix $tI_{n}+A\in\left(
\mathbb{K}\left[  t\right]  \right)  ^{n\times n}$. The latter matrix has the
advantage that its determinant $\det\left(  tI_{n}+A\right)  $ is regular
(i.e., not a zero-divisor) even when $\det A$ is not. This allows us to cancel
$\det\left(  tI_{n}+A\right)  $ from equalities, something that could not
generally be done with $\det A$. This trick is an instance of a strategy
deeply familiar to algebraic geometers (\textquotedblleft deforming away from
a singularity\textquotedblright) and even to elementary geometers (e.g.
viewing a tangent of a circle as a limiting case of a secant), but is rarely
seen in texts on linear algebra, particularly in the purely algebraic form
presented here. (Several authors perform an analogous analytic procedure,
which replaces a real matrix $A\in\mathbb{R}^{n\times n}$ by a
\textquotedblleft slightly wiggled\textquotedblright\ version $\varepsilon
I_{n}+A\in\mathbb{R}^{n\times n}$ for small $\varepsilon>0$. This often has
the same purpose as our replacement of $A$ by $tI_{n}+A$, but has the
disadvantage of working only for real or complex matrices.)

In a short final section (Section \ref{sec.further}), I will discuss a few
more properties of traces, including Almkvist's theorem on traces of nilpotent
matrices over commutative rings. The proofs are only outlined.

\section{\label{sec.thm}Notations and theorems}

\subsection{Notations}

Before we state the theorems that we will be occupying ourselves with, let us
agree on the notations.

\begin{definition}
\label{def.notations}Throughout this paper, the word \textquotedblleft
ring\textquotedblright\ will mean \textquotedblleft associative ring with
unity\textquotedblright. We will always let $\mathbb{K}$ denote a commutative
ring with unity. The word \textquotedblleft matrix\textquotedblright\ shall
always mean \textquotedblleft matrix over $\mathbb{K}$\textquotedblright,
unless explicitly stated otherwise.

As usual, we let $\mathbb{K}\left[  t\right]  $ denote the polynomial ring in
the indeterminate $t$ over $\mathbb{K}$.

If $f\in\mathbb{K}\left[  t\right]  $ is a polynomial and $n$ is an integer,
then $\left[  t^{n}\right]  f$ will denote the coefficient of $t^{n}$ in $f$.
(If $n$ is negative or greater than the degree of $f$, then this coefficient
is understood to be $0$.)

Let $\mathbb{N}$ denote the set $\left\{  0,1,2,\ldots\right\}  $.

If $n\in\mathbb{N}$ and $m\in\mathbb{N}$, and if we are given an element
$a_{i,j}\in\mathbb{K}$ for every $\left(  i,j\right)  \in\left\{
1,2,\ldots,n\right\}  \times\left\{  1,2,\ldots,m\right\}  $, then we use the
notation $\left(  a_{i,j}\right)  _{1\leq i\leq n,\ 1\leq j\leq m}$ for the
$n\times m$-matrix whose $\left(  i,j\right)  $-th entry is $a_{i,j}$ for all
$\left(  i,j\right)  \in\left\{  1,2,\ldots,n\right\}  \times\left\{
1,2,\ldots,m\right\}  $.

For every $n\in\mathbb{N}$, we denote the $n\times n$ identity matrix by
$I_{n}$.

For every $n\in\mathbb{N}$ and $m\in\mathbb{N}$, we denote the $n\times m$
zero matrix by $0_{n\times m}$.

If $A$ is any $n\times n$-matrix, then we let $\det A$ denote the determinant
of $A$, and we let $\operatorname*{Tr}A$ denote the trace of $A$. (Recall that
the trace of $A$ is defined to be the sum of the diagonal entries of $A$.)

We consider $\mathbb{K}$ as a subring of $\mathbb{K}\left[  t\right]  $. Thus,
for every $n\in\mathbb{N}$, every $n\times n$-matrix in $\mathbb{K}^{n\times
n}$ can be considered as a matrix in $\left(  \mathbb{K}\left[  t\right]
\right)  ^{n\times n}$.
\end{definition}

\subsection{The main claims}

We shall now state the results that we will prove further below. We begin with
a basic fact:

\begin{proposition}
\label{prop.ta+b}Let $n\in\mathbb{N}$. Let $A\in\mathbb{K}^{n\times n}$ and
$B\in\mathbb{K}^{n\times n}$ be two $n\times n$-matrices. Consider the matrix
$tA+B\in\left(  \mathbb{K}\left[  t\right]  \right)  ^{n\times n}$. \medskip

\textbf{(a)} Then, $\det\left(  tA+B\right)  \in\mathbb{K}\left[  t\right]  $
is a polynomial of degree $\leq n$ in $t$. \medskip

\textbf{(b)} We have $\left[  t^{0}\right]  \left(  \det\left(  tA+B\right)
\right)  =\det B$. \medskip

\textbf{(c)} We have $\left[  t^{n}\right]  \left(  \det\left(  tA+B\right)
\right)  =\det A$.
\end{proposition}

\begin{definition}
\label{def.chiA}Let $n\in\mathbb{N}$. Let $A\in\mathbb{K}^{n\times n}$ be an
$n\times n$-matrix. Then, we consider $A$ as a matrix in $\left(
\mathbb{K}\left[  t\right]  \right)  ^{n\times n}$ as well (as explained
above); thus, a matrix $tI_{n}-A\in\left(  \mathbb{K}\left[  t\right]
\right)  ^{n\times n}$ is defined. We let $\chi_{A}$ denote the polynomial
$\det\left(  tI_{n}-A\right)  \in\mathbb{K}\left[  t\right]  $; we call
$\chi_{A}$ the \textit{characteristic polynomial} of $A$.
\end{definition}

We notice that the notion of the characteristic polynomial is not standardized
across the literature. Our definition of $\chi_{A}$ is identical with the
definition in \cite[\S V.3]{Knapp2016} (except that we use $t$ instead of $X$
as the indeterminate), but the definition in \cite[Chapter Five, Section II,
Definition 3.9]{Heffer14} is different (it defines $\chi_{A}$ to be
$\det\left(  A-tI_{n}\right)  $ instead). The two definitions differ merely in
a sign (namely, one version of the characteristic polynomial is $\left(
-1\right)  ^{n}$ times the other), whence any statement about one of them can
easily be translated into a statement about the other; nevertheless this
discrepancy creates some occasions for confusion. I shall, of course, use
Definition \ref{def.chiA} throughout this paper.

\begin{corollary}
\label{cor.chiA}Let $n\in\mathbb{N}$. Let $A\in\mathbb{K}^{n\times n}$.
\medskip

\textbf{(a)} Then, $\chi_{A}\in\mathbb{K}\left[  t\right]  $ is a polynomial
of degree $\leq n$ in $t$. \medskip

\textbf{(b)} We have $\left[  t^{0}\right]  \chi_{A}=\left(  -1\right)
^{n}\det A$. \medskip

\textbf{(c)} We have $\left[  t^{n}\right]  \chi_{A}=1$.
\end{corollary}

Of course, combining parts \textbf{(a)} and \textbf{(c)} of Corollary
\ref{cor.chiA} shows that, for every $n\in\mathbb{N}$ and $A\in\mathbb{K}%
^{n\times n}$, the characteristic polynomial $\chi_{A}$ is a monic polynomial
of degree $n$.

Let me now state the main two theorems of this section:

\begin{theorem}
[Cayley-Hamilton theorem]\label{thm.CH}Let $n\in\mathbb{N}$. Let
$A\in\mathbb{K}^{n\times n}$. Then, $\chi_{A}\left(  A\right)  =0_{n\times n}%
$. (Here, $\chi_{A}\left(  A\right)  $ denotes the result of substituting $A$
for $t$ in the polynomial $\chi_{A}$. It does \textbf{not} denote the result
of substituting $A$ for $t$ in the expression $\det\left(  tI_{n}-A\right)  $;
in particular, $\chi_{A}\left(  A\right)  $ is an $n\times n$-matrix, not a determinant!)
\end{theorem}

\begin{theorem}
[trace Cayley-Hamilton theorem]\label{thm.TCH}Let $n\in\mathbb{N}$. Let
$A\in\mathbb{K}^{n\times n}$. For every $j\in\mathbb{Z}$, define an element
$c_{j}\in\mathbb{K}$ by $c_{j}=\left[  t^{n-j}\right]  \chi_{A}$. Then,%
\[
kc_{k}+\sum_{i=1}^{k}\operatorname*{Tr}\left(  A^{i}\right)  c_{k-i}%
=0\ \ \ \ \ \ \ \ \ \ \text{for every }k\in\mathbb{N}.
\]

\end{theorem}

Theorem \ref{thm.CH} is (as has already been said) well-known and a
cornerstone of linear algebra. It appears (with proofs) in \cite{Bernha11},
\cite[Theorem 7.23]{Brown93}, \cite{Buchhe84}, \cite[Theorem 2.16]{Camero08},
\cite[Theorem 23.1]{Climen13}, \cite[Theorem 4.7.12]{Ford24}, \cite[\S 28.10]%
{Garrett09}, \cite[Chapter Five, Section IV, Lemma 1.9]{Heffer14},
\cite[Theorem 5.9]{Knapp2016}, \cite[\S 5.15]{Loehr14}, \cite[Corollary
14.21]{Macduf56}, \cite[\S 4, Theorem 1]{Mate16}, \cite[Theorem I.8]%
{McDona84}, \cite[\S 6.1, Theorem 1.1]{Moore68}, \cite[Seconde m\'{e}thode
(\S 3)]{Sage08}, \cite{Shurma15}, \cite[Theorem 3.1]{Stoll17}, \cite{Straub83}%
, \cite[Theorem 7.10]{Willia89}, \cite[\S 3]{Zeilbe} and in many other
sources\footnote{All the sources we are citing (with the possible exception of
\cite[\S 28.10]{Garrett09}) prove Theorem \ref{thm.CH} in full generality,
although some of them do not \textbf{state} Theorem \ref{thm.CH} in full
generality (indeed, they often state it under the additional requirement that
$\mathbb{K}$ be a field). There are other sources which only prove Theorem
\ref{thm.CH} in the case when $\mathbb{K}$ is a field. The note \cite{Sage08}
gives four proofs of Theorem \ref{thm.CH} for the case when $\mathbb{K}%
=\mathbb{C}$; the first of these proofs works for every field $\mathbb{K}$,
whereas the second works for any commutative ring $\mathbb{K}$, and the third
and the fourth actually require $\mathbb{K}=\mathbb{C}$.
\par
Note that some authors decline to call Theorem \ref{thm.CH} the
Cayley-Hamilton theorem; they instead use this name for some related result.
For instance, Hefferon, in \cite{Heffer14}, uses the name \textquotedblleft
Cayley-Hamilton theorem\textquotedblright\ for a corollary.}. The proof we
will give below will essentially repeat the proof in \cite[Chapter Five,
Section IV, Lemma 1.9]{Heffer14}; this proof goes back to Buchheim's 1884 note
\cite{Buchhe84}, and may be the historically first proof of the theorem.

Theorem \ref{thm.TCH} is a less known result. It appears in \cite[Proposition
7.7.1]{Greub78} (with a proof using tensors\footnote{\cite[Chapter 7]{Greub78}
assumes $\mathbb{K}$ to be a field of characteristic $0$, but I hope the proof
does not rely on this assumption.}), in \cite[Chapter III, Exercise
14]{LomQui16} (with a sketch of a proof), in \cite[($C-H$)]{Zeilbe93} (with a
beautiful short proof using exterior algebra) and in \cite[Exercise 5]{Zeilbe}
(without proof); its particular case when $\mathbb{K}$ is a field also tends
to appear in representation-theoretical literature (mostly left as an exercise
to the reader\footnote{In the case when $\mathbb{K}$ is a field, a proof using
eigenvalues is outlined in \cite[\S IV.5, (44)]{Gantma77}. This proof can be
adapted to arbitrary commutative rings $\mathbb{K}$ using universal splitting
algebras (see \cite[\S III.4]{LomQui16}).
\par
Another proof when $\mathbb{K}$ is a field is presented in \cite{Kalman00}.}).
We will prove it similarly to Theorem \ref{thm.CH}; this proof, to my
knowledge, is new.

The names we use for Theorem \ref{thm.CH} and Theorem \ref{thm.TCH} should not
be misunderstood as attribution. Indeed, neither Cayley nor Hamilton seem to
have proved Theorem \ref{thm.CH}. Cayley, in his 1858 memoir \cite{Cayley58}
that introduced the theorem, verified it for $n\leq3$ and subsequently claimed
to \textquotedblleft have not thought it necessary to undertake the labour of
a formal proof of the theorem in the general case of a matrix of any
degree\textquotedblright. Hamilton seems to have done even less. Buchheim gave
what might be the first proof in 1884 \cite{Buchhe84} (in fact, the very proof
we shall give below), crediting (for unclear reasons) Tait's work on
quaternions. See \cite{Feldma62} and \cite[\S 14]{Macduf56} for more
historical references. The origins of Theorem \ref{thm.TCH} are even more
elusive: While it appears in Gantmacher's \cite[\S IV.5, (44)]{Gantma77}
(already in the 1960 edition) without much fanfare, it is not easy to find an
earlier source, perhaps in part because of the surprisingly late introduction
of the notion of the trace of a matrix (by Dedekind in 1882). Something along
the lines of Theorem \ref{thm.TCH} (in a very special case) appeared in
Frobenius's 1899 paper \cite[(4.) and (6.)]{Froben99} (one of the foundational
works of group representation theory). I would not be too surprised to find
some precursors in the vast literature on determinants and invariant theory.

\section{\label{sec.proofs}The proofs}

\subsection{Proposition \ref{prop.ta+b} and Corollary \ref{cor.chiA}}

Let us now begin proving the results stated above. As a warmup, we will prove
the (rather trivial) Proposition \ref{prop.ta+b}.

We first recall how the determinant of a matrix is defined: For any
$n\in\mathbb{N}$, let $S_{n}$ denote the $n$-th symmetric group (i.e., the
group of all permutations of $\left\{  1,2,\ldots,n\right\}  $). If
$n\in\mathbb{N}$ and $\sigma\in S_{n}$, then $\left(  -1\right)  ^{\sigma}$
denotes the sign of the permutation $\sigma$. If $n\in\mathbb{N}$, and if
$A=\left(  a_{i,j}\right)  _{1\leq i\leq n,\ 1\leq j\leq n}$ is an $n\times
n$-matrix, then%
\begin{equation}
\det A=\sum_{\sigma\in S_{n}}\left(  -1\right)  ^{\sigma}\prod_{i=1}%
^{n}a_{i,\sigma\left(  i\right)  }. \label{eq.det}%
\end{equation}

We prepare for the proof of Proposition \ref{prop.ta+b} by stating a simple lemma:

\begin{lemma}
\label{lem.ta+b.prod}Let $n\in\mathbb{N}$. Let $x_{1},x_{2},\ldots,x_{n}$ be
$n$ elements of $\mathbb{K}$. Let $y_{1},y_{2},\ldots,y_{n}$ be $n$ elements
of $\mathbb{K}$. Define a polynomial $f\in\mathbb{K}\left[  t\right]  $ by
$f=\prod_{i=1}^{n}\left(  tx_{i}+y_{i}\right)  $. \medskip

\textbf{(a)} Then, $f$ is a polynomial of degree $\leq n$. \medskip

\textbf{(b)} We have $\left[  t^{n}\right]  f=\prod_{i=1}^{n}x_{i}$. \medskip

\textbf{(c)} We have $\left[  t^{0}\right]  f=\prod_{i=1}^{n}y_{i}$.
\end{lemma}

\begin{vershort}
\begin{proof}
[Proof of Lemma \ref{lem.ta+b.prod}.]Obvious by multiplying out the product
$\prod_{i=1}^{n}\left(  tx_{i}+y_{i}\right)  $ (or, if one desires a formal
proof, by a straightforward induction over $n$).
\end{proof}
\end{vershort}

\begin{verlong}
Before we come to the proof of this lemma, let us recall the following
well-known facts about polynomials:

\begin{lemma}
\label{lem.poly.1}Let $g\in\mathbb{K}\left[  t\right]  $ and $h\in
\mathbb{K}\left[  t\right]  $ be two polynomials. \medskip

\textbf{(a)} We have%
\begin{equation}
\left[  t^{k}\right]  \left(  gh\right)  =\sum_{i=0}^{k}\left(  \left[
t^{i}\right]  g\right)  \left(  \left[  t^{k-i}\right]  h\right)
\ \ \ \ \ \ \ \ \ \ \text{for every }k\in\mathbb{N}. \label{eq.tkgh}%
\end{equation}
(This is just the definition of the product of two polynomials.) \medskip

\textbf{(b)} We have $\left[  t^{0}\right]  \left(  gh\right)  =\left(
\left[  t^{0}\right]  g\right)  \left(  \left[  t^{0}\right]  h\right)  $.
\medskip

\textbf{(c)} Let $p$ and $q$ be two nonnegative integers. Assume that $g$ is a
polynomial of degree $\leq p$. Assume that $h$ is a polynomial of degree $\leq
q$. Then, $gh$ is a polynomial of degree $\leq p+q$ and satisfies $\left[
t^{p+q}\right]  \left(  gh\right)  =\left(  \left[  t^{p}\right]  g\right)
\cdot\left(  \left[  t^{q}\right]  h\right)  $.
\end{lemma}

\begin{proof}
[Proof of Lemma \ref{lem.poly.1}.]\textbf{(a)} Lemma \ref{lem.poly.1} follows
immediately from the definition of the product $gh$ of the two polynomials $g$
and $h$. \medskip

\textbf{(b)} The equality (\ref{eq.tkgh}) (applied to $k=0$) yields%
\[
\left[  t^{0}\right]  \left(  gh\right)  =\sum_{i=0}^{0}\left(  \left[
t^{i}\right]  g\right)  \left(  \left[  t^{0-i}\right]  h\right)  =\left(
\left[  t^{0}\right]  g\right)  \left(  \underbrace{\left[  t^{0-0}\right]
}_{=\left[  t^{0}\right]  }h\right)  =\left(  \left[  t^{0}\right]  g\right)
\left(  \left[  t^{0}\right]  h\right)  .
\]
This proves Lemma \ref{lem.poly.1} \textbf{(b)}. \medskip

\textbf{(c)} The polynomial $g$ has degree $\leq p$. In other words,%
\begin{equation}
\left[  t^{i}\right]  g=0\ \ \ \ \ \ \ \ \ \ \text{for every integer }i>p.
\label{pf.lem.poly.1.c.1}%
\end{equation}

The polynomial $h$ has degree $\leq q$. In other words,%
\begin{equation}
\left[  t^{i}\right]  h=0\ \ \ \ \ \ \ \ \ \ \text{for every integer }i>q.
\label{pf.lem.poly.1.c.2}%
\end{equation}

Now, let $k$ be an integer such that $k\geq p+q$. From $q\geq0$, we obtain
$p+\underbrace{q}_{\geq0}\geq p\geq0$. Thus, $k\geq p+q\geq p$, so that $p\leq
k$. Also, $p\geq0$, so that $0\leq p$. Every $i\in\left\{  0,1,\ldots
,p-1\right\}  $ satisfies%
\begin{equation}
\left[  t^{k-i}\right]  h=0 \label{pf.lem.poly.1.c.2u}%
\end{equation}
\footnote{\textit{Proof of (\ref{pf.lem.poly.1.c.2u}):} Let $i\in\left\{
0,1,\ldots,p-1\right\}  $. Then, $i\leq p-1<p$ and thus $k-\underbrace{i}%
_{<p}>k-p\geq q$ (since $k\geq p+q$). Hence, (\ref{pf.lem.poly.1.c.2})
(applied to $k-i$ instead of $i$) yields $\left[  t^{k-i}\right]  h=0$. This
proves (\ref{pf.lem.poly.1.c.2u}).}. Finally, $k\geq p\geq0$, so that
$k\in\mathbb{N}$. Hence, (\ref{eq.tkgh}) yields%
\begin{align}
\left[  t^{k}\right]  \left(  gh\right)   &  =\sum_{i=0}^{k}\left(  \left[
t^{i}\right]  g\right)  \left(  \left[  t^{k-i}\right]  h\right) \nonumber\\
&  =\sum_{i=0}^{p}\left(  \left[  t^{i}\right]  g\right)  \left(  \left[
t^{k-i}\right]  h\right)  +\sum_{i=p+1}^{k}\underbrace{\left(  \left[
t^{i}\right]  g\right)  }_{\substack{=0\\\text{(by (\ref{pf.lem.poly.1.c.1}%
)}\\\text{(since }i\geq p+1>p\text{))}}}\left(  \left[  t^{k-i}\right]
h\right) \nonumber\\
&  \ \ \ \ \ \ \ \ \ \ \left(  \text{since }0\leq p\leq k\right) \nonumber\\
&  =\sum_{i=0}^{p}\left(  \left[  t^{i}\right]  g\right)  \left(  \left[
t^{k-i}\right]  h\right)  +\underbrace{\sum_{i=p+1}^{k}0\left(  \left[
t^{k-i}\right]  h\right)  }_{=0}=\sum_{i=0}^{p}\left(  \left[  t^{i}\right]
g\right)  \left(  \left[  t^{k-i}\right]  h\right) \nonumber\\
&  =\sum_{i=0}^{p-1}\left(  \left[  t^{i}\right]  g\right)
\underbrace{\left(  \left[  t^{k-i}\right]  h\right)  }%
_{\substack{=0\\\text{(by (\ref{pf.lem.poly.1.c.2u}))}}}+\left(  \left[
t^{p}\right]  g\right)  \left(  \left[  t^{k-p}\right]  h\right) \nonumber\\
&  \ \ \ \ \ \ \ \ \ \ \left(  \text{here, we have split off the addend for
}i=p\text{ from the sum}\right) \nonumber\\
&  =\underbrace{\sum_{i=0}^{p-1}\left(  \left[  t^{i}\right]  g\right)
0}_{=0}+\left(  \left[  t^{p}\right]  g\right)  \left(  \left[  t^{k-p}%
\right]  h\right)  =\left(  \left[  t^{p}\right]  g\right)  \left(  \left[
t^{k-p}\right]  h\right)  . \label{pf.lem.poly.1.c.5}%
\end{align}

Now, forget that we fixed $k$. We thus have shown that every integer $k$
satisfying $k\geq p+q$ satisfies (\ref{pf.lem.poly.1.c.5}).

Now, let $k$ be an integer satisfying $k>p+q$. From $k>p+q$, we obtain $k-p>q$
and thus $\left[  t^{k-p}\right]  h=0$ (by (\ref{pf.lem.poly.1.c.2}) (applied
to $i=k-p$)). But from $k>p+q$, we also obtain $k\geq p+q$. Hence,
(\ref{pf.lem.poly.1.c.5}) yields%
\[
\left[  t^{k}\right]  \left(  gh\right)  =\left(  \left[  t^{p}\right]
g\right)  \underbrace{\left(  \left[  t^{k-p}\right]  h\right)  }_{=0}=0.
\]

Now, forget that we fixed $k$. We thus have shown that $\left[  t^{k}\right]
\left(  gh\right)  =0$ for every integer $k>p+q$. In other words, $gh$ is a
polynomial of degree $\leq p+q$.

It remains to prove that $\left[  t^{p+q}\right]  \left(  gh\right)  =\left(
\left[  t^{p}\right]  g\right)  \cdot\left(  \left[  t^{q}\right]  h\right)
$. But (\ref{pf.lem.poly.1.c.5}) (applied to $k=p+q$) yields
\[
\left[  t^{p+q}\right]  \left(  gh\right)  =\left(  \left[  t^{p}\right]
g\right)  \left(  \underbrace{\left[  t^{p+q-p}\right]  }_{=\left[
t^{q}\right]  }h\right)  =\left(  \left[  t^{p}\right]  g\right)  \left(
\left[  t^{q}\right]  h\right)  =\left(  \left[  t^{p}\right]  g\right)
\cdot\left(  \left[  t^{q}\right]  h\right)  .
\]
This completes the proof of Lemma \ref{lem.poly.1} \textbf{(c)}.
\end{proof}

\begin{proof}
[Proof of Lemma \ref{lem.ta+b.prod}.]We shall prove Lemma \ref{lem.ta+b.prod}
by induction over $n$:

\textit{Induction base:} Lemma \ref{lem.ta+b.prod} holds in the case when
$n=0$\ \ \ \ \footnote{\textit{Proof.} Let $n$, $\left(  x_{1},x_{2}%
,\ldots,x_{n}\right)  $, $\left(  y_{1},y_{2},\ldots,y_{n}\right)  $ and $f$
be as in Lemma \ref{lem.ta+b.prod}. Assume that $n=0$. We must prove that
Lemma \ref{lem.ta+b.prod} holds in this case.
\par
We have
\begin{align*}
f  &  =\prod_{i=1}^{n}\left(  tx_{i}+y_{i}\right)  =\prod_{i=1}^{0}\left(
tx_{i}+y_{i}\right)  \ \ \ \ \ \ \ \ \ \ \left(  \text{since }n=0\right) \\
&  =\left(  \text{empty product}\right)  =1.
\end{align*}
Hence, $f$ is a polynomial of degree $\leq0$ (since $1$ is clearly a
polynomial of degree $\leq0$). In other words, $f$ is a polynomial of degree
$\leq n$ (since $n=0$). Thus, Lemma \ref{lem.ta+b.prod} \textbf{(a)} holds.
\par
Since $n=0$ and $f=1$, we have $\left[  t^{n}\right]  f=\left[  t^{0}\right]
1=1$. Comparing this with%
\begin{align*}
\prod_{i=1}^{n}x_{i}  &  =\prod_{i=1}^{0}x_{i}\ \ \ \ \ \ \ \ \ \ \left(
\text{since }n=0\right) \\
&  =\left(  \text{empty product}\right)  =1,
\end{align*}
we obtain $\left[  t^{n}\right]  f=\prod_{i=1}^{n}x_{i}$. Thus, Lemma
\ref{lem.ta+b.prod} \textbf{(b)} holds.
\par
Since $f=1$, we have $\left[  t^{0}\right]  f=\left[  t^{0}\right]  1=1$.
Comparing this with%
\begin{align*}
\prod_{i=1}^{n}y_{i}  &  =\prod_{i=1}^{0}y_{i}\ \ \ \ \ \ \ \ \ \ \left(
\text{since }n=0\right) \\
&  =\left(  \text{empty product}\right)  =1,
\end{align*}
we obtain $\left[  t^{0}\right]  f=\prod_{i=1}^{n}y_{i}$. Thus, Lemma
\ref{lem.ta+b.prod} \textbf{(c)} holds.
\par
Now, we have shown that all three parts \textbf{(a)}, \textbf{(b)} and
\textbf{(c)} of Lemma \ref{lem.ta+b.prod} hold. In other words, Lemma
\ref{lem.ta+b.prod} holds. Qed.}. This completes the induction base.

\textit{Induction step:} Let $N$ be a positive integer. Assume that Lemma
\ref{lem.ta+b.prod} holds in the case when $n=N-1$. We now must prove that
Lemma \ref{lem.ta+b.prod} holds in the case when $n=N$. In other words, we
must prove the following claim:

\begin{statement}
\textit{Claim 1:} Let $x_{1},x_{2},\ldots,x_{N}$ be $N$ elements of
$\mathbb{K}$. Let $y_{1},y_{2},\ldots,y_{N}$ be $N$ elements of $\mathbb{K}$.
Define a polynomial $f\in\mathbb{K}\left[  t\right]  $ by $f=\prod_{i=1}%
^{N}\left(  tx_{i}+y_{i}\right)  $. \medskip

\textbf{(a)} Then, $f$ is a polynomial of degree $\leq N$. \medskip

\textbf{(b)} We have $\left[  t^{N}\right]  f=\prod_{i=1}^{N}x_{i}$. \medskip

\textbf{(c)} We have $\left[  t^{0}\right]  f=\prod_{i=1}^{N}y_{i}$.
\end{statement}

But we have assumed that Lemma \ref{lem.ta+b.prod} holds in the case when
$n=N-1$. In other words, the following fact holds:

\begin{statement}
\textit{Fact 2:} Let $x_{1},x_{2},\ldots,x_{N-1}$ be $N-1$ elements of
$\mathbb{K}$. Let $y_{1},y_{2},\ldots,y_{N-1}$ be $N-1$ elements of
$\mathbb{K}$. Define a polynomial $f\in\mathbb{K}\left[  t\right]  $ by
$f=\prod_{i=1}^{N-1}\left(  tx_{i}+y_{i}\right)  $. \medskip

\textbf{(a)} Then, $f$ is a polynomial of degree $\leq N-1$. \medskip

\textbf{(b)} We have $\left[  t^{N-1}\right]  f=\prod_{i=1}^{N-1}x_{i}$.
\medskip

\textbf{(c)} We have $\left[  t^{0}\right]  f=\prod_{i=1}^{N-1}y_{i}$.
\end{statement}

Let us now prove Claim 1:

[\textit{Proof of Claim 1:} Define a polynomial $g\in\mathbb{K}\left[
t\right]  $ by $g=\prod_{i=1}^{N-1}\left(  tx_{i}+y_{i}\right)  $. Then,%
\begin{align*}
f  &  =\prod_{i=1}^{N}\left(  tx_{i}+y_{i}\right)  =\underbrace{\left(
\prod_{i=1}^{N-1}\left(  tx_{i}+y_{i}\right)  \right)  }_{=g}\cdot\left(
tx_{N}+y_{N}\right) \\
&  \ \ \ \ \ \ \ \ \ \ \left(
\begin{array}
[c]{c}%
\text{here, we have split off the factor}\\
\text{for }i=N\text{ from the product}%
\end{array}
\right) \\
&  =g\cdot\left(  tx_{N}+y_{N}\right)  .
\end{align*}

But recall that $g=\prod_{i=1}^{N-1}\left(  tx_{i}+y_{i}\right)  $. Hence, we
can apply Fact 2 to $g$ instead of $f$.

From Fact 2 \textbf{(a)} (applied to $g$ instead of $f$), we conclude that $g$
is a polynomial of degree $\leq N-1$.

The polynomial $tx_{N}+y_{N}\in\mathbb{K}\left[  t\right]  $ clearly has
degree $\leq1$ and satisfies \newline$\left[  t^{1}\right]  \left(
tx_{N}+y_{N}\right)  =x_{N}$ and $\left[  t^{0}\right]  \left(  tx_{N}%
+y_{N}\right)  =y_{N}$. Now, Lemma \ref{lem.poly.1} \textbf{(c)} (applied to
$h=tx_{N}+y_{N}$, $p=N-1$ and $q=1$) yields that $g\cdot\left(  tx_{N}%
+y_{N}\right)  $ is a polynomial of degree $\leq\left(  N-1\right)  +1$ and
satisfies $\left[  t^{\left(  N-1\right)  +1}\right]  \left(  g\cdot\left(
tx_{N}+y_{N}\right)  \right)  =\left(  \left[  t^{N-1}\right]  g\right)
\cdot\left(  \left[  t^{1}\right]  \left(  tx_{N}+y_{N}\right)  \right)  $.
Since $g\cdot\left(  tx_{N}+y_{N}\right)  =f$, $\left(  N-1\right)  +1=N$ and
$\left[  t^{1}\right]  \left(  tx_{N}+y_{N}\right)  =x_{N}$, this rewrites as
follows: The polynomial $f$ is a polynomial of degree $\leq N$ and satisfies
$\left[  t^{N}\right]  f=\left(  \left[  t^{N-1}\right]  g\right)  \cdot
x_{N}$. In particular, $f$ is a polynomial of degree $\leq N$. This proves
Claim 1 \textbf{(a)}.

Fact 2 \textbf{(b)} (applied to $g$ instead of $f$) shows that $\left[
t^{N-1}\right]  g=\prod_{i=1}^{N-1}x_{i}$. But
\begin{align*}
\prod_{i=1}^{N}x_{i}  &  =\underbrace{\left(  \prod_{i=1}^{N-1}x_{i}\right)
}_{=\left[  t^{N-1}\right]  g}\cdot x_{N}\ \ \ \ \ \ \ \ \ \ \left(
\begin{array}
[c]{c}%
\text{here, we have split off the factor}\\
\text{for }i=N\text{ from the product}%
\end{array}
\right) \\
&  =\left(  \left[  t^{N-1}\right]  g\right)  \cdot x_{N}.
\end{align*}
Comparing this with $\left[  t^{N}\right]  f=\left(  \left[  t^{N-1}\right]
g\right)  \cdot x_{N}$, we obtain $\left[  t^{N}\right]  f=\prod_{i=1}%
^{N}x_{i}$. Thus, Claim 1 \textbf{(b)} holds.

Fact 2 \textbf{(c)} (applied to $g$ instead of $f$) shows that $\left[
t^{0}\right]  g=\prod_{i=1}^{N-1}y_{i}$. But Lemma \ref{lem.poly.1}
\textbf{(b)} (applied to $h=tx_{N}+y_{N}$) yields $\left[  t^{0}\right]
\left(  g\cdot\left(  tx_{N}+y_{N}\right)  \right)  =\left(  \left[
t^{0}\right]  g\right)  \left(  \left[  t^{0}\right]  \left(  tx_{N}%
+y_{N}\right)  \right)  $. Since $g\cdot\left(  tx_{N}+y_{N}\right)  =f$ and
$\left[  t^{0}\right]  \left(  tx_{N}+y_{N}\right)  =y_{N}$, this rewrites as
$\left[  t^{0}\right]  f=\left(  \left[  t^{0}\right]  g\right)  y_{N}$.
Comparing this with%
\begin{align*}
\prod_{i=1}^{N}y_{i}  &  =\underbrace{\left(  \prod_{i=1}^{N-1}y_{i}\right)
}_{=\left[  t^{0}\right]  g}\cdot y_{N}\ \ \ \ \ \ \ \ \ \ \left(
\begin{array}
[c]{c}%
\text{here, we have split off the factor}\\
\text{for }i=N\text{ from the product}%
\end{array}
\right) \\
&  =\left(  \left[  t^{0}\right]  g\right)  y_{N},
\end{align*}
we obtain $\left[  t^{0}\right]  f=\prod_{i=1}^{N}y_{i}$. Thus, Claim 1
\textbf{(c)} is proven.

We now have proven all three parts \textbf{(a)}, \textbf{(b)} and \textbf{(c)}
of Claim 1. Thus, Claim 1 holds.]

We have now proven Claim 1. As we have explained, this completes the induction
step. Thus, Lemma \ref{lem.ta+b.prod} is proven by induction.
\end{proof}
\end{verlong}

\begin{proof}
[Proof of Proposition \ref{prop.ta+b}.]Write the $n\times n$-matrix $A$ in the
form $A=\left(  a_{i,j}\right)  _{1\leq i\leq n,\ 1\leq j\leq n}$. Thus,
$a_{i,j}\in\mathbb{K}$ for every $\left(  i,j\right)  \in\left\{
1,2,\ldots,n\right\}  ^{2}$ (since $A\in\mathbb{K}^{n\times n}$).

Write the $n\times n$-matrix $B$ in the form $B=\left(  b_{i,j}\right)
_{1\leq i\leq n,\ 1\leq j\leq n}$. Thus, $b_{i,j}\in\mathbb{K}$ for every
$\left(  i,j\right)  \in\left\{  1,2,\ldots,n\right\}  ^{2}$ (since
$B\in\mathbb{K}^{n\times n}$).

For every $\sigma\in S_{n}$, define a polynomial $f_{\sigma}\in\mathbb{K}%
\left[  t\right]  $ by%
\begin{equation}
f_{\sigma}=\prod_{i=1}^{n}\left(  ta_{i,\sigma\left(  i\right)  }%
+b_{i,\sigma\left(  i\right)  }\right)  . \label{pf.prop.ta+b.fsigma=}%
\end{equation}
The following holds:

\begin{statement}
\textit{Fact 1:} For every $\sigma\in S_{n}$, the polynomial $f_{\sigma}$ is a
polynomial of degree $\leq n$.
\end{statement}

[\textit{Proof of Fact 1:} Let $\sigma\in S_{n}$. Then, Lemma
\ref{lem.ta+b.prod} \textbf{(a)} (applied to $a_{i,\sigma\left(  i\right)  }$,
$b_{i,\sigma\left(  i\right)  }$ and $f_{\sigma}$ instead of $x_{i}$, $y_{i}$
and $f$) shows that $f_{\sigma}$ is a polynomial of degree $\leq n$. This
proves Fact 1.]

\begin{vershort}
From $A=\left(  a_{i,j}\right)  _{1\leq i\leq n,\ 1\leq j\leq n}$ and
$B=\left(  b_{i,j}\right)  _{1\leq i\leq n,\ 1\leq j\leq n}$, we obtain
$tA+B=\left(  ta_{i,j}+b_{i,j}\right)  _{1\leq i\leq n,\ 1\leq j\leq n}$.
Hence,%
\begin{align*}
\det\left(  tA+B\right)   &  =\sum_{\sigma\in S_{n}}\left(  -1\right)
^{\sigma}\underbrace{\prod_{i=1}^{n}\left(  ta_{i,\sigma\left(  i\right)
}+b_{i,\sigma\left(  i\right)  }\right)  }_{\substack{=f_{\sigma}\\\text{(by
(\ref{pf.prop.ta+b.fsigma=}))}}}\\
&  \ \ \ \ \ \ \ \ \ \ \left(
\begin{array}
[c]{c}%
\text{by (\ref{eq.det}), applied to }\mathbb{K}\left[  t\right]  \text{,
}tA+B\text{ and }ta_{i,j}+b_{i,j}\\
\text{instead of }\mathbb{K}\text{, }A\text{ and }a_{i,j}%
\end{array}
\right) \\
&  =\sum_{\sigma\in S_{n}}\left(  -1\right)  ^{\sigma}f_{\sigma}.
\end{align*}
Hence, $\det\left(  tA+B\right)  $ is a $\mathbb{K}$-linear combination of the
polynomials $f_{\sigma}$ for $\sigma\in S_{n}$. Since all of these polynomials
are polynomials of degree $\leq n$ (by Fact 1), we thus conclude that
$\det\left(  tA+B\right)  $ is a $\mathbb{K}$-linear combination of
polynomials of degree $\leq n$. Thus, $\det\left(  tA+B\right)  $ is itself a
polynomial of degree $\leq n$. This proves Proposition \ref{prop.ta+b}
\textbf{(a)}. \medskip
\end{vershort}

\begin{verlong}
We have%
\begin{align*}
t\underbrace{A}_{=\left(  a_{i,j}\right)  _{1\leq i\leq n,\ 1\leq j\leq n}%
}+\underbrace{B}_{=\left(  b_{i,j}\right)  _{1\leq i\leq n,\ 1\leq j\leq n}}
&  =\underbrace{t\left(  a_{i,j}\right)  _{1\leq i\leq n,\ 1\leq j\leq n}%
}_{=\left(  ta_{i,j}\right)  _{1\leq i\leq n,\ 1\leq j\leq n}}+\left(
b_{i,j}\right)  _{1\leq i\leq n,\ 1\leq j\leq n}\\
&  =\left(  ta_{i,j}\right)  _{1\leq i\leq n,\ 1\leq j\leq n}+\left(
b_{i,j}\right)  _{1\leq i\leq n,\ 1\leq j\leq n}\\
&  =\left(  ta_{i,j}+b_{i,j}\right)  _{1\leq i\leq n,\ 1\leq j\leq n}.
\end{align*}
Hence, (\ref{eq.det}) (applied to $\mathbb{K}\left[  t\right]  $, $tA+B$ and
$ta_{i,j}+b_{i,j}$ instead of $\mathbb{K}$, $A$ and $a_{i,j}$) shows that%
\[
\det\left(  tA+B\right)  =\sum_{\sigma\in S_{n}}\left(  -1\right)  ^{\sigma
}\underbrace{\prod_{i=1}^{n}\left(  ta_{i,\sigma\left(  i\right)
}+b_{i,\sigma\left(  i\right)  }\right)  }_{\substack{=f_{\sigma}\\\text{(by
(\ref{pf.prop.ta+b.fsigma=}))}}}=\sum_{\sigma\in S_{n}}\left(  -1\right)
^{\sigma}f_{\sigma}.
\]
Hence, $\det\left(  tA+B\right)  $ is a $\mathbb{K}$-linear combination of the
polynomials $f_{\sigma}$ for $\sigma\in S_{n}$. Since all of these polynomials
are polynomials of degree $\leq n$ (by Fact 1), we thus conclude that
$\det\left(  tA+B\right)  $ is a $\mathbb{K}$-linear combination of
polynomials of degree $\leq n$. Thus, $\det\left(  tA+B\right)  $ is itself a
polynomial of degree $\leq n$ (since any $\mathbb{K}$-linear combination of
polynomials of degree $\leq n$ is itself a polynomial of degree $\leq n$).
This proves Proposition \ref{prop.ta+b} \textbf{(a)}. \medskip
\end{verlong}

\textbf{(b)} We have%
\begin{align*}
\left[  t^{0}\right]  \underbrace{\left(  \det\left(  tA+B\right)  \right)
}_{=\sum_{\sigma\in S_{n}}\left(  -1\right)  ^{\sigma}f_{\sigma}}  &  =\left[
t^{0}\right]  \left(  \sum_{\sigma\in S_{n}}\left(  -1\right)  ^{\sigma
}f_{\sigma}\right)  =\sum_{\sigma\in S_{n}}\left(  -1\right)  ^{\sigma
}\underbrace{\left[  t^{0}\right]  f_{\sigma}}_{\substack{=\prod_{i=1}%
^{n}b_{i,\sigma\left(  i\right)  }\\\text{(by Lemma \ref{lem.ta+b.prod}
\textbf{(c)}}\\\text{(applied to }a_{i,\sigma\left(  i\right)  }\text{,
}b_{i,\sigma\left(  i\right)  }\text{ and }f_{\sigma}\\\text{instead of }%
x_{i}\text{, }y_{i}\text{ and }f\text{))}}}\\
&  =\sum_{\sigma\in S_{n}}\left(  -1\right)  ^{\sigma}\prod_{i=1}%
^{n}b_{i,\sigma\left(  i\right)  }.
\end{align*}
Comparing this with%
\[
\det B=\sum_{\sigma\in S_{n}}\left(  -1\right)  ^{\sigma}\prod_{i=1}%
^{n}b_{i,\sigma\left(  i\right)  }\ \ \ \ \ \ \ \ \ \ \left(
\begin{array}
[c]{c}%
\text{by (\ref{eq.det}), applied to }B\text{ and }b_{i,j}\\
\text{instead of }A\text{ and }a_{i,j}%
\end{array}
\right)  ,
\]
we obtain $\left[  t^{0}\right]  \left(  \det\left(  tA+B\right)  \right)
=\det B$. This proves Proposition \ref{prop.ta+b} \textbf{(b)}. \medskip

\textbf{(c)} We have%
\begin{align*}
\left[  t^{n}\right]  \underbrace{\left(  \det\left(  tA+B\right)  \right)
}_{=\sum_{\sigma\in S_{n}}\left(  -1\right)  ^{\sigma}f_{\sigma}}  &  =\left[
t^{n}\right]  \left(  \sum_{\sigma\in S_{n}}\left(  -1\right)  ^{\sigma
}f_{\sigma}\right)  =\sum_{\sigma\in S_{n}}\left(  -1\right)  ^{\sigma
}\underbrace{\left[  t^{n}\right]  f_{\sigma}}_{\substack{=\prod_{i=1}%
^{n}a_{i,\sigma\left(  i\right)  }\\\text{(by Lemma \ref{lem.ta+b.prod}
\textbf{(b)}}\\\text{(applied to }a_{i,\sigma\left(  i\right)  }\text{,
}b_{i,\sigma\left(  i\right)  }\text{ and }f_{\sigma}\\\text{instead of }%
x_{i}\text{, }y_{i}\text{ and }f\text{))}}}\\
&  =\sum_{\sigma\in S_{n}}\left(  -1\right)  ^{\sigma}\prod_{i=1}%
^{n}a_{i,\sigma\left(  i\right)  }.
\end{align*}
Comparing this with (\ref{eq.det}), we obtain $\left[  t^{n}\right]  \left(
\det\left(  tA+B\right)  \right)  =\det A$. This proves Proposition
\ref{prop.ta+b} \textbf{(c)}.
\end{proof}

\begin{vershort}
\begin{proof}
[Proof of Corollary \ref{cor.chiA}.]The definition of $\chi_{A}$ yields
\newline$\chi_{A}=\det\left(  \underbrace{tI_{n}-A}_{=tI_{n}+\left(
-A\right)  }\right)  =\det\left(  tI_{n}+\left(  -A\right)  \right)  $. Hence,
Corollary \ref{cor.chiA} follows from Proposition \ref{prop.ta+b} (applied to
$I_{n}$ and $-A$ instead of $A$ and $B$). (For part \textbf{(b)}, we need the
additional observation that $\det\left(  -A\right)  =\left(  -1\right)
^{n}\det A$.)
\end{proof}
\end{vershort}

\begin{verlong}
\begin{proof}
[Proof of Corollary \ref{cor.chiA}.]The definition of $\chi_{A}$ yields%
\[
\chi_{A}=\det\left(  \underbrace{tI_{n}-A}_{=tI_{n}+\left(  -A\right)
}\right)  =\det\left(  tI_{n}+\left(  -A\right)  \right)  .
\]

\textbf{(a)} Proposition \ref{prop.ta+b} \textbf{(a)} (applied to $I_{n}$ and
$-A$ instead of $A$ and $B$) yields that $\det\left(  tI_{n}+\left(
-A\right)  \right)  \in\mathbb{K}\left[  t\right]  $ is a polynomial of degree
$\leq n$ in $t$. In other words, $\chi_{A}$ is a polynomial of degree $\leq n$
in $t$ (since $\chi_{A}=\det\left(  tI_{n}+\left(  -A\right)  \right)  $).
This proves Corollary \ref{cor.chiA} \textbf{(a)}. \medskip

\textbf{(b)} Since $A$ is an $n\times n$-matrix, we have $\det\left(
-A\right)  =\left(  -1\right)  ^{n}\det A$\ \ \ \ \footnote{This is a
well-known fact, but let me give a \textit{proof} for the sake of
completeness:
\par
From \cite[Proposition 6.12]{detnotes}, we obtain that $\det\left(  \lambda
A\right)  =\lambda^{n}\det A$ for every $\lambda\in\mathbb{K}$. Applying this
to $\lambda=-1$, we obtain $\det\left(  \left(  -1\right)  A\right)  =\left(
-1\right)  ^{n}\det A$. Since $\left(  -1\right)  A=-A$, this rewrites as
$\det\left(  -A\right)  =\left(  -1\right)  ^{n}\det A$. Qed.}.

Proposition \ref{prop.ta+b} \textbf{(b)} (applied to $I_{n}$ and $-A$ instead
of $A$ and $B$) yields that \newline$\left[  t^{0}\right]  \left(  \det\left(
tI_{n}+\left(  -A\right)  \right)  \right)  =\det\left(  -A\right)  $. In
other words, $\left[  t^{0}\right]  \chi_{A}=\left(  -1\right)  ^{n}\det A$
(since $\chi_{A}=\det\left(  tI_{n}+\left(  -A\right)  \right)  $). Comparing
this with $\det\left(  -A\right)  =\left(  -1\right)  ^{n}\det A$, we obtain
$\left[  t^{0}\right]  \chi_{A}=\left(  -1\right)  ^{n}\det A$. This proves
Corollary \ref{cor.chiA} \textbf{(b)}. \medskip

\textbf{(c)} Proposition \ref{prop.ta+b} \textbf{(c)} (applied to $I_{n}$ and
$-A$ instead of $A$ and $B$) yields that $\left[  t^{n}\right]  \left(
\det\left(  tI_{n}+\left(  -A\right)  \right)  \right)  =\det\left(
I_{n}\right)  =1$. In other words, $\left[  t^{n}\right]  \chi_{A}=1$ (since
$\chi_{A}=\det\left(  tI_{n}+\left(  -A\right)  \right)  $). This proves
Corollary \ref{cor.chiA} \textbf{(c)}.
\end{proof}
\end{verlong}

Let me state one more trivial observation as a corollary:

\begin{corollary}
\label{cor.CH.chiA}Let $n\in\mathbb{N}$. Let $A\in\mathbb{K}^{n\times n}$. For
every $j\in\mathbb{Z}$, define an element $c_{j}\in\mathbb{K}$ by
$c_{j}=\left[  t^{n-j}\right]  \chi_{A}$. Then, $\chi_{A}=\sum_{k=0}%
^{n}c_{n-k}t^{k}$.
\end{corollary}

\begin{vershort}
\begin{proof}
[Proof of Corollary \ref{cor.CH.chiA}.]For every $k\in\mathbb{Z}$, the
definition of $c_{n-k}$ yields%
\begin{equation}
c_{n-k}=\left[  t^{n-\left(  n-k\right)  }\right]  \chi_{A}=\left[
t^{k}\right]  \chi_{A}. \label{pf.cor.CH.chiA.1}%
\end{equation}

We know that $\chi_{A}\in\mathbb{K}\left[  t\right]  $ is a polynomial of
degree $\leq n$ in $t$ (by Corollary \ref{cor.chiA} \textbf{(a)}). Hence,%
\[
\chi_{A}=\sum_{k=0}^{n}\underbrace{\left(  \left[  t^{k}\right]  \chi
_{A}\right)  }_{\substack{=c_{n-k}\\\text{(by (\ref{pf.cor.CH.chiA.1}))}%
}}t^{k}=\sum_{k=0}^{n}c_{n-k}t^{k}.
\]
This proves Corollary \ref{cor.CH.chiA}.
\end{proof}
\end{vershort}

\begin{verlong}
\begin{proof}
[Proof of Corollary \ref{cor.CH.chiA}.]For every $k\in\mathbb{Z}$, we have%
\begin{align}
c_{n-k}  &  =\underbrace{\left[  t^{n-\left(  n-k\right)  }\right]
}_{=\left[  t^{k}\right]  }\chi_{A}\ \ \ \ \ \ \ \ \ \ \left(  \text{by the
definition of }c_{n-k}\right) \nonumber\\
&  =\left[  t^{k}\right]  \chi_{A}. \label{pf.lem.CH.D.a.cn-k}%
\end{align}

We know that $\chi_{A}\in\mathbb{K}\left[  t\right]  $ is a polynomial of
degree $\leq n$ in $t$ (by Corollary \ref{cor.chiA} \textbf{(a)}). Hence,
\begin{equation}
\left[  t^{k}\right]  \chi_{A}=0\ \ \ \ \ \ \ \ \ \ \text{for every integer
}k>n. \label{pf.lem.CH.D.a.1}%
\end{equation}

Now, recall that every polynomial $f\in\mathbb{K}\left[  t\right]  $ satisfies
$f=\sum_{k\in\mathbb{N}}\left(  \left[  t^{k}\right]  f\right)  t^{k}$.
Applying this to $f=\chi_{A}$, we obtain%
\begin{align*}
\chi_{A}  &  =\sum_{k\in\mathbb{N}}\left(  \left[  t^{k}\right]  \chi
_{A}\right)  t^{k}=\sum_{k=0}^{n}\left(  \left[  t^{k}\right]  \chi
_{A}\right)  t^{k}+\sum_{k=n+1}^{\infty}\underbrace{\left(  \left[
t^{k}\right]  \chi_{A}\right)  }_{\substack{=0\\\text{(by
(\ref{pf.lem.CH.D.a.1})}\\\text{(since }k\geq n+1>n\text{))}}}t^{k}\\
&  =\sum_{k=0}^{n}\left(  \left[  t^{k}\right]  \chi_{A}\right)
t^{k}+\underbrace{\sum_{k=n+1}^{\infty}0t^{k}}_{=0}=\sum_{k=0}^{n}%
\underbrace{\left(  \left[  t^{k}\right]  \chi_{A}\right)  }%
_{\substack{=c_{n-k}\\\text{(by (\ref{pf.lem.CH.D.a.cn-k}))}}}t^{k}=\sum
_{k=0}^{n}c_{n-k}t^{k}.
\end{align*}
This proves Corollary \ref{cor.CH.chiA}.
\end{proof}
\end{verlong}

\subsection{Reminders on the adjugate}

Let us now briefly introduce the adjugate of a matrix and state some of its properties.

We first recall the definitions (mostly quoting them from \cite[Chapter
6]{detnotes}):

\begin{definition}
\label{def.submatrix}Let $n\in\mathbb{N}$ and $m\in\mathbb{N}$. Let $A=\left(
a_{i,j}\right)  _{1\leq i\leq n,\ 1\leq j\leq m}$ be an $n\times m$-matrix.
Let $i_{1},i_{2},\ldots,i_{u}$ be some elements of $\left\{  1,2,\ldots
,n\right\}  $; let $j_{1},j_{2},\ldots,j_{v}$ be some elements of $\left\{
1,2,\ldots,m\right\}  $. Then, we define $\operatorname*{sub}\nolimits_{i_{1}%
,i_{2},\ldots,i_{u}}^{j_{1},j_{2},\ldots,j_{v}}A$ to be the $u\times v$-matrix
$\left(  a_{i_{x},j_{y}}\right)  _{1\leq x\leq u,\ 1\leq y\leq v}$.
\end{definition}

\begin{definition}
\label{def.hat-omit}Let $n\in\mathbb{N}$. Let $a_{1},a_{2},\ldots,a_{n}$ be
$n$ objects. Let $i\in\left\{  1,2,\ldots,n\right\}  $. Then, $\left(
a_{1},a_{2},\ldots,\widehat{a_{i}},\ldots,a_{n}\right)  $ shall mean the list
$\left(  a_{1},a_{2},\ldots,a_{i-1},a_{i+1},a_{i+2},\ldots,a_{n}\right)  $
(that is, the list $\left(  a_{1},a_{2},\ldots,a_{n}\right)  $ with its $i$-th
entry removed). (Thus, the \textquotedblleft hat\textquotedblright\ over the
$a_{i}$ means that this $a_{i}$ is being omitted from the list.)

For example, $\left(  1^{2},2^{2},\ldots,\widehat{5}^{2},\ldots,8^{2}\right)
=\left(  1^{2},2^{2},3^{2},4^{2},6^{2},7^{2},8^{2}\right)  $.
\end{definition}

\begin{definition}
\label{def.submatrix.minor}Let $n\in\mathbb{N}$ and $m\in\mathbb{N}$. Let $A$
be an $n\times m$-matrix. For every $i\in\left\{  1,2,\ldots,n\right\}  $ and
$j\in\left\{  1,2,\ldots,m\right\}  $, we let $A_{\sim i,\sim j}$ be the
$\left(  n-1\right)  \times\left(  m-1\right)  $-matrix $\operatorname*{sub}%
\nolimits_{1,2,\ldots,\widehat{i},\ldots,n}^{1,2,\ldots,\widehat{j},\ldots
,m}A$. (Thus, $A_{\sim i,\sim j}$ is the matrix obtained from $A$ by crossing
out the $i$-th row and the $j$-th column.)
\end{definition}

\begin{definition}
\label{def.adj}Let $n\in\mathbb{N}$. Let $A$ be an $n\times n$-matrix. We
define a new $n\times n$-matrix $\operatorname*{adj}A$ by%
\[
\operatorname*{adj}A=\left(  \left(  -1\right)  ^{i+j}\det\left(  A_{\sim
j,\sim i}\right)  \right)  _{1\leq i\leq n,\ 1\leq j\leq n}.
\]
This matrix $\operatorname*{adj}A$ is called the \textit{adjugate} of the
matrix $A$.
\end{definition}

The main property of the adjugate is the following fact:

\begin{theorem}
\label{thm.adj.inverse}Let $n\in\mathbb{N}$. Let $A$ be an $n\times n$-matrix.
Then,%
\[
A\cdot\operatorname*{adj}A=\operatorname*{adj}A\cdot A=\det A\cdot I_{n}.
\]
(Recall that $I_{n}$ denotes the $n\times n$ identity matrix. Expressions such
as $\operatorname*{adj}A\cdot A$ and $\det A\cdot I_{n}$ have to be understood
as $\left(  \operatorname*{adj}A\right)  \cdot A$ and $\left(  \det A\right)
\cdot I_{n}$, respectively.)
\end{theorem}

Theorem \ref{thm.adj.inverse} appears in almost any text on linear algebra
that considers the adjugate; for example, it appears in \cite[Chapter Four,
Section III, Theorem 1.9]{Heffer14}, in \cite[Proposition 2.38]{Knapp2016}, in
\cite[Theorem 4.11]{Willia89} and in \cite[Theorem 6.100]{detnotes}. (Again,
most of these sources only state it in the case when $\mathbb{K}$ is a field,
but the proofs given apply in all generality. Different texts use different
notations. The source that is closest to my notations here is \cite{detnotes},
since Theorem \ref{thm.adj.inverse} above is a verbatim copy of \cite[Theorem
6.100]{detnotes}.)

Let us state a simple fact:

\begin{lemma}
\label{lem.adj.linear}Let $n\in\mathbb{N}$. Let $u$ and $v$ be two elements of
$\left\{  1,2,\ldots,n\right\}  $. Let $\lambda$ and $\mu$ be two elements of
$\mathbb{K}$. Let $A$ and $B$ be two $n\times n$-matrices. Then,%
\[
\left(  \lambda A+\mu B\right)  _{\sim u,\sim v}=\lambda A_{\sim u,\sim v}+\mu
B_{\sim u,\sim v}.
\]

\end{lemma}

\begin{vershort}
\begin{proof}
[Proof of Lemma \ref{lem.adj.linear}.]Obvious.
\end{proof}
\end{vershort}

\begin{verlong}
\begin{proof}
[Proof of Lemma \ref{lem.adj.linear}.]Let $\left(  p_{1},p_{2},\ldots
,p_{n-1}\right)  $ be the $\left(  n-1\right)  $-tuple $\left(  1,2,\ldots
,\widehat{u},\ldots,n\right)  $. Thus, $\left(  p_{1},p_{2},\ldots
,p_{n-1}\right)  =\left(  1,2,\ldots,\widehat{u},\ldots,n\right)  $.

Let $\left(  q_{1},q_{2},\ldots,q_{n-1}\right)  $ be the $\left(  n-1\right)
$-tuple $\left(  1,2,\ldots,\widehat{v},\ldots,n\right)  $. Thus, $\left(
q_{1},q_{2},\ldots,q_{n-1}\right)  =\left(  1,2,\ldots,\widehat{v}%
,\ldots,n\right)  $.

If $C=\left(  c_{i,j}\right)  _{1\leq i\leq n,\ 1\leq j\leq n}$ is any
$n\times n$-matrix, then
\begin{equation}
C_{\sim u,\sim v}=\left(  c_{p_{i},q_{j}}\right)  _{1\leq i\leq n-1,\ 1\leq
j\leq n-1} \label{pf.lem.adj.linear.1}%
\end{equation}
\footnote{\textit{Proof of (\ref{pf.lem.adj.linear.1}):} Let $C=\left(
c_{i,j}\right)  _{1\leq i\leq n,\ 1\leq j\leq n}$ be any $n\times n$-matrix.
The definition of $C_{\sim u,\sim v}$ yields
\begin{align*}
C_{\sim u,\sim v}  &  =\operatorname*{sub}\nolimits_{1,2,\ldots,\widehat{u}%
,\ldots,n}^{1,2,\ldots,\widehat{v},\ldots,n}C=\operatorname*{sub}%
\nolimits_{p_{1},p_{2},\ldots,p_{n-1}}^{1,2,\ldots,\widehat{v},\ldots,n}C\\
&  \ \ \ \ \ \ \ \ \ \ \left(  \text{since }\left(  1,2,\ldots,\widehat{u}%
,\ldots,n\right)  =\left(  p_{1},p_{2},\ldots,p_{n-1}\right)  \right) \\
&  =\operatorname*{sub}\nolimits_{p_{1},p_{2},\ldots,p_{n-1}}^{q_{1}%
,q_{2},\ldots,q_{n-1}}C\ \ \ \ \ \ \ \ \ \ \left(  \text{since }\left(
1,2,\ldots,\widehat{v},\ldots,n\right)  =\left(  q_{1},q_{2},\ldots
,q_{n-1}\right)  \right) \\
&  =\left(  c_{p_{x},q_{y}}\right)  _{1\leq x\leq n-1,\ 1\leq y\leq n-1}\\
&  \ \ \ \ \ \ \ \ \ \ \left(  \text{by the definition of }\operatorname*{sub}%
\nolimits_{p_{1},p_{2},\ldots,p_{n-1}}^{q_{1},q_{2},\ldots,q_{n-1}}C\text{,
since }C=\left(  c_{i,j}\right)  _{1\leq i\leq n,\ 1\leq j\leq n}\right) \\
&  =\left(  c_{p_{i},q_{j}}\right)  _{1\leq i\leq n-1,\ 1\leq j\leq n-1}%
\end{align*}
(here, we have renamed the index $\left(  x,y\right)  $ as $\left(
i,j\right)  $). This proves (\ref{pf.lem.adj.linear.1}).}.

Write the $n\times n$-matrix $A$ in the form $A=\left(  a_{i,j}\right)
_{1\leq i\leq n,\ 1\leq j\leq n}$. Write the $n\times n$-matrix $B$ in the
form $B=\left(  b_{i,j}\right)  _{1\leq i\leq n,\ 1\leq j\leq n}$.

From (\ref{pf.lem.adj.linear.1}) (applied to $A$ and $a_{i,j}$ instead of $C$
and $c_{i,j}$), we obtain $A_{\sim u,\sim v}=\left(  a_{p_{i},q_{j}}\right)
_{1\leq i\leq n-1,\ 1\leq j\leq n-1}$. From (\ref{pf.lem.adj.linear.1})
(applied to $B$ and $b_{i,j}$ instead of $C$ and $c_{i,j}$), we obtain
$B_{\sim u,\sim v}=\left(  b_{p_{i},q_{j}}\right)  _{1\leq i\leq n-1,\ 1\leq
j\leq n-1}$.

But
\begin{align*}
&  \lambda\underbrace{A}_{=\left(  a_{i,j}\right)  _{1\leq i\leq n,\ 1\leq
j\leq n}}+\,\mu\underbrace{B}_{=\left(  b_{i,j}\right)  _{1\leq i\leq
n,\ 1\leq j\leq n}}\\
&  =\underbrace{\lambda\left(  a_{i,j}\right)  _{1\leq i\leq n,\ 1\leq j\leq
n}}_{=\left(  \lambda a_{i,j}\right)  _{1\leq i\leq n,\ 1\leq j\leq n}%
}+\underbrace{\mu\left(  b_{i,j}\right)  _{1\leq i\leq n,\ 1\leq j\leq n}%
}_{=\left(  \mu b_{i,j}\right)  _{1\leq i\leq n,\ 1\leq j\leq n}}=\left(
\lambda a_{i,j}\right)  _{1\leq i\leq n,\ 1\leq j\leq n}+\left(  \mu
b_{i,j}\right)  _{1\leq i\leq n,\ 1\leq j\leq n}\\
&  =\left(  \lambda a_{i,j}+\mu b_{i,j}\right)  _{1\leq i\leq n,\ 1\leq j\leq
n}.
\end{align*}
Hence, (\ref{pf.lem.adj.linear.1}) (applied to $\lambda A+\mu B$ and $\lambda
a_{i,j}+\mu b_{i,j}$ instead of $C$ and $c_{i,j}$) shows that%
\[
\left(  \lambda A+\mu B\right)  _{\sim u,\sim v}=\left(  \lambda
a_{p_{i},q_{j}}+\mu b_{p_{i},q_{j}}\right)  _{1\leq i\leq n-1,\ 1\leq j\leq
n-1}.
\]
Comparing this with%
\begin{align*}
&  \lambda\underbrace{A_{\sim u,\sim v}}_{=\left(  a_{p_{i},q_{j}}\right)
_{1\leq i\leq n-1,\ 1\leq j\leq n-1}}+\,\mu\underbrace{B_{\sim u,\sim v}%
}_{=\left(  b_{p_{i},q_{j}}\right)  _{1\leq i\leq n-1,\ 1\leq j\leq n-1}}\\
&  =\underbrace{\lambda\left(  a_{p_{i},q_{j}}\right)  _{1\leq i\leq
n-1,\ 1\leq j\leq n-1}}_{=\left(  \lambda a_{p_{i},q_{j}}\right)  _{1\leq
i\leq n-1,\ 1\leq j\leq n-1}}+\underbrace{\mu\left(  b_{p_{i},q_{j}}\right)
_{1\leq i\leq n-1,\ 1\leq j\leq n-1}}_{=\left(  \mu b_{p_{i},q_{j}}\right)
_{1\leq i\leq n-1,\ 1\leq j\leq n-1}}\\
&  =\left(  \lambda a_{p_{i},q_{j}}\right)  _{1\leq i\leq n-1,\ 1\leq j\leq
n-1}+\left(  \mu b_{p_{i},q_{j}}\right)  _{1\leq i\leq n-1,\ 1\leq j\leq
n-1}\\
&  =\left(  \lambda a_{p_{i},q_{j}}+\mu b_{p_{i},q_{j}}\right)  _{1\leq i\leq
n-1,\ 1\leq j\leq n-1},
\end{align*}
we obtain $\left(  \lambda A+\mu B\right)  _{\sim u,\sim v}=\lambda A_{\sim
u,\sim v}+\mu B_{\sim u,\sim v}$. This proves Lemma \ref{lem.adj.linear}.
\end{proof}
\end{verlong}

Next, we prove a crucial, if simple, result:

\begin{proposition}
\label{prop.adj.poly}Let $n\in\mathbb{N}$. Let $A\in\mathbb{K}^{n\times n}$ be
an $n\times n$-matrix. Then, there exist $n$ matrices $D_{0},D_{1}%
,\ldots,D_{n-1}$ in $\mathbb{K}^{n\times n}$ such that%
\[
\operatorname*{adj}\left(  tI_{n}-A\right)  =\sum_{k=0}^{n-1}t^{k}%
D_{k}\ \ \ \ \ \ \ \ \ \ \text{in }\left(  \mathbb{K}\left[  t\right]
\right)  ^{n\times n}.
\]
(Here, of course, the matrix $D_{k}$ on the right hand side is understood as
an element of $\left(  \mathbb{K}\left[  t\right]  \right)  ^{n\times n}$.)
\end{proposition}

\begin{vershort}
\begin{proof}
[Proof of Proposition \ref{prop.adj.poly}.]Fix $\left(  u,v\right)
\in\left\{  1,2,\ldots,n\right\}  ^{2}$. Then, Proposition \ref{prop.ta+b}
\textbf{(a)} (applied to $n-1$, $\left(  I_{n}\right)  _{\sim u,\sim v}$ and
$\left(  -A\right)  _{\sim u,\sim v}$ instead of $n$, $A$ and $B$) shows that
\newline$\det\left(  t\left(  I_{n}\right)  _{\sim u,\sim v}+\left(
-A\right)  _{\sim u,\sim v}\right)  \in\mathbb{K}\left[  t\right]  $ is a
polynomial of degree $\leq n-1$ in $t$. In other words, there exists an
$n$-tuple $\left(  d_{u,v,0},d_{u,v,1},\ldots,d_{u,v,n-1}\right)
\in\mathbb{K}^{n}$ such that
\[
\det\left(  t\left(  I_{n}\right)  _{\sim u,\sim v}+\left(  -A\right)  _{\sim
u,\sim v}\right)  =\sum_{k=0}^{n-1}d_{u,v,k}t^{k}.
\]
Consider this $\left(  d_{u,v,0},d_{u,v,1},\ldots,d_{u,v,n-1}\right)  $. But
Lemma \ref{lem.adj.linear} (applied to $\mathbb{K}\left[  t\right]  $, $t$,
$1$, $I_{n}$ and $-A$ instead of $\mathbb{K}$, $\lambda$, $\mu$, $A$ and $B$)
yields $\left(  tI_{n}-A\right)  _{\sim u,\sim v}=t\left(  I_{n}\right)
_{\sim u,\sim v}+\left(  -A\right)  _{\sim u,\sim v}$ (after some
simplifications). Thus,%
\begin{align}
\det\left(  \left(  tI_{n}-A\right)  _{\sim u,\sim v}\right)   &  =\det\left(
t\left(  I_{n}\right)  _{\sim u,\sim v}+\left(  -A\right)  _{\sim u,\sim
v}\right) \nonumber\\
&  =\sum_{k=0}^{n-1}d_{u,v,k}t^{k}. \label{pf.prop.adj.poly.short.2}%
\end{align}

Now, forget that we fixed $\left(  u,v\right)  $. Thus, for every $\left(
u,v\right)  \in\left\{  1,2,\ldots,n\right\}  ^{2}$, we have constructed an
$n$-tuple $\left(  d_{u,v,0},d_{u,v,1},\ldots,d_{u,v,n-1}\right)
\in\mathbb{K}^{n}$ satisfying (\ref{pf.prop.adj.poly.short.2}).

Now, the definition of $\operatorname*{adj}\left(  tI_{n}-A\right)  $ yields
\begin{align*}
\operatorname*{adj}\left(  tI_{n}-A\right)   &  =\left(  \left(  -1\right)
^{i+j}\underbrace{\det\left(  \left(  tI_{n}-A\right)  _{\sim j,\sim
i}\right)  }_{\substack{=\sum_{k=0}^{n-1}d_{j,i,k}t^{k}\\\text{(by
(\ref{pf.prop.adj.poly.short.2}), applied to }\left(  u,v\right)  =\left(
j,i\right)  \text{)}}}\right)  _{1\leq i\leq n,\ 1\leq j\leq n}\\
&  =\left(  \underbrace{\left(  -1\right)  ^{i+j}\sum_{k=0}^{n-1}%
d_{j,i,k}t^{k}}_{=\sum_{k=0}^{n-1}t^{k}\left(  -1\right)  ^{i+j}d_{j,i,k}%
}\right)  _{1\leq i\leq n,\ 1\leq j\leq n}=\left(  \sum_{k=0}^{n-1}%
t^{k}\left(  -1\right)  ^{i+j}d_{j,i,k}\right)  _{1\leq i\leq n,\ 1\leq j\leq
n}.
\end{align*}
Comparing this with%
\[
\sum_{k=0}^{n-1}t^{k}\left(  \left(  -1\right)  ^{i+j}d_{j,i,k}\right)
_{1\leq i\leq n,\ 1\leq j\leq n}=\left(  \sum_{k=0}^{n-1}t^{k}\left(
-1\right)  ^{i+j}d_{j,i,k}\right)  _{1\leq i\leq n,\ 1\leq j\leq n},
\]
we obtain $\operatorname*{adj}\left(  tI_{n}-A\right)  =\sum_{k=0}^{n-1}%
t^{k}\left(  \left(  -1\right)  ^{i+j}d_{j,i,k}\right)  _{1\leq i\leq
n,\ 1\leq j\leq n}$. Hence, there exist $n$ matrices $D_{0},D_{1}%
,\ldots,D_{n-1}$ in $\mathbb{K}^{n\times n}$ such that%
\[
\operatorname*{adj}\left(  tI_{n}-A\right)  =\sum_{k=0}^{n-1}t^{k}%
D_{k}\ \ \ \ \ \ \ \ \ \ \text{in }\left(  \mathbb{K}\left[  t\right]
\right)  ^{n\times n}%
\]
(namely, $D_{k}=\left(  \left(  -1\right)  ^{i+j}d_{j,i,k}\right)  _{1\leq
i\leq n,\ 1\leq j\leq n}$ for every $k\in\left\{  0,1,\ldots,n-1\right\}  $).
This proves Proposition \ref{prop.adj.poly}.
\end{proof}
\end{vershort}

\begin{verlong}
\begin{proof}
[Proof of Proposition \ref{prop.adj.poly}.]Fix $\left(  u,v\right)
\in\left\{  1,2,\ldots,n\right\}  ^{2}$. We shall define an $n$-tuple
\newline$\left(  d_{u,v,0},d_{u,v,1},\ldots,d_{u,v,n-1}\right)  \in
\mathbb{K}^{n}$ satisfying%
\begin{equation}
\det\left(  \left(  tI_{n}-A\right)  _{\sim u,\sim v}\right)  =\sum
_{k=0}^{n-1}d_{u,v,k}t^{k}. \label{pf.prop.adj.poly.d-goal}%
\end{equation}

Indeed, $u$ and $v$ are two elements of $\left\{  1,2,\ldots,n\right\}  $
(since $\left(  u,v\right)  \in\left\{  1,2,\ldots,n\right\}  ^{2}$). Hence,
Lemma \ref{lem.adj.linear} (applied to $\mathbb{K}\left[  t\right]  $, $t$,
$1$, $I_{n}$ and $-A$ instead of $\mathbb{K}$, $\lambda$, $\mu$, $A$ and $B$)
yields
\[
\left(  tI_{n}+1\left(  -A\right)  \right)  _{\sim u,\sim v}=t\left(
I_{n}\right)  _{\sim u,\sim v}+1\left(  -A\right)  _{\sim u,\sim v}=t\left(
I_{n}\right)  _{\sim u,\sim v}+\left(  -A\right)  _{\sim u,\sim v}.
\]
Since $tI_{n}+1\left(  -A\right)  =tI_{n}+\left(  -A\right)  =tI_{n}-A$, this
rewrites as
\begin{equation}
\left(  tI_{n}-A\right)  _{\sim u,\sim v}=t\left(  I_{n}\right)  _{\sim u,\sim
v}+\left(  -A\right)  _{\sim u,\sim v}. \label{pf.prop.adj.poly.3}%
\end{equation}

But $u\in\left\{  1,2,\ldots,n\right\}  $, so that $1\leq u\leq n$ and thus
$n\geq1$. Hence, $n-1\in\mathbb{N}$. Thus, Proposition \ref{prop.ta+b}
\textbf{(a)} (applied to $n-1$, $\left(  I_{n}\right)  _{\sim u,\sim v}$ and
$\left(  -A\right)  _{\sim u,\sim v}$ instead of $n$, $A$ and $B$) shows that
$\det\left(  t\left(  I_{n}\right)  _{\sim u,\sim v}+\left(  -A\right)  _{\sim
u,\sim v}\right)  \in\mathbb{K}\left[  t\right]  $ is a polynomial of degree
$\leq n-1$ in $t$. In other words, there exists some $n$-tuple $\left(
b_{0},b_{1},\ldots,b_{n-1}\right)  \in\mathbb{K}^{n}$ such that $\det\left(
t\left(  I_{n}\right)  _{\sim u,\sim v}+\left(  -A\right)  _{\sim u,\sim
v}\right)  =\sum_{k=0}^{n-1}b_{k}t^{k}$. Consider this $\left(  b_{0}%
,b_{1},\ldots,b_{n-1}\right)  $.

Taking determinants on both sides of (\ref{pf.prop.adj.poly.3}), we obtain%
\begin{align}
\det\left(  \left(  tI_{n}-A\right)  _{\sim u,\sim v}\right)   &  =\det\left(
t\left(  I_{n}\right)  _{\sim u,\sim v}+\left(  -A\right)  _{\sim u,\sim
v}\right) \nonumber\\
&  =\sum_{k=0}^{n-1}b_{k}t^{k}. \label{pf.prop.adj.poly.5}%
\end{align}
Now, define an $n$-tuple $\left(  d_{u,v,0},d_{u,v,1},\ldots,d_{u,v,n-1}%
\right)  \in\mathbb{K}^{n}$ by $\left(  d_{u,v,0},d_{u,v,1},\ldots
,d_{u,v,n-1}\right)  =\left(  b_{0},b_{1},\ldots,b_{n-1}\right)  $. Thus,
$d_{u,v,k}=b_{k}$ for every $k\in\left\{  0,1,\ldots,n-1\right\}  $. Hence,
(\ref{pf.prop.adj.poly.5}) becomes%
\[
\det\left(  \left(  tI_{n}-A\right)  _{\sim u,\sim v}\right)  =\sum
_{k=0}^{n-1}\underbrace{b_{k}}_{=d_{u,v,k}}t^{k}=\sum_{k=0}^{n-1}%
d_{u,v,k}t^{k}.
\]
In other words, (\ref{pf.prop.adj.poly.d-goal}) is satisfied.

Now, forget that we have fixed $\left(  u,v\right)  $. Thus, for every
$\left(  u,v\right)  \in\left\{  1,2,\ldots,n\right\}  ^{2}$, we have defined
an $n$-tuple $\left(  d_{u,v,0},d_{u,v,1},\ldots,d_{u,v,n-1}\right)
\in\mathbb{K}^{n}$ satisfying (\ref{pf.prop.adj.poly.d-goal}).

Now, the definition of $\operatorname*{adj}\left(  tI_{n}-A\right)  $ yields
\begin{align*}
\operatorname*{adj}\left(  tI_{n}-A\right)   &  =\left(  \left(  -1\right)
^{i+j}\underbrace{\det\left(  \left(  tI_{n}-A\right)  _{\sim j,\sim
i}\right)  }_{\substack{=\sum_{k=0}^{n-1}d_{j,i,k}t^{k}\\\text{(by
(\ref{pf.prop.adj.poly.d-goal}), applied to }\left(  u,v\right)  =\left(
j,i\right)  \text{)}}}\right)  _{1\leq i\leq n,\ 1\leq j\leq n}\\
&  =\left(  \underbrace{\left(  -1\right)  ^{i+j}\sum_{k=0}^{n-1}%
d_{j,i,k}t^{k}}_{\substack{=\sum_{k=0}^{n-1}\left(  -1\right)  ^{i+j}%
d_{j,i,k}t^{k}\\=\sum_{k=0}^{n-1}t^{k}\left(  -1\right)  ^{i+j}d_{j,i,k}%
}}\right)  _{1\leq i\leq n,\ 1\leq j\leq n}=\left(  \sum_{k=0}^{n-1}%
t^{k}\left(  -1\right)  ^{i+j}d_{j,i,k}\right)  _{1\leq i\leq n,\ 1\leq j\leq
n}.
\end{align*}
Comparing this with%
\begin{align*}
\sum_{k=0}^{n-1}\underbrace{t^{k}\left(  \left(  -1\right)  ^{i+j}%
d_{j,i,k}\right)  _{1\leq i\leq n,\ 1\leq j\leq n}}_{=\left(  t^{k}\left(
-1\right)  ^{i+j}d_{j,i,k}\right)  _{1\leq i\leq n,\ 1\leq j\leq n}}  &
=\sum_{k=0}^{n-1}\left(  t^{k}\left(  -1\right)  ^{i+j}d_{j,i,k}\right)
_{1\leq i\leq n,\ 1\leq j\leq n}\\
&  =\left(  \sum_{k=0}^{n-1}t^{k}\left(  -1\right)  ^{i+j}d_{j,i,k}\right)
_{1\leq i\leq n,\ 1\leq j\leq n},
\end{align*}
we obtain $\operatorname*{adj}\left(  tI_{n}-A\right)  =\sum_{k=0}^{n-1}%
t^{k}\left(  \left(  -1\right)  ^{i+j}d_{j,i,k}\right)  _{1\leq i\leq
n,\ 1\leq j\leq n}$. Hence, there exist $n$ matrices $D_{0},D_{1}%
,\ldots,D_{n-1}$ in $\mathbb{K}^{n\times n}$ such that%
\[
\operatorname*{adj}\left(  tI_{n}-A\right)  =\sum_{k=0}^{n-1}t^{k}%
D_{k}\ \ \ \ \ \ \ \ \ \ \text{in }\left(  \mathbb{K}\left[  t\right]
\right)  ^{n\times n}%
\]
(namely, $D_{k}=\left(  \left(  -1\right)  ^{i+j}d_{j,i,k}\right)  _{1\leq
i\leq n,\ 1\leq j\leq n}$ for every $k\in\left\{  0,1,\ldots,n-1\right\}  $).
This proves Proposition \ref{prop.adj.poly}.
\end{proof}
\end{verlong}

\begin{noncompile}
For any two objects $i$ and $j$, define an element $\delta_{i,j}\in\mathbb{K}$
by $\delta_{i,j}=%
\begin{cases}
1, & \text{if }i=j;\\
0, & \text{if }i\neq j
\end{cases}
$. Then, $I_{n}=\left(  \delta_{i,j}\right)  _{1\leq i\leq n,\ 1\leq j\leq n}$
(by the definition of the identity matrix $I_{n}$). Hence, $t\underbrace{I_{n}%
}_{=\left(  \delta_{i,j}\right)  _{1\leq i\leq n,\ 1\leq j\leq n}}=t\left(
\delta_{i,j}\right)  _{1\leq i\leq n,\ 1\leq j\leq n}=\left(  t\delta
_{i,j}\right)  _{1\leq i\leq n,\ 1\leq j\leq n}$.
\end{noncompile}

\subsection{Polynomials with matrix entries: a trivial lemma}

\begin{lemma}
\label{lem.polymat.=}Let $n\in\mathbb{N}$ and $m\in\mathbb{N}$. Let $\left(
B_{0},B_{1},\ldots,B_{m}\right)  \in\left(  \mathbb{K}^{n\times n}\right)
^{m+1}$ and $\left(  C_{0},C_{1},\ldots,C_{m}\right)  \in\left(
\mathbb{K}^{n\times n}\right)  ^{m+1}$ be two $\left(  m+1\right)  $-tuples of
matrices in $\mathbb{K}^{n\times n}$. Assume that%
\[
\sum_{k=0}^{m}t^{k}B_{k}=\sum_{k=0}^{m}t^{k}C_{k}\ \ \ \ \ \ \ \ \ \ \text{in
}\left(  \mathbb{K}\left[  t\right]  \right)  ^{n\times n}.
\]
Then, $B_{k}=C_{k}$ for every $k\in\left\{  0,1,\ldots,m\right\}  $.
\end{lemma}

\begin{vershort}
\begin{proof}
[Proof of Lemma \ref{lem.polymat.=}.]For every $k\in\left\{  0,1,\ldots
,m\right\}  $, write the matrix $B_{k}\in\mathbb{K}^{n\times n}$ in the form
$B_{k}=\left(  b_{k,i,j}\right)  _{1\leq i\leq n,\ 1\leq j\leq n}$, and write
the matrix $C_{k}\in\mathbb{K}^{n\times n}$ in the form $C_{k}=\left(
c_{k,i,j}\right)  _{1\leq i\leq n,\ 1\leq j\leq n}$.

Now, $\sum_{k=0}^{m}t^{k}B_{k}=\left(  \sum_{k=0}^{m}t^{k}b_{k,i,j}\right)
_{1\leq i\leq n,\ 1\leq j\leq n}$ (since $B_{k}=\left(  b_{k,i,j}\right)
_{1\leq i\leq n,\ 1\leq j\leq n}$ for every $k\in\left\{  0,1,\ldots
,m\right\}  $). Similarly, $\sum_{k=0}^{m}t^{k}C_{k}=\left(  \sum_{k=0}%
^{m}t^{k}c_{k,i,j}\right)  _{1\leq i\leq n,\ 1\leq j\leq n}$. Thus,%
\[
\left(  \sum_{k=0}^{m}t^{k}b_{k,i,j}\right)  _{1\leq i\leq n,\ 1\leq j\leq
n}=\sum_{k=0}^{m}t^{k}B_{k}=\sum_{k=0}^{m}t^{k}C_{k}=\left(  \sum_{k=0}%
^{m}t^{k}c_{k,i,j}\right)  _{1\leq i\leq n,\ 1\leq j\leq n}.
\]
In other words,%
\[
\sum_{k=0}^{m}t^{k}b_{k,i,j}=\sum_{k=0}^{m}t^{k}c_{k,i,j}%
\]
for every $\left(  i,j\right)  \in\left\{  1,2,\ldots,n\right\}  ^{2}$.
Comparing coefficients on both sides of this equality, we obtain%
\[
b_{k,i,j}=c_{k,i,j}%
\]
for every $k\in\left\{  0,1,\ldots,m\right\}  $ for every $\left(  i,j\right)
\in\left\{  1,2,\ldots,n\right\}  ^{2}$. Now, every $k\in\left\{
0,1,\ldots,m\right\}  $ satisfies%
\[
B_{k}=\left(  \underbrace{b_{k,i,j}}_{=c_{k,i,j}}\right)  _{1\leq i\leq
n,\ 1\leq j\leq n}=\left(  c_{k,i,j}\right)  _{1\leq i\leq n,\ 1\leq j\leq
n}=C_{k}.
\]
This proves Lemma \ref{lem.polymat.=}.
\end{proof}
\end{vershort}

\begin{verlong}
\begin{proof}
[Proof of Lemma \ref{lem.polymat.=}.]For every $k\in\left\{  0,1,\ldots
,m\right\}  $, write the matrix $B_{k}\in\mathbb{K}^{n\times n}$ in the form
$B_{k}=\left(  b_{k,i,j}\right)  _{1\leq i\leq n,\ 1\leq j\leq n}$. For every
$k\in\left\{  0,1,\ldots,m\right\}  $, write the matrix $C_{k}\in
\mathbb{K}^{n\times n}$ in the form $C_{k}=\left(  c_{k,i,j}\right)  _{1\leq
i\leq n,\ 1\leq j\leq n}$.

We have%
\begin{align*}
\sum_{k=0}^{m}t^{k}\underbrace{B_{k}}_{=\left(  b_{k,i,j}\right)  _{1\leq
i\leq n,\ 1\leq j\leq n}}  &  =\sum_{k=0}^{m}\underbrace{t^{k}\left(
b_{k,i,j}\right)  _{1\leq i\leq n,\ 1\leq j\leq n}}_{=\left(  t^{k}%
b_{k,i,j}\right)  _{1\leq i\leq n,\ 1\leq j\leq n}}=\sum_{k=0}^{m}\left(
t^{k}b_{k,i,j}\right)  _{1\leq i\leq n,\ 1\leq j\leq n}\\
&  =\left(  \sum_{k=0}^{m}t^{k}b_{k,i,j}\right)  _{1\leq i\leq n,\ 1\leq j\leq
n}.
\end{align*}
Comparing this with%
\begin{align*}
\sum_{k=0}^{m}t^{k}B_{k}  &  =\sum_{k=0}^{m}t^{k}\underbrace{C_{k}}_{=\left(
c_{k,i,j}\right)  _{1\leq i\leq n,\ 1\leq j\leq n}}=\sum_{k=0}^{m}%
\underbrace{t^{k}\left(  c_{k,i,j}\right)  _{1\leq i\leq n,\ 1\leq j\leq n}%
}_{=\left(  t^{k}c_{k,i,j}\right)  _{1\leq i\leq n,\ 1\leq j\leq n}}\\
&  =\sum_{k=0}^{m}\left(  t^{k}c_{k,i,j}\right)  _{1\leq i\leq n,\ 1\leq j\leq
n}=\left(  \sum_{k=0}^{m}t^{k}c_{k,i,j}\right)  _{1\leq i\leq n,\ 1\leq j\leq
n},
\end{align*}
we obtain%
\[
\left(  \sum_{k=0}^{m}t^{k}b_{k,i,j}\right)  _{1\leq i\leq n,\ 1\leq j\leq
n}=\left(  \sum_{k=0}^{m}t^{k}c_{k,i,j}\right)  _{1\leq i\leq n,\ 1\leq j\leq
n}.
\]
In other words,%
\begin{equation}
\sum_{k=0}^{m}t^{k}b_{k,i,j}=\sum_{k=0}^{m}t^{k}c_{k,i,j}
\label{pf.lem.polymat.=.1}%
\end{equation}
for every $\left(  i,j\right)  \in\left\{  1,2,\ldots,n\right\}  ^{2}$.

Now, fix $\left(  i,j\right)  \in\left\{  1,2,\ldots,n\right\}  ^{2}$. Then,%
\begin{align*}
\sum_{k=0}^{m}\underbrace{b_{k,i,j}t^{k}}_{=t^{k}b_{k,i,j}}  &  =\sum
_{k=0}^{m}t^{k}b_{k,i,j}=\sum_{k=0}^{m}\underbrace{t^{k}c_{k,i,j}}%
_{=c_{k,i,j}t^{k}}\ \ \ \ \ \ \ \ \ \ \left(  \text{by
(\ref{pf.lem.polymat.=.1})}\right) \\
&  =\sum_{k=0}^{m}c_{k,i,j}t^{k}%
\end{align*}
in $\mathbb{K}\left[  t\right]  $. Comparing coefficients on both sides of
this equality, we obtain%
\begin{equation}
b_{k,i,j}=c_{k,i,j}\ \ \ \ \ \ \ \ \ \ \text{for every }k\in\left\{
0,1,\ldots,m\right\}  . \label{pf.lem.polymat.=.2}%
\end{equation}

Now, forget that we fixed $\left(  i,j\right)  $. We thus have proven
(\ref{pf.lem.polymat.=.2}) for every $\left(  i,j\right)  \in\left\{
1,2,\ldots,n\right\}  ^{2}$ and every $k\in\left\{  0,1,\ldots,m\right\}  $.

Now, every $k\in\left\{  0,1,\ldots,m\right\}  $ satisfies%
\[
B_{k}=\left(  \underbrace{b_{k,i,j}}_{\substack{=c_{k,i,j}\\\text{(by
(\ref{pf.lem.polymat.=.2}))}}}\right)  _{1\leq i\leq n,\ 1\leq j\leq
n}=\left(  c_{k,i,j}\right)  _{1\leq i\leq n,\ 1\leq j\leq n}=C_{k}%
\]
(since $C_{k}=\left(  c_{k,i,j}\right)  _{1\leq i\leq n,\ 1\leq j\leq n}$).
This proves Lemma \ref{lem.polymat.=}.
\end{proof}
\end{verlong}

\subsection{Proof of the Cayley-Hamilton theorem}

We are now fully prepared for the proof of the Cayley-Hamilton theorem.
However, we are going to organize the crucial part of this proof as a lemma,
so that we can use it later in our proof of the trace Cayley-Hamilton theorem.

\begin{lemma}
\label{lem.CH.D}Let $n\in\mathbb{N}$. Let $A\in\mathbb{K}^{n\times n}$. For
every $j\in\mathbb{Z}$, define an element $c_{j}\in\mathbb{K}$ by
$c_{j}=\left[  t^{n-j}\right]  \chi_{A}$.

Let $D_{0},D_{1},\ldots,D_{n-1}$ be $n$ matrices in $\mathbb{K}^{n\times n}$
such that%
\begin{equation}
\operatorname*{adj}\left(  tI_{n}-A\right)  =\sum_{k=0}^{n-1}t^{k}%
D_{k}\ \ \ \ \ \ \ \ \ \ \text{in }\left(  \mathbb{K}\left[  t\right]
\right)  ^{n\times n}. \label{eq.lem.CH.D.adj}%
\end{equation}
Thus, an $n$-tuple $\left(  D_{0},D_{1},\ldots,D_{n-1}\right)  $ of matrices
in $\mathbb{K}^{n\times n}$ is defined. Extend this $n$-tuple to a family
$\left(  D_{k}\right)  _{k\in\mathbb{Z}}$ of matrices in $\mathbb{K}^{n\times
n}$ by setting%
\begin{equation}
\left(  D_{k}=0_{n\times n}\ \ \ \ \ \ \ \ \ \ \text{for every }k\in
\mathbb{Z}\setminus\left\{  0,1,\ldots,n-1\right\}  \right)  .
\label{eq.lem.CH.D.0}%
\end{equation}
Then: \medskip

\textbf{(a)} We have $\chi_{A}=\sum_{k=0}^{n}c_{n-k}t^{k}$. \medskip

\textbf{(b)} For every integer $k$, we have $c_{n-k}I_{n}=D_{k-1}-AD_{k}$.
\medskip

\textbf{(c)} Every $k\in\mathbb{N}$ satisfies%
\[
\sum_{i=0}^{k}c_{k-i}A^{i}=D_{n-1-k}.
\]

\end{lemma}

\begin{vershort}
\begin{proof}
[Proof of Lemma \ref{lem.CH.D}.]\textbf{(a)} Lemma \ref{lem.CH.D} \textbf{(a)}
is just Corollary \ref{cor.CH.chiA}. \medskip

\textbf{(b)} We have%
\begin{align}
\sum_{k=0}^{n}t^{k}D_{k-1}  &  =t^{0}\underbrace{D_{0-1}}_{\substack{=D_{-1}%
=0_{n\times n}\\\text{(by (\ref{eq.lem.CH.D.0}))}}}+\sum_{k=1}^{n}t^{k}%
D_{k-1}=\sum_{k=1}^{n}t^{k}D_{k-1}=\sum_{k=0}^{n-1}\underbrace{t^{k+1}%
}_{=tt^{k}}\underbrace{D_{\left(  k+1\right)  -1}}_{=D_{k}}\nonumber\\
&  \ \ \ \ \ \ \ \ \ \ \left(  \text{here, we have substituted }k+1\text{ for
}k\text{ in the sum}\right) \nonumber\\
&  =\sum_{k=0}^{n-1}tt^{k}D_{k}=t\underbrace{\sum_{k=0}^{n-1}t^{k}D_{k}%
}_{\substack{=\operatorname*{adj}\left(  tI_{n}-A\right)  \\\text{(by
(\ref{eq.lem.CH.D.adj}))}}}=t\operatorname*{adj}\left(  tI_{n}-A\right)
\label{pf.lem.CH.D.short.1}%
\end{align}
and%
\begin{align}
\sum_{k=0}^{n}t^{k}D_{k}  &  =t^{n}\underbrace{D_{n}}_{\substack{=0_{n\times
n}\\\text{(by (\ref{eq.lem.CH.D.0}))}}}+\sum_{k=0}^{n-1}t^{k}D_{k}=\sum
_{k=0}^{n-1}t^{k}D_{k}\nonumber\\
&  =\operatorname*{adj}\left(  tI_{n}-A\right)  \ \ \ \ \ \ \ \ \ \ \left(
\text{by (\ref{eq.lem.CH.D.adj})}\right)  . \label{pf.lem.CH.D.short.2}%
\end{align}

But Theorem \ref{thm.adj.inverse} (applied to $\mathbb{K}\left[  t\right]  $
and $tI_{n}-A$ instead of $\mathbb{K}$ and $A$) shows that%
\[
\left(  tI_{n}-A\right)  \cdot\operatorname*{adj}\left(  tI_{n}-A\right)
=\operatorname*{adj}\left(  tI_{n}-A\right)  \cdot\left(  tI_{n}-A\right)
=\det\left(  tI_{n}-A\right)  \cdot I_{n}.
\]
Thus, in particular,%
\[
\left(  tI_{n}-A\right)  \cdot\operatorname*{adj}\left(  tI_{n}-A\right)
=\underbrace{\det\left(  tI_{n}-A\right)  }_{\substack{=\chi_{A}\\\text{(by
the definition of }\chi_{A}\text{)}}}\cdot\,I_{n}=\chi_{A}\cdot I_{n},
\]
so that%
\begin{align*}
\chi_{A}\cdot I_{n}  &  =\left(  tI_{n}-A\right)  \cdot\operatorname*{adj}%
\left(  tI_{n}-A\right) \\
&  =\underbrace{t\operatorname*{adj}\left(  tI_{n}-A\right)  }%
_{\substack{=\sum_{k=0}^{n}t^{k}D_{k-1}\\\text{(by (\ref{pf.lem.CH.D.short.1}%
))}}}-\,A\underbrace{\operatorname*{adj}\left(  tI_{n}-A\right)
}_{\substack{=\sum_{k=0}^{n}t^{k}D_{k}\\\text{(by (\ref{pf.lem.CH.D.short.2}%
))}}}=\sum_{k=0}^{n}t^{k}D_{k-1}-\underbrace{A\sum_{k=0}^{n}t^{k}D_{k}}%
_{=\sum_{k=0}^{n}t^{k}AD_{k}}\\
&  =\sum_{k=0}^{n}t^{k}D_{k-1}-\sum_{k=0}^{n}t^{k}AD_{k}=\sum_{k=0}^{n}%
t^{k}\left(  D_{k-1}-AD_{k}\right)  .
\end{align*}
Thus,%
\[
\sum_{k=0}^{n}t^{k}\left(  D_{k-1}-AD_{k}\right)  =\underbrace{\chi_{A}%
}_{\substack{=\sum_{k=0}^{n}c_{n-k}t^{k}\\\text{(by Lemma \ref{lem.CH.D}
\textbf{(a)})}}}\cdot\,I_{n}=\left(  \sum_{k=0}^{n}c_{n-k}t^{k}\right)  \cdot
I_{n}=\sum_{k=0}^{n}t^{k}c_{n-k}I_{n}.
\]
Lemma \ref{lem.polymat.=} (applied to $m=n$, $B_{k}=D_{k-1}-AD_{k}$ and
$C_{k}=c_{n-k}I_{n}$) thus shows that%
\begin{equation}
D_{k-1}-AD_{k}=c_{n-k}I_{n}\ \ \ \ \ \ \ \ \ \ \text{for every }k\in\left\{
0,1,\ldots,n\right\}  . \label{pf.lem.CH.D.short.b.5}%
\end{equation}

Now, let $k$ be an integer. We must prove that $c_{n-k}I_{n}=D_{k-1}-AD_{k}$.

If $k\in\left\{  0,1,\ldots,n\right\}  $, then this follows from
(\ref{pf.lem.CH.D.short.b.5}). Thus, we WLOG assume that $k\notin\left\{
0,1,\ldots,n\right\}  $. Hence, $k-1\in\mathbb{Z}\setminus\left\{
0,1,\ldots,n-1\right\}  $, so that (\ref{eq.lem.CH.D.0}) (applied to $k-1$
instead of $k$) yields $D_{k-1}=0_{n\times n}$. Also, $k\notin\left\{
0,1,\ldots,n\right\}  $ leads to $k\in\mathbb{Z}\setminus\left\{
0,1,\ldots,n-1\right\}  $; therefore, (\ref{eq.lem.CH.D.0}) yields
$D_{k}=0_{n\times n}$. Now, $\underbrace{D_{k-1}}_{=0_{n\times n}%
}-\,A\underbrace{D_{k}}_{=0_{n\times n}}=0_{n\times n}-0_{n\times
n}=0_{n\times n}$.

On the other hand, $c_{n-k}=0$\ \ \ \ \footnote{\textit{Proof.} Recall that
$\chi_{A}$ is a polynomial of degree $\leq n$ (by Corollary \ref{cor.chiA}
\textbf{(a)}). Hence, $\left[  t^{k}\right]  \chi_{A}=0$ (since $k\notin%
\left\{  0,1,\ldots,n\right\}  $). Now, (\ref{pf.cor.CH.chiA.1}) yields
$c_{n-k}=\left[  t^{k}\right]  \chi_{A}=0$.}. Hence, $\underbrace{c_{n-k}%
}_{=0}I_{n}=0_{n\times n}$. Compared with $D_{k-1}-AD_{k}=0_{n\times n}$, this
yields $c_{n-k}I_{n}=D_{k-1}-AD_{k}$.

Hence, $c_{n-k}I_{n}=D_{k-1}-AD_{k}$ is proven. In other words, Lemma
\ref{lem.CH.D} \textbf{(b)} is proven. \medskip

\textbf{(c)} Let $k\in\mathbb{N}$. Then,%
\begin{align*}
\sum_{i=0}^{k}c_{k-i}A^{i}  &  =\sum_{i=n-k}^{n}\underbrace{c_{k-\left(
k-n+i\right)  }}_{=c_{n-i}}A^{k-n+i}\\
&  \ \ \ \ \ \ \ \ \ \ \left(  \text{here, we have substituted }k-n+i\text{
for }i\text{ in the sum}\right) \\
&  =\sum_{i=n-k}^{n}\underbrace{c_{n-i}A^{k-n+i}}_{=A^{k-n+i}c_{n-i}I_{n}%
}=\sum_{i=n-k}^{n}A^{k-n+i}\underbrace{c_{n-i}I_{n}}_{\substack{=D_{i-1}%
-AD_{i}\\\text{(by Lemma \ref{lem.CH.D} \textbf{(b)},}\\\text{applied to
}i\text{ instead of }k\text{)}}}\\
&  =\sum_{i=n-k}^{n}\underbrace{A^{k-n+i}\left(  D_{i-1}-AD_{i}\right)
}_{=A^{k-n+i}D_{i-1}-A^{k-n+i}AD_{i}}\\
&  =\sum_{i=n-k}^{n}\left(  A^{k-n+i}D_{i-1}-\underbrace{A^{k-n+i}%
A}_{=A^{k-n+i+1}=A^{k-n+\left(  i+1\right)  }}\ \ \underbrace{D_{i}%
}_{=D_{\left(  i+1\right)  -1}}\right) \\
&  =\sum_{i=n-k}^{n}\left(  A^{k-n+i}D_{i-1}-A^{k-n+\left(  i+1\right)
}D_{\left(  i+1\right)  -1}\right) \\
&  =\underbrace{A^{k-n+\left(  n-k\right)  }}_{=A^{0}=I_{n}}D_{n-k-1}%
-A^{k-n+\left(  n+1\right)  }\underbrace{D_{\left(  n+1\right)  -1}%
}_{\substack{=D_{n}=0_{n\times n}\\\text{(by (\ref{eq.lem.CH.D.0}))}}}\\
&  \ \ \ \ \ \ \ \ \ \ \left(  \text{by the telescope principle}\right) \\
&  =D_{n-k-1}=D_{n-1-k}.
\end{align*}
This proves Lemma \ref{lem.CH.D} \textbf{(c)}.
\end{proof}
\end{vershort}

\begin{verlong}
\begin{proof}
[Proof of Lemma \ref{lem.CH.D}.]Combining $n\in\mathbb{Z}$ with $n\notin%
\left\{  0,1,\ldots,n-1\right\}  $ (since $n>n-1$), we obtain $n\in
\mathbb{Z}\setminus\left\{  0,1,\ldots,n-1\right\}  $. Hence, $D_{n}%
=0_{n\times n}$ (by (\ref{eq.lem.CH.D.0}), applied to $k=n$).

Combining $-1\in\mathbb{Z}$ with $-1\notin\left\{  0,1,\ldots,n-1\right\}  $
(since $-1<0$), we obtain $-1\in\mathbb{Z}\setminus\left\{  0,1,\ldots
,n-1\right\}  $. Hence, $D_{-1}=0_{n\times n}$ (by (\ref{eq.lem.CH.D.0}),
applied to $k=-1$). \medskip

\textbf{(a)} Lemma \ref{lem.CH.D} \textbf{(a)} is identical with Corollary
\ref{cor.CH.chiA}, and therefore has already been proven. \medskip

\textbf{(b)} We have $n\in\mathbb{N}$, so that $0\in\left\{  0,1,\ldots
,n\right\}  $. Hence, we can split off the addend for $k=0$ from the sum
$\sum_{k=0}^{n}t^{k}D_{k-1}$. We thus obtain%
\begin{align}
\sum_{k=0}^{n}t^{k}D_{k-1}  &  =t^{0}\underbrace{D_{0-1}}_{=D_{-1}=0_{n\times
n}}+\sum_{k=1}^{n}t^{k}D_{k-1}=\underbrace{t^{0}0_{n\times n}}_{=0_{n\times
n}}+\sum_{k=1}^{n}t^{k}D_{k-1}\nonumber\\
&  =\sum_{k=1}^{n}t^{k}D_{k-1}=\sum_{k=0}^{n-1}\underbrace{t^{k+1}}_{=tt^{k}%
}\underbrace{D_{\left(  k+1\right)  -1}}_{=D_{k}}\nonumber\\
&  \ \ \ \ \ \ \ \ \ \ \left(  \text{here, we have substituted }k+1\text{ for
}k\text{ in the sum}\right) \nonumber\\
&  =\sum_{k=0}^{n-1}tt^{k}D_{k}=t\underbrace{\sum_{k=0}^{n-1}t^{k}D_{k}%
}_{\substack{=\operatorname*{adj}\left(  tI_{n}-A\right)  \\\text{(by
(\ref{eq.lem.CH.D.adj}))}}}=t\operatorname*{adj}\left(  tI_{n}-A\right)  .
\label{pf.lem.CH.D.1}%
\end{align}

Also, $n\in\mathbb{N}$, so that $n\in\left\{  0,1,\ldots,n\right\}  $. Thus,
we can split off the addend for $k=n$ from the sum $\sum_{k=0}^{n}t^{k}D_{k}$.
We thus obtain%
\begin{align}
\sum_{k=0}^{n}t^{k}D_{k}  &  =t^{n}\underbrace{D_{n}}_{=0_{n\times n}}%
+\sum_{k=0}^{n-1}t^{k}D_{k}=\underbrace{t^{n}0_{n\times n}}_{=0_{n\times n}%
}+\sum_{k=0}^{n-1}t^{k}D_{k}\nonumber\\
&  =\sum_{k=0}^{n-1}t^{k}D_{k}=\operatorname*{adj}\left(  tI_{n}-A\right)
\ \ \ \ \ \ \ \ \ \ \left(  \text{by (\ref{eq.lem.CH.D.adj})}\right)  .
\label{pf.lem.CH.D.2}%
\end{align}

But Theorem \ref{thm.adj.inverse} (applied to $\mathbb{K}\left[  t\right]  $
and $tI_{n}-A$ instead of $\mathbb{K}$ and $A$) shows that%
\[
\left(  tI_{n}-A\right)  \cdot\operatorname*{adj}\left(  tI_{n}-A\right)
=\operatorname*{adj}\left(  tI_{n}-A\right)  \cdot\left(  tI_{n}-A\right)
=\det\left(  tI_{n}-A\right)  \cdot I_{n}.
\]
Thus, in particular,%
\[
\left(  tI_{n}-A\right)  \cdot\operatorname*{adj}\left(  tI_{n}-A\right)
=\underbrace{\det\left(  tI_{n}-A\right)  }_{\substack{=\chi_{A}\\\text{(since
}\chi_{A}=\det\left(  tI_{n}-A\right)  \\\text{(by the definition of }\chi
_{A}\text{))}}}\cdot\,I_{n}=\chi_{A}\cdot I_{n},
\]
so that%
\begin{align*}
\chi_{A}\cdot I_{n}  &  =\left(  tI_{n}-A\right)  \cdot\operatorname*{adj}%
\left(  tI_{n}-A\right)  =t\underbrace{I_{n}\operatorname*{adj}\left(
tI_{n}-A\right)  }_{=\operatorname*{adj}\left(  tI_{n}-A\right)
}-A\operatorname*{adj}\left(  tI_{n}-A\right) \\
&  =\underbrace{t\operatorname*{adj}\left(  tI_{n}-A\right)  }%
_{\substack{=\sum_{k=0}^{n}t^{k}D_{k-1}\\\text{(by (\ref{pf.lem.CH.D.1}))}%
}}-\,A\underbrace{\operatorname*{adj}\left(  tI_{n}-A\right)  }%
_{\substack{=\sum_{k=0}^{n}t^{k}D_{k}\\\text{(by (\ref{pf.lem.CH.D.2}))}}}\\
&  =\sum_{k=0}^{n}t^{k}D_{k-1}-\underbrace{A\sum_{k=0}^{n}t^{k}D_{k}}%
_{=\sum_{k=0}^{n}At^{k}D_{k}=\sum_{k=0}^{n}t^{k}AD_{k}}=\sum_{k=0}^{n}%
t^{k}D_{k-1}-\sum_{k=0}^{n}t^{k}AD_{k}\\
&  =\sum_{k=0}^{n}t^{k}\left(  D_{k-1}-AD_{k}\right)  .
\end{align*}
Thus,%
\begin{align*}
\sum_{k=0}^{n}t^{k}\left(  D_{k-1}-AD_{k}\right)   &  =\underbrace{\chi_{A}%
}_{\substack{=\sum_{k=0}^{n}c_{n-k}t^{k}\\\text{(by Lemma \ref{lem.CH.D}
\textbf{(a)})}}}\cdot\,I_{n}=\left(  \sum_{k=0}^{n}c_{n-k}t^{k}\right)  \cdot
I_{n}\\
&  =\sum_{k=0}^{n}\underbrace{c_{n-k}t^{k}}_{=t^{k}c_{n-k}}I_{n}=\sum
_{k=0}^{n}t^{k}c_{n-k}I_{n}.
\end{align*}
Lemma \ref{lem.polymat.=} (applied to $m=n$, $B_{k}=D_{k-1}-AD_{k}$ and
$C_{k}=c_{n-k}I_{n}$) thus shows that%
\begin{equation}
D_{k-1}-AD_{k}=c_{n-k}I_{n}\ \ \ \ \ \ \ \ \ \ \text{for every }k\in\left\{
0,1,\ldots,n\right\}  . \label{pf.lem.CH.D.b.5}%
\end{equation}

Now, let $k$ be an integer. We must prove that $c_{n-k}I_{n}=D_{k-1}-AD_{k}$.

If $k\in\left\{  0,1,\ldots,n\right\}  $, then this follows from
(\ref{pf.lem.CH.D.b.5}). Thus, we can WLOG assume that we don't have
$k\in\left\{  0,1,\ldots,n\right\}  $. Assume this.

We have $k\notin\left\{  0,1,\ldots,n\right\}  $ (since we don't have
$k\in\left\{  0,1,\ldots,n\right\}  $). We have $k-1\notin\left\{
0,1,\ldots,n-1\right\}  $\ \ \ \ \footnote{\textit{Proof.} Assume the
contrary. Thus, $k-1\in\left\{  0,1,\ldots,n-1\right\}  $, so that $\left(
k-1\right)  +1\in\left\{  1,2,\ldots,n\right\}  $. Now, $k=\left(  k-1\right)
+1\in\left\{  1,2,\ldots,n\right\}  \subseteq\left\{  0,1,\ldots,n\right\}  $.
This contradicts $k\notin\left\{  0,1,\ldots,n\right\}  $. This contradiction
proves that our assumption was wrong, qed.}. Combining $k-1\in\mathbb{Z}$ with
$k-1\notin\left\{  0,1,\ldots,n-1\right\}  $, we obtain $k-1\in\mathbb{Z}%
\setminus\left\{  0,1,\ldots,n-1\right\}  $. Thus, (\ref{eq.lem.CH.D.0})
(applied to $k-1$ instead of $k$) yields $D_{k-1}=0_{n\times n}$.

We have $k\notin\left\{  0,1,\ldots,n-1\right\}  $%
\ \ \ \ \footnote{\textit{Proof.} Assume the contrary. Thus, $k\in\left\{
0,1,\ldots,n-1\right\}  $, so that $k\in\left\{  0,1,\ldots,n-1\right\}
\subseteq\left\{  0,1,\ldots,n\right\}  $. This contradicts $k\notin\left\{
0,1,\ldots,n\right\}  $. This contradiction proves that our assumption was
wrong, qed.}. Combining $k\in\mathbb{Z}$ with $k\notin\left\{  0,1,\ldots
,n-1\right\}  $, we obtain $k\in\mathbb{Z}\setminus\left\{  0,1,\ldots
,n-1\right\}  $. Thus, (\ref{eq.lem.CH.D.0}) yields $D_{k}=0_{n\times n}$.

Now, $\underbrace{D_{k-1}}_{=0_{n\times n}}-\,A\underbrace{D_{k}}_{=0_{n\times
n}}=0_{n\times n}-0_{n\times n}=0_{n\times n}$.

On the other hand, $c_{n-k}=0$\ \ \ \ \footnote{\textit{Proof.} We are in one
of the following two cases:
\par
\textit{Case 1:} We have $k\geq0$.
\par
\textit{Case 2:} We have $k<0$.
\par
Let us consider Case 1 first. In this case, we have $k\geq0$. Hence,
$k\in\mathbb{N}$. Combined with $k\notin\left\{  0,1,\ldots,n\right\}  $, this
yields $k\in\mathbb{N}\setminus\left\{  0,1,\ldots,n\right\}  =\left\{
n+1,n+2,n+3,\ldots\right\}  $, so that $k>n$. Thus, $\left[  t^{k}\right]
\chi_{A}=0$ (by (\ref{pf.lem.CH.D.a.1})). Now, (\ref{pf.lem.CH.D.a.cn-k})
yields $c_{n-k}=\left[  t^{k}\right]  \chi_{A}=0$. Thus, $c_{n-k}=0$ is proven
in Case 1.
\par
Let us now consider Case 2. In this case, we have $k<0$. Thus, $k$ is a
negative integer. Now, every $f\in\mathbb{K}\left[  t\right]  $ and every
negative integer $m$ satisfy $\left[  t^{m}\right]  f=0$ (since polynomials
have no monomials $t^{m}$ with negative $m$). Applying this to $f=\chi_{A}$
and $m=k$, we obtain $\left[  t^{k}\right]  \chi_{A}=0$. Now,
(\ref{pf.lem.CH.D.a.cn-k}) yields $c_{n-k}=\left[  t^{k}\right]  \chi_{A}=0$.
Thus, $c_{n-k}=0$ is proven in Case 2.
\par
Now, we have proven $c_{n-k}=0$ in each of the two Cases 1 and 2. Since these
two Cases cover all possibilities, this shows that $c_{n-k}=0$ always holds.
Qed.}. Hence, $\underbrace{c_{n-k}}_{=0}I_{n}=0I_{n}=0_{n\times n}$. Compared
with $D_{k-1}-AD_{k}=0_{n\times n}$, this yields $c_{n-k}I_{n}=D_{k-1}-AD_{k}$.

Hence, $c_{n-k}I_{n}=D_{k-1}-AD_{k}$ is proven. In other words, Lemma
\ref{lem.CH.D} \textbf{(b)} is proven. \medskip

\textbf{(c)} Let $k\in\mathbb{N}$. We have $k\geq0$ (since $k\in\mathbb{N}$)
and thus $n-\underbrace{k}_{\geq0}\leq n$. Hence, we can split off the addend
for $i=n$ from the sum $\sum_{i=n-k}^{n}A^{\left(  k-\left(  n-i\right)
\right)  +1}D_{i}$. We thus obtain%
\begin{align}
\sum_{i=n-k}^{n}A^{\left(  k-\left(  n-i\right)  \right)  +1}D_{i}  &
=A^{\left(  k-\left(  n-n\right)  \right)  +1}\underbrace{D_{n}}_{=0_{n\times
n}}+\sum_{i=n-k}^{n-1}A^{\left(  k-\left(  n-i\right)  \right)  +1}%
D_{i}\nonumber\\
&  =\underbrace{A^{\left(  k-\left(  n-n\right)  \right)  +1}0_{n\times n}%
}_{=0_{n\times n}}+\sum_{i=n-k}^{n-1}A^{\left(  k-\left(  n-i\right)  \right)
+1}D_{i}\nonumber\\
&  =\sum_{i=n-k}^{n-1}A^{\left(  k-\left(  n-i\right)  \right)  +1}D_{i}%
=\sum_{i=n-k+1}^{n}\underbrace{A^{\left(  k-\left(  n-\left(  i-1\right)
\right)  \right)  +1}}_{\substack{=A^{k-n+i}\\\text{(since }\left(  k-\left(
n-\left(  i-1\right)  \right)  \right)  +1=k-n+i\text{)}}}D_{i-1}\nonumber\\
&  \ \ \ \ \ \ \ \ \ \ \left(  \text{here, we have substituted }i-1\text{ for
}i\text{ in the sum}\right) \nonumber\\
&  =\sum_{i=n-k+1}^{n}A^{k-n+i}D_{i-1}. \label{pf.lem.CH.D.c.1}%
\end{align}

Also, $n-k\leq n$. Hence, we can split off the addend for $i=n-k$ from the sum
$\sum_{i=n-k}^{n}A^{k-\left(  n-i\right)  }D_{i-1}$. We thus obtain%
\begin{align}
\sum_{i=n-k}^{n}A^{k-\left(  n-i\right)  }D_{i-1}  &
=\underbrace{A^{k-\left(  n-\left(  n-k\right)  \right)  }}_{\substack{=A^{0}%
\\\text{(since }k-\left(  n-\left(  n-k\right)  \right)  =0\text{)}}%
}D_{n-k-1}+\sum_{i=n-k+1}^{n}\underbrace{A^{k-\left(  n-i\right)  }%
}_{\substack{=A^{k-n+i}\\\text{(since }k-\left(  n-i\right)  =k-n+i\text{)}%
}}D_{i-1}\nonumber\\
&  =\underbrace{A^{0}}_{=I_{n}}D_{n-k-1}+\sum_{i=n-k+1}^{n}A^{k-n+i}%
D_{i-1}\nonumber\\
&  =D_{n-k-1}+\sum_{i=n-k+1}^{n}A^{k-n+i}D_{i-1}. \label{pf.lem.CH.D.c.2}%
\end{align}

Now,%
\begin{align*}
\sum_{i=0}^{k}c_{k-i}A^{i}  &  =\sum_{i=0}^{k}\underbrace{c_{k-\left(
k-i\right)  }}_{=c_{i}}A^{k-i}\ \ \ \ \ \ \ \ \ \ \left(  \text{here, we have
substituted }k-i\text{ for }i\text{ in the sum}\right) \\
&  =\sum_{i=0}^{k}c_{i}A^{k-i}=\sum_{i=n-k}^{n}c_{n-i}\underbrace{A^{k-\left(
n-i\right)  }}_{=A^{k-\left(  n-i\right)  }I_{n}}\\
&  \ \ \ \ \ \ \ \ \ \ \left(  \text{here, we have substituted }n-i\text{ for
}i\text{ in the sum}\right) \\
&  =\sum_{i=n-k}^{n}\underbrace{c_{n-i}A^{k-\left(  n-i\right)  }I_{n}%
}_{=A^{k-\left(  n-i\right)  }c_{n-i}I_{n}}=\sum_{i=n-k}^{n}A^{k-\left(
n-i\right)  }\underbrace{c_{n-i}I_{n}}_{\substack{=D_{i-1}-AD_{i}\\\text{(by
Lemma \ref{lem.CH.D} \textbf{(b)},}\\\text{applied to }i\text{ instead of
}k\text{)}}}\\
&  =\sum_{i=n-k}^{n}\underbrace{A^{k-\left(  n-i\right)  }\left(
D_{i-1}-AD_{i}\right)  }_{=A^{k-\left(  n-i\right)  }D_{i-1}-A^{k-\left(
n-i\right)  }AD_{i}}\\
&  =\sum_{i=n-k}^{n}\left(  A^{k-\left(  n-i\right)  }D_{i-1}-A^{k-\left(
n-i\right)  }AD_{i}\right) \\
&  =\sum_{i=n-k}^{n}A^{k-\left(  n-i\right)  }D_{i-1}-\sum_{i=n-k}%
^{n}\underbrace{A^{k-\left(  n-i\right)  }A}_{=A^{\left(  k-\left(
n-i\right)  \right)  +1}}D_{i}\\
&  =\underbrace{\sum_{i=n-k}^{n}A^{k-\left(  n-i\right)  }D_{i-1}%
}_{\substack{=D_{n-k-1}+\sum_{i=n-k+1}^{n}A^{k-n+i}D_{i-1}\\\text{(by
(\ref{pf.lem.CH.D.c.2}))}}}-\underbrace{\sum_{i=n-k}^{n}A^{\left(  k-\left(
n-i\right)  \right)  +1}D_{i}}_{\substack{=\sum_{i=n-k+1}^{n}A^{k-n+i}%
D_{i-1}\\\text{(by (\ref{pf.lem.CH.D.c.1}))}}}\\
&  =\left(  D_{n-k-1}+\sum_{i=n-k+1}^{n}A^{k-n+i}D_{i-1}\right)  -\left(
\sum_{i=n-k+1}^{n}A^{k-n+i}D_{i-1}\right) \\
&  =D_{n-k-1}=D_{n-1-k}%
\end{align*}
(since $n-k-1=n-1-k$). This proves Lemma \ref{lem.CH.D} \textbf{(c)}.
\end{proof}
\end{verlong}

\begin{vershort}
\begin{proof}
[Proof of Theorem \ref{thm.CH}.]For every $j\in\mathbb{Z}$, define an element
$c_{j}\in\mathbb{K}$ by $c_{j}=\left[  t^{n-j}\right]  \chi_{A}$.

Proposition \ref{prop.adj.poly} shows that there exist $n$ matrices
$D_{0},D_{1},\ldots,D_{n-1}$ in $\mathbb{K}^{n\times n}$ such that%
\[
\operatorname*{adj}\left(  tI_{n}-A\right)  =\sum_{k=0}^{n-1}t^{k}%
D_{k}\ \ \ \ \ \ \ \ \ \ \text{in }\left(  \mathbb{K}\left[  t\right]
\right)  ^{n\times n}.
\]
Consider these $D_{0},D_{1},\ldots,D_{n-1}$. Thus, an $n$-tuple $\left(
D_{0},D_{1},\ldots,D_{n-1}\right)  $ of matrices in $\mathbb{K}^{n\times n}$
is defined. Extend this $n$-tuple to a family $\left(  D_{k}\right)
_{k\in\mathbb{Z}}$ of matrices in $\mathbb{K}^{n\times n}$ by setting%
\[
D_{k}=0_{n\times n}\ \ \ \ \ \ \ \ \ \ \text{for every }k\in\mathbb{Z}%
\setminus\left\{  0,1,\ldots,n-1\right\}  .
\]
Thus, in particular, $D_{-1}=0_{n\times n}$.

Lemma \ref{lem.CH.D} \textbf{(a)} shows that $\chi_{A}=\sum_{k=0}^{n}%
c_{n-k}t^{k}=\sum_{i=0}^{n}c_{n-i}t^{i}$. Substituting $A$ for $t$ in this
equality, we obtain%
\begin{align*}
\chi_{A}\left(  A\right)   &  =\sum_{i=0}^{n}c_{n-i}A^{i}=D_{n-1-n}%
\ \ \ \ \ \ \ \ \ \ \left(  \text{by Lemma \ref{lem.CH.D} \textbf{(c)},
applied to }k=n\right) \\
&  =D_{-1}=0_{n\times n}.
\end{align*}
This proves Theorem \ref{thm.CH}.
\end{proof}
\end{vershort}

\begin{verlong}
\begin{proof}
[Proof of Theorem \ref{thm.CH}.]For every $j\in\mathbb{Z}$, define an element
$c_{j}\in\mathbb{K}$ by $c_{j}=\left[  t^{n-j}\right]  \chi_{A}$.

Proposition \ref{prop.adj.poly} shows that there exist $n$ matrices
$D_{0},D_{1},\ldots,D_{n-1}$ in $\mathbb{K}^{n\times n}$ such that%
\[
\operatorname*{adj}\left(  tI_{n}-A\right)  =\sum_{k=0}^{n-1}t^{k}%
D_{k}\ \ \ \ \ \ \ \ \ \ \text{in }\left(  \mathbb{K}\left[  t\right]
\right)  ^{n\times n}.
\]
Consider these $D_{0},D_{1},\ldots,D_{n-1}$. Thus, an $n$-tuple $\left(
D_{0},D_{1},\ldots,D_{n-1}\right)  $ of matrices in $\mathbb{K}^{n\times n}$
is defined. Extend this $n$-tuple to a family $\left(  D_{k}\right)
_{k\in\mathbb{Z}}$ of matrices in $\mathbb{K}^{n\times n}$ by setting%
\begin{equation}
\left(  D_{k}=0_{n\times n}\ \ \ \ \ \ \ \ \ \ \text{for every }k\in
\mathbb{Z}\setminus\left\{  0,1,\ldots,n-1\right\}  \right)  .
\label{pf.thm.CH.0}%
\end{equation}

Combining $-1\in\mathbb{Z}$ with $-1\notin\left\{  0,1,\ldots,n-1\right\}  $
(since $-1<0$), we obtain $-1\in\mathbb{Z}\setminus\left\{  0,1,\ldots
,n-1\right\}  $. Applying (\ref{pf.thm.CH.0}) to $k=-1$ therefore yields
$D_{-1}=0_{n\times n}$.

Lemma \ref{lem.CH.D} \textbf{(a)} shows that $\chi_{A}=\sum_{k=0}^{n}%
c_{n-k}t^{k}$. Substituting $A$ for $t$ in this equality, we obtain%
\begin{align*}
\chi_{A}\left(  A\right)   &  =\sum_{k=0}^{n}c_{n-k}A^{k}=\sum_{i=0}%
^{n}c_{n-i}A^{i}\ \ \ \ \ \ \ \ \ \ \left(
\begin{array}
[c]{c}%
\text{here, we have renamed the summation}\\
\text{index }k\text{ as }i
\end{array}
\right) \\
&  =D_{n-1-n}\ \ \ \ \ \ \ \ \ \ \left(  \text{by Lemma \ref{lem.CH.D}
\textbf{(c)}, applied to }k=n\right) \\
&  =D_{-1}\ \ \ \ \ \ \ \ \ \ \left(  \text{since }n-1-n=-1\right) \\
&  =0_{n\times n}.
\end{align*}
This proves Theorem \ref{thm.CH}.
\end{proof}
\end{verlong}

\subsection{Derivations and determinants}

Now, let us make what seems to be a detour, and define $\mathbb{K}%
$-derivations of a $\mathbb{K}$-algebra\footnote{See \cite[Convention
1.1]{Grinbe16} for what we mean by a \textquotedblleft$\mathbb{K}%
$-algebra\textquotedblright. In a nutshell, we require $\mathbb{K}$-algebras
to be associative and unital, and we require the multiplication map on a
$\mathbb{K}$-algebra to be $\mathbb{K}$-bilinear.}:

\begin{definition}
\label{def.deriv}Let $\mathbb{L}$ be a $\mathbb{K}$-algebra. A $\mathbb{K}%
$-linear map $f:\mathbb{L}\rightarrow\mathbb{L}$ is said to be a $\mathbb{K}%
$\textit{-derivation} if it satisfies%
\begin{equation}
\left(  f\left(  ab\right)  =af\left(  b\right)  +f\left(  a\right)
b\ \ \ \ \ \ \ \ \ \ \text{for every }a\in\mathbb{L}\text{ and }b\in
\mathbb{L}\right)  . \label{eq.def.deriv.axiom}%
\end{equation}

\end{definition}

The notion of a \textquotedblleft$\mathbb{K}$-derivation\textquotedblright\ is
a particular case of the notion of a \textquotedblleft$\mathbf{k}%
$-derivation\textquotedblright\ defined in \cite[Definition 1.5]{Grinbe16};
specifically, it is obtained from the latter when setting $\mathbf{k}%
=\mathbb{K}$, $A=\mathbb{L}$ and $M=\mathbb{L}$. This particular case will
suffice for us. Examples of $\mathbb{K}$-derivations abound (there are several
in \cite{Grinbe16}), but the only one we will need is the following:

\begin{proposition}
\label{prop.deriv.del}Let $\partial:\mathbb{K}\left[  t\right]  \rightarrow
\mathbb{K}\left[  t\right]  $ be the differentiation operator (i.e., the map
that sends every polynomial $f\in\mathbb{K}\left[  t\right]  $ to the
derivative of $f$). Then, $\partial:\mathbb{K}\left[  t\right]  \rightarrow
\mathbb{K}\left[  t\right]  $ is a $\mathbb{K}$-derivation.
\end{proposition}

\begin{vershort}
\begin{proof}
[Proof of Proposition \ref{prop.deriv.del}.]This follows from the fact that
$\partial\left(  ab\right)  =a\partial\left(  b\right)  +\partial\left(
a\right)  b$ for any two polynomials $a$ and $b$ (the well-known Leibniz law).
\end{proof}
\end{vershort}

\begin{verlong}
\begin{proof}
[Proof of Proposition \ref{prop.deriv.del}.]The well-known Leibniz law for the
derivative of a product yields that%
\begin{align*}
\partial\left(  ab\right)   &  =\partial\left(  a\right)  b+a\partial\left(
b\right) \\
&  =a\partial\left(  b\right)  +\partial\left(  a\right)
b\ \ \ \ \ \ \ \ \ \ \text{for every }a\in\mathbb{K}\left[  t\right]  \text{
and }b\in\mathbb{K}\left[  t\right]  .
\end{align*}
But the map $\partial:\mathbb{K}\left[  t\right]  \rightarrow\mathbb{K}\left[
t\right]  $ is $\mathbb{K}$-linear. Hence, $\partial$ is a $\mathbb{K}%
$-derivation if and only if it satisfies%
\begin{equation}
\left(  \partial\left(  ab\right)  =a\partial\left(  b\right)  +\partial
\left(  a\right)  b\ \ \ \ \ \ \ \ \ \ \text{for every }a\in\mathbb{K}\left[
t\right]  \text{ and }b\in\mathbb{K}\left[  t\right]  \right)  .
\label{pf.prop.deriv.del.1}%
\end{equation}
Hence, $\partial$ is a $\mathbb{K}$-derivation (since $\partial$ satisfies
(\ref{pf.prop.deriv.del.1})). This proves Proposition \ref{prop.deriv.del}.
\end{proof}
\end{verlong}

A fundamental fact about $\mathbb{K}$-derivations is the following:

\begin{proposition}
\label{prop.deriv.leibniz}Let $\mathbb{L}$ be a $\mathbb{K}$-algebra. Let
$f:\mathbb{L}\rightarrow\mathbb{L}$ be a $\mathbb{K}$-derivation. Let
$n\in\mathbb{N}$, and let $a_{1},a_{2},\ldots,a_{n}\in\mathbb{L}$. Then,%
\[
f\left(  a_{1}a_{2}\cdots a_{n}\right)  =\sum_{i=1}^{n}a_{1}a_{2}\cdots
a_{i-1}f\left(  a_{i}\right)  a_{i+1}a_{i+2}\cdots a_{n}.
\]

\end{proposition}

This proposition is a particular case of \cite[Theorem 1.14]{Grinbe16}
(obtained by setting $\mathbf{k}=\mathbb{K}$, $A=\mathbb{L}$ and
$M=\mathbb{L}$); it is also easy to prove\footnote{First one should show that
$f\left(  1\right)  =0$ (by applying (\ref{eq.def.deriv.axiom}) to $a=1$ and
$b=1$). Then, one can prove Proposition \ref{prop.deriv.leibniz} by
straightforward induction on $n$.}.

What we are going to need is a formula for how a derivation acts on the
determinant of a matrix. We first introduce a notation:

\begin{definition}
\label{def.f-mat}Let $n\in\mathbb{N}$ and $m\in\mathbb{N}$. Let $\mathbb{L}$
and $\mathbb{M}$ be rings. Let $f:\mathbb{L}\rightarrow\mathbb{M}$ be any map.
Then, $f^{n\times m}$ will denote the map from $\mathbb{L}^{n\times m}$ to
$\mathbb{M}^{n\times m}$ which sends every matrix $\left(  a_{i,j}\right)
_{1\leq i\leq n,\ 1\leq j\leq m}\in\mathbb{L}^{n\times m}$ to the matrix
$\left(  f\left(  a_{i,j}\right)  \right)  _{1\leq i\leq n,\ 1\leq j\leq m}%
\in\mathbb{M}^{n\times m}$. (In other words, $f^{n\times m}$ is the map which
takes an $n\times m$-matrix in $\mathbb{L}^{n\times m}$, and applies $f$ to
each entry of this matrix.)
\end{definition}

\begin{theorem}
\label{thm.deriv.ddet}Let $\mathbb{L}$ be a commutative $\mathbb{K}$-algebra.
Let $f:\mathbb{L}\rightarrow\mathbb{L}$ be a $\mathbb{K}$-derivation. Let
$n\in\mathbb{N}$. Let $A\in\mathbb{L}^{n\times n}$. Then,%
\[
f\left(  \det A\right)  =\operatorname*{Tr}\left(  f^{n\times n}\left(
A\right)  \cdot\operatorname*{adj}A\right)  .
\]

\end{theorem}

Proving Theorem \ref{thm.deriv.ddet} will take us a while. Let us begin by
stating three lemmas:

\begin{lemma}
\label{lem.Tr(AB)}Let $n\in\mathbb{N}$ and $m\in\mathbb{N}$. Let $A=\left(
a_{i,j}\right)  _{1\leq i\leq n,\ 1\leq j\leq m}\in\mathbb{K}^{n\times m}$ and
$B=\left(  b_{i,j}\right)  _{1\leq i\leq m,\ 1\leq j\leq n}\in\mathbb{K}%
^{m\times n}$. Then,%
\[
\operatorname*{Tr}\left(  AB\right)  =\sum_{i=1}^{n}\sum_{j=1}^{m}%
a_{i,j}b_{j,i}.
\]

\end{lemma}

\begin{vershort}
\begin{proof}
[Proof of Lemma \ref{lem.Tr(AB)}.]The definition of $AB$ yields $AB=\left(
\sum_{k=1}^{m}a_{i,k}b_{k,j}\right)  _{1\leq i\leq n,\ 1\leq j\leq n}$ (since
$A=\left(  a_{i,j}\right)  _{1\leq i\leq n,\ 1\leq j\leq m}$ and $B=\left(
b_{i,j}\right)  _{1\leq i\leq m,\ 1\leq j\leq n}$). Hence,%
\[
\operatorname*{Tr}\left(  AB\right)  =\sum_{i=1}^{n}\sum_{k=1}^{m}%
a_{i,k}b_{k,i}=\sum_{i=1}^{n}\sum_{j=1}^{m}a_{i,j}b_{j,i}%
\]
(here, we have renamed the summation index $k$ as $j$ in the second sum). This
proves Lemma \ref{lem.Tr(AB)}.
\end{proof}
\end{vershort}

\begin{verlong}
\begin{proof}
[Proof of Lemma \ref{lem.Tr(AB)}.]Recall that the trace $\operatorname*{Tr}C$
of an $n\times n$-matrix $C$ is the sum of all diagonal entries of $C$
(because this is how $\operatorname*{Tr}C$ is defined). In other words, if
$C=\left(  c_{i,j}\right)  _{1\leq i\leq n,\ 1\leq j\leq n}$ is an $n\times
n$-matrix, then%
\begin{equation}
\operatorname*{Tr}C=\sum_{i=1}^{n}c_{i,i}. \label{pf.lem.Tr(AB).TrC=}%
\end{equation}

From $A=\left(  a_{i,j}\right)  _{1\leq i\leq n,\ 1\leq j\leq m}$ and
$B=\left(  b_{i,j}\right)  _{1\leq i\leq m,\ 1\leq j\leq n}$, we obtain
\[
AB=\left(  \sum_{k=1}^{m}a_{i,k}b_{k,j}\right)  _{1\leq i\leq n,\ 1\leq j\leq
n}%
\]
(by the definition of the matrix $AB$). Hence, (\ref{pf.lem.Tr(AB).TrC=})
(applied to $AB$ and $\sum_{k=1}^{m}a_{i,k}b_{k,j}$ instead of $C$ and
$c_{i,j}$) yields%
\[
\operatorname*{Tr}\left(  AB\right)  =\sum_{i=1}^{n}\underbrace{\sum_{k=1}%
^{m}a_{i,k}b_{k,i}}_{\substack{=\sum_{j=1}^{m}a_{i,j}b_{j,i}\\\text{(here, we
have renamed the}\\\text{summation index }k\text{ as }j\text{)}}}=\sum
_{i=1}^{n}\sum_{j=1}^{m}a_{i,j}b_{j,i}.
\]
This proves Lemma \ref{lem.Tr(AB)}.
\end{proof}
\end{verlong}

\begin{lemma}
\label{lem.deriv.leibniz-comm}Let $\mathbb{L}$ be a commutative $\mathbb{K}%
$-algebra. Let $f:\mathbb{L}\rightarrow\mathbb{L}$ be a $\mathbb{K}%
$-derivation. Let $n\in\mathbb{N}$, and let $a_{1},a_{2},\ldots,a_{n}%
\in\mathbb{L}$. Then,%
\[
f\left(  a_{1}a_{2}\cdots a_{n}\right)  =\sum_{k=1}^{n}f\left(  a_{k}\right)
\prod_{\substack{i\in\left\{  1,2,\ldots,n\right\}  ;\\i\neq k}}a_{i}.
\]

\end{lemma}

\begin{vershort}
\begin{proof}
[Proof of Lemma \ref{lem.deriv.leibniz-comm}.]Proposition
\ref{prop.deriv.leibniz} yields%
\begin{align*}
f\left(  a_{1}a_{2}\cdots a_{n}\right)   &  =\sum_{i=1}^{n}a_{1}a_{2}\cdots
a_{i-1}f\left(  a_{i}\right)  a_{i+1}a_{i+2}\cdots a_{n}\\
&  =\sum_{k=1}^{n}\underbrace{a_{1}a_{2}\cdots a_{k-1}f\left(  a_{k}\right)
a_{k+1}a_{k+2}\cdots a_{n}}_{\substack{=f\left(  a_{k}\right)  \left(
a_{1}a_{2}\cdots a_{k-1}\right)  \left(  a_{k+1}a_{k+2}\cdots a_{n}\right)
\\\text{(since }\mathbb{L}\text{ is commutative)}}}\\
&  \ \ \ \ \ \ \ \ \ \ \left(  \text{here, we have renamed the summation index
}i\text{ as }k\right) \\
&  =\sum_{k=1}^{n}f\left(  a_{k}\right)  \underbrace{\left(  a_{1}a_{2}\cdots
a_{k-1}\right)  \left(  a_{k+1}a_{k+2}\cdots a_{n}\right)  }_{=\prod
_{\substack{i\in\left\{  1,2,\ldots,n\right\}  ;\\i\neq k}}a_{i}}=\sum
_{k=1}^{n}f\left(  a_{k}\right)  \prod_{\substack{i\in\left\{  1,2,\ldots
,n\right\}  ;\\i\neq k}}a_{i}.
\end{align*}
This proves Lemma \ref{lem.deriv.leibniz-comm}.
\end{proof}
\end{vershort}

\begin{verlong}
\begin{proof}
[Proof of Lemma \ref{lem.deriv.leibniz-comm}.]Let $k\in\left\{  1,2,\ldots
,n\right\}  $.

We recall a basic property of products (in a commutative ring): If $U$ and $V$
are two disjoint finite sets, and if $b_{i}$ is an element of $\mathbb{L}$ for
every $i\in U\cup V$, then%
\begin{equation}
\prod_{i\in U\cup V}b_{i}=\left(  \prod_{i\in U}b_{i}\right)  \left(
\prod_{i\in V}b_{i}\right)  . \label{pf.lem.deriv.leibniz-comm.1}%
\end{equation}
(This relies on the fact that the ring $\mathbb{L}$ is commutative.)

Now, the sets $\left\{  1,2,\ldots,k-1\right\}  $ and $\left\{  k+1,k+2,\ldots
,n\right\}  $ are disjoint\footnote{\textit{Proof.} Let $g\in\left\{
1,2,\ldots,k-1\right\}  \cap\left\{  k+1,k+2,\ldots,n\right\}  $. Then,
$g\in\left\{  1,2,\ldots,k-1\right\}  \cap\left\{  k+1,k+2,\ldots,n\right\}
\subseteq\left\{  1,2,\ldots,k-1\right\}  $, so that $g\leq k-1$. Also,
$g\in\left\{  1,2,\ldots,k-1\right\}  \cap\left\{  k+1,k+2,\ldots,n\right\}
\subseteq\left\{  k+1,k+2,\ldots,n\right\}  $, so that $g\geq k+1$. Combining
$g\leq k-1\leq k$ with $g\geq k+1>k$, we obtain a contradiction.
\par
Now, forget that we fixed $g$. We thus have found a contradiction for each
$g\in\left\{  1,2,\ldots,k-1\right\}  \cap\left\{  k+1,k+2,\ldots,n\right\}
$. Hence, there exists no $g\in\left\{  1,2,\ldots,k-1\right\}  \cap\left\{
k+1,k+2,\ldots,n\right\}  $. In other words, the set $\left\{  1,2,\ldots
,k-1\right\}  \cap\left\{  k+1,k+2,\ldots,n\right\}  $ is empty. In other
words, the sets $\left\{  1,2,\ldots,k-1\right\}  $ and $\left\{
k+1,k+2,\ldots,n\right\}  $ are disjoint. Qed.} and finite. Hence,
(\ref{pf.lem.deriv.leibniz-comm.1}) (applied to $U=\left\{  1,2,\ldots
,k-1\right\}  $, $V=\left\{  k+1,k+2,\ldots,n\right\}  $ and $b_{i}=a_{i}$)
yields%
\[
\prod_{i\in\left\{  1,2,\ldots,k-1\right\}  \cup\left\{  k+1,k+2,\ldots
,n\right\}  }a_{i}=\left(  \prod_{i\in\left\{  1,2,\ldots,k-1\right\}  }%
a_{i}\right)  \left(  \prod_{i\in\left\{  k+1,k+2,\ldots,n\right\}  }%
a_{i}\right)  .
\]

Now,%
\begin{align*}
\underbrace{\prod_{\substack{i\in\left\{  1,2,\ldots,n\right\}  ;\\i\neq k}%
}}_{=\prod_{i\in\left\{  1,2,\ldots,n\right\}  \setminus\left\{  k\right\}  }%
}a_{i}  &  =\underbrace{\prod_{i\in\left\{  1,2,\ldots,n\right\}
\setminus\left\{  k\right\}  }}_{\substack{=\prod_{i\in\left\{  1,2,\ldots
,k-1\right\}  \cup\left\{  k+1,k+2,\ldots,n\right\}  }\\\text{(since }\left\{
1,2,\ldots,n\right\}  \setminus\left\{  k\right\}  \\=\left\{  1,2,\ldots
,k-1,k+1,k+2,\ldots,n\right\}  \\=\left\{  1,2,\ldots,k-1\right\}
\cup\left\{  k+1,k+2,\ldots,n\right\}  \text{)}}}a_{i}\\
&  =\prod_{i\in\left\{  1,2,\ldots,k-1\right\}  \cup\left\{  k+1,k+2,\ldots
,n\right\}  }a_{i}\\
&  =\underbrace{\left(  \prod_{i\in\left\{  1,2,\ldots,k-1\right\}  }%
a_{i}\right)  }_{=a_{1}a_{2}\cdots a_{k-1}}\underbrace{\left(  \prod
_{i\in\left\{  k+1,k+2,\ldots,n\right\}  }a_{i}\right)  }_{=a_{k+1}%
a_{k+2}\cdots a_{n}}\\
&  =\left(  a_{1}a_{2}\cdots a_{k-1}\right)  \left(  a_{k+1}a_{k+2}\cdots
a_{n}\right)  .
\end{align*}
Hence,%
\begin{align}
&  f\left(  a_{k}\right)  \underbrace{\prod_{\substack{i\in\left\{
1,2,\ldots,n\right\}  ;\\i\neq k}}a_{i}}_{=\left(  a_{1}a_{2}\cdots
a_{k-1}\right)  \left(  a_{k+1}a_{k+2}\cdots a_{n}\right)  }\nonumber\\
&  =\underbrace{f\left(  a_{k}\right)  \left(  a_{1}a_{2}\cdots a_{k-1}%
\right)  }_{\substack{=\left(  a_{1}a_{2}\cdots a_{k-1}\right)  f\left(
a_{k}\right)  \\\text{(since }\mathbb{L}\text{ is commutative)}}}\left(
a_{k+1}a_{k+2}\cdots a_{n}\right) \nonumber\\
&  =\left(  a_{1}a_{2}\cdots a_{k-1}\right)  f\left(  a_{k}\right)  \left(
a_{k+1}a_{k+2}\cdots a_{n}\right) \nonumber\\
&  =a_{1}a_{2}\cdots a_{k-1}f\left(  a_{k}\right)  a_{k+1}a_{k+2}\cdots a_{n}.
\label{pf.lem.deriv.leibniz-comm.3}%
\end{align}

Now, forget that we fixed $k$. We thus have shown that
(\ref{pf.lem.deriv.leibniz-comm.3}) holds for every $k\in\left\{
1,2,\ldots,n\right\}  $.

Now, Proposition \ref{prop.deriv.leibniz} yields%
\begin{align*}
f\left(  a_{1}a_{2}\cdots a_{n}\right)   &  =\sum_{i=1}^{n}a_{1}a_{2}\cdots
a_{i-1}f\left(  a_{i}\right)  a_{i+1}a_{i+2}\cdots a_{n}\\
&  =\sum_{k=1}^{n}\underbrace{a_{1}a_{2}\cdots a_{k-1}f\left(  a_{k}\right)
a_{k+1}a_{k+2}\cdots a_{n}}_{\substack{=f\left(  a_{k}\right)  \prod
_{\substack{i\in\left\{  1,2,\ldots,n\right\}  ;\\i\neq k}}a_{i}\\\text{(by
(\ref{pf.lem.deriv.leibniz-comm.3}))}}}\\
&  \ \ \ \ \ \ \ \ \ \ \left(  \text{here, we have renamed the summation index
}i\text{ as }k\right) \\
&  =\sum_{k=1}^{n}f\left(  a_{k}\right)  \prod_{\substack{i\in\left\{
1,2,\ldots,n\right\}  ;\\i\neq k}}a_{i}.
\end{align*}
This proves Lemma \ref{lem.deriv.leibniz-comm}.
\end{proof}
\end{verlong}

\begin{lemma}
\label{lem.laplace.Apq}Let $n\in\mathbb{N}$. Let $A=\left(  a_{i,j}\right)
_{1\leq i\leq n,\ 1\leq j\leq n}$ be an $n\times n$-matrix. Let $p\in\left\{
1,2,\ldots,n\right\}  $ and $q\in\left\{  1,2,\ldots,n\right\}  $. Then,%
\[
\sum_{\substack{\sigma\in S_{n};\\\sigma\left(  p\right)  =q}}\left(
-1\right)  ^{\sigma}\prod_{\substack{i\in\left\{  1,2,\ldots,n\right\}
;\\i\neq p}}a_{i,\sigma\left(  i\right)  }=\left(  -1\right)  ^{p+q}%
\det\left(  A_{\sim p,\sim q}\right)  .
\]

\end{lemma}

Lemma \ref{lem.laplace.Apq} is \cite[Lemma 6.84]{detnotes}; it is also easy to
prove (it is the main step in the proof of the Laplace expansion formula for
the determinant).

\begin{vershort}
\begin{proof}
[Proof of Theorem \ref{thm.deriv.ddet}.]Write the matrix $A\in\mathbb{L}%
^{n\times n}$ in the form $A=\left(  a_{i,j}\right)  _{1\leq i\leq n,\ 1\leq
j\leq n}$. Hence, $f^{n\times n}\left(  A\right)  =\left(  f\left(
a_{i,j}\right)  \right)  _{1\leq i\leq n,\ 1\leq j\leq n}$ (by the definition
of $f^{n\times n}$). The definition of $\operatorname*{adj}A$ shows that
$\operatorname*{adj}A=\left(  \left(  -1\right)  ^{i+j}\det\left(  A_{\sim
j,\sim i}\right)  \right)  _{1\leq i\leq n,\ 1\leq j\leq n}$. Hence, Lemma
\ref{lem.Tr(AB)} (applied to $\mathbb{L}$, $n$, $f^{n\times n}\left(
A\right)  $, $f\left(  a_{i,j}\right)  $, $\operatorname*{adj}A$ and $\left(
-1\right)  ^{i+j}\det\left(  A_{\sim j,\sim i}\right)  $ instead of
$\mathbb{K}$, $m$, $A$, $a_{i,j}$, $B$ and $b_{i,j}$) yields%
\begin{align}
\operatorname*{Tr}\left(  f^{n\times n}\left(  A\right)  \cdot
\operatorname*{adj}A\right)   &  =\sum_{i=1}^{n}\sum_{j=1}^{n}f\left(
a_{i,j}\right)  \left(  -1\right)  ^{j+i}\det\left(  A_{\sim i,\sim j}\right)
\nonumber\\
&  =\sum_{k=1}^{n}\sum_{j=1}^{n}f\left(  a_{k,j}\right)  \underbrace{\left(
-1\right)  ^{j+k}}_{=\left(  -1\right)  ^{k+j}}\det\left(  A_{\sim k,\sim
j}\right) \nonumber\\
&  \ \ \ \ \ \ \ \ \ \ \left(
\begin{array}
[c]{c}%
\text{here, we have renamed the summation index }i\\
\text{as }k\text{ in the outer sum}%
\end{array}
\right) \nonumber\\
&  =\sum_{k=1}^{n}\sum_{j=1}^{n}f\left(  a_{k,j}\right)  \left(  -1\right)
^{k+j}\det\left(  A_{\sim k,\sim j}\right)  .
\label{pf.thm.deriv.ddet.short.Tr=}%
\end{align}

But the map $f$ is a $\mathbb{K}$-derivation, and thus is $\mathbb{K}$-linear.
Now, (\ref{eq.det}) (applied to $\mathbb{L}$ instead of $\mathbb{K}$) yields
$\det A=\sum_{\sigma\in S_{n}}\left(  -1\right)  ^{\sigma}\prod_{i=1}%
^{n}a_{i,\sigma\left(  i\right)  }$. Applying $f$ to both sides of this
equality, we find%
\begin{align}
&  f\left(  \det A\right) \nonumber\\
&  =f\left(  \sum_{\sigma\in S_{n}}\left(  -1\right)  ^{\sigma}\prod_{i=1}%
^{n}a_{i,\sigma\left(  i\right)  }\right) \nonumber\\
&  =\sum_{\sigma\in S_{n}}\left(  -1\right)  ^{\sigma}f\left(
\underbrace{\prod_{i=1}^{n}a_{i,\sigma\left(  i\right)  }}_{=a_{1,\sigma
\left(  1\right)  }a_{2,\sigma\left(  2\right)  }\cdots a_{n,\sigma\left(
n\right)  }}\right)  \ \ \ \ \ \ \ \ \ \ \left(  \text{since the map }f\text{
is }\mathbb{K}\text{-linear}\right) \nonumber\\
&  =\sum_{\sigma\in S_{n}}\left(  -1\right)  ^{\sigma}\underbrace{f\left(
a_{1,\sigma\left(  1\right)  }a_{2,\sigma\left(  2\right)  }\cdots
a_{n,\sigma\left(  n\right)  }\right)  }_{\substack{=\sum_{k=1}^{n}f\left(
a_{k,\sigma\left(  k\right)  }\right)  \prod_{\substack{i\in\left\{
1,2,\ldots,n\right\}  ;\\i\neq k}}a_{i,\sigma\left(  i\right)  }\\\text{(by
Lemma \ref{lem.deriv.leibniz-comm}, applied to}\\a_{i,\sigma\left(  i\right)
}\text{ instead of }a_{i}\text{)}}}\nonumber\\
&  =\sum_{\sigma\in S_{n}}\left(  -1\right)  ^{\sigma}\sum_{k=1}^{n}f\left(
a_{k,\sigma\left(  k\right)  }\right)  \prod_{\substack{i\in\left\{
1,2,\ldots,n\right\}  ;\\i\neq k}}a_{i,\sigma\left(  i\right)  }\nonumber\\
&  =\sum_{k=1}^{n}\sum_{\sigma\in S_{n}}\left(  -1\right)  ^{\sigma}f\left(
a_{k,\sigma\left(  k\right)  }\right)  \prod_{\substack{i\in\left\{
1,2,\ldots,n\right\}  ;\\i\neq k}}a_{i,\sigma\left(  i\right)  }.
\label{pf.thm.deriv.ddet.short.1}%
\end{align}
But every $k\in\left\{  1,2,\ldots,n\right\}  $ satisfies%
\begin{align*}
&  \underbrace{\sum_{\sigma\in S_{n}}}_{\substack{=\sum_{j\in\left\{
1,2,\ldots,n\right\}  }\sum_{\substack{\sigma\in S_{n};\\\sigma\left(
k\right)  =j}}\\\text{(since }\sigma\left(  k\right)  \in\left\{
1,2,\ldots,n\right\}  \text{ for each }\sigma\in S_{n}\text{)}}}\left(
-1\right)  ^{\sigma}f\left(  a_{k,\sigma\left(  k\right)  }\right)
\prod_{\substack{i\in\left\{  1,2,\ldots,n\right\}  ;\\i\neq k}}a_{i,\sigma
\left(  i\right)  }\\
&  =\sum_{j\in\left\{  1,2,\ldots,n\right\}  }\sum_{\substack{\sigma\in
S_{n};\\\sigma\left(  k\right)  =j}}\left(  -1\right)  ^{\sigma}f\left(
\underbrace{a_{k,\sigma\left(  k\right)  }}_{\substack{=a_{k,j}\\\text{(since
}\sigma\left(  k\right)  =j\text{)}}}\right)  \prod_{\substack{i\in\left\{
1,2,\ldots,n\right\}  ;\\i\neq k}}a_{i,\sigma\left(  i\right)  }\\
&  =\sum_{j\in\left\{  1,2,\ldots,n\right\}  }\sum_{\substack{\sigma\in
S_{n};\\\sigma\left(  k\right)  =j}}\left(  -1\right)  ^{\sigma}f\left(
a_{k,j}\right)  \prod_{\substack{i\in\left\{  1,2,\ldots,n\right\}  ;\\i\neq
k}}a_{i,\sigma\left(  i\right)  }\\
&  =\underbrace{\sum_{j\in\left\{  1,2,\ldots,n\right\}  }}_{=\sum_{j=1}^{n}%
}f\left(  a_{k,j}\right)  \underbrace{\sum_{\substack{\sigma\in S_{n}%
;\\\sigma\left(  k\right)  =j}}\left(  -1\right)  ^{\sigma}\prod
_{\substack{i\in\left\{  1,2,\ldots,n\right\}  ;\\i\neq k}}a_{i,\sigma\left(
i\right)  }}_{\substack{=\left(  -1\right)  ^{k+j}\det\left(  A_{\sim k,\sim
j}\right)  \\\text{(by Lemma \ref{lem.laplace.Apq}, applied to}\\\mathbb{L}%
\text{, }k\text{ and }j\text{ instead of }\mathbb{K}\text{, }p\text{ and
}q\text{)}}}\\
&  =\sum_{j=1}^{n}f\left(  a_{k,j}\right)  \left(  -1\right)  ^{k+j}%
\det\left(  A_{\sim k,\sim j}\right)  .
\end{align*}
Hence, (\ref{pf.thm.deriv.ddet.short.1}) becomes%
\begin{align*}
f\left(  \det A\right)   &  =\sum_{k=1}^{n}\underbrace{\sum_{\sigma\in S_{n}%
}\left(  -1\right)  ^{\sigma}f\left(  a_{k,\sigma\left(  k\right)  }\right)
\prod_{\substack{i\in\left\{  1,2,\ldots,n\right\}  ;\\i\neq k}}a_{i,\sigma
\left(  i\right)  }}_{=\sum_{j=1}^{n}f\left(  a_{k,j}\right)  \left(
-1\right)  ^{k+j}\det\left(  A_{\sim k,\sim j}\right)  }\\
&  =\sum_{k=1}^{n}\sum_{j=1}^{n}f\left(  a_{k,j}\right)  \left(  -1\right)
^{k+j}\det\left(  A_{\sim k,\sim j}\right)  =\operatorname*{Tr}\left(
f^{n\times n}\left(  A\right)  \cdot\operatorname*{adj}A\right)
\end{align*}
(by (\ref{pf.thm.deriv.ddet.short.Tr=})). This proves Theorem
\ref{thm.deriv.ddet}.
\end{proof}
\end{vershort}

\begin{verlong}
\begin{proof}
[Proof of Theorem \ref{thm.deriv.ddet}.]Write the matrix $A\in\mathbb{L}%
^{n\times n}$ in the form $A=\left(  a_{i,j}\right)  _{1\leq i\leq n,\ 1\leq
j\leq n}$. Thus, $a_{i,j}\in\mathbb{L}$ for every $\left(  i,j\right)
\in\left\{  1,2,\ldots,n\right\}  ^{2}$.

From $A=\left(  a_{i,j}\right)  _{1\leq i\leq n,\ 1\leq j\leq n}$, we obtain
$f^{n\times n}\left(  A\right)  =\left(  f\left(  a_{i,j}\right)  \right)
_{1\leq i\leq n,\ 1\leq j\leq n}$ (by the definition of $f^{n\times n}$).
Recall that $\operatorname*{adj}A=\left(  \left(  -1\right)  ^{i+j}\det\left(
A_{\sim j,\sim i}\right)  \right)  _{1\leq i\leq n,\ 1\leq j\leq n}$ (by the
definition of $\operatorname*{adj}A$). Hence, Lemma \ref{lem.Tr(AB)} (applied
to $\mathbb{L}$, $n$, $f^{n\times n}\left(  A\right)  $, $f\left(
a_{i,j}\right)  $, $\operatorname*{adj}A$ and $\left(  -1\right)  ^{i+j}%
\det\left(  A_{\sim j,\sim i}\right)  $ instead of $\mathbb{K}$, $m$, $A$,
$a_{i,j}$, $B$ and $b_{i,j}$) yields%
\begin{align}
\operatorname*{Tr}\left(  f^{n\times n}\left(  A\right)  \cdot
\operatorname*{adj}A\right)   &  =\sum_{i=1}^{n}\sum_{j=1}^{n}f\left(
a_{i,j}\right)  \left(  -1\right)  ^{j+i}\det\left(  A_{\sim i,\sim j}\right)
\nonumber\\
&  =\sum_{k=1}^{n}\sum_{j=1}^{n}f\left(  a_{k,j}\right)  \underbrace{\left(
-1\right)  ^{j+k}}_{\substack{=\left(  -1\right)  ^{k+j}\\\text{(since
}j+k=k+j\text{)}}}\det\left(  A_{\sim k,\sim j}\right) \nonumber\\
&  \ \ \ \ \ \ \ \ \ \ \left(
\begin{array}
[c]{c}%
\text{here, we have renamed the summation index }i\\
\text{as }k\text{ in the outer sum}%
\end{array}
\right) \nonumber\\
&  =\sum_{k=1}^{n}\sum_{j=1}^{n}f\left(  a_{k,j}\right)  \left(  -1\right)
^{k+j}\det\left(  A_{\sim k,\sim j}\right)  . \label{pf.thm.deriv.ddet.Tr=}%
\end{align}

But the map $f$ is a $\mathbb{K}$-derivation, and thus is $\mathbb{K}$-linear.
Now, (\ref{eq.det}) (applied to $\mathbb{L}$ instead of $\mathbb{K}$) shows
that%
\[
\det A=\sum_{\sigma\in S_{n}}\left(  -1\right)  ^{\sigma}\prod_{i=1}%
^{n}a_{i,\sigma\left(  i\right)  }.
\]
Applying the map $f$ to both sides of this equation, we obtain%
\begin{align}
&  f\left(  \det A\right) \nonumber\\
&  =f\left(  \sum_{\sigma\in S_{n}}\left(  -1\right)  ^{\sigma}\prod_{i=1}%
^{n}a_{i,\sigma\left(  i\right)  }\right) \nonumber\\
&  =\sum_{\sigma\in S_{n}}\left(  -1\right)  ^{\sigma}f\left(
\underbrace{\prod_{i=1}^{n}a_{i,\sigma\left(  i\right)  }}_{=a_{1,\sigma
\left(  1\right)  }a_{2,\sigma\left(  2\right)  }\cdots a_{n,\sigma\left(
n\right)  }}\right)  \ \ \ \ \ \ \ \ \ \ \left(  \text{since the map }f\text{
is }\mathbb{K}\text{-linear}\right) \nonumber\\
&  =\sum_{\sigma\in S_{n}}\left(  -1\right)  ^{\sigma}\underbrace{f\left(
a_{1,\sigma\left(  1\right)  }a_{2,\sigma\left(  2\right)  }\cdots
a_{n,\sigma\left(  n\right)  }\right)  }_{\substack{=\sum_{k=1}^{n}f\left(
a_{k,\sigma\left(  k\right)  }\right)  \prod_{\substack{i\in\left\{
1,2,\ldots,n\right\}  ;\\i\neq k}}a_{i,\sigma\left(  i\right)  }\\\text{(by
Lemma \ref{lem.deriv.leibniz-comm},}\\\text{applied to }a_{i,\sigma\left(
i\right)  }\text{ instead of }a_{i}\text{)}}}\nonumber\\
&  =\sum_{\sigma\in S_{n}}\left(  -1\right)  ^{\sigma}\sum_{k=1}^{n}f\left(
a_{k,\sigma\left(  k\right)  }\right)  \prod_{\substack{i\in\left\{
1,2,\ldots,n\right\}  ;\\i\neq k}}a_{i,\sigma\left(  i\right)  }\nonumber\\
&  =\underbrace{\sum_{\sigma\in S_{n}}\sum_{k=1}^{n}}_{=\sum_{k=1}^{n}%
\sum_{\sigma\in S_{n}}}\left(  -1\right)  ^{\sigma}f\left(  a_{k,\sigma\left(
k\right)  }\right)  \prod_{\substack{i\in\left\{  1,2,\ldots,n\right\}
;\\i\neq k}}a_{i,\sigma\left(  i\right)  }\nonumber\\
&  =\sum_{k=1}^{n}\sum_{\sigma\in S_{n}}\left(  -1\right)  ^{\sigma}f\left(
a_{k,\sigma\left(  k\right)  }\right)  \prod_{\substack{i\in\left\{
1,2,\ldots,n\right\}  ;\\i\neq k}}a_{i,\sigma\left(  i\right)  }.
\label{pf.thm.deriv.ddet.1}%
\end{align}
But every $k\in\left\{  1,2,\ldots,n\right\}  $ satisfies%
\begin{align*}
&  \underbrace{\sum_{\sigma\in S_{n}}}_{\substack{=\sum_{j\in\left\{
1,2,\ldots,n\right\}  }\sum_{\substack{\sigma\in S_{n};\\\sigma\left(
k\right)  =j}}\\\text{(since }\sigma\left(  k\right)  \in\left\{
1,2,\ldots,n\right\}  \text{ for each }\sigma\in S_{n}\text{)}}}\left(
-1\right)  ^{\sigma}f\left(  a_{k,\sigma\left(  k\right)  }\right)
\prod_{\substack{i\in\left\{  1,2,\ldots,n\right\}  ;\\i\neq k}}a_{i,\sigma
\left(  i\right)  }\\
&  =\sum_{j\in\left\{  1,2,\ldots,n\right\}  }\sum_{\substack{\sigma\in
S_{n};\\\sigma\left(  k\right)  =j}}\left(  -1\right)  ^{\sigma}f\left(
\underbrace{a_{k,\sigma\left(  k\right)  }}_{\substack{=a_{k,j}\\\text{(since
}\sigma\left(  k\right)  =j\text{)}}}\right)  \prod_{\substack{i\in\left\{
1,2,\ldots,n\right\}  ;\\i\neq k}}a_{i,\sigma\left(  i\right)  }\\
&  =\sum_{j\in\left\{  1,2,\ldots,n\right\}  }\underbrace{\sum
_{\substack{\sigma\in S_{n};\\\sigma\left(  k\right)  =j}}\left(  -1\right)
^{\sigma}f\left(  a_{k,j}\right)  \prod_{\substack{i\in\left\{  1,2,\ldots
,n\right\}  ;\\i\neq k}}a_{i,\sigma\left(  i\right)  }}_{=f\left(
a_{k,j}\right)  \sum_{\substack{\sigma\in S_{n};\\\sigma\left(  k\right)
=j}}\left(  -1\right)  ^{\sigma}\prod_{\substack{i\in\left\{  1,2,\ldots
,n\right\}  ;\\i\neq k}}a_{i,\sigma\left(  i\right)  }}\\
&  =\underbrace{\sum_{j\in\left\{  1,2,\ldots,n\right\}  }}_{=\sum_{j=1}^{n}%
}f\left(  a_{k,j}\right)  \underbrace{\sum_{\substack{\sigma\in S_{n}%
;\\\sigma\left(  k\right)  =j}}\left(  -1\right)  ^{\sigma}\prod
_{\substack{i\in\left\{  1,2,\ldots,n\right\}  ;\\i\neq k}}a_{i,\sigma\left(
i\right)  }}_{\substack{=\left(  -1\right)  ^{k+j}\det\left(  A_{\sim k,\sim
j}\right)  \\\text{(by Lemma \ref{lem.laplace.Apq}, applied to}\\\mathbb{L}%
\text{, }k\text{ and }j\text{ instead of }\mathbb{K}\text{, }p\text{ and
}q\text{)}}}\\
&  =\sum_{j=1}^{n}f\left(  a_{k,j}\right)  \left(  -1\right)  ^{k+j}%
\det\left(  A_{\sim k,\sim j}\right)  .
\end{align*}
Hence, (\ref{pf.thm.deriv.ddet.1}) becomes%
\begin{align*}
&  f\left(  \det A\right) \\
&  =\sum_{k=1}^{n}\underbrace{\sum_{\sigma\in S_{n}}\left(  -1\right)
^{\sigma}f\left(  a_{k,\sigma\left(  k\right)  }\right)  \prod_{\substack{i\in
\left\{  1,2,\ldots,n\right\}  ;\\i\neq k}}a_{i,\sigma\left(  i\right)  }%
}_{=\sum_{j=1}^{n}f\left(  a_{k,j}\right)  \left(  -1\right)  ^{k+j}%
\det\left(  A_{\sim k,\sim j}\right)  }\\
&  =\sum_{k=1}^{n}\sum_{j=1}^{n}f\left(  a_{k,j}\right)  \left(  -1\right)
^{k+j}\det\left(  A_{\sim k,\sim j}\right)  =\operatorname*{Tr}\left(
f^{n\times n}\left(  A\right)  \cdot\operatorname*{adj}A\right)
\end{align*}
(by (\ref{pf.thm.deriv.ddet.Tr=})). This proves Theorem \ref{thm.deriv.ddet}.
\end{proof}
\end{verlong}

\subsection{The derivative of the characteristic polynomial}

The characteristic polynomial $\chi_{A}$ of a square matrix $A$ is, first of
all, a polynomial; and a polynomial has a derivative. We shall have need for a
formula for this derivative:

\begin{theorem}
\label{thm.ddet}Let $n\in\mathbb{N}$. Let $A\in\mathbb{K}^{n\times n}$. Let
$\partial:\mathbb{K}\left[  t\right]  \rightarrow\mathbb{K}\left[  t\right]  $
be the differentiation operator (i.e., the map that sends every polynomial
$f\in\mathbb{K}\left[  t\right]  $ to the derivative of $f$). Then,%
\[
\partial\chi_{A}=\operatorname*{Tr}\left(  \operatorname*{adj}\left(
tI_{n}-A\right)  \right)  .
\]

\end{theorem}

\begin{vershort}
\begin{proof}
[Proof of Theorem \ref{thm.ddet}.]Proposition \ref{prop.deriv.del} shows that
$\partial:\mathbb{K}\left[  t\right]  \rightarrow\mathbb{K}\left[  t\right]  $
is a $\mathbb{K}$-derivation. Now, consider the map $\partial^{n\times
n}:\left(  \mathbb{K}\left[  t\right]  \right)  ^{n\times n}\rightarrow\left(
\mathbb{K}\left[  t\right]  \right)  ^{n\times n}$ (defined according to
Definition \ref{def.f-mat}). It is easy to see that
\begin{equation}
\partial^{n\times n}\left(  tB-A\right)  =B \label{pf.thm.ddet.short.1}%
\end{equation}
for any $n\times n$-matrix $B\in\mathbb{K}^{n\times n}$%
\ \ \ \ \footnote{\textit{Proof.} Let $B\in\mathbb{K}^{n\times n}$ be an
$n\times n$-matrix. Write the matrix $B$ in the form $B=\left(  b_{i,j}%
\right)  _{1\leq i\leq n,\ 1\leq j\leq n}$. Write the matrix $A$ in the form
$A=\left(  a_{i,j}\right)  _{1\leq i\leq n,\ 1\leq j\leq n}$. Both matrices
$A$ and $B$ belong to $\mathbb{K}^{n\times n}$; thus, every $\left(
i,j\right)  \in\left\{  1,2,\ldots,n\right\}  ^{2}$ satisfies $a_{i,j}%
\in\mathbb{K}$ and $b_{i,j}\in\mathbb{K}$ and therefore $\partial\left(
tb_{i,j}-a_{i,j}\right)  =b_{i,j}$ (since $\partial$ is the differentiation
operator).
\par
Now,%
\begin{align*}
&  t\underbrace{B}_{=\left(  b_{i,j}\right)  _{1\leq i\leq n,\ 1\leq j\leq n}%
}-\underbrace{A}_{=\left(  a_{i,j}\right)  _{1\leq i\leq n,\ 1\leq j\leq n}}\\
&  =t\left(  b_{i,j}\right)  _{1\leq i\leq n,\ 1\leq j\leq n}-\left(
a_{i,j}\right)  _{1\leq i\leq n,\ 1\leq j\leq n}=\left(  tb_{i,j}%
-a_{i,j}\right)  _{1\leq i\leq n,\ 1\leq j\leq n}.
\end{align*}
Hence, the definition of the map $\partial^{n\times n}$ yields%
\[
\partial^{n\times n}\left(  tB-A\right)  =\left(  \underbrace{\partial\left(
tb_{i,j}-a_{i,j}\right)  }_{=b_{i,j}}\right)  _{1\leq i\leq n,\ 1\leq j\leq
n}=\left(  b_{i,j}\right)  _{1\leq i\leq n,\ 1\leq j\leq n}=B,
\]
qed.}. Applying this to $B=I_{n}$, we obtain $\partial^{n\times n}\left(
tI_{n}-A\right)  =I_{n}$.

The definition of $\chi_{A}$ yields $\chi_{A}=\det\left(  tI_{n}-A\right)  $.
Applying the map $\partial$ to both sides of this equality, we obtain%
\begin{align*}
\partial\chi_{A}  &  =\partial\left(  \det\left(  tI_{n}-A\right)  \right)
=\operatorname*{Tr}\left(  \underbrace{\partial^{n\times n}\left(
tI_{n}-A\right)  }_{=I_{n}}\cdot\operatorname*{adj}\left(  tI_{n}-A\right)
\right) \\
&  \ \ \ \ \ \ \ \ \ \ \left(
\begin{array}
[c]{c}%
\text{by Theorem \ref{thm.deriv.ddet} (applied to }\mathbb{K}\left[  t\right]
\text{, }\partial\text{ and }tI_{n}-A\\
\text{instead of }\mathbb{L}\text{, }f\text{ and }A\text{)}%
\end{array}
\right) \\
&  =\operatorname*{Tr}\left(  \underbrace{I_{n}\cdot\operatorname*{adj}\left(
tI_{n}-A\right)  }_{=\operatorname*{adj}\left(  tI_{n}-A\right)  }\right)
=\operatorname*{Tr}\left(  \operatorname*{adj}\left(  tI_{n}-A\right)
\right)  .
\end{align*}
This proves Theorem \ref{thm.ddet}.
\end{proof}
\end{vershort}

\begin{verlong}
\begin{proof}
[Proof of Theorem \ref{thm.ddet}.]Proposition \ref{prop.deriv.del} shows that
$\partial:\mathbb{K}\left[  t\right]  \rightarrow\mathbb{K}\left[  t\right]  $
is a $\mathbb{K}$-derivation. A map $\partial^{n\times n}:\left(
\mathbb{K}\left[  t\right]  \right)  ^{n\times n}\rightarrow\left(
\mathbb{K}\left[  t\right]  \right)  ^{n\times n}$ is defined (according to
Definition \ref{def.f-mat}).

It is easy to see that $\partial^{n\times n}\left(  tB-A\right)  =B$ for any
$n\times n$-matrix $B\in\mathbb{K}^{n\times n}$%
\ \ \ \ \footnote{\textit{Proof.} Let $B\in\mathbb{K}^{n\times n}$ be an
$n\times n$-matrix. Write the matrix $B$ in the form $B=\left(  b_{i,j}%
\right)  _{1\leq i\leq n,\ 1\leq j\leq n}$. Thus,
\begin{equation}
b_{i,j}\in\mathbb{K}\ \ \ \ \ \ \ \ \ \ \text{for every }\left(  i,j\right)
\in\left\{  1,2,\ldots,n\right\}  ^{2} \label{pf.thm.ddet.fn1.1}%
\end{equation}
(since $B\in\mathbb{K}^{n\times n}$). Also, $t\underbrace{B}_{=\left(
b_{i,j}\right)  _{1\leq i\leq n,\ 1\leq j\leq n}}=t\left(  b_{i,j}\right)
_{1\leq i\leq n,\ 1\leq j\leq n}=\left(  tb_{i,j}\right)  _{1\leq i\leq
n,\ 1\leq j\leq n}$.
\par
Write the matrix $A$ in the form $A=\left(  a_{i,j}\right)  _{1\leq i\leq
n,\ 1\leq j\leq n}$. Thus,%
\begin{equation}
a_{i,j}\in\mathbb{K}\ \ \ \ \ \ \ \ \ \ \text{for every }\left(  i,j\right)
\in\left\{  1,2,\ldots,n\right\}  ^{2} \label{pf.thm.ddet.fn1.2}%
\end{equation}
(since $A\in\mathbb{K}^{n\times n}$).
\par
Fix $\left(  i,j\right)  \in\left\{  1,2,\ldots,n\right\}  ^{2}$. Then,
$b_{i,j}\in\mathbb{K}$ (by (\ref{pf.thm.ddet.fn1.1})) and $a_{i,j}%
\in\mathbb{K}$ (by (\ref{pf.thm.ddet.fn1.2})). Hence, the polynomial
$tb_{i,j}-a_{i,j}\in\mathbb{K}\left[  t\right]  $ has the coefficient
$b_{i,j}$ in front of $t^{1}$ and the coefficient $-a_{i,j}$ in front of
$t^{0}$, and all its other coefficients are $0$. Thus,
\begin{equation}
\partial\left(  tb_{i,j}-a_{i,j}\right)  =b_{i,j} \label{pf.thm.ddet.fn1.3}%
\end{equation}
(since $\partial$ is the differentiation operator on $\mathbb{K}\left[
t\right]  $).
\par
Now, forget that we fixed $\left(  i,j\right)  $. We thus have proven
(\ref{pf.thm.ddet.fn1.3}) for every $\left(  i,j\right)  \in\left\{
1,2,\ldots,n\right\}  ^{2}$.
\par
Now,%
\[
\underbrace{tB}_{=\left(  tb_{i,j}\right)  _{1\leq i\leq n,\ 1\leq j\leq n}%
}-\underbrace{A}_{=\left(  a_{i,j}\right)  _{1\leq i\leq n,\ 1\leq j\leq n}%
}=\left(  tb_{i,j}\right)  _{1\leq i\leq n,\ 1\leq j\leq n}-\left(
a_{i,j}\right)  _{1\leq i\leq n,\ 1\leq j\leq n}=\left(  tb_{i,j}%
-a_{i,j}\right)  _{1\leq i\leq n,\ 1\leq j\leq n}.
\]
Hence, the definition of the map $\partial^{n\times n}$ yields%
\[
\partial^{n\times n}\left(  tB-A\right)  =\left(  \underbrace{\partial\left(
tb_{i,j}-a_{i,j}\right)  }_{\substack{=b_{i,j}\\\text{(by
(\ref{pf.thm.ddet.fn1.3}))}}}\right)  _{1\leq i\leq n,\ 1\leq j\leq n}=\left(
b_{i,j}\right)  _{1\leq i\leq n,\ 1\leq j\leq n}=B,
\]
qed.}. Applying this to $B=I_{n}$, we obtain $\partial^{n\times n}\left(
tI_{n}-A\right)  =I_{n}$.

The definition of $\chi_{A}$ yields $\chi_{A}=\det\left(  tI_{n}-A\right)  $.
Applying the map $\partial$ to both sides of this equality, we obtain%
\begin{align*}
\partial\chi_{A}  &  =\partial\left(  \det\left(  tI_{n}-A\right)  \right)
=\operatorname*{Tr}\left(  \underbrace{\partial^{n\times n}\left(
tI_{n}-A\right)  }_{=I_{n}}\cdot\operatorname*{adj}\left(  tI_{n}-A\right)
\right) \\
&  \ \ \ \ \ \ \ \ \ \ \left(
\begin{array}
[c]{c}%
\text{by Theorem \ref{thm.deriv.ddet} (applied to }\mathbb{K}\left[  t\right]
\text{, }\partial\text{ and }tI_{n}-A\\
\text{instead of }\mathbb{L}\text{, }f\text{ and }A\text{)}%
\end{array}
\right) \\
&  =\operatorname*{Tr}\left(  \underbrace{I_{n}\cdot\operatorname*{adj}\left(
tI_{n}-A\right)  }_{=\operatorname*{adj}\left(  tI_{n}-A\right)  }\right)
=\operatorname*{Tr}\left(  \operatorname*{adj}\left(  tI_{n}-A\right)
\right)  .
\end{align*}
This proves Theorem \ref{thm.ddet}.
\end{proof}
\end{verlong}

We can use Theorem \ref{thm.ddet} to obtain the following result:

\begin{proposition}
\label{prop.TrD}Let $n\in\mathbb{N}$. Let $A\in\mathbb{K}^{n\times n}$. For
every $j\in\mathbb{Z}$, define an element $c_{j}\in\mathbb{K}$ by
$c_{j}=\left[  t^{n-j}\right]  \chi_{A}$.

Let $D_{0},D_{1},\ldots,D_{n-1}$ be $n$ matrices in $\mathbb{K}^{n\times n}$
satisfying (\ref{eq.lem.CH.D.adj}). Thus, an $n$-tuple $\left(  D_{0}%
,D_{1},\ldots,D_{n-1}\right)  $ of matrices in $\mathbb{K}^{n\times n}$ is
defined. Extend this $n$-tuple to a family $\left(  D_{k}\right)
_{k\in\mathbb{Z}}$ of matrices in $\mathbb{K}^{n\times n}$ by setting
(\ref{eq.lem.CH.D.0}). Then, every $k\in\mathbb{Z}$ satisfies%
\begin{equation}
\operatorname*{Tr}\left(  D_{k}\right)  =\left(  k+1\right)  c_{n-\left(
k+1\right)  }. \label{eq.prop.TrD.claim}%
\end{equation}

\end{proposition}

\begin{vershort}
\begin{proof}
[Proof of Proposition \ref{prop.TrD}.]Let $\partial:\mathbb{K}\left[
t\right]  \rightarrow\mathbb{K}\left[  t\right]  $ be the differentiation
operator (i.e., the map that sends every polynomial $f\in\mathbb{K}\left[
t\right]  $ to the derivative of $f$).

Lemma \ref{lem.CH.D} \textbf{(a)} yields $\chi_{A}=\sum_{k=0}^{n}c_{n-k}t^{k}%
$. Applying the map $\partial$ to both sides of this equality, we obtain%
\begin{align*}
\partial\chi_{A}  &  =\partial\left(  \sum_{k=0}^{n}c_{n-k}t^{k}\right)
=\sum_{k=1}^{n}c_{n-k}kt^{k-1}\ \ \ \ \ \ \ \ \ \ \left(
\begin{array}
[c]{c}%
\text{since }\partial\text{ is the differentiation}\\
\text{operator}%
\end{array}
\right) \\
&  =\sum_{k=1}^{n}kc_{n-k}t^{k-1}=\sum_{k=0}^{n-1}\left(  k+1\right)
c_{n-\left(  k+1\right)  }t^{k}%
\end{align*}
(here, we have substituted $k+1$ for $k$ in the sum). Meanwhile, Theorem
\ref{thm.ddet} yields%
\[
\partial\chi_{A}=\operatorname*{Tr}\left(  \underbrace{\operatorname*{adj}%
\left(  tI_{n}-A\right)  }_{\substack{=\sum_{k=0}^{n-1}t^{k}D_{k}\\\text{(by
(\ref{eq.lem.CH.D.adj}))}}}\right)  =\operatorname*{Tr}\left(  \sum
_{k=0}^{n-1}t^{k}D_{k}\right)  =\sum_{k=0}^{n-1}t^{k}\operatorname*{Tr}\left(
D_{k}\right)  =\sum_{k=0}^{n-1}\operatorname*{Tr}\left(  D_{k}\right)  t^{k}.
\]
Comparing these two equalities, we obtain%
\[
\sum_{k=0}^{n-1}\operatorname*{Tr}\left(  D_{k}\right)  t^{k}=\sum_{k=0}%
^{n-1}\left(  k+1\right)  c_{n-\left(  k+1\right)  }t^{k}.
\]
This is an identity between two polynomials in $\mathbb{K}\left[  t\right]  $.
Comparing coefficients on both sides of this identity, we conclude that%
\begin{equation}
\operatorname*{Tr}\left(  D_{k}\right)  =\left(  k+1\right)  c_{n-\left(
k+1\right)  }\ \ \ \ \ \ \ \ \ \ \text{for every }k\in\left\{  0,1,\ldots
,n-1\right\}  . \label{pf.prop.TrD.short.c=T}%
\end{equation}

Now, let $k\in\mathbb{Z}$. We must prove (\ref{eq.prop.TrD.claim}).

If $k\in\left\{  0,1,\ldots,n-1\right\}  $, then (\ref{eq.prop.TrD.claim})
follows immediately from (\ref{pf.prop.TrD.short.c=T}). Hence, for the rest of
this proof, we WLOG assume that we don't have $k\in\left\{  0,1,\ldots
,n-1\right\}  $.

We don't have $k\in\left\{  0,1,\ldots,n-1\right\}  $. Thus, $k\in
\mathbb{Z}\setminus\left\{  0,1,\ldots,n-1\right\}  $. Hence,
(\ref{eq.lem.CH.D.0}) yields $D_{k}=0_{n\times n}$, so that
$\operatorname*{Tr}\left(  D_{k}\right)  =\operatorname*{Tr}\left(  0_{n\times
n}\right)  =0$.

Recall again that $k\in\mathbb{Z}\setminus\left\{  0,1,\ldots,n-1\right\}  $.
In other words, we have either $k<0$ or $k\geq n$. Thus, we are in one of the
following two cases:

\textit{Case 1:} We have $k<0$.

\textit{Case 2:} We have $k\geq n$.

Let us first consider Case 1. In this case, we have $k<0$. If $k=-1$, then
(\ref{eq.prop.TrD.claim}) holds\footnote{\textit{Proof.} Assume that $k=-1$.
Then, $k+1=0$, so that $\underbrace{\left(  k+1\right)  }_{=0}c_{n-\left(
k+1\right)  }=0$. Comparing this with $\operatorname*{Tr}\left(  D_{k}\right)
=0$, we obtain $\operatorname*{Tr}\left(  D_{k}\right)  =\left(  k+1\right)
c_{n-\left(  k+1\right)  }$; hence, (\ref{eq.prop.TrD.claim}) holds, qed.}.
Hence, for the rest of this proof, we WLOG assume that $k\neq-1$. Combining
$k<0$ with $k\neq-1$, we obtain $k<-1$. Hence, $k+1<0$.

The definition of $c_{n-\left(  k+1\right)  }$ yields $c_{n-\left(
k+1\right)  }=\left[  \underbrace{t^{n-\left(  n-\left(  k+1\right)  \right)
}}_{=t^{k+1}}\right]  \chi_{A}=\left[  t^{k+1}\right]  \chi_{A}=0$ (since
$k+1<0$, but $\chi_{A}$ is a polynomial). Hence, $\left(  k+1\right)
\underbrace{c_{n-\left(  k+1\right)  }}_{=0}=0$. Comparing this with
$\operatorname*{Tr}\left(  D_{k}\right)  =0$, we obtain $\operatorname*{Tr}%
\left(  D_{k}\right)  =\left(  k+1\right)  c_{n-\left(  k+1\right)  }$. Hence,
(\ref{eq.prop.TrD.claim}) is proven in Case 1.

Let us now consider Case 2. In this case, we have $k\geq n$. Thus, $k+1\geq
n+1>n$.

But $\chi_{A}$ is a polynomial of degree $\leq n$. Hence, $\left[
t^{m}\right]  \chi_{A}=0$ for every integer $m>n$. Applying this to $m=k+1$,
we obtain $\left[  t^{k+1}\right]  \chi_{A}=0$ (since $k+1>n$).

The definition of $c_{n-\left(  k+1\right)  }$ yields $c_{n-\left(
k+1\right)  }=\left[  \underbrace{t^{n-\left(  n-\left(  k+1\right)  \right)
}}_{=t^{k+1}}\right]  \chi_{A}=\left[  t^{k+1}\right]  \chi_{A}=0$. Hence,
$\left(  k+1\right)  \underbrace{c_{n-\left(  k+1\right)  }}_{=0}=0$.
Comparing this with $\operatorname*{Tr}\left(  D_{k}\right)  =0$, we obtain
$\operatorname*{Tr}\left(  D_{k}\right)  =\left(  k+1\right)  c_{n-\left(
k+1\right)  }$. Hence, (\ref{eq.prop.TrD.claim}) is proven in Case 2.

We have now proven (\ref{eq.prop.TrD.claim}) in each of the two Cases 1 and 2.
Thus, (\ref{eq.prop.TrD.claim}) always holds. Thus, Proposition \ref{prop.TrD}
is proven.
\end{proof}
\end{vershort}

\begin{verlong}
\begin{proof}
[Proof of Proposition \ref{prop.TrD}.]For every commutative ring $\mathbb{L}$,
we consider the map $\operatorname*{Tr}:\mathbb{L}^{n\times n}\rightarrow
\mathbb{L}$ which sends every matrix $B\in\mathbb{L}^{n\times n}$ to
$\operatorname*{Tr}B\in\mathbb{L}$. This map $\operatorname*{Tr}%
:\mathbb{L}^{n\times n}\rightarrow\mathbb{L}$ is $\mathbb{L}$-linear. Applying
this to $\mathbb{L}=\mathbb{K}\left[  t\right]  $, we conclude that the map
$\operatorname*{Tr}:\left(  \mathbb{K}\left[  t\right]  \right)  ^{n\times
n}\rightarrow\mathbb{K}\left[  t\right]  $ is $\mathbb{K}\left[  t\right]  $-linear.

Let $\partial:\mathbb{K}\left[  t\right]  \rightarrow\mathbb{K}\left[
t\right]  $ be the differentiation operator (i.e., the map that sends every
polynomial $f\in\mathbb{K}\left[  t\right]  $ to the derivative of $f$).

Lemma \ref{lem.CH.D} \textbf{(a)} yields $\chi_{A}=\sum_{k=0}^{n}c_{n-k}t^{k}%
$. Applying the map $\partial$ to both sides of this equality, we obtain%
\begin{align*}
\partial\chi_{A}  &  =\partial\left(  \sum_{k=0}^{n}c_{n-k}t^{k}\right)
=\sum_{k=1}^{n}c_{n-k}kt^{k-1}\ \ \ \ \ \ \ \ \ \ \left(
\begin{array}
[c]{c}%
\text{since }\partial\text{ is the differentiation}\\
\text{operator}%
\end{array}
\right) \\
&  =\sum_{k=0}^{n-1}\underbrace{c_{n-\left(  k+1\right)  }\left(  k+1\right)
}_{=\left(  k+1\right)  c_{n-\left(  k+1\right)  }}\underbrace{t^{\left(
k+1\right)  -1}}_{\substack{=t^{k}\\\text{(since }\left(  k+1\right)
-1=k\text{)}}}\\
&  \ \ \ \ \ \ \ \ \ \ \left(  \text{here, we have substituted }k+1\text{ for
}k\text{ in the sum}\right) \\
&  =\sum_{k=0}^{n-1}\left(  k+1\right)  c_{n-\left(  k+1\right)  }t^{k}.
\end{align*}
Thus,%
\begin{align*}
&  \sum_{k=0}^{n-1}\left(  k+1\right)  c_{n-\left(  k+1\right)  }t^{k}\\
&  =\partial\chi_{A}=\operatorname*{Tr}\left(  \underbrace{\operatorname*{adj}%
\left(  tI_{n}-A\right)  }_{\substack{=\sum_{k=0}^{n-1}t^{k}D_{k}\\\text{(by
(\ref{eq.lem.CH.D.adj}))}}}\right)  \ \ \ \ \ \ \ \ \ \ \left(  \text{by
Theorem \ref{thm.ddet}}\right) \\
&  =\operatorname*{Tr}\left(  \sum_{k=0}^{n-1}t^{k}D_{k}\right)  =\sum
_{k=0}^{n-1}\underbrace{t^{k}\operatorname*{Tr}\left(  D_{k}\right)
}_{=\operatorname*{Tr}\left(  D_{k}\right)  t^{k}}\ \ \ \ \ \ \ \ \ \ \left(
\begin{array}
[c]{c}%
\text{since the map }\operatorname*{Tr}:\left(  \mathbb{K}\left[  t\right]
\right)  ^{n\times n}\rightarrow\mathbb{K}\left[  t\right] \\
\text{is }\mathbb{K}\left[  t\right]  \text{-linear}%
\end{array}
\right) \\
&  =\sum_{k=0}^{n-1}\operatorname*{Tr}\left(  D_{k}\right)  t^{k}.
\end{align*}
This is an identity between two polynomials in $\mathbb{K}\left[  t\right]  $.
Comparing coefficients on both sides of this identity, we conclude that%
\begin{equation}
\left(  k+1\right)  c_{n-\left(  k+1\right)  }=\operatorname*{Tr}\left(
D_{k}\right)  \ \ \ \ \ \ \ \ \ \ \text{for every }k\in\left\{  0,1,\ldots
,n-1\right\}  . \label{pf.prop.TrD.c=T}%
\end{equation}

On the other hand, we have%
\begin{equation}
c_{n-k}=0\ \ \ \ \ \ \ \ \ \ \text{for every }k\in\mathbb{Z}\setminus\left\{
0,1,\ldots,n\right\}  \label{pf.prop.TrD.c=0}%
\end{equation}
\footnote{\textit{Proof of (\ref{pf.prop.TrD.c=0}):} We know that $\chi_{A}%
\in\mathbb{K}\left[  t\right]  $ is a polynomial of degree $\leq n$ in $t$ (by
Corollary \ref{cor.chiA} \textbf{(a)}). Hence,
\begin{equation}
\left[  t^{k}\right]  \chi_{A}=0\ \ \ \ \ \ \ \ \ \ \text{for every integer
}k>n. \label{pf.prop.TrD.c=0.pf.1}%
\end{equation}
Also, clearly, $\left[  t^{k}\right]  f=0$ for every negative integer $k$ and
every polynomial $f\in\mathbb{K}\left[  t\right]  $. Applying this to
$f=\chi_{A}$, we conclude that%
\begin{equation}
\left[  t^{k}\right]  \chi_{A}=0\ \ \ \ \ \ \ \ \ \ \text{for every negative
integer }k. \label{pf.prop.TrD.c=0.pf.2}%
\end{equation}
\par
Now, let $k\in\mathbb{Z}\setminus\left\{  0,1,\ldots,n\right\}  $. We must
prove that $c_{n-k}=0$. We have%
\begin{align*}
c_{n-k}  &  =\underbrace{\left[  t^{n-\left(  n-k\right)  }\right]
}_{=\left[  t^{k}\right]  }\chi_{A}\ \ \ \ \ \ \ \ \ \ \left(  \text{by the
definition of }c_{n-k}\right) \\
&  =\left[  t^{k}\right]  \chi_{A}.
\end{align*}
\par
Now, we are in one of the following two cases:
\par
\textit{Case 1:} We have $k<0$.
\par
\textit{Case 2:} We have $k\geq0$.
\par
Let us first consider Case 1. In this case, we have $k<0$. Hence, $k$ is a
negative integer. Thus, (\ref{pf.prop.TrD.c=0.pf.2}) yields $\left[
t^{k}\right]  \chi_{A}=0$. Thus, $c_{n-k}=\left[  t^{k}\right]  \chi_{A}=0$.
Hence, $c_{n-k}=0$ is proven in Case 1.
\par
Let us now consider Case 2. In this case, we have $k\geq0$. Hence,
$k\in\mathbb{N}$. But $k\in\mathbb{Z}\setminus\left\{  0,1,\ldots,n\right\}
$; in other words, $k\in\mathbb{Z}$ and $k\notin\left\{  0,1,\ldots,n\right\}
$. Combining $k\in\mathbb{N}$ with $k\notin\left\{  0,1,\ldots,n\right\}  $,
we obtain $k\in\mathbb{N}\setminus\left\{  0,1,\ldots,n\right\}  =\left\{
n+1,n+2,n+3,\ldots\right\}  $. Hence, $k\geq n+1>n$, and thus $\left[
t^{k}\right]  \chi_{A}=0$ (by (\ref{pf.prop.TrD.c=0.pf.1})). Thus,
$c_{n-k}=\left[  t^{k}\right]  \chi_{A}=0$. Hence, $c_{n-k}=0$ is proven in
Case 2.
\par
Now, we have proven $c_{n-k}=0$ in each of the two Cases 1 and 2. Since these
two Cases cover all possibilities, this shows that $c_{n-k}=0$ always holds.
This proves (\ref{pf.prop.TrD.c=0}).}.

Now, let $k\in\mathbb{Z}$. We must prove (\ref{eq.prop.TrD.claim}).

If $k\in\left\{  0,1,\ldots,n-1\right\}  $, then (\ref{eq.prop.TrD.claim})
follows immediately from (\ref{pf.prop.TrD.c=T}). Hence, for the rest of this
proof, we can WLOG assume that we don't have $k\in\left\{  0,1,\ldots
,n-1\right\}  $. Assume this.

We are in one of the following three cases:

\textit{Case 1:} We have $k+1\in\left\{  0,1,\ldots,n\right\}  $.

\textit{Case 2:} We have $k+1\notin\left\{  0,1,\ldots,n\right\}  $.

Let us first consider Case 1. In this case, we have $k+1\in\left\{
0,1,\ldots,n\right\}  $. Hence, $\left(  k+1\right)  -1\in\left\{
-1,0,\ldots,n-1\right\}  $. Thus, $k=\left(  k+1\right)  -1\in\left\{
-1,0,\ldots,n-1\right\}  $.

But $k\notin\left\{  0,1,\ldots,n-1\right\}  $ (since we don't have
$k\in\left\{  0,1,\ldots,n-1\right\}  $). Combining $k\in\left\{
-1,0,\ldots,n-1\right\}  $ with $k\notin\left\{  0,1,\ldots,n-1\right\}  $, we
obtain%
\[
k\in\left\{  -1,0,\ldots,n-1\right\}  \setminus\left\{  0,1,\ldots
,n-1\right\}  \subseteq\left\{  -1\right\}  .
\]
In other words, $k=-1$.

Now, $k=-1<0$, so that $k\notin\left\{  0,1,\ldots,n-1\right\}  $. Combining
$k\in\mathbb{Z}$ with $k\notin\left\{  0,1,\ldots,n-1\right\}  $, we obtain
$k\in\mathbb{Z}\setminus\left\{  0,1,\ldots,n-1\right\}  $. Thus,
(\ref{eq.lem.CH.D.0}) yields $D_{k}=0_{n\times n}$, so that
$\operatorname*{Tr}\left(  \underbrace{D_{k}}_{=0_{n\times n}}\right)
=\operatorname*{Tr}\left(  0_{n\times n}\right)  =0$. Comparing this with
$\underbrace{\left(  k+1\right)  }_{\substack{=0\\\text{(since }k=-1\text{)}%
}}c_{n-\left(  k+1\right)  }=0$, we obtain $\operatorname*{Tr}\left(
D_{k}\right)  =\left(  k+1\right)  c_{n-\left(  k+1\right)  }$. Thus,
(\ref{eq.prop.TrD.claim}) is proven in Case 1.

Let us now consider Case 2. In this case, we have $k+1\notin\left\{
0,1,\ldots,n\right\}  $. Combining $k+1\in\mathbb{Z}$ with $k+1\notin\left\{
0,1,\ldots,n\right\}  $, we obtain $k+1\in\mathbb{Z}\setminus\left\{
0,1,\ldots,n\right\}  $. Thus, (\ref{pf.prop.TrD.c=0}) (applied to $k+1$
instead of $k$) yields $c_{n-\left(  k+1\right)  }=0$. On the other hand,
$k\notin\left\{  0,1,\ldots,n-1\right\}  $\ \ \ \ \footnote{\textit{Proof.}
Assume the contrary. Thus, $k\in\left\{  0,1,\ldots,n-1\right\}  $. Hence,
$k+1\in\left\{  1,2,\ldots,n\right\}  \subseteq\left\{  0,1,\ldots,n\right\}
$. This contradicts $k+1\notin\left\{  0,1,\ldots,n\right\}  $. This
contradiction proves that our assumption was wrong, qed.}. Combining
$k\in\mathbb{Z}$ with $k\notin\left\{  0,1,\ldots,n-1\right\}  $, we obtain
$k\in\mathbb{Z}\setminus\left\{  0,1,\ldots,n-1\right\}  $. Hence,
(\ref{eq.lem.CH.D.0}) yields $D_{k}=0_{n\times n}$. Thus, $\operatorname*{Tr}%
\left(  \underbrace{D_{k}}_{=0_{n\times n}}\right)  =\operatorname*{Tr}\left(
0_{n\times n}\right)  =0$. Comparing this with $\left(  k+1\right)
\underbrace{c_{n-\left(  k+1\right)  }}_{=0}=0$, we obtain $\operatorname*{Tr}%
\left(  D_{k}\right)  =\left(  k+1\right)  c_{n-\left(  k+1\right)  }$. Thus,
(\ref{eq.prop.TrD.claim}) is proven in Case 2.

We have now proven (\ref{eq.prop.TrD.claim}) in each of the two Cases 1 and 2.
Since these two Cases cover all possibilities, this shows that
(\ref{eq.prop.TrD.claim}) always holds. Thus, Proposition \ref{prop.TrD} is proven.
\end{proof}
\end{verlong}

\subsection{Proof of the trace Cayley-Hamilton theorem}

Now, we can finally prove the trace Cayley-Hamilton theorem itself:

\begin{vershort}
\begin{proof}
[Proof of Theorem \ref{thm.TCH}.]Proposition \ref{prop.adj.poly} shows that
there exist $n$ matrices $D_{0},D_{1},\ldots,D_{n-1}$ in $\mathbb{K}^{n\times
n}$ such that%
\[
\operatorname*{adj}\left(  tI_{n}-A\right)  =\sum_{k=0}^{n-1}t^{k}%
D_{k}\ \ \ \ \ \ \ \ \ \ \text{in }\left(  \mathbb{K}\left[  t\right]
\right)  ^{n\times n}.
\]
Consider these $D_{0},D_{1},\ldots,D_{n-1}$. Thus, an $n$-tuple $\left(
D_{0},D_{1},\ldots,D_{n-1}\right)  $ of matrices in $\mathbb{K}^{n\times n}$
is defined. Extend this $n$-tuple to a family $\left(  D_{k}\right)
_{k\in\mathbb{Z}}$ of matrices in $\mathbb{K}^{n\times n}$ by setting%
\[
\left(  D_{k}=0_{n\times n}\ \ \ \ \ \ \ \ \ \ \text{for every }k\in
\mathbb{Z}\setminus\left\{  0,1,\ldots,n-1\right\}  \right)  .
\]

Now, let $k\in\mathbb{N}$. Then, Proposition \ref{prop.TrD} (applied to
$n-1-k$ instead of $k$) yields
\[
\operatorname*{Tr}\left(  D_{n-1-k}\right)  =\underbrace{\left(  \left(
n-1-k\right)  +1\right)  }_{=n-k}\underbrace{c_{n-\left(  \left(
n-1-k\right)  +1\right)  }}_{\substack{=c_{k}\\\text{(since }n-\left(  \left(
n-1-k\right)  +1\right)  =k\text{)}}}=\left(  n-k\right)  c_{k}.
\]
Thus,%
\begin{align*}
\left(  n-k\right)  c_{k}  &  =\operatorname*{Tr}\left(  \underbrace{D_{n-1-k}%
}_{\substack{=\sum_{i=0}^{k}c_{k-i}A^{i}\\\text{(by Lemma \ref{lem.CH.D}
\textbf{(c)})}}}\right)  =\operatorname*{Tr}\left(  \sum_{i=0}^{k}c_{k-i}%
A^{i}\right) \\
&  =\sum_{i=0}^{k}c_{k-i}\operatorname*{Tr}\left(  A^{i}\right)
=\underbrace{c_{k-0}}_{=c_{k}}\operatorname*{Tr}\left(  \underbrace{A^{0}%
}_{=I_{n}}\right)  +\sum_{i=1}^{k}c_{k-i}\operatorname*{Tr}\left(
A^{i}\right) \\
&  \ \ \ \ \ \ \ \ \ \ \left(  \text{here, we have split off the addend for
}i=0\text{ from the sum}\right) \\
&  =c_{k}\underbrace{\operatorname*{Tr}\left(  I_{n}\right)  }_{=n}+\sum
_{i=1}^{k}\underbrace{c_{k-i}\operatorname*{Tr}\left(  A^{i}\right)
}_{=\operatorname*{Tr}\left(  A^{i}\right)  c_{k-i}}=c_{k}n+\sum_{i=1}%
^{k}\operatorname*{Tr}\left(  A^{i}\right)  c_{k-i}.
\end{align*}
Solving this equation for $\sum_{i=1}^{k}\operatorname*{Tr}\left(
A^{i}\right)  c_{k-i}$, we obtain%
\[
\sum_{i=1}^{k}\operatorname*{Tr}\left(  A^{i}\right)  c_{k-i}%
=\underbrace{\left(  n-k\right)  c_{k}}_{=nc_{k}-kc_{k}}-\underbrace{c_{k}%
n}_{=nc_{k}}=nc_{k}-kc_{k}-nc_{k}=-kc_{k}.
\]
Adding $kc_{k}$ to both sides of this equation, we obtain $kc_{k}+\sum
_{i=1}^{k}\operatorname*{Tr}\left(  A^{i}\right)  c_{k-i}=0$. This proves
Theorem \ref{thm.TCH}.
\end{proof}
\end{vershort}

\begin{verlong}
\begin{proof}
[Proof of Theorem \ref{thm.TCH}.]Consider the map $\operatorname*{Tr}%
:\mathbb{K}^{n\times n}\rightarrow\mathbb{K}$ which sends every matrix
$B\in\mathbb{K}^{n\times n}$ to $\operatorname*{Tr}B\in\mathbb{K}$. This map
$\operatorname*{Tr}:\mathbb{K}^{n\times n}\rightarrow\mathbb{K}$ is
$\mathbb{K}$-linear.

The identity matrix $I_{n}$ has diagonal entries $1,1,\ldots,1$. Hence, its
trace is%
\[
\operatorname*{Tr}\left(  I_{n}\right)  =\underbrace{1+1+\cdots+1}_{n\text{
times}}=n\cdot1=n.
\]

Proposition \ref{prop.adj.poly} shows that there exist $n$ matrices
$D_{0},D_{1},\ldots,D_{n-1}$ in $\mathbb{K}^{n\times n}$ such that%
\[
\operatorname*{adj}\left(  tI_{n}-A\right)  =\sum_{k=0}^{n-1}t^{k}%
D_{k}\ \ \ \ \ \ \ \ \ \ \text{in }\left(  \mathbb{K}\left[  t\right]
\right)  ^{n\times n}.
\]
Consider these $D_{0},D_{1},\ldots,D_{n-1}$. Thus, an $n$-tuple $\left(
D_{0},D_{1},\ldots,D_{n-1}\right)  $ of matrices in $\mathbb{K}^{n\times n}$
is defined. Extend this $n$-tuple to a family $\left(  D_{k}\right)
_{k\in\mathbb{Z}}$ of matrices in $\mathbb{K}^{n\times n}$ by setting%
\[
\left(  D_{k}=0_{n\times n}\ \ \ \ \ \ \ \ \ \ \text{for every }k\in
\mathbb{Z}\setminus\left\{  0,1,\ldots,n-1\right\}  \right)  .
\]

Now, let $k\in\mathbb{N}$. Then, Proposition \ref{prop.TrD} (applied to
$n-1-k$ instead of $k$) yields
\[
\operatorname*{Tr}\left(  D_{n-1-k}\right)  =\underbrace{\left(  \left(
n-1-k\right)  +1\right)  }_{=n-k}\underbrace{c_{n-\left(  \left(
n-1-k\right)  +1\right)  }}_{\substack{=c_{k}\\\text{(since }n-\left(  \left(
n-1-k\right)  +1\right)  =k\text{)}}}=\left(  n-k\right)  c_{k}.
\]
Thus,%
\begin{align*}
\left(  n-k\right)  c_{k}  &  =\operatorname*{Tr}\left(  \underbrace{D_{n-1-k}%
}_{\substack{=\sum_{i=0}^{k}c_{k-i}A^{i}\\\text{(by Lemma \ref{lem.CH.D}
\textbf{(c)})}}}\right)  =\operatorname*{Tr}\left(  \sum_{i=0}^{k}c_{k-i}%
A^{i}\right) \\
&  =\sum_{i=0}^{k}c_{k-i}\operatorname*{Tr}\left(  A^{i}\right)
\ \ \ \ \ \ \ \ \ \ \left(  \text{since the map }\operatorname*{Tr}%
:\mathbb{K}^{n\times n}\rightarrow\mathbb{K}\text{ is }\mathbb{K}%
\text{-linear}\right) \\
&  =\underbrace{c_{k-0}}_{=c_{k}}\operatorname*{Tr}\left(  \underbrace{A^{0}%
}_{=I_{n}}\right)  +\sum_{i=1}^{k}c_{k-i}\operatorname*{Tr}\left(
A^{i}\right) \\
&  \ \ \ \ \ \ \ \ \ \ \left(
\begin{array}
[c]{c}%
\text{here, we have split off the addend for }i=0\text{ from the sum}\\
\text{(since }0\in\left\{  0,1,\ldots,k\right\}  \text{ (since }k\in
\mathbb{N}\text{))}%
\end{array}
\right) \\
&  =c_{k}\underbrace{\operatorname*{Tr}\left(  I_{n}\right)  }_{=n}+\sum
_{i=1}^{k}\underbrace{c_{k-i}\operatorname*{Tr}\left(  A^{i}\right)
}_{=\operatorname*{Tr}\left(  A^{i}\right)  c_{k-i}}=c_{k}n+\sum_{i=1}%
^{k}\operatorname*{Tr}\left(  A^{i}\right)  c_{k-i}.
\end{align*}
Solving this equation for $\sum_{i=1}^{k}\operatorname*{Tr}\left(
A^{i}\right)  c_{k-i}$, we obtain%
\[
\sum_{i=1}^{k}\operatorname*{Tr}\left(  A^{i}\right)  c_{k-i}%
=\underbrace{\left(  n-k\right)  c_{k}}_{=nc_{k}-kc_{k}}-\underbrace{c_{k}%
n}_{=nc_{k}}=nc_{k}-kc_{k}-nc_{k}=-kc_{k}.
\]
Adding $kc_{k}$ to both sides of this equation, we obtain $kc_{k}+\sum
_{i=1}^{k}\operatorname*{Tr}\left(  A^{i}\right)  c_{k-i}=0$. This proves
Theorem \ref{thm.TCH}.
\end{proof}
\end{verlong}

\subsection{A corollary}

The following fact (which can also be easily proven by other means) follows
readily from Theorem \ref{thm.TCH}:

\begin{corollary}
\label{cor.TCH.c1}Let $n\in\mathbb{N}$. Let $A\in\mathbb{K}^{n\times n}$.
Then, $\left[  t^{n-1}\right]  \chi_{A}=-\operatorname*{Tr}A$.
\end{corollary}

\begin{proof}
[Proof of Corollary \ref{cor.TCH.c1}.]For every $j\in\mathbb{Z}$, define an
element $c_{j}\in\mathbb{K}$ by $c_{j}=\left[  t^{n-j}\right]  \chi_{A}$. The
definition of $c_{1}$ yields $c_{1}=\left[  t^{n-1}\right]  \chi_{A}$. The
definition of $c_{0}$ yields $c_{0}=\underbrace{\left[  t^{n-0}\right]
}_{=\left[  t^{n}\right]  }\chi_{A}=\left[  t^{n}\right]  \chi_{A}=1$ (by
Corollary \ref{cor.chiA} \textbf{(c)}).

Theorem \ref{thm.TCH} (applied to $k=1$) yields $1c_{1}+\sum_{i=1}%
^{1}\operatorname*{Tr}\left(  A^{i}\right)  c_{1-i}=0$. Thus,%
\[
1c_{1}=-\underbrace{\sum_{i=1}^{1}\operatorname*{Tr}\left(  A^{i}\right)
c_{1-i}}_{=\operatorname*{Tr}\left(  A^{1}\right)  c_{1-1}}%
=-\operatorname*{Tr}\left(  \underbrace{A^{1}}_{=A}\right)
\underbrace{c_{1-1}}_{=c_{0}=1}=-\operatorname*{Tr}A.
\]
Comparing this with $1c_{1}=c_{1}=\left[  t^{n-1}\right]  \chi_{A}$, we obtain
$\left[  t^{n-1}\right]  \chi_{A}=-\operatorname*{Tr}A$. This proves Corollary
\ref{cor.TCH.c1}.
\end{proof}

\section{\label{sec.nilp}Application: Nilpotency and traces}

\subsection{A nilpotency criterion}

As an application of Theorem \ref{thm.TCH}, let us now prove the following
fact (generalizing \cite{m.se1798703} and part of \cite[Corollary 1]{Robins61}):

\begin{corollary}
\label{cor.nilpcrit}Let $n\in\mathbb{N}$. Let $A\in\mathbb{K}^{n\times n}$.
Assume that
\begin{equation}
\operatorname*{Tr}\left(  A^{i}\right)  =0\ \ \ \ \ \ \ \ \ \ \text{for every
}i\in\left\{  1,2,\ldots,n\right\}  . \label{eq.cor.nilpcrit.cond}%
\end{equation}

\textbf{(a)} Then, $n!A^{n}=0_{n\times n}$. \medskip

\textbf{(b)} If $\mathbb{K}$ is a commutative $\mathbb{Q}$-algebra, then
$A^{n}=0_{n\times n}$. \medskip

\textbf{(c)} We have $n!\chi_{A}=n!t^{n}$. \medskip

\textbf{(d)} If $\mathbb{K}$ is a commutative $\mathbb{Q}$-algebra, then
$\chi_{A}=t^{n}$.
\end{corollary}

\begin{proof}
[Proof of Corollary \ref{cor.nilpcrit}.]For every $j\in\mathbb{Z}$, define an
element $c_{j}\in\mathbb{K}$ by $c_{j}=\left[  t^{n-j}\right]  \chi_{A}$. The
definition of $c_{0}$ yields $c_{0}=\underbrace{\left[  t^{n-0}\right]
}_{=\left[  t^{n}\right]  }\chi_{A}=\left[  t^{n}\right]  \chi_{A}=1$ (by
Corollary \ref{cor.chiA} \textbf{(c)}).

We now claim that%
\begin{equation}
kc_{k}=0\ \ \ \ \ \ \ \ \ \ \text{for every }k\in\left\{  1,2,\ldots
,n\right\}  . \label{pf.cor.nilpcrit.main}%
\end{equation}

\begin{vershort}
[\textit{Proof of (\ref{pf.cor.nilpcrit.main}):} Let $k\in\left\{
1,2,\ldots,n\right\}  $. Then, every $i\in\left\{  1,2,\ldots,k\right\}  $
satisfies $i\in\left\{  1,2,\ldots,n\right\}  $ and therefore also
\begin{equation}
\operatorname*{Tr}\left(  A^{i}\right)  =0
\label{pf.cor.nilpcrit.main.pf.short.2}%
\end{equation}
(by (\ref{eq.cor.nilpcrit.cond})). Now, Theorem \ref{thm.TCH} yields%
\[
kc_{k}+\sum_{i=1}^{k}\operatorname*{Tr}\left(  A^{i}\right)  c_{k-i}=0.
\]
Solving this equation for $kc_{k}$, we obtain%
\[
kc_{k}=-\sum_{i=1}^{k}\underbrace{\operatorname*{Tr}\left(  A^{i}\right)
}_{\substack{=0\\\text{(by (\ref{pf.cor.nilpcrit.main.pf.short.2}))}}%
}c_{k-i}=-\underbrace{\sum_{i=1}^{k}0c_{k-i}}_{=0}=-0=0.
\]
This proves (\ref{pf.cor.nilpcrit.main}).]
\end{vershort}

\begin{verlong}
[\textit{Proof of (\ref{pf.cor.nilpcrit.main}):} Let $k\in\left\{
1,2,\ldots,n\right\}  $. Then, $k\leq n$, so that $\left\{  1,2,\ldots
,k\right\}  \subseteq\left\{  1,2,\ldots,n\right\}  $. Thus, every
$i\in\left\{  1,2,\ldots,k\right\}  $ satisfies $i\in\left\{  1,2,\ldots
,k\right\}  \subseteq\left\{  1,2,\ldots,n\right\}  $ and therefore also
\begin{equation}
\operatorname*{Tr}\left(  A^{i}\right)  =0 \label{pf.cor.nilpcrit.main.pf.2}%
\end{equation}
(by (\ref{eq.cor.nilpcrit.cond})). Now, Theorem \ref{thm.TCH} yields%
\[
kc_{k}+\sum_{i=1}^{k}\operatorname*{Tr}\left(  A^{i}\right)  c_{k-i}=0.
\]
Solving this equation for $kc_{k}$, we obtain%
\[
kc_{k}=-\sum_{i=1}^{k}\underbrace{\operatorname*{Tr}\left(  A^{i}\right)
}_{\substack{=0\\\text{(by (\ref{pf.cor.nilpcrit.main.pf.2}))}}}c_{k-i}%
=-\underbrace{\sum_{i=1}^{k}0c_{k-i}}_{=0}=-0=0.
\]
This proves (\ref{pf.cor.nilpcrit.main}).]
\end{verlong}

Now, we claim that%
\begin{equation}
n!c_{k}=0\ \ \ \ \ \ \ \ \ \ \text{for every }k\in\left\{  1,2,\ldots
,n\right\}  . \label{pf.cor.nilpcrit.main-n}%
\end{equation}

\begin{vershort}
[\textit{Proof of (\ref{pf.cor.nilpcrit.main-n}):} Let $k\in\left\{
1,2,\ldots,n\right\}  $. The product $1\cdot2\cdot\cdots\cdot n$ contains $k$
as a factor, and thus is a multiple of $k$; in other words, $n!$ is a multiple
of $k$ (since $n!=1\cdot2\cdot\cdots\cdot n$). Hence, $n!c_{k}$ is a multiple
of $kc_{k}$. Thus, (\ref{pf.cor.nilpcrit.main-n}) follows from
(\ref{pf.cor.nilpcrit.main}).]
\end{vershort}

\begin{verlong}
[\textit{Proof of (\ref{pf.cor.nilpcrit.main-n}):} Let $k\in\left\{
1,2,\ldots,n\right\}  $. Recall that
\[
n!=1\cdot2\cdot\cdots\cdot n=\prod_{i\in\left\{  1,2,\ldots,n\right\}
}i=k\cdot\prod_{\substack{i\in\left\{  1,2,\ldots,n\right\}  ;\\i\neq k}}i
\]
(here, we have split off the factor for $i=k$ from the product (since
$k\in\left\{  1,2,\ldots,n\right\}  $)). Thus,%
\[
\underbrace{n!}_{=k\cdot\prod_{\substack{i\in\left\{  1,2,\ldots,n\right\}
;\\i\neq k}}i}c_{k}=k\cdot\left(  \prod_{\substack{i\in\left\{  1,2,\ldots
,n\right\}  ;\\i\neq k}}i\right)  c_{k}=\left(  \prod_{\substack{i\in\left\{
1,2,\ldots,n\right\}  ;\\i\neq k}}i\right)  \underbrace{kc_{k}}%
_{\substack{=0\\\text{(by (\ref{pf.cor.nilpcrit.main}))}}}=0.
\]
This proves (\ref{pf.cor.nilpcrit.main-n}).]
\end{verlong}

Finally, we observe that%
\begin{equation}
n!c_{n-k}=0\ \ \ \ \ \ \ \ \ \ \text{for every }k\in\left\{  0,1,\ldots
,n-1\right\}  . \label{pf.cor.nilpcrit.main-n2}%
\end{equation}

[\textit{Proof of (\ref{pf.cor.nilpcrit.main-n2}):} Let $k\in\left\{
0,1,\ldots,n-1\right\}  $. Then, $n-k\in\left\{  1,2,\ldots,n\right\}  $.
Hence, (\ref{pf.cor.nilpcrit.main-n}) (applied to $n-k$ instead of $k$) yields
$n!c_{n-k}=0$. This proves (\ref{pf.cor.nilpcrit.main-n2}).]

\begin{vershort}
Now, Corollary \ref{cor.CH.chiA} yields $\chi_{A}=\sum_{k=0}^{n}c_{n-k}t^{k}$.
Substituting $A$ for $t$ in this equality, we obtain $\chi_{A}\left(
A\right)  =\sum_{k=0}^{n}c_{n-k}A^{k}$. Multiplying both sides of the latter
equality by $n!$, we obtain%
\begin{align*}
n!\chi_{A}\left(  A\right)   &  =n!\sum_{k=0}^{n}c_{n-k}A^{k}=\sum_{k=0}%
^{n}n!c_{n-k}A^{k}=\sum_{k=0}^{n-1}\underbrace{n!c_{n-k}}%
_{\substack{=0\\\text{(by (\ref{pf.cor.nilpcrit.main-n2}))}}}A^{k}%
+n!\underbrace{c_{n-n}}_{=c_{0}=1}A^{n}\\
&  \ \ \ \ \ \ \ \ \ \ \left(  \text{here, we have split off the addend for
}k=n\text{ from the sum}\right) \\
&  =\underbrace{\sum_{k=0}^{n-1}0A^{k}}_{=0}+\,n!A^{n}=n!A^{n}.
\end{align*}
Hence,
\[
n!A^{n}=n!\underbrace{\chi_{A}\left(  A\right)  }_{\substack{=0_{n\times
n}\\\text{(by Theorem \ref{thm.CH})}}}=0_{n\times n}.
\]
This proves Corollary \ref{cor.nilpcrit} \textbf{(a)}. \medskip
\end{vershort}

\begin{verlong}
Now, Corollary \ref{cor.CH.chiA} yields $\chi_{A}=\sum_{k=0}^{n}c_{n-k}t^{k}$.
Substituting $A$ for $t$ in this equality, we obtain $\chi_{A}\left(
A\right)  =\sum_{k=0}^{n}c_{n-k}A^{k}$. Multiplying both sides of the latter
equality by $n!$, we obtain%
\begin{align*}
n!\chi_{A}\left(  A\right)   &  =n!\sum_{k=0}^{n}c_{n-k}A^{k}=\sum_{k=0}%
^{n}n!c_{n-k}A^{k}=\sum_{k=0}^{n-1}\underbrace{n!c_{n-k}}%
_{\substack{=0\\\text{(by (\ref{pf.cor.nilpcrit.main-n2}))}}}A^{k}%
+n!\underbrace{c_{n-n}}_{=c_{0}=1}A^{n}\\
&  \ \ \ \ \ \ \ \ \ \ \left(
\begin{array}
[c]{c}%
\text{here, we have split off the addend for }k=n\text{ from the sum}\\
\text{(since }n\in\left\{  0,1,\ldots,n\right\}  \text{ (since }n\in
\mathbb{N}\text{))}%
\end{array}
\right) \\
&  =\underbrace{\sum_{k=0}^{n-1}0A^{k}}_{=0}+\,n!A^{n}=n!A^{n}.
\end{align*}
Hence,
\[
n!A^{n}=n!\underbrace{\chi_{A}\left(  A\right)  }_{\substack{=0_{n\times
n}\\\text{(by Theorem \ref{thm.CH})}}}=0_{n\times n}.
\]
This proves Corollary \ref{cor.nilpcrit} \textbf{(a)}. \medskip
\end{verlong}

\textbf{(b)} Assume that $\mathbb{K}$ is a commutative $\mathbb{Q}$-algebra.
Corollary \ref{cor.nilpcrit} \textbf{(a)} yields $n!A^{n}=0_{n\times n}$. Now,
$\dfrac{1}{n!}\in\mathbb{Q}$, so that we can multiply an $n\times n$-matrix in
$\mathbb{K}^{n\times n}$ by $\dfrac{1}{n!}$ (since $\mathbb{K}$ is a
$\mathbb{Q}$-algebra). We have $\underbrace{\dfrac{1}{n!}n!}_{=1}A^{n}=A^{n}$.
Hence, $A^{n}=\dfrac{1}{n!}\underbrace{n!A^{n}}_{=0_{n\times n}}=\dfrac{1}%
{n!}0_{n\times n}=0_{n\times n}$. This proves Corollary \ref{cor.nilpcrit}
\textbf{(b)}. \medskip

\begin{vershort}
\textbf{(c)} Multiplying the equality $\chi_{A}=\sum_{k=0}^{n}c_{n-k}t^{k}$ by
$n!$, we obtain%
\begin{align*}
n!\chi_{A}  &  =n!\sum_{k=0}^{n}c_{n-k}t^{k}=\sum_{k=0}^{n}n!c_{n-k}t^{k}%
=\sum_{k=0}^{n-1}\underbrace{n!c_{n-k}}_{\substack{=0\\\text{(by
(\ref{pf.cor.nilpcrit.main-n2}))}}}t^{k}+n!\underbrace{c_{n-n}}_{=c_{0}%
=1}t^{n}\\
&  \ \ \ \ \ \ \ \ \ \ \left(  \text{here, we have split off the addend for
}k=n\text{ from the sum}\right) \\
&  =\underbrace{\sum_{k=0}^{n-1}0t^{k}}_{=0}+\,n!t^{n}=n!t^{n}.
\end{align*}
This proves Corollary \ref{cor.nilpcrit} \textbf{(c)}. \medskip
\end{vershort}

\begin{verlong}
\textbf{(c)} Multiplying both sides of the equality $\chi_{A}=\sum_{k=0}%
^{n}c_{n-k}t^{k}$ by $n!$, we obtain%
\begin{align*}
n!\chi_{A}  &  =n!\sum_{k=0}^{n}c_{n-k}t^{k}=\sum_{k=0}^{n}n!c_{n-k}t^{k}%
=\sum_{k=0}^{n-1}\underbrace{n!c_{n-k}}_{\substack{=0\\\text{(by
(\ref{pf.cor.nilpcrit.main-n2}))}}}t^{k}+n!\underbrace{c_{n-n}}_{=c_{0}%
=1}t^{n}\\
&  \ \ \ \ \ \ \ \ \ \ \left(
\begin{array}
[c]{c}%
\text{here, we have split off the addend for }k=n\text{ from the sum}\\
\text{(since }n\in\left\{  0,1,\ldots,n\right\}  \text{ (since }n\in
\mathbb{N}\text{))}%
\end{array}
\right) \\
&  =\underbrace{\sum_{k=0}^{n-1}0t^{k}}_{=0}+\,n!t^{n}=n!t^{n}.
\end{align*}
This proves Corollary \ref{cor.nilpcrit} \textbf{(c)}. \medskip
\end{verlong}

\textbf{(d)} Assume that $\mathbb{K}$ is a commutative $\mathbb{Q}$-algebra.
Corollary \ref{cor.nilpcrit} \textbf{(c)} yields $n!\chi_{A}=n!t^{n}$.

Now, $\dfrac{1}{n!}\in\mathbb{Q}$, so that we can multiply any polynomial in
$\mathbb{K}\left[  t\right]  $ by $\dfrac{1}{n!}$ (since $\mathbb{K}$ is a
$\mathbb{Q}$-algebra). We have $\underbrace{\dfrac{1}{n!}n!}_{=1}\chi_{A}%
=\chi_{A}$. Hence, $\chi_{A}=\dfrac{1}{n!}\underbrace{n!\chi_{A}}_{=n!t^{n}%
}=\dfrac{1}{n!}n!t^{n}=t^{n}$. This proves Corollary \ref{cor.nilpcrit}
\textbf{(d)}.
\end{proof}

\subsection{A converse direction}

The following result -- in a sense, a converse of Corollary \ref{cor.nilpcrit}
\textbf{(d)} -- also follows from Theorem \ref{thm.TCH}:

\begin{corollary}
\label{cor.nilpcrit-conv}Let $n\in\mathbb{N}$. Let $A\in\mathbb{K}^{n\times
n}$. Assume that $\chi_{A}=t^{n}$. Then, $\operatorname*{Tr}\left(
A^{i}\right)  =0$ for every positive integer $i$.
\end{corollary}

\begin{proof}
[Proof of Corollary \ref{cor.nilpcrit-conv}.]For every $j\in\mathbb{Z}$,
define an element $c_{j}\in\mathbb{K}$ by $c_{j}=\left[  t^{n-j}\right]
\chi_{A}$. Then, each positive integer $j$ satisfies%
\begin{equation}
c_{j}=0. \label{pf.cor.nilpcrit-conv.1}%
\end{equation}

\begin{vershort}
[\textit{Proof of (\ref{pf.cor.nilpcrit-conv.1}):} Let $j$ be a positive
integer. Thus, $j\neq0$, so that $n-j\neq n$ and thus $\left[  t^{n-j}\right]
\left(  t^{n}\right)  =0$. In view of $\chi_{A}=t^{n}$, this rewrites as
$\left[  t^{n-j}\right]  \chi_{A}=0$. But the definition of $c_{j}$ yields
$c_{j}=\left[  t^{n-j}\right]  \chi_{A}=0$. This proves
(\ref{pf.cor.nilpcrit-conv.1}).]
\end{vershort}

\begin{verlong}
[\textit{Proof of (\ref{pf.cor.nilpcrit-conv.1}):} Let $j$ be a positive
integer. Thus, $j\neq0$. Hence, $n-j\neq n$.

It is thus easy to see that $\left[  t^{n-j}\right]  \left(  t^{n}\right)
=0$\ \ \ \ \footnote{\textit{Proof.} We are in one of the following two cases:
\par
\textit{Case 1:} We have $n-j\geq0$.
\par
\textit{Case 2:} We have $n-j<0$.
\par
Let us first consider Case 1. In this case, we have $n-j\geq0$. Thus, $\left[
t^{n-j}\right]  \left(  t^{n}\right)  $ is the coefficient of $t^{n-j}$ in
$t^{n}$ (by the definition of $\left[  t^{n-j}\right]  \left(  t^{n}\right)
$). Thus,%
\[
\left[  t^{n-j}\right]  \left(  t^{n}\right)  =\left(  \text{the coefficient
of }t^{n-j}\text{ in }t^{n}\right)  =0
\]
(since $t^{n-j}$ and $t^{n}$ are two distinct monomials). Thus, $\left[
t^{n-j}\right]  \left(  t^{n}\right)  =0$ is proven in Case 1.
\par
Let us now consider Case 2. In this case, we have $n-j<0$. Thus, $n-j$ is a
negative integer. Now, every $f\in\mathbb{K}\left[  t\right]  $ and every
negative integer $m$ satisfy $\left[  t^{m}\right]  f=0$ (since polynomials
have no monomials $t^{m}$ with negative $m$). Applying this to $f=t^{n}$ and
$m=n-j$, we obtain $\left[  t^{n-j}\right]  \left(  t^{n}\right)  =0$. Hence,
$\left[  t^{n-j}\right]  \left(  t^{n}\right)  =0$ is proven in Case 2.
\par
Now, we have proven that $\left[  t^{n-j}\right]  \left(  t^{n}\right)  =0$ in
each of the two Cases 1 and 2. Since these two Cases cover all possibilities,
this shows that $\left[  t^{n-j}\right]  \left(  t^{n}\right)  =0$ always
holds. Qed.}. But the definition of $c_{j}$ yields $c_{j}=\left[
t^{n-j}\right]  \underbrace{\chi_{A}}_{=t^{n}}=\left[  t^{n-j}\right]  \left(
t^{n}\right)  =0$. This proves (\ref{pf.cor.nilpcrit-conv.1}).]
\end{verlong}

The definition of $c_{0}$ yields $c_{0}=\underbrace{\left[  t^{n-0}\right]
}_{=\left[  t^{n}\right]  }\chi_{A}=\left[  t^{n}\right]  \chi_{A}=1$ (by
Corollary \ref{cor.chiA} \textbf{(c)}).

Now, we claim that%
\begin{equation}
\operatorname*{Tr}\left(  A^{p}\right)  =0\ \ \ \ \ \ \ \ \ \ \text{for every
positive integer }p. \label{pf.cor.nilpcrit-conv.2}%
\end{equation}

[\textit{Proof of (\ref{pf.cor.nilpcrit-conv.2}):} We shall prove
(\ref{pf.cor.nilpcrit-conv.2}) by strong induction on $p$:

\textit{Induction step:} Fix a positive integer $k$. Assume (as the induction
hypothesis) that (\ref{pf.cor.nilpcrit-conv.2}) holds whenever $p<k$. We must
now prove that (\ref{pf.cor.nilpcrit-conv.2}) holds for $p=k$.

From (\ref{pf.cor.nilpcrit-conv.1}) (applied to $j=k$), we obtain $c_{k}=0$.

\begin{vershort}
We have assumed that (\ref{pf.cor.nilpcrit-conv.2}) holds whenever $p<k$. In
other words,%
\begin{equation}
\operatorname*{Tr}\left(  A^{p}\right)  =0\ \ \ \ \ \ \ \ \ \ \text{for every
positive integer }p<k. \label{pf.cor.nilpcrit-conv.2.pf.short.IH}%
\end{equation}

Now, Theorem \ref{thm.TCH} yields%
\[
kc_{k}+\sum_{i=1}^{k}\operatorname*{Tr}\left(  A^{i}\right)  c_{k-i}=0.
\]
Hence,%
\begin{align*}
0  &  =k\underbrace{c_{k}}_{=0}+\sum_{i=1}^{k}\operatorname*{Tr}\left(
A^{i}\right)  c_{k-i}=\sum_{i=1}^{k}\operatorname*{Tr}\left(  A^{i}\right)
c_{k-i}=\sum_{i=1}^{k-1}\underbrace{\operatorname*{Tr}\left(  A^{i}\right)
}_{\substack{=0\\\text{(by (\ref{pf.cor.nilpcrit-conv.2.pf.short.IH}%
),}\\\text{applied}\\\text{to }p=i\text{)}}}c_{k-i}+\operatorname*{Tr}\left(
A^{k}\right)  \underbrace{c_{k-k}}_{=c_{0}=1}\\
&  \ \ \ \ \ \ \ \ \ \ \left(  \text{here, we have split off the addend for
}i=k\text{ from the sum}\right) \\
&  =\underbrace{\sum_{i=1}^{k-1}0c_{k-i}}_{=0}+\operatorname*{Tr}\left(
A^{k}\right)  =\operatorname*{Tr}\left(  A^{k}\right)  .
\end{align*}

\end{vershort}

\begin{verlong}
We have assumed that (\ref{pf.cor.nilpcrit-conv.2}) holds whenever $p<k$. In
other words,%
\begin{equation}
\operatorname*{Tr}\left(  A^{p}\right)  =0\ \ \ \ \ \ \ \ \ \ \text{for every
positive integer }p<k. \label{pf.cor.nilpcrit-conv.2.pf.IH}%
\end{equation}
Thus, for each $i\in\left\{  1,2,\ldots,k-1\right\}  $, we have%
\begin{equation}
\operatorname*{Tr}\left(  A^{i}\right)  =0 \label{pf.cor.nilpcrit-conv.2.pf.3}%
\end{equation}
\footnote{\textit{Proof of (\ref{pf.cor.nilpcrit-conv.2.pf.3}):} Let
$i\in\left\{  1,2,\ldots,k-1\right\}  $. Thus, $1\leq i\leq k-1$. Hence, $i$
is a positive integer (since $i\in\left\{  1,2,\ldots,k-1\right\}  $) and
satisfies $i\leq k-1<k$. Therefore, (\ref{pf.cor.nilpcrit-conv.2.pf.IH})
(applied to $p=i$) yields $\operatorname*{Tr}\left(  A^{i}\right)  =0$. Qed.}.

Now, Theorem \ref{thm.TCH} yields%
\[
kc_{k}+\sum_{i=1}^{k}\operatorname*{Tr}\left(  A^{i}\right)  c_{k-i}=0.
\]
Hence,%
\[
\sum_{i=1}^{k}\operatorname*{Tr}\left(  A^{i}\right)  c_{k-i}%
=-\underbrace{c_{k}}_{=0}=-0=0.
\]
Thus,%
\begin{align*}
0  &  =\sum_{i=1}^{k}\operatorname*{Tr}\left(  A^{i}\right)  c_{k-i}%
=\sum_{i=1}^{k-1}\underbrace{\operatorname*{Tr}\left(  A^{i}\right)
}_{\substack{=0\\\text{(by (\ref{pf.cor.nilpcrit-conv.2.pf.3}))}}%
}c_{k-i}+\operatorname*{Tr}\left(  A^{k}\right)  \underbrace{c_{k-k}}%
_{=c_{0}=1}\\
&  \ \ \ \ \ \ \ \ \ \ \left(
\begin{array}
[c]{c}%
\text{here, we have split off the addend for }i=k\text{ from the sum}\\
\text{(since }k\in\left\{  1,2,\ldots,k\right\}  \text{ (since }k\text{ is a
positive integer))}%
\end{array}
\right) \\
&  =\underbrace{\sum_{i=1}^{k-1}0c_{k-i}}_{=0}+\operatorname*{Tr}\left(
A^{k}\right)  =\operatorname*{Tr}\left(  A^{k}\right)  .
\end{align*}

\end{verlong}

Thus, $\operatorname*{Tr}\left(  A^{k}\right)  =0$. In other words,
(\ref{pf.cor.nilpcrit-conv.2}) holds for $p=k$. This completes the induction
step. Thus, (\ref{pf.cor.nilpcrit-conv.2}) is proven by strong induction.]

We have thus proven that $\operatorname*{Tr}\left(  A^{p}\right)  =0$ for
every positive integer $p$. Renaming the variable $p$ as $i$ in this
statement, we conclude that $\operatorname*{Tr}\left(  A^{i}\right)  =0$ for
every positive integer $i$. This proves Corollary \ref{cor.nilpcrit-conv}.
\end{proof}

\section{\label{sec.adj}More on the adjugate}

I shall now discuss various other properties of the adjugate
$\operatorname*{adj}A$ of a square matrix $A$.

\subsection{Functoriality}

For any $n\in\mathbb{N}$ and $m\in\mathbb{N}$, a homomorphism $f:\mathbb{L}%
\rightarrow\mathbb{M}$ between two rings $\mathbb{L}$ and $\mathbb{M}$ gives
rise to a map $f^{n\times m}:\mathbb{L}^{n\times m}\rightarrow\mathbb{M}%
^{n\times m}$ (as defined in Definition \ref{def.f-mat}). We recall some
classical properties of these maps $f^{n\times m}$:

\begin{proposition}
\label{prop.functor.clas}Let $\mathbb{L}$ and $\mathbb{M}$ be two commutative
rings. Let $f:\mathbb{L}\rightarrow\mathbb{M}$ be a ring homomorphism.
\medskip

\textbf{(a)} For every $n\in\mathbb{N}$ and $m\in\mathbb{N}$, the map
$f^{n\times m}:\mathbb{L}^{n\times m}\rightarrow\mathbb{M}^{n\times m}$ is a
homomorphism of additive groups. \medskip

\textbf{(b)} Every $n\in\mathbb{N}$ satisfies $f^{n\times n}\left(
I_{n}\right)  =I_{n}$. \medskip

\textbf{(c)} For every $n\in\mathbb{N}$, $m\in\mathbb{N}$, $p\in\mathbb{N}$,
$A\in\mathbb{L}^{n\times m}$ and $B\in\mathbb{L}^{m\times p}$, we have
$f^{n\times p}\left(  AB\right)  =f^{n\times m}\left(  A\right)  \cdot
f^{m\times p}\left(  B\right)  $. \medskip

\textbf{(d)} For every $n\in\mathbb{N}$ and $m\in\mathbb{N}$ and every
$A\in\mathbb{L}^{n\times m}$ and $\lambda\in\mathbb{L}$, we have $f^{n\times
m}\left(  \lambda A\right)  =f\left(  \lambda\right)  f^{n\times m}\left(
A\right)  $.
\end{proposition}

Now, let me state the classical (and simple) fact which is often (somewhat
incompletely) subsumed under the slogan \textquotedblleft ring homomorphisms
preserve determinants and adjugates\textquotedblright:

\begin{proposition}
\label{prop.functor.det}Let $\mathbb{L}$ and $\mathbb{M}$ be two commutative
rings. Let $f:\mathbb{L}\rightarrow\mathbb{M}$ be a ring homomorphism. Let
$n\in\mathbb{N}$. Let $A\in\mathbb{L}^{n\times n}$. \medskip

\textbf{(a)} We have $f\left(  \det A\right)  =\det\left(  f^{n\times
n}\left(  A\right)  \right)  $. \medskip

\textbf{(b)} Any two elements $u$ and $v$ of $\left\{  1,2,\ldots,n\right\}  $
satisfy $f^{\left(  n-1\right)  \times\left(  n-1\right)  }\left(  A_{\sim
u,\sim v}\right)  =\left(  f^{n\times n}\left(  A\right)  \right)  _{\sim
u,\sim v}$. \medskip

\textbf{(c)} We have $f^{n\times n}\left(  \operatorname*{adj}A\right)
=\operatorname*{adj}\left(  f^{n\times n}\left(  A\right)  \right)  $.
\end{proposition}

\begin{vershort}
\begin{proof}
[Proof of Proposition \ref{prop.functor.det}.]Proving Proposition
\ref{prop.functor.det} is completely straightforward, and left to the reader.
\end{proof}
\end{vershort}

\begin{verlong}
\begin{proof}
[Proof of Proposition \ref{prop.functor.det}.]Write the $n\times n$-matrix $A$
in the form $A=\left(  a_{i,j}\right)  _{1\leq i\leq n,\ 1\leq j\leq n}$.
Then,
\[
f^{n\times n}\left(  \underbrace{A}_{=\left(  a_{i,j}\right)  _{1\leq i\leq
n,\ 1\leq j\leq n}}\right)  =f^{n\times n}\left(  \left(  a_{i,j}\right)
_{1\leq i\leq n,\ 1\leq j\leq n}\right)  =\left(  f\left(  a_{i,j}\right)
\right)  _{1\leq i\leq n,\ 1\leq j\leq n}%
\]
(by the definition of $f^{n\times n}$). Hence, (\ref{eq.det}) (applied to
$\mathbb{M}$, $f^{n\times n}\left(  A\right)  $ and $f\left(  a_{i,j}\right)
$ instead of $\mathbb{K}$, $A$ and $a_{i,j}$) yields%
\begin{equation}
\det\left(  f^{n\times n}\left(  A\right)  \right)  =\sum_{\sigma\in S_{n}%
}\left(  -1\right)  ^{\sigma}\prod_{i=1}^{n}f\left(  a_{i,\sigma\left(
i\right)  }\right)  . \label{pf.prop.functor.det.a.1}%
\end{equation}

\textbf{(a)} The equality (\ref{eq.det}) (applied to $\mathbb{L}$ instead of
$\mathbb{K}$) yields%
\[
\det A=\sum_{\sigma\in S_{n}}\left(  -1\right)  ^{\sigma}\prod_{i=1}%
^{n}a_{i,\sigma\left(  i\right)  }.
\]
Applying the map $f$ to both sides of this equality, we obtain%
\begin{align*}
f\left(  \det A\right)   &  =f\left(  \sum_{\sigma\in S_{n}}\left(  -1\right)
^{\sigma}\prod_{i=1}^{n}a_{i,\sigma\left(  i\right)  }\right)  =\sum
_{\sigma\in S_{n}}\left(  -1\right)  ^{\sigma}\prod_{i=1}^{n}f\left(
a_{i,\sigma\left(  i\right)  }\right) \\
&  \ \ \ \ \ \ \ \ \ \ \left(  \text{since }f\text{ is a ring homomorphism}%
\right) \\
&  =\det\left(  f^{n\times n}\left(  A\right)  \right)
\ \ \ \ \ \ \ \ \ \ \left(  \text{by (\ref{pf.prop.functor.det.a.1})}\right)
.
\end{align*}
This proves Proposition \ref{prop.functor.det} \textbf{(a)}. \medskip

\textbf{(b)} Let $u$ and $v$ be two elements of $\left\{  1,2,\ldots
,n\right\}  $.

Let $\left(  p_{1},p_{2},\ldots,p_{n-1}\right)  $ be the $\left(  n-1\right)
$-tuple $\left(  1,2,\ldots,\widehat{u},\ldots,n\right)  $. Thus, $\left(
p_{1},p_{2},\ldots,p_{n-1}\right)  =\left(  1,2,\ldots,\widehat{u}%
,\ldots,n\right)  $.

Let $\left(  q_{1},q_{2},\ldots,q_{n-1}\right)  $ be the $\left(  n-1\right)
$-tuple $\left(  1,2,\ldots,\widehat{v},\ldots,n\right)  $. Thus, $\left(
q_{1},q_{2},\ldots,q_{n-1}\right)  =\left(  1,2,\ldots,\widehat{v}%
,\ldots,n\right)  $.

If $C=\left(  c_{i,j}\right)  _{1\leq i\leq n,\ 1\leq j\leq n}$ is any
$n\times n$-matrix over a commutative ring $\mathbb{K}$, then
\begin{equation}
C_{\sim u,\sim v}=\left(  c_{p_{i},q_{j}}\right)  _{1\leq i\leq n-1,\ 1\leq
j\leq n-1} \label{pf.prop.functor.det.b.1}%
\end{equation}
\footnote{\textit{Proof of (\ref{pf.prop.functor.det.b.1}):} The equality
(\ref{pf.prop.functor.det.b.1}) is precisely the equality
(\ref{pf.lem.adj.linear.1}), and thus has already been proven.}.

From (\ref{pf.prop.functor.det.b.1}) (applied to $\mathbb{L}$, $A$ and
$a_{i,j}$ instead of $\mathbb{K}$, $C$ and $c_{i,j}$), we obtain $A_{\sim
u,\sim v}=\left(  a_{p_{i},q_{j}}\right)  _{1\leq i\leq n-1,\ 1\leq j\leq
n-1}$. Hence, the definition of the map $f^{\left(  n-1\right)  \times\left(
n-1\right)  }$ yields%
\begin{equation}
f^{\left(  n-1\right)  \times\left(  n-1\right)  }\left(  A_{\sim u,\sim
v}\right)  =\left(  f\left(  a_{p_{i},q_{j}}\right)  \right)  _{1\leq i\leq
n-1,\ 1\leq j\leq n-1}. \label{pf.prop.functor.det.b.2}%
\end{equation}

From (\ref{pf.prop.functor.det.b.1}) (applied to $\mathbb{M}$, $f^{n\times
n}\left(  A\right)  $ and $f\left(  a_{i,j}\right)  $ instead of $\mathbb{K}$,
$C$ and $c_{i,j}$), we obtain $\left(  f^{n\times n}\left(  A\right)  \right)
_{\sim u,\sim v}=\left(  f\left(  a_{p_{i},q_{j}}\right)  \right)  _{1\leq
i\leq n-1,\ 1\leq j\leq n-1}$ (since $f^{n\times n}\left(  A\right)  =\left(
f\left(  a_{i,j}\right)  \right)  _{1\leq i\leq n,\ 1\leq j\leq n}$).
Comparing this with (\ref{pf.prop.functor.det.b.2}), we obtain $f^{\left(
n-1\right)  \times\left(  n-1\right)  }\left(  A_{\sim u,\sim v}\right)
=\left(  f^{n\times n}\left(  A\right)  \right)  _{\sim u,\sim v}$. This
proves Proposition \ref{prop.functor.det} \textbf{(b)}. \medskip

\textbf{(c)} Let $\left(  i,j\right)  \in\left\{  1,2,\ldots,n\right\}  ^{2}$.
Then, $A_{\sim j,\sim i}\in\mathbb{L}^{\left(  n-1\right)  \times\left(
n-1\right)  }$. Hence, Proposition \ref{prop.functor.det} \textbf{(a)}
(applied to $n-1$ and $A_{\sim j,\sim i}$ instead of $n$ and $A$) yields%
\[
f\left(  \det\left(  A_{\sim j,\sim i}\right)  \right)  =\det\left(
\underbrace{f^{\left(  n-1\right)  \times\left(  n-1\right)  }\left(  A_{\sim
j,\sim i}\right)  }_{\substack{=\left(  f^{n\times n}\left(  A\right)
\right)  _{\sim j,\sim i}\\\text{(by Proposition \ref{prop.functor.det}
\textbf{(b)}}\\\text{(applied to }u=j\text{ and }v=i\text{))}}}\right)
=\det\left(  \left(  f^{n\times n}\left(  A\right)  \right)  _{\sim j,\sim
i}\right)  .
\]
But $f$ is a ring homomorphism; hence,%
\begin{align}
&  f\left(  \left(  -1\right)  ^{i+j}\det\left(  A_{\sim j,\sim i}\right)
\right) \nonumber\\
&  =\left(  -1\right)  ^{i+j}\underbrace{f\left(  \det\left(  A_{\sim j,\sim
i}\right)  \right)  }_{=\det\left(  \left(  f^{n\times n}\left(  A\right)
\right)  _{\sim j,\sim i}\right)  }=\left(  -1\right)  ^{i+j}\det\left(
\left(  f^{n\times n}\left(  A\right)  \right)  _{\sim j,\sim i}\right)  .
\label{pf.prop.functor.det.c.2}%
\end{align}

Now, forget that we fixed $\left(  i,j\right)  $. We thus have shown that
(\ref{pf.prop.functor.det.c.2}) holds for every $\left(  i,j\right)
\in\left\{  1,2,\ldots,n\right\}  ^{2}$.

The definition of $\operatorname*{adj}A$ yields $\operatorname*{adj}A=\left(
\left(  -1\right)  ^{i+j}\det\left(  A_{\sim j,\sim i}\right)  \right)
_{1\leq i\leq n,\ 1\leq j\leq n}$. Hence, the definition of the map
$f^{n\times n}$ yields%
\begin{align}
f^{n\times n}\left(  \operatorname*{adj}A\right)   &  =\left(
\underbrace{f\left(  \left(  -1\right)  ^{i+j}\det\left(  A_{\sim j,\sim
i}\right)  \right)  }_{\substack{=\left(  -1\right)  ^{i+j}\det\left(  \left(
f^{n\times n}\left(  A\right)  \right)  _{\sim j,\sim i}\right)  \\\text{(by
(\ref{pf.prop.functor.det.c.2}))}}}\right)  _{1\leq i\leq n,\ 1\leq j\leq
n}\nonumber\\
&  =\left(  \left(  -1\right)  ^{i+j}\det\left(  \left(  f^{n\times n}\left(
A\right)  \right)  _{\sim j,\sim i}\right)  \right)  _{1\leq i\leq n,\ 1\leq
j\leq n}. \label{pf.prop.functor.det.c.3}%
\end{align}

But the definition of $\operatorname*{adj}\left(  f^{n\times n}\left(
A\right)  \right)  $ yields%
\begin{align*}
\operatorname*{adj}\left(  f^{n\times n}\left(  A\right)  \right)   &
=\left(  \left(  -1\right)  ^{i+j}\det\left(  \left(  f^{n\times n}\left(
A\right)  \right)  _{\sim j,\sim i}\right)  \right)  _{1\leq i\leq n,\ 1\leq
j\leq n}\\
&  =f^{n\times n}\left(  \operatorname*{adj}A\right)
\ \ \ \ \ \ \ \ \ \ \left(  \text{by (\ref{pf.prop.functor.det.c.3})}\right)
.
\end{align*}
This proves Proposition \ref{prop.functor.det} \textbf{(c)}.
\end{proof}
\end{verlong}

\subsection{The evaluation homomorphism}

We shall apply the above to relate the determinant and the adjugate of a
matrix $A$ with those of the matrix $tI_{n}+A$:

\begin{proposition}
\label{prop.tInA.ev}Let $\varepsilon:\mathbb{K}\left[  t\right]
\rightarrow\mathbb{K}$ be the map which sends every polynomial $p\in
\mathbb{K}\left[  t\right]  $ to its value $p\left(  0\right)  $. It is
well-known that $\varepsilon$ is a $\mathbb{K}$-algebra homomorphism.

Let $n\in\mathbb{N}$. Let $A\in\mathbb{K}^{n\times n}$. Consider the matrix
$tI_{n}+A\in\left(  \mathbb{K}\left[  t\right]  \right)  ^{n\times n}$. Then:
\medskip

\textbf{(a)} We have $\varepsilon\left(  \det\left(  tI_{n}+A\right)  \right)
=\det A$. \medskip

\textbf{(b)} We have $\varepsilon^{n\times n}\left(  \operatorname*{adj}%
\left(  tI_{n}+A\right)  \right)  =\operatorname*{adj}A$. \medskip

\textbf{(c)} We have $\varepsilon^{n\times n}\left(  tI_{n}+A\right)  =A$.
\end{proposition}

\begin{vershort}
\begin{proof}
[Proof of Proposition \ref{prop.tInA.ev}.]We have%
\[
\varepsilon^{n\times n}\left(  tB+A\right)  =A
\]
for every $B\in\mathbb{K}^{n\times n}$\ \ \ \ \footnote{\textit{Proof.} This
equality is similar to (\ref{pf.thm.ddet.short.1}), and is proven
analogously.}. Applying this to $B=I_{n}$, we obtain $\varepsilon^{n\times
n}\left(  tI_{n}+A\right)  =A$. This proves Proposition \ref{prop.tInA.ev}
\textbf{(c)}. \medskip

\textbf{(a)} Proposition \ref{prop.functor.det} \textbf{(a)} (applied to
$\mathbb{K}\left[  t\right]  $, $\mathbb{K}$, $\varepsilon$ and $tI_{n}+A$
instead of $\mathbb{L}$, $\mathbb{M}$, $f$ and $A$) yields%
\[
\varepsilon\left(  \det\left(  tI_{n}+A\right)  \right)  =\det\left(
\underbrace{\varepsilon^{n\times n}\left(  tI_{n}+A\right)  }_{=A}\right)
=\det A.
\]
This proves Proposition \ref{prop.tInA.ev} \textbf{(a)}. \medskip

\textbf{(b)} Proposition \ref{prop.functor.det} \textbf{(c)} (applied to
$\mathbb{K}\left[  t\right]  $, $\mathbb{K}$, $\varepsilon$ and $tI_{n}+A$
instead of $\mathbb{L}$, $\mathbb{M}$, $f$ and $A$) yields%
\[
\varepsilon^{n\times n}\left(  \operatorname*{adj}\left(  tI_{n}+A\right)
\right)  =\operatorname*{adj}\left(  \underbrace{\varepsilon^{n\times
n}\left(  tI_{n}+A\right)  }_{=A}\right)  =\operatorname*{adj}A.
\]
This proves Proposition \ref{prop.tInA.ev} \textbf{(b)}.
\end{proof}
\end{vershort}

\begin{verlong}
\begin{proof}
[Proof of Proposition \ref{prop.tInA.ev}.]We have%
\begin{equation}
\varepsilon^{n\times n}\left(  tB+A\right)  =A \label{pf.prop.tInA.ev.1}%
\end{equation}
for every $B\in\mathbb{K}^{n\times n}$\ \ \ \ \footnote{\textit{Proof of
(\ref{pf.prop.tInA.ev.1}):} Let $B\in\mathbb{K}^{n\times n}$. Write the
$n\times n$-matrix $B\in\mathbb{K}^{n\times n}$ in the form $B=\left(
b_{i,j}\right)  _{1\leq i\leq n,\ 1\leq j\leq n}$. Thus,
\begin{equation}
b_{i,j}\in\mathbb{K}\ \ \ \ \ \ \ \ \ \ \text{for every }\left(  i,j\right)
\in\left\{  1,2,\ldots,n\right\}  ^{2} \label{pf.prop.tInA.ev.1.pf.1}%
\end{equation}
(since $B\in\mathbb{K}^{n\times n}$). Also, $t\underbrace{B}_{=\left(
b_{i,j}\right)  _{1\leq i\leq n,\ 1\leq j\leq n}}=t\left(  b_{i,j}\right)
_{1\leq i\leq n,\ 1\leq j\leq n}=\left(  tb_{i,j}\right)  _{1\leq i\leq
n,\ 1\leq j\leq n}$.
\par
Write the $n\times n$-matrix $A\in\mathbb{K}^{n\times n}$ in the form
$A=\left(  a_{i,j}\right)  _{1\leq i\leq n,\ 1\leq j\leq n}$. Thus,%
\begin{equation}
a_{i,j}\in\mathbb{K}\ \ \ \ \ \ \ \ \ \ \text{for every }\left(  i,j\right)
\in\left\{  1,2,\ldots,n\right\}  ^{2} \label{pf.prop.tInA.ev.1.pf.2}%
\end{equation}
(since $A\in\mathbb{K}^{n\times n}$).
\par
Fix $\left(  i,j\right)  \in\left\{  1,2,\ldots,n\right\}  ^{2}$. Then,
$b_{i,j}\in\mathbb{K}$ (by (\ref{pf.prop.tInA.ev.1.pf.1})) and $a_{i,j}%
\in\mathbb{K}$ (by (\ref{pf.prop.tInA.ev.1.pf.2})). Hence, the value of the
polynomial $tb_{i,j}-a_{i,j}\in\mathbb{K}\left[  t\right]  $ at $0$ is
$\left(  tb_{i,j}+a_{i,j}\right)  \left(  0\right)  =a_{i,j}$. But the
definition of $\varepsilon$ yields%
\begin{equation}
\varepsilon\left(  tb_{i,j}+a_{i,j}\right)  =\left(  tb_{i,j}+a_{i,j}\right)
\left(  0\right)  =a_{i,j}. \label{pf.prop.tInA.ev.1.pf.4}%
\end{equation}
\par
Now, forget that we fixed $\left(  i,j\right)  $. We thus have proven
(\ref{pf.prop.tInA.ev.1.pf.4}) for every $\left(  i,j\right)  \in\left\{
1,2,\ldots,n\right\}  ^{2}.$
\par
Now,%
\[
\underbrace{tB}_{=\left(  tb_{i,j}\right)  _{1\leq i\leq n,\ 1\leq j\leq n}%
}+\underbrace{A}_{=\left(  a_{i,j}\right)  _{1\leq i\leq n,\ 1\leq j\leq n}%
}=\left(  tb_{i,j}\right)  _{1\leq i\leq n,\ 1\leq j\leq n}+\left(
a_{i,j}\right)  _{1\leq i\leq n,\ 1\leq j\leq n}=\left(  tb_{i,j}%
+a_{i,j}\right)  _{1\leq i\leq n,\ 1\leq j\leq n}.
\]
Hence, the definition of the map $\varepsilon^{n\times n}$ yields%
\[
\varepsilon^{n\times n}\left(  tB+A\right)  =\left(  \underbrace{\varepsilon
\left(  tb_{i,j}+a_{i,j}\right)  }_{\substack{=a_{i,j}\\\text{(by
(\ref{pf.prop.tInA.ev.1.pf.4}))}}}\right)  _{1\leq i\leq n,\ 1\leq j\leq
n}=\left(  a_{i,j}\right)  _{1\leq i\leq n,\ 1\leq j\leq n}=A.
\]
Hence, (\ref{pf.prop.tInA.ev.1}) is proven.}. Applying this to $B=I_{n}$, we
obtain $\varepsilon^{n\times n}\left(  tI_{n}+A\right)  =A$. This proves
Proposition \ref{prop.tInA.ev} \textbf{(c)}. \medskip

The map $\varepsilon$ is a $\mathbb{K}$-algebra homomorphism, thus a ring
homomorphism. \medskip

\textbf{(a)} Proposition \ref{prop.functor.det} \textbf{(a)} (applied to
$\mathbb{K}\left[  t\right]  $, $\mathbb{K}$, $\varepsilon$ and $tI_{n}+A$
instead of $\mathbb{L}$, $\mathbb{M}$, $f$ and $A$) yields%
\[
\varepsilon\left(  \det\left(  tI_{n}+A\right)  \right)  =\det\left(
\underbrace{\varepsilon^{n\times n}\left(  tI_{n}+A\right)  }_{=A}\right)
=\det A.
\]
This proves Proposition \ref{prop.tInA.ev} \textbf{(a)}. \medskip

\textbf{(b)} Proposition \ref{prop.functor.det} \textbf{(c)} (applied to
$\mathbb{K}\left[  t\right]  $, $\mathbb{K}$, $\varepsilon$ and $tI_{n}+A$
instead of $\mathbb{L}$, $\mathbb{M}$, $f$ and $A$) yields%
\[
\varepsilon^{n\times n}\left(  \operatorname*{adj}\left(  tI_{n}+A\right)
\right)  =\operatorname*{adj}\left(  \underbrace{\varepsilon^{n\times
n}\left(  tI_{n}+A\right)  }_{=A}\right)  =\operatorname*{adj}A.
\]
This proves Proposition \ref{prop.tInA.ev} \textbf{(b)}.
\end{proof}
\end{verlong}

If $A\in\mathbb{K}^{n\times n}$ is a square matrix, then the matrix
$tI_{n}+A\in\left(  \mathbb{K}\left[  t\right]  \right)  ^{n\times n}$ has a
property which the matrix $A$ might not have: namely, its determinant is
regular. Let us first define what this means:

\begin{definition}
\label{def.regular}Let $\mathbb{A}$ be a commutative ring. Let $a\in
\mathbb{A}$. The element $a$ of $\mathbb{A}$ is said to be \textit{regular} if
and only if every $x\in\mathbb{A}$ satisfying $ax=0$ satisfies $x=0$.

Instead of saying that $a$ is regular, one can also say that \textquotedblleft%
$a$ is \textit{cancellable}\textquotedblright, or that \textquotedblleft$a$ is
a \textit{non-zero-divisor}\textquotedblright.
\end{definition}

A basic property of regular elements is the following:

\begin{lemma}
\label{lem.regular.cancel}Let $\mathbb{A}$ be a commutative ring. Let $a$ be a
regular element of $\mathbb{A}$. Let $b$ and $c$ be two elements of
$\mathbb{A}$ such that $ab=ac$. Then, $b=c$.
\end{lemma}

\begin{proof}
[Proof of Lemma \ref{lem.regular.cancel}.]We have $a\left(  b-c\right)
=\underbrace{ab}_{=ac}-ac=ac-ac=0$.

Now, recall that the element $a$ of $\mathbb{A}$ is regular if and only if
every $x\in\mathbb{A}$ satisfying $ax=0$ satisfies $x=0$ (by the definition of
\textquotedblleft regular\textquotedblright). Hence, every $x\in\mathbb{A}$
satisfying $ax=0$ satisfies $x=0$ (because the element $a$ of $\mathbb{A}$ is
regular). Applying this to $x=b-c$, we obtain $b-c=0$ (since $a\left(
b-c\right)  =0$). Thus, $b=c$. This proves Lemma \ref{lem.regular.cancel}.
\end{proof}

Regular elements, of course, can also be cancelled from matrix equations:

\begin{lemma}
\label{lem.regular.mat-cancel}Let $n\in\mathbb{N}$ and $m\in\mathbb{N}$. Let
$a$ be a regular element of $\mathbb{K}$. Let $B\in\mathbb{K}^{n\times m}$ and
$C\in\mathbb{K}^{n\times m}$ be such that $aB=aC$. Then, $B=C$.
\end{lemma}

\begin{vershort}
\begin{proof}
[Proof of Lemma \ref{lem.regular.mat-cancel}.]Write the $n\times m$-matrices
$B$ and $C$ in the forms $B=\left(  b_{i,j}\right)  _{1\leq i\leq n,\ 1\leq
j\leq m}$ and $C=\left(  c_{i,j}\right)  _{1\leq i\leq n,\ 1\leq j\leq m}$.
Then, $aB=\left(  ab_{i,j}\right)  _{1\leq i\leq n,\ 1\leq j\leq m}$ and
$aC=\left(  ac_{i,j}\right)  _{1\leq i\leq n,\ 1\leq j\leq m}$. Hence,
\[
\left(  ab_{i,j}\right)  _{1\leq i\leq n,\ 1\leq j\leq m}=aB=aC=\left(
ac_{i,j}\right)  _{1\leq i\leq n,\ 1\leq j\leq m}.
\]
In other words,
\[
ab_{i,j}=ac_{i,j}\ \ \ \ \ \ \ \ \ \ \text{for every }\left(  i,j\right)
\in\left\{  1,2,\ldots,n\right\}  \times\left\{  1,2,\ldots,m\right\}  .
\]
Thus,
\[
b_{i,j}=c_{i,j}\ \ \ \ \ \ \ \ \ \ \text{for every }\left(  i,j\right)
\in\left\{  1,2,\ldots,n\right\}  \times\left\{  1,2,\ldots,m\right\}
\]
(by Lemma \ref{lem.regular.cancel}, applied to $b=b_{i,j}$ and $c=c_{i,j}$).
Hence, $\left(  b_{i,j}\right)  _{1\leq i\leq n,\ 1\leq j\leq m}=\left(
c_{i,j}\right)  _{1\leq i\leq n,\ 1\leq j\leq m}$. Thus, $B=\left(
b_{i,j}\right)  _{1\leq i\leq n,\ 1\leq j\leq m}=\left(  c_{i,j}\right)
_{1\leq i\leq n,\ 1\leq j\leq m}=C$. Lemma \ref{lem.regular.mat-cancel} is proven.
\end{proof}
\end{vershort}

\begin{verlong}
\begin{proof}
[Proof of Lemma \ref{lem.regular.mat-cancel}.]Write the $n\times m$-matrix
$B\in\mathbb{K}^{n\times m}$ in the form $B=\left(  b_{i,j}\right)  _{1\leq
i\leq n,\ 1\leq j\leq m}$.

Write the $n\times m$-matrix $C\in\mathbb{K}^{n\times m}$ in the form
$C=\left(  c_{i,j}\right)  _{1\leq i\leq n,\ 1\leq j\leq m}$.

Now,%
\[
a\underbrace{B}_{=\left(  b_{i,j}\right)  _{1\leq i\leq n,\ 1\leq j\leq m}%
}=a\left(  b_{i,j}\right)  _{1\leq i\leq n,\ 1\leq j\leq m}=\left(
ab_{i,j}\right)  _{1\leq i\leq n,\ 1\leq j\leq m},
\]
so that%
\[
\left(  ab_{i,j}\right)  _{1\leq i\leq n,\ 1\leq j\leq m}=aB=a\underbrace{C}%
_{=\left(  c_{i,j}\right)  _{1\leq i\leq n,\ 1\leq j\leq m}}=a\left(
c_{i,j}\right)  _{1\leq i\leq n,\ 1\leq j\leq m}=\left(  ac_{i,j}\right)
_{1\leq i\leq n,\ 1\leq j\leq m}.
\]
In other words,
\begin{equation}
ab_{i,j}=ac_{i,j}\ \ \ \ \ \ \ \ \ \ \text{for every }\left(  i,j\right)
\in\left\{  1,2,\ldots,n\right\}  \times\left\{  1,2,\ldots,m\right\}  .
\label{pf.lem.regular.det.cancel.a.1}%
\end{equation}

Now,
\begin{equation}
b_{i,j}=c_{i,j}\ \ \ \ \ \ \ \ \ \ \text{for every }\left(  i,j\right)
\in\left\{  1,2,\ldots,n\right\}  \times\left\{  1,2,\ldots,m\right\}
\label{pf.lem.regular.det.cancel.a.2}%
\end{equation}
\footnote{\textit{Proof of (\ref{pf.lem.regular.det.cancel.a.2}):} Let
$\left(  i,j\right)  \in\left\{  1,2,\ldots,n\right\}  \times\left\{
1,2,\ldots,m\right\}  $. From (\ref{pf.lem.regular.det.cancel.a.1}), we obtain
$ab_{i,j}=ac_{i,j}$. Lemma \ref{lem.regular.cancel} (applied to $\mathbb{K}$,
$b_{i,j}$ and $c_{i,j}$ instead of $\mathbb{A}$, $b$ and $c$) thus yields
$b_{i,j}=c_{i,j}$ (since the element $a$ of $\mathbb{K}$ is regular). This
proves (\ref{pf.lem.regular.det.cancel.a.2}).}. Hence, $\left(  b_{i,j}%
\right)  _{1\leq i\leq n,\ 1\leq j\leq m}=\left(  c_{i,j}\right)  _{1\leq
i\leq n,\ 1\leq j\leq m}$. Thus, $B=\left(  b_{i,j}\right)  _{1\leq i\leq
n,\ 1\leq j\leq m}=\left(  c_{i,j}\right)  _{1\leq i\leq n,\ 1\leq j\leq m}%
=C$. This proves Lemma \ref{lem.regular.mat-cancel}.
\end{proof}
\end{verlong}

One important way to construct regular elements is the following fact:

\begin{proposition}
\label{prop.pol.monic-reg}Let $n\in\mathbb{N}$. Let $p\in\mathbb{K}\left[
t\right]  $ be a monic polynomial of degree $n$. Then, the element $p$ of
$\mathbb{K}\left[  t\right]  $ is regular.
\end{proposition}

\begin{vershort}
\begin{proof}
[Proof of Proposition \ref{prop.pol.monic-reg}.]Proposition
\ref{prop.pol.monic-reg} is precisely \cite[Corollary 3.15]{Grinbe16b}%
.\footnote{In a nutshell, the proof in \cite[Corollary 3.15]{Grinbe16b}
proceeds as follows: The polynomial $p$ is monic of degree $n$; thus,
$p=t^{n}+\left(  \text{lower order terms}\right)  $ (where \textquotedblleft
lower order terms\textquotedblright\ means a $\mathbb{K}$-linear combination
of monomials $t^{i}$ with $i<n$). On the other hand, if $x\in\mathbb{K}\left[
t\right]  $ is such that $px=0$, then we can write $x$ as $x=at^{m}+\left(
\text{lower order terms}\right)  $ for some $a\in\mathbb{K}$ and some
$m\in\mathbb{N}$ (where \textquotedblleft lower order terms\textquotedblright%
\ means a $\mathbb{K}$-linear combination of monomials $t^{i}$ with $i<m$).
Multiplying the equalities $p=t^{n}+\left(  \text{lower order terms}\right)  $
and $x=at^{m}+\left(  \text{lower order terms}\right)  $, we find%
\[
px=at^{n+m}+\left(  \text{lower order terms}\right)  .
\]
Hence, from $px=0$, we obtain $a=0$. Thus, if $x\in\mathbb{K}\left[  t\right]
$ is such that $px=0$, then the highest-degree term of $x$ is $0$. But this
clearly entails $x=0$. Thus we have shown that $p$ is regular.}
\end{proof}
\end{vershort}

\begin{verlong}
\begin{proof}
[Proof of Proposition \ref{prop.pol.monic-reg}.]Proposition
\ref{prop.pol.monic-reg} is precisely \cite[Corollary 3.15]{Grinbe16b}. (The
notations used in \cite{Grinbe16b} differ from ours as follows: Our ring
$\mathbb{K}$ and our indeterminate $t$ correspond to the ring $A$ and the
indeterminate $X$ in \cite{Grinbe16b}.)\footnote{In a nutshell, the proof in
\cite[Corollary 3.15]{Grinbe16b} proceeds as follows: The polynomial $p$ is
monic of degree $n$; thus, $p=t^{n}+\left(  \text{lower order terms}\right)  $
(where \textquotedblleft lower order terms\textquotedblright\ means a
$\mathbb{K}$-linear combination of monomials $t^{i}$ with $i<n$). On the other
hand, if $x\in\mathbb{K}\left[  t\right]  $ is such that $px=0$, then we can
write $x$ as $x=at^{m}+\left(  \text{lower order terms}\right)  $ for some
$a\in\mathbb{K}$ and some $m\in\mathbb{N}$ (where \textquotedblleft lower
order terms\textquotedblright\ means a $\mathbb{K}$-linear combination of
monomials $t^{i}$ with $i<m$). Multiplying the equalities $p=t^{n}+\left(
\text{lower order terms}\right)  $ and $x=at^{m}+\left(  \text{lower order
terms}\right)  $, we find%
\[
px=at^{n+m}+\left(  \text{lower order terms}\right)  .
\]
Hence, from $px=0$, we obtain $a=0$. Thus, if $x\in\mathbb{K}\left[  t\right]
$ is such that $px=0$, then the highest-degree term of $x$ is $0$. But this
clearly entails $x=0$. Thus we have shown that $p$ is regular.}
\end{proof}
\end{verlong}

\begin{corollary}
\label{cor.tInA.reg}Let $n\in\mathbb{N}$. Let $A\in\mathbb{K}^{n\times n}$.
Consider the matrix $tI_{n}+A\in\left(  \mathbb{K}\left[  t\right]  \right)
^{n\times n}$.

Then, the element $\det\left(  tI_{n}+A\right)  $ of $\mathbb{K}\left[
t\right]  $ is regular.
\end{corollary}

\begin{proof}
[Proof of Corollary \ref{cor.tInA.reg}.]Proposition \ref{prop.ta+b}
\textbf{(a)} (applied to $I_{n}$ and $A$ instead of $A$ and $B$) yields that
$\det\left(  tI_{n}+A\right)  \in\mathbb{K}\left[  t\right]  $ is a polynomial
of degree $\leq n$ in $t$. Proposition \ref{prop.ta+b} \textbf{(c)} (applied
to $I_{n}$ and $A$ instead of $A$ and $B$) yields that $\left[  t^{n}\right]
\left(  \det\left(  tI_{n}+A\right)  \right)  =\det\left(  I_{n}\right)  =1$.

So we know that the polynomial $\det\left(  tI_{n}+A\right)  \in
\mathbb{K}\left[  t\right]  $ is a polynomial of degree $\leq n$, and that the
coefficient of $t^{n}$ in this polynomial is $\left[  t^{n}\right]  \left(
\det\left(  tI_{n}+A\right)  \right)  =1$. In other words, the polynomial
$\det\left(  tI_{n}+A\right)  \in\mathbb{K}\left[  t\right]  $ is monic of
degree $n$. Thus, Proposition \ref{prop.pol.monic-reg} (applied to
$p=\det\left(  tI_{n}+A\right)  $) shows that the element $\det\left(
tI_{n}+A\right)  $ of $\mathbb{K}\left[  t\right]  $ is regular. This proves
Corollary \ref{cor.tInA.reg}.
\end{proof}

A square matrix whose determinant is regular can be cancelled from equations,
as the following lemma shows:

\begin{lemma}
\label{lem.regular.det.cancel}Let $n\in\mathbb{N}$. Let $A\in\mathbb{K}%
^{n\times n}$. Assume that the element $\det A$ of $\mathbb{K}$ is regular.
Let $m\in\mathbb{N}$. \medskip

\textbf{(a)} If $B\in\mathbb{K}^{n\times m}$ and $C\in\mathbb{K}^{n\times m}$
are such that $AB=AC$, then $B=C$. \medskip

\textbf{(b)} If $B\in\mathbb{K}^{m\times n}$ and $C\in\mathbb{K}^{m\times n}$
are such that $BA=CA$, then $B=C$.
\end{lemma}

\begin{vershort}
\begin{proof}
[Proof of Lemma \ref{lem.regular.det.cancel}.]Define an element $a$ of
$\mathbb{K}$ by $a=\det A$. Recall that the element $\det A$ of $\mathbb{K}$
is regular. In other words, the element $a$ of $\mathbb{K}$ is regular (since
$a=\det A$). Theorem \ref{thm.adj.inverse} yields $A\cdot\operatorname*{adj}%
A=\operatorname*{adj}A\cdot A=\det A\cdot I_{n}$. \medskip

\textbf{(a)} Let $B\in\mathbb{K}^{n\times m}$ and $C\in\mathbb{K}^{n\times m}$
be such that $AB=AC$. We must prove that $B=C$.

We have
\[
\underbrace{\operatorname*{adj}A\cdot A}_{=\det A\cdot I_{n}}%
B=\underbrace{\det A}_{=a}\cdot\underbrace{I_{n}B}_{=B}=aB.
\]
Thus,
\[
aB=\operatorname*{adj}A\cdot\underbrace{AB}_{=AC}%
=\underbrace{\operatorname*{adj}A\cdot A}_{=\det A\cdot I_{n}}%
C=\underbrace{\det A}_{=a}\cdot\underbrace{I_{n}C}_{=C}=aC.
\]

Lemma \ref{lem.regular.mat-cancel} thus yields $B=C$. This proves Lemma
\ref{lem.regular.det.cancel} \textbf{(a)}. \medskip

\textbf{(b)} The proof of Lemma \ref{lem.regular.det.cancel} \textbf{(b)} is
similar to the proof of Lemma \ref{lem.regular.det.cancel} \textbf{(a)} (but
now we need to work with $BA\cdot\operatorname*{adj}A$ and $CA\cdot
\operatorname*{adj}A$ instead of $\operatorname*{adj}A\cdot AB$ and
$\operatorname*{adj}A\cdot AC$). The details are left to the reader.
\end{proof}
\end{vershort}

\begin{verlong}
\begin{proof}
[Proof of Lemma \ref{lem.regular.det.cancel}.]Define an element $a$ of
$\mathbb{K}$ by $a=\det A$. Recall that the element $\det A$ of $\mathbb{K}$
is regular. In other words, the element $a$ of $\mathbb{K}$ is regular (since
$a=\det A$).

Theorem \ref{thm.adj.inverse} yields $A\cdot\operatorname*{adj}%
A=\operatorname*{adj}A\cdot A=\det A\cdot I_{n}$. \medskip

\textbf{(a)} Let $B\in\mathbb{K}^{n\times m}$ and $C\in\mathbb{K}^{n\times m}$
be such that $AB=AC$. We must prove that $B=C$.

We have
\[
\underbrace{\operatorname*{adj}A\cdot A}_{=\det A\cdot I_{n}}%
B=\underbrace{\det A}_{=a}\cdot\underbrace{I_{n}B}_{=B}=aB.
\]
Hence,%
\[
aB=\operatorname*{adj}A\cdot\underbrace{AB}_{=AC}%
=\underbrace{\operatorname*{adj}A\cdot A}_{=\det A\cdot I_{n}}%
C=\underbrace{\det A}_{=a}\cdot\underbrace{I_{n}C}_{=C}=aC.
\]

Lemma \ref{lem.regular.mat-cancel} thus yields $B=C$ (since the element $a$ of
$\mathbb{K}$ is regular). This proves Lemma \ref{lem.regular.det.cancel}
\textbf{(a)}. \medskip

\textbf{(b)} Let $B\in\mathbb{K}^{m\times n}$ and $C\in\mathbb{K}^{m\times n}$
be such that $BA=CA$. We must prove that $B=C$.

We have
\[
B\underbrace{A\cdot\operatorname*{adj}A}_{=\det A\cdot I_{n}}=B\cdot\det
A\cdot I_{n}=\underbrace{\det A}_{=a}\cdot\,B=aB.
\]
Hence,%
\[
aB=\underbrace{BA}_{=CA}\cdot\operatorname*{adj}A=C\underbrace{A\cdot
\operatorname*{adj}A}_{=\det A\cdot I_{n}}=C\cdot\det A\cdot I_{n}%
=\underbrace{\det A}_{=a}\cdot\,C=aC.
\]
Lemma \ref{lem.regular.mat-cancel} (applied to $m$ and $n$ instead of $n$ and
$m$) thus yields $B=C$ (since the element $a$ of $\mathbb{K}$ is regular).
This proves Lemma \ref{lem.regular.det.cancel} \textbf{(b)}.
\end{proof}
\end{verlong}

\subsection{The adjugate of a product}

Corollary \ref{cor.tInA.reg} can be put to use in several circumstances. Here
is a simple example:

\begin{theorem}
\label{thm.adj.adj(AB)}Let $n\in\mathbb{N}$. Let $A$ and $B$ be two $n\times
n$-matrices. Then,%
\[
\operatorname*{adj}\left(  AB\right)  =\operatorname*{adj}B\cdot
\operatorname*{adj}A.
\]

\end{theorem}

Theorem \ref{thm.adj.adj(AB)} is the statement of \cite[Exercise
6.33]{detnotes}; see \cite[solution of Exercise 6.33]{detnotes} for a proof of
this theorem. We shall show a different proof of it now.

We begin by showing a particular case of Theorem \ref{thm.adj.adj(AB)}:

\begin{lemma}
\label{lem.adj.adj(AB).reg}Let $n\in\mathbb{N}$. Let $A$ and $B$ be two
$n\times n$-matrices. Assume that the elements $\det A$ and $\det B$ of
$\mathbb{K}$ are regular. Then, $\operatorname*{adj}\left(  AB\right)
=\operatorname*{adj}B\cdot\operatorname*{adj}A$.
\end{lemma}

\begin{proof}
[Proof of Lemma \ref{lem.adj.adj(AB).reg}.]Theorem \ref{thm.adj.inverse}
yields%
\[
A\cdot\operatorname*{adj}A=\operatorname*{adj}A\cdot A=\det A\cdot I_{n}.
\]
Theorem \ref{thm.adj.inverse} (applied to $B$ instead of $A$) yields%
\[
B\cdot\operatorname*{adj}B=\operatorname*{adj}B\cdot B=\det B\cdot I_{n}.
\]
Theorem \ref{thm.adj.inverse} (applied to $AB$ instead of $A$) yields%
\[
AB\cdot\operatorname*{adj}\left(  AB\right)  =\operatorname*{adj}\left(
AB\right)  \cdot AB=\det\left(  AB\right)  \cdot I_{n}.
\]

Now,
\begin{align*}
A\underbrace{B\cdot\operatorname*{adj}B}_{=\det B\cdot I_{n}}\cdot
\operatorname*{adj}A  &  =\underbrace{A\cdot\det B\cdot I_{n}}_{=\det B\cdot
A}\cdot\operatorname*{adj}A\\
&  =\det B\cdot\underbrace{A\cdot\operatorname*{adj}A}_{=\det A\cdot I_{n}%
}=\det B\cdot\det A\cdot I_{n}\\
&  =\det A\cdot\det B\cdot I_{n}.
\end{align*}
Comparing this with%
\[
AB\cdot\operatorname*{adj}\left(  AB\right)  =\underbrace{\det\left(
AB\right)  }_{\substack{=\det A\cdot\det B\\\text{(by \cite[Theorem
6.23]{detnotes})}}}\cdot\,I_{n}=\det A\cdot\det B\cdot I_{n},
\]
we obtain $AB\cdot\operatorname*{adj}B\cdot\operatorname*{adj}A=AB\cdot
\operatorname*{adj}\left(  AB\right)  $. Lemma \ref{lem.regular.det.cancel}
\textbf{(a)} (applied to $n$, $B\cdot\operatorname*{adj}B\cdot
\operatorname*{adj}A$ and $B\cdot\operatorname*{adj}\left(  AB\right)  $
instead of $m$, $B$ and $C$) therefore yields $B\cdot\operatorname*{adj}%
B\cdot\operatorname*{adj}A=B\cdot\operatorname*{adj}\left(  AB\right)  $
(since the element $\det A$ of $\mathbb{K}$ is regular). Thus, Lemma
\ref{lem.regular.det.cancel} \textbf{(a)} (applied to $n$, $B$,
$\operatorname*{adj}B\cdot\operatorname*{adj}A$ and $\operatorname*{adj}%
\left(  AB\right)  $ instead of $m$, $A$, $B$ and $C$) yields
$\operatorname*{adj}B\cdot\operatorname*{adj}A=\operatorname*{adj}\left(
AB\right)  $ (since the element $\det B$ of $\mathbb{K}$ is regular). This
proves Lemma \ref{lem.adj.adj(AB).reg}.
\end{proof}

We now derive Theorem \ref{thm.adj.adj(AB)} from this lemma:

\begin{vershort}
\begin{proof}
[Proof of Theorem \ref{thm.adj.adj(AB)}.]Define the $\mathbb{K}$-algebra
homomorphism $\varepsilon:\mathbb{K}\left[  t\right]  \rightarrow\mathbb{K}$
as in Proposition \ref{prop.tInA.ev}.

Define two matrices $\widetilde{A}$ and $\widetilde{B}$ in $\left(
\mathbb{K}\left[  t\right]  \right)  ^{n\times n}$ by $\widetilde{A}=tI_{n}+A$
and $\widetilde{B}=tI_{n}+B$.

From $\widetilde{A}=tI_{n}+A$, we obtain $\varepsilon^{n\times n}\left(
\operatorname*{adj}\widetilde{A}\right)  =\varepsilon^{n\times n}\left(
\operatorname*{adj}\left(  tI_{n}+A\right)  \right)  =\operatorname*{adj}A$
(by Proposition \ref{prop.tInA.ev} \textbf{(b)}). Similarly, $\varepsilon
^{n\times n}\left(  \operatorname*{adj}\widetilde{B}\right)
=\operatorname*{adj}B$.

From $\widetilde{A}=tI_{n}+A$, we obtain $\varepsilon^{n\times n}\left(
\widetilde{A}\right)  =\varepsilon^{n\times n}\left(  tI_{n}+A\right)  =A$ (by
Proposition \ref{prop.tInA.ev} \textbf{(c)}). Similarly, $\varepsilon^{n\times
n}\left(  \widetilde{B}\right)  =B$.

Corollary \ref{cor.tInA.reg} shows that the element $\det\left(
tI_{n}+A\right)  $ of $\mathbb{K}\left[  t\right]  $ is regular. In other
words, the element $\det\widetilde{A}$ of $\mathbb{K}\left[  t\right]  $ is
regular (since $\widetilde{A}=tI_{n}+A$). Similarly, the element
$\det\widetilde{B}$ of $\mathbb{K}\left[  t\right]  $ is regular. Lemma
\ref{lem.adj.adj(AB).reg} (applied to $\mathbb{K}\left[  t\right]  $,
$\widetilde{A}$ and $\widetilde{B}$ instead of $\mathbb{K}$, $A$ and $B$) thus
yields%
\[
\operatorname*{adj}\left(  \widetilde{A}\widetilde{B}\right)
=\operatorname*{adj}\widetilde{B}\cdot\operatorname*{adj}\widetilde{A}.
\]
Applying the map $\varepsilon^{n\times n}$ to both sides of this equality, we
obtain%
\begin{align*}
\varepsilon^{n\times n}\left(  \operatorname*{adj}\left(  \widetilde{A}%
\widetilde{B}\right)  \right)   &  =\varepsilon^{n\times n}\left(
\operatorname*{adj}\widetilde{B}\cdot\operatorname*{adj}\widetilde{A}\right)
=\underbrace{\varepsilon^{n\times n}\left(  \operatorname*{adj}\widetilde{B}%
\right)  }_{\substack{=\operatorname*{adj}B}}\cdot\underbrace{\varepsilon
^{n\times n}\left(  \operatorname*{adj}\widetilde{A}\right)  }%
_{=\operatorname*{adj}A}\\
&  \ \ \ \ \ \ \ \ \ \ \ \ \ \ \ \ \ \ \ \ \left(
\begin{array}
[c]{c}%
\text{by Proposition \ref{prop.functor.clas} \textbf{(c)}, applied to}\\
\mathbb{K}\left[  t\right]  \text{, }\mathbb{K}\text{, }\varepsilon\text{,
}n\text{, }n\text{, }\operatorname*{adj}\widetilde{B}\text{ and }%
\operatorname*{adj}\widetilde{A}\\
\text{instead of }\mathbb{L}\text{, }\mathbb{M}\text{, }f\text{, }m\text{,
}p\text{, }A\text{ and }B
\end{array}
\right) \\
&  =\operatorname*{adj}B\cdot\operatorname*{adj}A.
\end{align*}
Hence,%
\begin{equation}
\operatorname*{adj}B\cdot\operatorname*{adj}A=\varepsilon^{n\times n}\left(
\operatorname*{adj}\left(  \widetilde{A}\widetilde{B}\right)  \right)
=\operatorname*{adj}\left(  \varepsilon^{n\times n}\left(  \widetilde{A}%
\widetilde{B}\right)  \right)  \label{pf.thm.adj.adj(AB).short.1}%
\end{equation}
(by Proposition \ref{prop.functor.det} \textbf{(c)}, applied to $\mathbb{K}%
\left[  t\right]  $, $\mathbb{K}$, $\varepsilon$ and $\widetilde{A}%
\widetilde{B}$ instead of $\mathbb{L}$, $\mathbb{M}$, $f$ and $A$).

But Proposition \ref{prop.functor.clas} \textbf{(c)} (applied to
$\mathbb{K}\left[  t\right]  $, $\mathbb{K}$, $\varepsilon$, $n$, $n$,
$\widetilde{A}$ and $\widetilde{B}$ instead of $\mathbb{L}$, $\mathbb{M}$,
$f$, $m$, $p$, $A$ and $B$) shows that%
\[
\varepsilon^{n\times n}\left(  \widetilde{A}\widetilde{B}\right)
=\underbrace{\varepsilon^{n\times n}\left(  \widetilde{A}\right)  }_{=A}%
\cdot\underbrace{\varepsilon^{n\times n}\left(  \widetilde{B}\right)  }%
_{=B}=AB.
\]
Hence, (\ref{pf.thm.adj.adj(AB).short.1}) becomes%
\[
\operatorname*{adj}B\cdot\operatorname*{adj}A=\operatorname*{adj}\left(
\underbrace{\varepsilon^{n\times n}\left(  \widetilde{A}\widetilde{B}\right)
}_{=AB}\right)  =\operatorname*{adj}\left(  AB\right)  .
\]
This proves Theorem \ref{thm.adj.adj(AB)}.
\end{proof}
\end{vershort}

\begin{verlong}
\begin{proof}
[Proof of Theorem \ref{thm.adj.adj(AB)}.]Define the $\mathbb{K}$-algebra
homomorphism $\varepsilon:\mathbb{K}\left[  t\right]  \rightarrow\mathbb{K}$
as in Proposition \ref{prop.tInA.ev}.

Consider the matrices $tI_{n}+A$ and $tI_{n}+B$ in $\left(  \mathbb{K}\left[
t\right]  \right)  ^{n\times n}$. Corollary \ref{cor.tInA.reg} shows that the
element $\det\left(  tI_{n}+A\right)  $ of $\mathbb{K}\left[  t\right]  $ is
regular. Corollary \ref{cor.tInA.reg} (applied to $B$ instead of $A$) shows
that the element $\det\left(  tI_{n}+B\right)  $ of $\mathbb{K}\left[
t\right]  $ is regular. Lemma \ref{lem.adj.adj(AB).reg} (applied to
$\mathbb{K}\left[  t\right]  $, $tI_{n}+A$ and $tI_{n}+B$ instead of
$\mathbb{K}$, $A$ and $B$) thus yields%
\[
\operatorname*{adj}\left(  \left(  tI_{n}+A\right)  \left(  tI_{n}+B\right)
\right)  =\operatorname*{adj}\left(  tI_{n}+B\right)  \cdot\operatorname*{adj}%
\left(  tI_{n}+A\right)  .
\]
Applying the map $\varepsilon^{n\times n}$ to both sides of this equality, we
obtain%
\begin{align*}
&  \varepsilon^{n\times n}\left(  \operatorname*{adj}\left(  \left(
tI_{n}+A\right)  \left(  tI_{n}+B\right)  \right)  \right) \\
&  =\varepsilon^{n\times n}\left(  \operatorname*{adj}\left(  tI_{n}+B\right)
\cdot\operatorname*{adj}\left(  tI_{n}+A\right)  \right) \\
&  =\underbrace{\varepsilon^{n\times n}\left(  \operatorname*{adj}\left(
tI_{n}+B\right)  \right)  }_{\substack{=\operatorname*{adj}B\\\text{(by
Proposition \ref{prop.tInA.ev} \textbf{(b)}}\\\text{(applied to }B\text{
instead of }A\text{))}}}\cdot\underbrace{\varepsilon^{n\times n}\left(
\operatorname*{adj}\left(  tI_{n}+A\right)  \right)  }%
_{\substack{=\operatorname*{adj}A\\\text{(by Proposition \ref{prop.tInA.ev}
\textbf{(b)})}}}\\
&  \ \ \ \ \ \ \ \ \ \ \ \ \ \ \ \ \ \ \ \ \left(
\begin{array}
[c]{c}%
\text{by Proposition \ref{prop.functor.clas} \textbf{(c)} (applied to}\\
\mathbb{K}\left[  t\right]  \text{, }\mathbb{K}\text{, }\varepsilon\text{,
}n\text{, }n\text{, }\operatorname*{adj}\left(  tI_{n}+B\right)  \text{ and
}\operatorname*{adj}\left(  tI_{n}+A\right) \\
\text{instead of }\mathbb{L}\text{, }\mathbb{M}\text{, }f\text{, }m\text{,
}p\text{, }A\text{ and }B\text{)}%
\end{array}
\right) \\
&  =\operatorname*{adj}B\cdot\operatorname*{adj}A.
\end{align*}
Hence,%
\begin{align}
\operatorname*{adj}B\cdot\operatorname*{adj}A  &  =\varepsilon^{n\times
n}\left(  \operatorname*{adj}\left(  \left(  tI_{n}+A\right)  \left(
tI_{n}+B\right)  \right)  \right) \nonumber\\
&  =\operatorname*{adj}\left(  \varepsilon^{n\times n}\left(  \left(
tI_{n}+A\right)  \left(  tI_{n}+B\right)  \right)  \right)
\label{pf.thm.adj.adj(AB).1}%
\end{align}
(by Proposition \ref{prop.functor.det} \textbf{(c)} (applied to $\mathbb{K}%
\left[  t\right]  $, $\mathbb{K}$, $\varepsilon$ and $\left(  tI_{n}+A\right)
\left(  tI_{n}+B\right)  $ instead of $\mathbb{L}$, $\mathbb{M}$, $f$ and $A$)).

But Proposition \ref{prop.functor.clas} \textbf{(c)} (applied to
$\mathbb{K}\left[  t\right]  $, $\mathbb{K}$, $\varepsilon$, $n$, $n$,
$tI_{n}+A$ and $tI_{n}+B$ instead of $\mathbb{L}$, $\mathbb{M}$, $f$, $m$,
$p$, $A$ and $B$) shows that%
\[
\varepsilon^{n\times n}\left(  \left(  tI_{n}+A\right)  \left(  tI_{n}%
+B\right)  \right)  =\underbrace{\varepsilon^{n\times n}\left(  tI_{n}%
+A\right)  }_{\substack{=A\\\text{(by Proposition \ref{prop.tInA.ev}
\textbf{(c)})}}}\cdot\underbrace{\varepsilon^{n\times n}\left(  tI_{n}%
+B\right)  }_{\substack{=B\\\text{(by Proposition \ref{prop.tInA.ev}
\textbf{(c)}}\\\text{(applied to }B\text{ instead of }A\text{))}}}=AB.
\]
Hence, (\ref{pf.thm.adj.adj(AB).1}) becomes%
\[
\operatorname*{adj}B\cdot\operatorname*{adj}A=\operatorname*{adj}\left(
\underbrace{\varepsilon^{n\times n}\left(  \left(  tI_{n}+A\right)  \left(
tI_{n}+B\right)  \right)  }_{=AB}\right)  =\operatorname*{adj}\left(
AB\right)  .
\]
This proves Theorem \ref{thm.adj.adj(AB)}.
\end{proof}
\end{verlong}

\subsection{Determinant and adjugate of an adjugate}

Our next target is the following result:

\begin{theorem}
\label{thm.adj.adjadj}Let $n\in\mathbb{N}$. Let $A$ be an $n\times n$-matrix.
\medskip

\textbf{(a)} If $n\geq1$, then $\det\left(  \operatorname*{adj}A\right)
=\left(  \det A\right)  ^{n-1}$. \medskip

\textbf{(b)} If $n\geq2$, then $\operatorname*{adj}\left(  \operatorname*{adj}%
A\right)  =\left(  \det A\right)  ^{n-2}A$.
\end{theorem}

Again, we shall first prove it in a particular case:

\begin{lemma}
\label{lem.adj.adjadj.reg}Let $n\in\mathbb{N}$. Let $A$ be an $n\times
n$-matrix. Assume that the element $\det A$ of $\mathbb{K}$ is regular.
\medskip

\textbf{(a)} If $n\geq1$, then $\det\left(  \operatorname*{adj}A\right)
=\left(  \det A\right)  ^{n-1}$. \medskip

\textbf{(b)} If $n\geq2$, then $\operatorname*{adj}\left(  \operatorname*{adj}%
A\right)  =\left(  \det A\right)  ^{n-2}A$.
\end{lemma}

Before we start proving Lemma \ref{lem.adj.adjadj.reg}, let us first recall
the following fact: If $n\in\mathbb{N}$, $\lambda\in\mathbb{K}$ and
$C\in\mathbb{K}^{n\times n}$, then%
\begin{equation}
\det\left(  \lambda C\right)  =\lambda^{n}\det C.
\label{pf.lem.adj.adjadj.reg.lam1}%
\end{equation}
(In fact, this is precisely \cite[Proposition 6.12]{detnotes} (applied to $C$
instead of $A$).)

\begin{proof}
[Proof of Lemma \ref{lem.adj.adjadj.reg}.]Theorem \ref{thm.adj.inverse} yields%
\[
A\cdot\operatorname*{adj}A=\operatorname*{adj}A\cdot A=\det A\cdot I_{n}.
\]

\textbf{(a)} Assume that $n\geq1$. Now,%
\begin{align*}
\det\left(  \underbrace{A\cdot\operatorname*{adj}A}_{=\det A\cdot I_{n}%
}\right)   &  =\det\left(  \det A\cdot I_{n}\right)  =\left(  \det A\right)
^{n}\underbrace{\det\left(  I_{n}\right)  }_{=1}\\
&  \ \ \ \ \ \ \ \ \ \ \left(  \text{by (\ref{pf.lem.adj.adjadj.reg.lam1})
(applied to }\det A\text{ and }I_{n}\text{ instead of }\lambda\text{ and
}C\text{)}\right) \\
&  =\left(  \det A\right)  ^{n}=\det A\cdot\left(  \det A\right)  ^{n-1}.
\end{align*}
Thus,%
\[
\det A\cdot\left(  \det A\right)  ^{n-1}=\det\left(  A\cdot\operatorname*{adj}%
A\right)  =\det A\cdot\det\left(  \operatorname*{adj}A\right)  .
\]
Hence, Lemma \ref{lem.regular.cancel} (applied to $\mathbb{A}=\mathbb{K}$,
$a=\det A$, $b=\left(  \det A\right)  ^{n-1}$ and $c=\det\left(
\operatorname*{adj}A\right)  $) yields $\left(  \det A\right)  ^{n-1}%
=\det\left(  \operatorname*{adj}A\right)  $ (since $\det A$ is a regular
element of $\mathbb{K}$). This proves Lemma \ref{lem.adj.adjadj.reg}
\textbf{(a)}. \medskip

\textbf{(b)} Assume that $n\geq2$. Thus, $n-1\geq1$ and $n\geq2\geq1$. Now,
Lemma \ref{lem.adj.adjadj.reg} \textbf{(a)} yields
\begin{align*}
\det\left(  \operatorname*{adj}A\right)   &  =\left(  \det A\right)
^{n-1}=\det A\cdot\underbrace{\left(  \det A\right)  ^{\left(  n-1\right)
-1}}_{=\left(  \det A\right)  ^{n-2}}\ \ \ \ \ \ \ \ \ \ \left(  \text{since
}n-1\geq1\right) \\
&  =\det A\cdot\left(  \det A\right)  ^{n-2}.
\end{align*}

But Theorem \ref{thm.adj.inverse} (applied to $\operatorname*{adj}A$ instead
of $A$) yields%
\[
\operatorname*{adj}A\cdot\operatorname*{adj}\left(  \operatorname*{adj}%
A\right)  =\operatorname*{adj}\left(  \operatorname*{adj}A\right)
\cdot\operatorname*{adj}A=\det\left(  \operatorname*{adj}A\right)  \cdot
I_{n}.
\]
Now,%
\begin{align*}
A\cdot\underbrace{\operatorname*{adj}A\cdot\operatorname*{adj}\left(
\operatorname*{adj}A\right)  }_{=\det\left(  \operatorname*{adj}A\right)
\cdot I_{n}}  &  =A\cdot\det\left(  \operatorname*{adj}A\right)  \cdot
I_{n}=\underbrace{\det\left(  \operatorname*{adj}A\right)  }_{=\det
A\cdot\left(  \det A\right)  ^{n-2}}A\\
&  =\det A\cdot\left(  \det A\right)  ^{n-2}A.
\end{align*}
Hence,%
\begin{align*}
\det A\cdot\left(  \det A\right)  ^{n-2}A  &  =\underbrace{A\cdot
\operatorname*{adj}A}_{=\det A\cdot I_{n}}\cdot\operatorname*{adj}\left(
\operatorname*{adj}A\right)  =\det A\cdot I_{n}\cdot\operatorname*{adj}\left(
\operatorname*{adj}A\right) \\
&  =\det A\cdot\operatorname*{adj}\left(  \operatorname*{adj}A\right)  .
\end{align*}
Hence, Lemma \ref{lem.regular.mat-cancel} (applied to $n$, $\det A$, $\left(
\det A\right)  ^{n-2}A$ and $\operatorname*{adj}\left(  \operatorname*{adj}%
A\right)  $ instead of $m$, $a$, $B$ and $C$) yields $\left(  \det A\right)
^{n-2}A=\operatorname*{adj}\left(  \operatorname*{adj}A\right)  $ (since $\det
A$ is a regular element of $\mathbb{K}$). This proves Lemma
\ref{lem.adj.adjadj.reg} \textbf{(b)}.
\end{proof}

Let us now derive Theorem \ref{thm.adj.adjadj} from this lemma:

\begin{proof}
[Proof of Theorem \ref{thm.adj.adjadj}.]Define the $\mathbb{K}$-algebra
homomorphism $\varepsilon:\mathbb{K}\left[  t\right]  \rightarrow\mathbb{K}$
as in Proposition \ref{prop.tInA.ev}.

Define a matrix $\widetilde{A}\in\left(  \mathbb{K}\left[  t\right]  \right)
^{n\times n}$ by $\widetilde{A}=tI_{n}+A$. Corollary \ref{cor.tInA.reg} shows
that the element $\det\left(  tI_{n}+A\right)  $ of $\mathbb{K}\left[
t\right]  $ is regular. In other words, the element $\det\widetilde{A}$ of
$\mathbb{K}\left[  t\right]  $ is regular (since $\widetilde{A}=tI_{n}+A$).

\begin{vershort}
From $\widetilde{A}=tI_{n}+A$, we obtain $\varepsilon^{n\times n}\left(
\operatorname*{adj}\widetilde{A}\right)  =\varepsilon^{n\times n}\left(
\operatorname*{adj}\left(  tI_{n}+A\right)  \right)  =\operatorname*{adj}A$
(by Proposition \ref{prop.tInA.ev} \textbf{(b)}) and $\varepsilon\left(
\det\widetilde{A}\right)  =\varepsilon\left(  \det\left(  tI_{n}+A\right)
\right)  =\det A$ (by Proposition \ref{prop.tInA.ev} \textbf{(a)}). \medskip
\end{vershort}

\begin{verlong}
We have $\varepsilon^{n\times n}\left(  \operatorname*{adj}%
\underbrace{\widetilde{A}}_{=tI_{n}+A}\right)  =\varepsilon^{n\times n}\left(
\operatorname*{adj}\left(  tI_{n}+A\right)  \right)  =\operatorname*{adj}A$
(by Proposition \ref{prop.tInA.ev} \textbf{(b)}). Also, $\varepsilon\left(
\det\underbrace{\widetilde{A}}_{=tI_{n}+A}\right)  =\varepsilon\left(
\det\left(  tI_{n}+A\right)  \right)  =\det A$ (by Proposition
\ref{prop.tInA.ev} \textbf{(a)}). \medskip
\end{verlong}

\textbf{(a)} Assume that $n\geq1$. Lemma \ref{lem.adj.adjadj.reg} \textbf{(a)}
(applied to $\mathbb{K}\left[  t\right]  $ and $\widetilde{A}$ instead of
$\mathbb{K}$ and $A$) yields $\det\left(  \operatorname*{adj}\widetilde{A}%
\right)  =\left(  \det\widetilde{A}\right)  ^{n-1}$.

Now, Proposition \ref{prop.functor.det} \textbf{(a)} (applied to
$\mathbb{K}\left[  t\right]  $, $\mathbb{K}$, $\varepsilon$ and
$\operatorname*{adj}\widetilde{A}$ instead of $\mathbb{L}$, $\mathbb{M}$, $f$
and $A$) yields%
\[
\varepsilon\left(  \det\left(  \operatorname*{adj}\widetilde{A}\right)
\right)  =\det\left(  \underbrace{\varepsilon^{n\times n}\left(
\operatorname*{adj}\widetilde{A}\right)  }_{=\operatorname*{adj}A}\right)
=\det\left(  \operatorname*{adj}A\right)  .
\]
Hence,%
\begin{align*}
\det\left(  \operatorname*{adj}A\right)   &  =\varepsilon\left(
\underbrace{\det\left(  \operatorname*{adj}\widetilde{A}\right)  }_{=\left(
\det\widetilde{A}\right)  ^{n-1}}\right)  =\varepsilon\left(  \left(
\det\widetilde{A}\right)  ^{n-1}\right)  =\left(  \underbrace{\varepsilon
\left(  \det\widetilde{A}\right)  }_{=\det A}\right)  ^{n-1}\\
&  \ \ \ \ \ \ \ \ \ \ \ \ \ \ \ \ \ \ \ \ \left(  \text{since }%
\varepsilon\text{ is a }\mathbb{K}\text{-algebra homomorphism}\right) \\
&  =\left(  \det A\right)  ^{n-1}.
\end{align*}
This proves Theorem \ref{thm.adj.adjadj} \textbf{(a)}. \medskip

\textbf{(b)} Assume that $n\geq2$. Lemma \ref{lem.adj.adjadj.reg} \textbf{(b)}
(applied to $\mathbb{K}\left[  t\right]  $ and $\widetilde{A}$ instead of
$\mathbb{K}$ and $A$) yields $\operatorname*{adj}\left(  \operatorname*{adj}%
\widetilde{A}\right)  =\left(  \det\widetilde{A}\right)  ^{n-2}\widetilde{A}$.
We have $\varepsilon^{n\times n}\left(  \underbrace{\widetilde{A}}_{=tI_{n}%
+A}\right)  =\varepsilon^{n\times n}\left(  tI_{n}+A\right)  =A$ (by
Proposition \ref{prop.tInA.ev} \textbf{(c)}). Proposition
\ref{prop.functor.det} \textbf{(c)} (applied to $\mathbb{K}\left[  t\right]
$, $\mathbb{K}$, $\varepsilon$ and $\operatorname*{adj}\widetilde{A}$ instead
of $\mathbb{L}$, $\mathbb{M}$, $f$ and $A$) yields%
\[
\varepsilon^{n\times n}\left(  \operatorname*{adj}\left(  \operatorname*{adj}%
\widetilde{A}\right)  \right)  =\operatorname*{adj}\left(
\underbrace{\varepsilon^{n\times n}\left(  \operatorname*{adj}\widetilde{A}%
\right)  }_{=\operatorname*{adj}A}\right)  =\operatorname*{adj}\left(
\operatorname*{adj}A\right)  .
\]
Thus,%
\begin{align*}
\operatorname*{adj}\left(  \operatorname*{adj}A\right)   &  =\varepsilon
^{n\times n}\left(  \underbrace{\operatorname*{adj}\left(  \operatorname*{adj}%
\widetilde{A}\right)  }_{=\left(  \det\widetilde{A}\right)  ^{n-2}%
\widetilde{A}}\right)  =\varepsilon^{n\times n}\left(  \left(  \det
\widetilde{A}\right)  ^{n-2}\widetilde{A}\right) \\
&  =\underbrace{\varepsilon\left(  \left(  \det\widetilde{A}\right)
^{n-2}\right)  }_{\substack{=\left(  \varepsilon\left(  \det\widetilde{A}%
\right)  \right)  ^{n-2}\\\text{(since }\varepsilon\text{ is a }%
\mathbb{K}\text{-algebra}\\\text{homomorphism)}}}\underbrace{\varepsilon
^{n\times n}\left(  \widetilde{A}\right)  }_{=A}\\
&  \ \ \ \ \ \ \ \ \ \ \left(
\begin{array}
[c]{c}%
\text{by Proposition \ref{prop.functor.clas} \textbf{(d)} (applied to
}\mathbb{K}\left[  t\right]  \text{, }\mathbb{K}\text{, }\varepsilon\text{,
}n\text{, }\widetilde{A}\\
\text{and }\left(  \det\widetilde{A}\right)  ^{n-2}\text{ instead of
}\mathbb{L}\text{, }\mathbb{M}\text{, }f\text{, }m\text{, }A\text{ and
}\lambda\text{)}%
\end{array}
\right) \\
&  =\left(  \underbrace{\varepsilon\left(  \det\widetilde{A}\right)  }_{=\det
A}\right)  ^{n-2}A=\left(  \det A\right)  ^{n-2}A.
\end{align*}
This proves Theorem \ref{thm.adj.adjadj} \textbf{(b)}.
\end{proof}

\subsection{The adjugate of $A$ as a polynomial in $A$}

Next, let us show that the adjugate of a square matrix $A$ is a polynomial in
$A$ (with coefficients that depend on $A$, but are scalars -- not matrices):

\begin{theorem}
\label{thm.adj.poly}Let $n\in\mathbb{N}$. Let $A\in\mathbb{K}^{n\times n}$.
For every $j\in\mathbb{Z}$, define an element $c_{j}\in\mathbb{K}$ by
$c_{j}=\left[  t^{n-j}\right]  \chi_{A}$. Then,%
\[
\operatorname*{adj}A=\left(  -1\right)  ^{n-1}\sum_{i=0}^{n-1}c_{n-1-i}A^{i}.
\]

\end{theorem}

One consequence of Theorem \ref{thm.adj.poly} is that every $n\times n$-matrix
which commutes with a given $n\times n$-matrix $A$ must also commute with
$\operatorname*{adj}A$.

We prepare for the proof of Theorem \ref{thm.adj.poly} with two really simple facts:

\begin{lemma}
\label{lem.adj.lam}Let $n\in\mathbb{N}$. Let $u$ and $v$ be two elements of
$\left\{  1,2,\ldots,n\right\}  $. Let $\lambda\in\mathbb{K}$. Let $A$ be an
$n\times n$-matrix. Then,%
\[
\left(  \lambda A\right)  _{\sim u,\sim v}=\lambda A_{\sim u,\sim v}.
\]

\end{lemma}

\begin{vershort}
\begin{proof}
[Proof of Lemma \ref{lem.adj.lam}.]This follows from Lemma
\ref{lem.adj.linear} (applied to $\mu=0$ and $B=A$).
\end{proof}
\end{vershort}

\begin{verlong}
\begin{proof}
[Proof of Lemma \ref{lem.adj.lam}.]Lemma \ref{lem.adj.linear} (applied to
$\mu=0$ and $B=A$) yields%
\[
\left(  \lambda A+0A\right)  _{\sim u,\sim v}=\lambda A_{\sim u,\sim
v}+\underbrace{0A_{\sim u,\sim v}}_{=0_{\left(  n-1\right)  \times\left(
n-1\right)  }}=\lambda A_{\sim u,\sim v}.
\]
Since $\lambda A+\underbrace{0A}_{=0_{n\times n}}=\lambda A$, this rewrites as
$\left(  \lambda A\right)  _{\sim u,\sim v}=\lambda A_{\sim u,\sim v}$. This
proves Lemma \ref{lem.adj.lam}.
\end{proof}
\end{verlong}

\begin{proposition}
\label{prop.adj.lamA}Let $n$ be a positive integer. Let $A\in\mathbb{K}%
^{n\times n}$ and $\lambda\in\mathbb{K}$. Then, $\operatorname*{adj}\left(
\lambda A\right)  =\lambda^{n-1}\operatorname*{adj}A$.
\end{proposition}

\begin{vershort}
\begin{proof}
[Proof of Proposition \ref{prop.adj.lamA}.]Recalling the definitions of
$\operatorname*{adj}\left(  \lambda A\right)  $ and $\operatorname*{adj}A$
(and using Lemma \ref{lem.adj.lam}), the reader can easily reduce Proposition
\ref{prop.adj.lamA} to (\ref{pf.lem.adj.adjadj.reg.lam1}) (applied to $n-1$
and $A_{\sim j,\sim i}$ instead of $n$ and $C$).
\end{proof}
\end{vershort}

\begin{verlong}
\begin{proof}
[Proof of Proposition \ref{prop.adj.lamA}.]Every $\left(  i,j\right)
\in\left\{  1,2,\ldots,n\right\}  ^{2}$ satisfies
\begin{equation}
\det\left(  \left(  \lambda A\right)  _{\sim j,\sim i}\right)  =\lambda
^{n-1}\det\left(  A_{\sim j,\sim i}\right)  \label{pf.prop.adj.lamA.1}%
\end{equation}
\footnote{\textit{Proof of (\ref{pf.prop.adj.lamA.1}):} Let $\left(
i,j\right)  \in\left\{  1,2,\ldots,n\right\}  ^{2}$. Thus, $i$ and $j$ are two
elements of $\left\{  1,2,\ldots,n\right\}  $. Hence, Lemma \ref{lem.adj.lam}
(applied to $u=j$ and $v=i$) yields $\left(  \lambda A\right)  _{\sim j,\sim
i}=\lambda A_{\sim j,\sim i}$. But $i\in\left\{  1,2,\ldots,n\right\}  $, so
that $1\leq i\leq n$ and thus $n\geq1$. Hence, $n-1\in\mathbb{N}$. Also,
$A_{\sim j,\sim i}\in\mathbb{K}^{\left(  n-1\right)  \times\left(  n-1\right)
}$. Hence, (\ref{pf.lem.adj.adjadj.reg.lam1}) (applied to $n-1$ and $A_{\sim
j,\sim i}$ instead of $n$ and $C$) yields $\det\left(  \lambda A_{\sim j,\sim
i}\right)  =\lambda^{n-1}\det\left(  A_{\sim j,\sim i}\right)  $. Now,%
\[
\det\left(  \underbrace{\left(  \lambda A\right)  _{\sim j,\sim i}}_{=\lambda
A_{\sim j,\sim i}}\right)  =\det\left(  \lambda A_{\sim j,\sim i}\right)
=\lambda^{n-1}\det\left(  A_{\sim j,\sim i}\right)  .
\]
This proves (\ref{pf.prop.adj.lamA.1}).}. Now, the definition of
$\operatorname*{adj}A$ yields $\operatorname*{adj}A=\left(  \left(  -1\right)
^{i+j}\det\left(  A_{\sim j,\sim i}\right)  \right)  _{1\leq i\leq n,\ 1\leq
j\leq n}$. Multiplying both sides of this equality by $\lambda^{n-1}$, we
obtain%
\begin{align}
\lambda^{n-1}\operatorname*{adj}A  &  =\lambda^{n-1}\left(  \left(  -1\right)
^{i+j}\det\left(  A_{\sim j,\sim i}\right)  \right)  _{1\leq i\leq n,\ 1\leq
j\leq n}\nonumber\\
&  =\left(  \underbrace{\lambda^{n-1}\left(  -1\right)  ^{i+j}}_{=\left(
-1\right)  ^{i+j}\lambda^{n-1}}\det\left(  A_{\sim j,\sim i}\right)  \right)
_{1\leq i\leq n,\ 1\leq j\leq n}\nonumber\\
&  =\left(  \left(  -1\right)  ^{i+j}\lambda^{n-1}\det\left(  A_{\sim j,\sim
i}\right)  \right)  _{1\leq i\leq n,\ 1\leq j\leq n}.
\label{pf.prop.adj.lamA.4}%
\end{align}

On the other hand, the definition of $\operatorname*{adj}\left(  \lambda
A\right)  $ yields%
\begin{align*}
\operatorname*{adj}\left(  \lambda A\right)   &  =\left(  \left(  -1\right)
^{i+j}\underbrace{\det\left(  \left(  \lambda A\right)  _{\sim j,\sim
i}\right)  }_{\substack{=\lambda^{n-1}\det\left(  A_{\sim j,\sim i}\right)
\\\text{(by (\ref{pf.prop.adj.lamA.1}))}}}\right)  _{1\leq i\leq n,\ 1\leq
j\leq n}\\
&  =\left(  \left(  -1\right)  ^{i+j}\lambda^{n-1}\det\left(  A_{\sim j,\sim
i}\right)  \right)  _{1\leq i\leq n,\ 1\leq j\leq n}=\lambda^{n-1}%
\operatorname*{adj}A
\end{align*}
(by (\ref{pf.prop.adj.lamA.4})). This proves Proposition \ref{prop.adj.lamA}.
\end{proof}
\end{verlong}

Now, let me show a slightly simpler variant of Theorem \ref{thm.adj.poly}:

\begin{lemma}
\label{lem.adj.poly.-A}Let $n$ be a positive integer. Let $A\in\mathbb{K}%
^{n\times n}$. For every $j\in\mathbb{Z}$, define an element $c_{j}%
\in\mathbb{K}$ by $c_{j}=\left[  t^{n-j}\right]  \chi_{A}$. Then,%
\[
\operatorname*{adj}\left(  -A\right)  =\sum_{i=0}^{n-1}c_{n-1-i}A^{i}.
\]

\end{lemma}

\begin{vershort}
\begin{proof}
[Proof of Lemma \ref{lem.adj.poly.-A}.]For every $j\in\mathbb{Z}$, define an
element $c_{j}\in\mathbb{K}$ by $c_{j}=\left[  t^{n-j}\right]  \chi_{A}$.

Proposition \ref{prop.adj.poly} shows that there exist $n$ matrices
$D_{0},D_{1},\ldots,D_{n-1}$ in $\mathbb{K}^{n\times n}$ such that%
\begin{equation}
\operatorname*{adj}\left(  tI_{n}-A\right)  =\sum_{k=0}^{n-1}t^{k}%
D_{k}\ \ \ \ \ \ \ \ \ \ \text{in }\left(  \mathbb{K}\left[  t\right]
\right)  ^{n\times n}. \label{pf.lem.adj.poly.-A.short.1}%
\end{equation}
Consider these $D_{0},D_{1},\ldots,D_{n-1}$. Thus, an $n$-tuple $\left(
D_{0},D_{1},\ldots,D_{n-1}\right)  $ of matrices in $\mathbb{K}^{n\times n}$
is defined. Extend this $n$-tuple to a family $\left(  D_{k}\right)
_{k\in\mathbb{Z}}$ of matrices in $\mathbb{K}^{n\times n}$ by setting
(\ref{eq.lem.CH.D.0}). Lemma \ref{lem.CH.D} \textbf{(c)} (applied to $k=n-1$)
yields%
\begin{equation}
\sum_{i=0}^{n-1}c_{n-1-i}A^{i}=D_{n-1-\left(  n-1\right)  }=D_{0}.
\label{pf.lem.adj.poly.-A.short.2}%
\end{equation}

On the other hand, define the $\mathbb{K}$-algebra homomorphism $\varepsilon
:\mathbb{K}\left[  t\right]  \rightarrow\mathbb{K}$ as in Proposition
\ref{prop.tInA.ev}. This homomorphism $\varepsilon$ satisfies $\varepsilon
\left(  t\right)  =0$. Also, it satisfies $\varepsilon\left(  u\right)  =u$
for every $u\in\mathbb{K}$. Hence, the map $\varepsilon^{n\times n}:\left(
\mathbb{K}\left[  t\right]  \right)  ^{n\times n}\rightarrow\mathbb{K}%
^{n\times n}$ (defined as in Definition \ref{def.f-mat}) satisfies%
\begin{equation}
\varepsilon^{n\times n}\left(  F\right)  =F\ \ \ \ \ \ \ \ \ \ \text{for every
}F\in\mathbb{K}^{n\times n}. \label{pf.lem.adj.poly.-A.short.eF}%
\end{equation}

But Proposition \ref{prop.functor.clas} \textbf{(a)} (applied to
$\mathbb{L}=\mathbb{K}\left[  t\right]  $, $\mathbb{M}=\mathbb{K}$,
$f=\varepsilon$ and $m=n$) yields that the map $\varepsilon^{n\times
n}:\left(  \mathbb{K}\left[  t\right]  \right)  ^{n\times n}\rightarrow
\mathbb{K}^{n\times n}$ is a homomorphism of additive groups. Hence,%
\begin{align}
\varepsilon^{n\times n}\left(  \sum_{k=0}^{n-1}t^{k}D_{k}\right)   &
=\sum_{k=0}^{n-1}\underbrace{\varepsilon^{n\times n}\left(  t^{k}D_{k}\right)
}_{\substack{=\varepsilon\left(  t^{k}\right)  \varepsilon^{n\times n}\left(
D_{k}\right)  \\\text{(by Proposition \ref{prop.functor.clas} \textbf{(d)}
(applied to}\\\mathbb{K}\left[  t\right]  \text{, }\mathbb{K}\text{,
}\varepsilon\text{, }n\text{, }D_{k}\text{ and }t^{k}\\\text{instead of
}\mathbb{L}\text{, }\mathbb{M}\text{, }f\text{, }m\text{, }A\text{ and
}\lambda\text{))}}}\nonumber\\
&  =\sum_{k=0}^{n-1}\underbrace{\varepsilon\left(  t^{k}\right)
}_{\substack{=\left(  \varepsilon\left(  t\right)  \right)  ^{k}\\\text{(since
}\varepsilon\text{ is a ring}\\\text{homomorphism)}}%
}\ \ \underbrace{\varepsilon^{n\times n}\left(  D_{k}\right)  }%
_{\substack{=D_{k}\\\text{(by (\ref{pf.lem.adj.poly.-A.short.eF}%
)}\\\text{(applied to }F=D_{k}\text{))}}}\nonumber\\
&  =\sum_{k=0}^{n-1}\left(  \underbrace{\varepsilon\left(  t\right)  }%
_{=0}\right)  ^{k}D_{k}=\sum_{k=0}^{n-1}0^{k}D_{k}=\underbrace{0^{0}}%
_{=1}D_{0}+\sum_{k=1}^{n-1}\underbrace{0^{k}}_{\substack{=0\\\text{(since
}k\geq1\text{)}}}D_{k}\nonumber\\
&  \ \ \ \ \ \ \ \ \ \ \left(
\begin{array}
[c]{c}%
\text{here, we have split off the addend for }k=0\\
\text{from the sum (since }0\in\left\{  0,1,\ldots,n-1\right\}  \text{)}%
\end{array}
\right) \nonumber\\
&  =D_{0}+\underbrace{\sum_{k=1}^{n-1}0D_{k}}_{=0_{n\times n}}=D_{0}.
\label{pf.lem.adj.poly.-A.short.7}%
\end{align}
But applying the map $\varepsilon^{n\times n}$ to both sides of the equality
(\ref{pf.lem.adj.poly.-A.short.1}), we obtain%
\[
\varepsilon^{n\times n}\left(  \operatorname*{adj}\left(  tI_{n}-A\right)
\right)  =\varepsilon^{n\times n}\left(  \sum_{k=0}^{n-1}t^{k}D_{k}\right)
=D_{0}%
\]
(by (\ref{pf.lem.adj.poly.-A.short.7})). Thus,%
\[
D_{0}=\varepsilon^{n\times n}\left(  \operatorname*{adj}\left(
\underbrace{tI_{n}-A}_{=tI_{n}+\left(  -A\right)  }\right)  \right)
=\varepsilon^{n\times n}\left(  \operatorname*{adj}\left(  tI_{n}+\left(
-A\right)  \right)  \right)  =\operatorname*{adj}\left(  -A\right)
\]
(by Proposition \ref{prop.tInA.ev} \textbf{(b)}, applied to $-A$ instead of
$A$). Hence, (\ref{pf.lem.adj.poly.-A.short.2}) becomes%
\[
\sum_{i=0}^{n-1}c_{n-1-i}A^{i}=D_{0}=\operatorname*{adj}\left(  -A\right)  .
\]
This proves Lemma \ref{lem.adj.poly.-A}.
\end{proof}
\end{vershort}

\begin{verlong}
\begin{proof}
[Proof of Lemma \ref{lem.adj.poly.-A}.]For every $j\in\mathbb{Z}$, define an
element $c_{j}\in\mathbb{K}$ by $c_{j}=\left[  t^{n-j}\right]  \chi_{A}$.

Proposition \ref{prop.adj.poly} shows that there exist $n$ matrices
$D_{0},D_{1},\ldots,D_{n-1}$ in $\mathbb{K}^{n\times n}$ such that%
\begin{equation}
\operatorname*{adj}\left(  tI_{n}-A\right)  =\sum_{k=0}^{n-1}t^{k}%
D_{k}\ \ \ \ \ \ \ \ \ \ \text{in }\left(  \mathbb{K}\left[  t\right]
\right)  ^{n\times n}. \label{pf.lem.adj.poly.-A.1}%
\end{equation}
Consider these $D_{0},D_{1},\ldots,D_{n-1}$. Thus, an $n$-tuple $\left(
D_{0},D_{1},\ldots,D_{n-1}\right)  $ of matrices in $\mathbb{K}^{n\times n}$
is defined. Extend this $n$-tuple to a family $\left(  D_{k}\right)
_{k\in\mathbb{Z}}$ of matrices in $\mathbb{K}^{n\times n}$ by setting
(\ref{eq.lem.CH.D.0}). Lemma \ref{lem.CH.D} \textbf{(c)} (applied to $k=n-1$)
yields%
\begin{equation}
\sum_{i=0}^{n-1}c_{n-1-i}A^{i}=D_{n-1-\left(  n-1\right)  }=D_{0}%
\ \ \ \ \ \ \ \ \ \ \left(  \text{since }n-1-\left(  n-1\right)  =0\right)  .
\label{pf.lem.adj.poly.-A.2}%
\end{equation}

On the other hand, define the $\mathbb{K}$-algebra homomorphism $\varepsilon
:\mathbb{K}\left[  t\right]  \rightarrow\mathbb{K}$ as in Proposition
\ref{prop.tInA.ev}. This homomorphism $\varepsilon$ satisfies $\varepsilon
\left(  t\right)  =0$. Also, it satisfies%
\begin{equation}
\varepsilon\left(  u\right)  =u\ \ \ \ \ \ \ \ \ \ \text{for every }%
u\in\mathbb{K}. \label{pf.lem.adj.poly.-A.eu}%
\end{equation}
Hence, the map $\varepsilon^{n\times n}:\left(  \mathbb{K}\left[  t\right]
\right)  ^{n\times n}\rightarrow\mathbb{K}^{n\times n}$ (defined as in
Definition \ref{def.f-mat}) satisfies%
\begin{equation}
\varepsilon^{n\times n}\left(  F\right)  =F\ \ \ \ \ \ \ \ \ \ \text{for every
}F\in\mathbb{K}^{n\times n} \label{pf.lem.adj.poly.-A.eF}%
\end{equation}
\footnote{\textit{Proof of (\ref{pf.lem.adj.poly.-A.eF}):} Let $F\in
\mathbb{K}^{n\times n}$. Write the $n\times n$-matrix $F$ in the form
$F=\left(  f_{i,j}\right)  _{1\leq i\leq n,\ 1\leq j\leq n}$. Then,
$f_{i,j}\in\mathbb{K}$ for every $\left(  i,j\right)  \in\left\{
1,2,\ldots,n\right\}  ^{2}$ (since $F\in\mathbb{K}^{n\times n}$). Hence,%
\begin{equation}
\varepsilon\left(  f_{i,j}\right)  =f_{i,j}\ \ \ \ \ \ \ \ \ \ \text{for every
}\left(  i,j\right)  \in\left\{  1,2,\ldots,n\right\}  ^{2}
\label{pf.lem.adj.poly.-A.eF.pf.1}%
\end{equation}
(by (\ref{pf.lem.adj.poly.-A.eu}) (applied to $u=f_{i,j}$)).
\par
Applying the map $\varepsilon^{n\times n}$ to both sides of the equality
$F=\left(  f_{i,j}\right)  _{1\leq i\leq n,\ 1\leq j\leq n}$, we obtain%
\begin{align*}
\varepsilon^{n\times n}\left(  F\right)   &  =\varepsilon^{n\times n}\left(
\left(  f_{i,j}\right)  _{1\leq i\leq n,\ 1\leq j\leq n}\right)  =\left(
\underbrace{\varepsilon\left(  f_{i,j}\right)  }_{\substack{=f_{i,j}%
\\\text{(by (\ref{pf.lem.adj.poly.-A.eF.pf.1}))}}}\right)  _{1\leq i\leq
n,\ 1\leq j\leq n}\ \ \ \ \ \ \ \ \ \ \left(  \text{by the definition of the
map }\varepsilon^{n\times n}\right) \\
&  =\left(  f_{i,j}\right)  _{1\leq i\leq n,\ 1\leq j\leq n}=F.
\end{align*}
This proves (\ref{pf.lem.adj.poly.-A.eF}).}.

The integer $n$ is positive. Thus, $n-1\in\mathbb{N}$, so that $0\in\left\{
0,1,\ldots,n-1\right\}  $.

But $\varepsilon$ is a ring homomorphism (since $\varepsilon$ is a
$\mathbb{K}$-algebra homomorphism). Hence, Proposition \ref{prop.functor.clas}
\textbf{(a)} (applied to $\mathbb{L}=\mathbb{K}\left[  t\right]  $,
$\mathbb{M}=\mathbb{K}$, $f=\varepsilon$ and $m=n$) yields that the map
$\varepsilon^{n\times n}:\left(  \mathbb{K}\left[  t\right]  \right)
^{n\times n}\rightarrow\mathbb{K}^{n\times n}$ is a homomorphism of additive
groups. Hence,%
\begin{align}
\varepsilon^{n\times n}\left(  \sum_{k=0}^{n-1}t^{k}D_{k}\right)   &
=\sum_{k=0}^{n-1}\underbrace{\varepsilon^{n\times n}\left(  t^{k}D_{k}\right)
}_{\substack{=\varepsilon\left(  t^{k}\right)  \varepsilon^{n\times n}\left(
D_{k}\right)  \\\text{(by Proposition \ref{prop.functor.clas} \textbf{(d)}%
}\\\text{(applied to }\mathbb{K}\left[  t\right]  \text{, }\mathbb{K}\text{,
}\varepsilon\text{, }n\text{, }D_{k}\text{ and }t^{k}\\\text{instead of
}\mathbb{L}\text{, }\mathbb{M}\text{, }f\text{, }m\text{, }A\text{ and
}\lambda\text{))}}}\nonumber\\
&  =\sum_{k=0}^{n-1}\underbrace{\varepsilon\left(  t^{k}\right)
}_{\substack{=\left(  \varepsilon\left(  t\right)  \right)  ^{k}\\\text{(since
}\varepsilon\text{ is a ring}\\\text{homomorphism)}}}\underbrace{\varepsilon
^{n\times n}\left(  D_{k}\right)  }_{\substack{=D_{k}\\\text{(by
(\ref{pf.lem.adj.poly.-A.eF})}\\\text{(applied to }F=D_{k}\text{))}%
}}\nonumber\\
&  =\sum_{k=0}^{n-1}\left(  \underbrace{\varepsilon\left(  t\right)  }%
_{=0}\right)  ^{k}D_{k}=\sum_{k=0}^{n-1}0^{k}D_{k}=\underbrace{0^{0}}%
_{=1}D_{0}+\sum_{k=1}^{n-1}\underbrace{0^{k}}_{\substack{=0\\\text{(since
}k\geq1\text{)}}}D_{k}\nonumber\\
&  \ \ \ \ \ \ \ \ \ \ \left(
\begin{array}
[c]{c}%
\text{here, we have split off the addend for }k=0\text{ from the sum}\\
\text{(since }0\in\left\{  0,1,\ldots,n-1\right\}  \text{)}%
\end{array}
\right) \nonumber\\
&  =D_{0}+\underbrace{\sum_{k=1}^{n-1}0D_{k}}_{=0_{n\times n}}=D_{0}.
\label{pf.lem.adj.poly.-A.7}%
\end{align}
But applying the map $\varepsilon^{n\times n}$ to both sides of the equality
(\ref{pf.lem.adj.poly.-A.1}), we obtain%
\[
\varepsilon^{n\times n}\left(  \operatorname*{adj}\left(  tI_{n}-A\right)
\right)  =\varepsilon^{n\times n}\left(  \sum_{k=0}^{n-1}t^{k}D_{k}\right)
=D_{0}%
\]
(by (\ref{pf.lem.adj.poly.-A.7})). Thus,%
\[
D_{0}=\varepsilon^{n\times n}\left(  \operatorname*{adj}\left(
\underbrace{tI_{n}-A}_{=tI_{n}+\left(  -A\right)  }\right)  \right)
=\varepsilon^{n\times n}\left(  \operatorname*{adj}\left(  tI_{n}+\left(
-A\right)  \right)  \right)  =\operatorname*{adj}\left(  -A\right)
\]
(by Proposition \ref{prop.tInA.ev} \textbf{(b)} (applied to $-A$ instead of
$A$)). Hence, (\ref{pf.lem.adj.poly.-A.2}) becomes%
\[
\sum_{i=0}^{n-1}c_{n-1-i}A^{i}=D_{0}=\operatorname*{adj}\left(  -A\right)  .
\]
This proves Lemma \ref{lem.adj.poly.-A}.
\end{proof}
\end{verlong}

Finally, we are ready to prove Theorem \ref{thm.adj.poly}:

\begin{vershort}
\begin{proof}
[Proof of Theorem \ref{thm.adj.poly}.]We must prove the equality
$\operatorname*{adj}A=\left(  -1\right)  ^{n-1}\sum_{i=0}^{n-1}c_{n-1-i}A^{i}%
$. This is an equality between two $n\times n$-matrices, and thus obviously
holds if $n=0$. Hence, we WLOG assume that $n\neq0$. Thus, $n$ is a positive
integer. Hence, Proposition \ref{prop.adj.lamA} (applied to $\lambda=-1$)
yields
\[
\operatorname*{adj}\left(  -A\right)  =\left(  -1\right)  ^{n-1}%
\operatorname*{adj}A.
\]
Therefore,%
\[
\operatorname*{adj}A=\left(  -1\right)  ^{n-1}\underbrace{\operatorname*{adj}%
\left(  -A\right)  }_{\substack{=\sum_{i=0}^{n-1}c_{n-1-i}A^{i}\\\text{(by
Lemma \ref{lem.adj.poly.-A})}}}=\left(  -1\right)  ^{n-1}\sum_{i=0}%
^{n-1}c_{n-1-i}A^{i}.
\]
This proves Theorem \ref{thm.adj.poly}.
\end{proof}
\end{vershort}

\begin{verlong}
\begin{proof}
[Proof of Theorem \ref{thm.adj.poly}.]We must prove the equality
$\operatorname*{adj}A=\left(  -1\right)  ^{n-1}\sum_{i=0}^{n-1}c_{n-1-i}A^{i}%
$. If $n=0$, then this equality is obviously true\footnote{\textit{Proof.}
This equality is an equality between two $n\times n$-matrices. Thus, if $n=0$,
then this equality is an equality between two $0\times0$-matrices, and
therefore is true (because any equality between two $0\times0$-matrices is
true (since there exists only one $0\times0$-matrix)). Qed.}. Hence, for the
rest of this proof, we can WLOG assume that we don't have $n=0$. Assume this.

We have $n\neq0$ (since we don't have $n=0$). Thus, $n$ is a positive integer
(since $n\in\mathbb{N}$). Hence, Proposition \ref{prop.adj.lamA} (applied to
$\lambda=-1$) yields
\[
\operatorname*{adj}\left(  \left(  -1\right)  A\right)  =\left(  -1\right)
^{n-1}\operatorname*{adj}A.
\]
Now,%
\[
\left(  -1\right)  ^{n-1}\operatorname*{adj}\left(  \underbrace{-A}_{=\left(
-1\right)  A}\right)  =\left(  -1\right)  ^{n-1}%
\underbrace{\operatorname*{adj}\left(  \left(  -1\right)  A\right)
}_{=\left(  -1\right)  ^{n-1}\operatorname*{adj}A}=\underbrace{\left(
-1\right)  ^{n-1}\left(  -1\right)  ^{n-1}}_{\substack{=\left(  -1\right)
^{\left(  n-1\right)  +\left(  n-1\right)  }=1\\\text{(since }\left(
n-1\right)  +\left(  n-1\right)  =2\left(  n-1\right)  \\\text{is even)}%
}}\operatorname*{adj}A=\operatorname*{adj}A,
\]
so that%
\[
\operatorname*{adj}A=\left(  -1\right)  ^{n-1}\underbrace{\operatorname*{adj}%
\left(  -A\right)  }_{\substack{=\sum_{i=0}^{n-1}c_{n-1-i}A^{i}\\\text{(by
Lemma \ref{lem.adj.poly.-A})}}}=\left(  -1\right)  ^{n-1}\sum_{i=0}%
^{n-1}c_{n-1-i}A^{i}.
\]
This proves Theorem \ref{thm.adj.poly}.
\end{proof}
\end{verlong}

\subsection{Minors of the adjugate: Jacobi's theorem}

A \textit{minor} of a matrix $A$ is defined to be a determinant of a square
submatrix of $A$. A theorem due to Jacobi connects the minors of
$\operatorname*{adj}A$ (for a square matrix $A$) with the minors of $A$.
Before we can state this theorem, let us introduce some notations:

\begin{definition}
\label{def.submatrix2}Let $n\in\mathbb{N}$ and $m\in\mathbb{N}$. Let
$A=\left(  a_{i,j}\right)  _{1\leq i\leq n,\ 1\leq j\leq m}$ be an $n\times
m$-matrix. Let $i_{1},i_{2},\ldots,i_{u}$ be some elements of $\left\{
1,2,\ldots,n\right\}  $; let $j_{1},j_{2},\ldots,j_{v}$ be some elements of
$\left\{  1,2,\ldots,m\right\}  $. Then, we shall use $\operatorname*{sub}%
\nolimits_{\left(  i_{1},i_{2},\ldots,i_{u}\right)  }^{\left(  j_{1}%
,j_{2},\ldots,j_{v}\right)  }A$ as a synonym for the $u\times v$-matrix
$\operatorname*{sub}\nolimits_{i_{1},i_{2},\ldots,i_{u}}^{j_{1},j_{2}%
,\ldots,j_{v}}A$. Thus, for every $\mathbf{i}\in\left\{  1,2,\ldots,n\right\}
^{u}$ and $\mathbf{j}\in\left\{  1,2,\ldots,m\right\}  ^{v}$, a $u\times
v$-matrix $\operatorname*{sub}\nolimits_{\mathbf{i}}^{\mathbf{j}}A$ is defined.
\end{definition}

\begin{definition}
If $I$ is a finite set of integers, then $\sum I$ shall denote the sum of all
elements of $I$. (Thus, $\sum I=\sum_{i\in I}i$.)
\end{definition}

\begin{definition}
If $I$ is a finite set of integers, then $w\left(  I\right)  $ shall denote
the list of all elements of $I$ in increasing order (with no repetitions).
(For example, $w\left(  \left\{  3,4,8\right\}  \right)  =\left(
3,4,8\right)  $.)
\end{definition}

The following fact is obvious:

\begin{remark}
\label{rmk.w(I).in}Let $n\in\mathbb{N}$. Let $I$ be a subset of $\left\{
1,2,\ldots,n\right\}  $. Then, $w\left(  I\right)  \in\left\{  1,2,\ldots
,n\right\}  ^{\left\vert I\right\vert }$.
\end{remark}

\begin{verlong}
\begin{proof}
[Proof of Remark \ref{rmk.w(I).in}.]We know that $w\left(  I\right)  $ is the
list of all elements of $I$ in increasing order (with no repetitions) (by the
definition of $w\left(  I\right)  $). Thus, $w\left(  I\right)  $ is a list of
$\left\vert I\right\vert $ elements.

Also, $w\left(  I\right)  $ is a list of elements of $I$ (since $w\left(
I\right)  $ is the list of all elements of $I$ in increasing order (with no
repetitions)). Thus, $w\left(  I\right)  $ is a list of $\left\vert
I\right\vert $ elements of $I$. In other words, $w\left(  I\right)  \in
I^{\left\vert I\right\vert }$. But $I\subseteq\left\{  1,2,\ldots,n\right\}  $
(since $I$ is a subset of $\left\{  1,2,\ldots,n\right\}  $), and thus
$I^{\left\vert I\right\vert }\subseteq\left\{  1,2,\ldots,n\right\}
^{\left\vert I\right\vert }$. Hence, $w\left(  I\right)  \in I^{\left\vert
I\right\vert }\subseteq\left\{  1,2,\ldots,n\right\}  ^{\left\vert
I\right\vert }$. This proves Remark \ref{rmk.w(I).in}.
\end{proof}
\end{verlong}

Now, we can state Jacobi's theorem\footnote{This is \cite[Corollary
7.256]{detnotes}. It also appears in \cite[Theorem 2.5.2]{Prasolov} (in a
different form).}:

\begin{theorem}
\label{thm.jacobi-complement}Let $n\in\mathbb{N}$. For any subset $I$ of
$\left\{  1,2,\ldots,n\right\}  $, we let $\widetilde{I}$ denote the
complement $\left\{  1,2,\ldots,n\right\}  \setminus I$ of $I$.

Let $A$ be an $n\times n$-matrix.

Let $P$ and $Q$ be two subsets of $\left\{  1,2,\ldots,n\right\}  $ such that
$\left\vert P\right\vert =\left\vert Q\right\vert \geq1$. Then,%
\[
\det\left(  \operatorname*{sub}\nolimits_{w\left(  P\right)  }^{w\left(
Q\right)  }\left(  \operatorname*{adj}A\right)  \right)  =\left(  -1\right)
^{\sum P+\sum Q}\left(  \det A\right)  ^{\left\vert Q\right\vert -1}%
\det\left(  \operatorname*{sub}\nolimits_{w\left(  \widetilde{Q}\right)
}^{w\left(  \widetilde{P}\right)  }A\right)  .
\]

\end{theorem}

We shall not give a standalone proof of this theorem; instead, we will merely
derive it from results proven in \cite{detnotes}. Namely, in \cite[Corollary
7.255]{detnotes}, the following was proven:

\begin{lemma}
\label{lem.sol.addexe.jacobi-complement.adjform}Let $n\in\mathbb{N}$. For any
subset $I$ of $\left\{  1,2,\ldots,n\right\}  $, we let $\widetilde{I}$ denote
the complement $\left\{  1,2,\ldots,n\right\}  \setminus I$ of $I$.

Let $A$ be an $n\times n$-matrix.

Let $P$ and $Q$ be two subsets of $\left\{  1,2,\ldots,n\right\}  $ such that
$\left\vert P\right\vert =\left\vert Q\right\vert $. Then,%
\[
\det A\cdot\det\left(  \operatorname*{sub}\nolimits_{w\left(  P\right)
}^{w\left(  Q\right)  }\left(  \operatorname*{adj}A\right)  \right)  =\left(
-1\right)  ^{\sum P+\sum Q}\left(  \det A\right)  ^{\left\vert Q\right\vert
}\det\left(  \operatorname*{sub}\nolimits_{w\left(  \widetilde{Q}\right)
}^{w\left(  \widetilde{P}\right)  }A\right)  .
\]

\end{lemma}

We shall also use the following obvious lemma:

\begin{lemma}
\label{lem.functor.subij}Let $\mathbb{L}$ and $\mathbb{M}$ be two commutative
rings. Let $f:\mathbb{L}\rightarrow\mathbb{M}$ be any map. Let $n\in
\mathbb{N}$ and $m\in\mathbb{N}$. Let $A\in\mathbb{L}^{n\times m}$.

Let $u\in\mathbb{N}$ and $v\in\mathbb{N}$. Let $\mathbf{i}\in\left\{
1,2,\ldots,n\right\}  ^{u}$ and $\mathbf{j}\in\left\{  1,2,\ldots,m\right\}
^{v}$. Then,%
\[
f^{u\times v}\left(  \operatorname*{sub}\nolimits_{\mathbf{i}}^{\mathbf{j}%
}A\right)  =\operatorname*{sub}\nolimits_{\mathbf{i}}^{\mathbf{j}}\left(
f^{n\times m}\left(  A\right)  \right)  .
\]

\end{lemma}

\begin{verlong}
\begin{proof}
[Proof of Lemma \ref{lem.functor.subij}.]We have $\mathbf{i}\in\left\{
1,2,\ldots,n\right\}  ^{u}$. Thus, we can write $\mathbf{i}$ in the form
$\mathbf{i}=\left(  i_{1},i_{2},\ldots,i_{u}\right)  $ for some $u$ elements
$i_{1},i_{2},\ldots,i_{u}$ of $\left\{  1,2,\ldots,n\right\}  $. Consider
these $i_{1},i_{2},\ldots,i_{u}$.

We have $\mathbf{j}\in\left\{  1,2,\ldots,m\right\}  ^{v}$. Thus, we can write
$\mathbf{j}$ in the form $\mathbf{j}=\left(  j_{1},j_{2},\ldots,j_{v}\right)
$ for some $v$ elements $j_{1},j_{2},\ldots,j_{v}$ of $\left\{  1,2,\ldots
,m\right\}  $. Consider these $j_{1},j_{2},\ldots,j_{v}$.

Write the $n\times m$-matrix $A$ in the form $A=\left(  a_{i,j}\right)
_{1\leq i\leq n,\ 1\leq j\leq m}$. Then, $f^{n\times m}\left(  A\right)
=\left(  f\left(  a_{i,j}\right)  \right)  _{1\leq i\leq n,\ 1\leq j\leq m}$
(by the definition of the map $f^{n\times m}$).

But%
\begin{align*}
\operatorname*{sub}\nolimits_{\mathbf{i}}^{\mathbf{j}}A  &
=\operatorname*{sub}\nolimits_{\left(  i_{1},i_{2},\ldots,i_{u}\right)
}^{\left(  j_{1},j_{2},\ldots,j_{v}\right)  }A\ \ \ \ \ \ \ \ \ \ \left(
\text{since }\mathbf{i}=\left(  i_{1},i_{2},\ldots,i_{u}\right)  \text{ and
}\mathbf{j}=\left(  j_{1},j_{2},\ldots,j_{v}\right)  \right) \\
&  =\operatorname*{sub}\nolimits_{i_{1},i_{2},\ldots,i_{u}}^{j_{1}%
,j_{2},\ldots,j_{v}}A=\left(  a_{i_{x},j_{y}}\right)  _{1\leq x\leq u,\ 1\leq
y\leq v}%
\end{align*}
(by the definition of $\operatorname*{sub}\nolimits_{i_{1},i_{2},\ldots,i_{u}%
}^{j_{1},j_{2},\ldots,j_{v}}A$, since $A=\left(  a_{i,j}\right)  _{1\leq i\leq
n,\ 1\leq j\leq m}$). Applying the map $f^{u\times v}$ to both sides of this
equality, we find%
\[
f^{u\times v}\left(  \operatorname*{sub}\nolimits_{\mathbf{i}}^{\mathbf{j}%
}A\right)  =f^{u\times v}\left(  \left(  a_{i_{x},j_{y}}\right)  _{1\leq x\leq
u,\ 1\leq y\leq v}\right)  =\left(  f\left(  a_{i_{x},j_{y}}\right)  \right)
_{1\leq x\leq u,\ 1\leq y\leq v}%
\]
(by the definition of $f^{u\times v}$). Comparing this with%
\begin{align*}
\operatorname*{sub}\nolimits_{\mathbf{i}}^{\mathbf{j}}\left(  f^{n\times
m}\left(  A\right)  \right)   &  =\operatorname*{sub}\nolimits_{\left(
i_{1},i_{2},\ldots,i_{u}\right)  }^{\left(  j_{1},j_{2},\ldots,j_{v}\right)
}\left(  f^{n\times m}\left(  A\right)  \right) \\
&  \ \ \ \ \ \ \ \ \ \ \left(  \text{since }\mathbf{i}=\left(  i_{1}%
,i_{2},\ldots,i_{u}\right)  \text{ and }\mathbf{j}=\left(  j_{1},j_{2}%
,\ldots,j_{v}\right)  \right) \\
&  =\operatorname*{sub}\nolimits_{i_{1},i_{2},\ldots,i_{u}}^{j_{1}%
,j_{2},\ldots,j_{v}}\left(  f^{n\times m}\left(  A\right)  \right)  =\left(
f\left(  a_{i_{x},j_{y}}\right)  \right)  _{1\leq x\leq u,\ 1\leq y\leq v}\\
&  \ \ \ \ \ \ \ \ \ \ \left(
\begin{array}
[c]{c}%
\text{by the definition of }\operatorname*{sub}\nolimits_{i_{1},i_{2}%
,\ldots,i_{u}}^{j_{1},j_{2},\ldots,j_{v}}\left(  f^{n\times m}\left(
A\right)  \right)  \text{,}\\
\text{since }f^{n\times m}\left(  A\right)  =\left(  f\left(  a_{i,j}\right)
\right)  _{1\leq i\leq n,\ 1\leq j\leq m}%
\end{array}
\right)  ,
\end{align*}
we obtain $f^{u\times v}\left(  \operatorname*{sub}\nolimits_{\mathbf{i}%
}^{\mathbf{j}}A\right)  =\operatorname*{sub}\nolimits_{\mathbf{i}}%
^{\mathbf{j}}\left(  f^{n\times m}\left(  A\right)  \right)  $. This proves
Lemma \ref{lem.functor.subij}.
\end{proof}
\end{verlong}

\begin{proof}
[Proof of Theorem \ref{thm.jacobi-complement}.]Define the $\mathbb{K}$-algebra
homomorphism $\varepsilon:\mathbb{K}\left[  t\right]  \rightarrow\mathbb{K}$
as in Proposition \ref{prop.tInA.ev}.

Define a matrix $\widetilde{A}\in\left(  \mathbb{K}\left[  t\right]  \right)
^{n\times n}$ by $\widetilde{A}=tI_{n}+A$. Corollary \ref{cor.tInA.reg} shows
that the element $\det\left(  tI_{n}+A\right)  $ of $\mathbb{K}\left[
t\right]  $ is regular. In other words, the element $\det\widetilde{A}$ of
$\mathbb{K}\left[  t\right]  $ is regular (since $\widetilde{A}=tI_{n}+A$).

We have $\left\vert Q\right\vert -1\in\mathbb{N}$ (since $\left\vert
Q\right\vert \geq1$). Lemma \ref{lem.sol.addexe.jacobi-complement.adjform}
(applied to $\mathbb{K}\left[  t\right]  $ and $\widetilde{A}$ instead of
$\mathbb{K}$ and $A$) yields%
\begin{align*}
\det\widetilde{A}\cdot\det\left(  \operatorname*{sub}\nolimits_{w\left(
P\right)  }^{w\left(  Q\right)  }\left(  \operatorname*{adj}\widetilde{A}%
\right)  \right)   &  =\left(  -1\right)  ^{\sum P+\sum Q}\underbrace{\left(
\det\widetilde{A}\right)  ^{\left\vert Q\right\vert }}_{\substack{=\left(
\det\widetilde{A}\right)  \left(  \det\widetilde{A}\right)  ^{\left\vert
Q\right\vert -1}\\\text{(since }\left\vert Q\right\vert \geq1\text{)}}%
}\det\left(  \operatorname*{sub}\nolimits_{w\left(  \widetilde{Q}\right)
}^{w\left(  \widetilde{P}\right)  }\widetilde{A}\right) \\
&  =\left(  -1\right)  ^{\sum P+\sum Q}\left(  \det\widetilde{A}\right)
\left(  \det\widetilde{A}\right)  ^{\left\vert Q\right\vert -1}\det\left(
\operatorname*{sub}\nolimits_{w\left(  \widetilde{Q}\right)  }^{w\left(
\widetilde{P}\right)  }\widetilde{A}\right) \\
&  =\det\widetilde{A}\cdot\left(  -1\right)  ^{\sum P+\sum Q}\left(
\det\widetilde{A}\right)  ^{\left\vert Q\right\vert -1}\det\left(
\operatorname*{sub}\nolimits_{w\left(  \widetilde{Q}\right)  }^{w\left(
\widetilde{P}\right)  }\widetilde{A}\right)  .
\end{align*}
Hence, Lemma \ref{lem.regular.cancel} (applied to $\mathbb{A}=\mathbb{K}%
\left[  t\right]  $, $a=\det\widetilde{A}$, $b=\det\left(  \operatorname*{sub}%
\nolimits_{w\left(  P\right)  }^{w\left(  Q\right)  }\left(
\operatorname*{adj}\widetilde{A}\right)  \right)  $ and $c=\left(  -1\right)
^{\sum P+\sum Q}\left(  \det\widetilde{A}\right)  ^{\left\vert Q\right\vert
-1}\det\left(  \operatorname*{sub}\nolimits_{w\left(  \widetilde{Q}\right)
}^{w\left(  \widetilde{P}\right)  }\widetilde{A}\right)  $) yields%
\[
\det\left(  \operatorname*{sub}\nolimits_{w\left(  P\right)  }^{w\left(
Q\right)  }\left(  \operatorname*{adj}\widetilde{A}\right)  \right)  =\left(
-1\right)  ^{\sum P+\sum Q}\left(  \det\widetilde{A}\right)  ^{\left\vert
Q\right\vert -1}\det\left(  \operatorname*{sub}\nolimits_{w\left(
\widetilde{Q}\right)  }^{w\left(  \widetilde{P}\right)  }\widetilde{A}\right)
\]
(since the element $\det\widetilde{A}$ of $\mathbb{K}\left[  t\right]  $ is
regular). Applying the map $\varepsilon$ to both sides of this equality, we
obtain%
\begin{align}
&  \varepsilon\left(  \det\left(  \operatorname*{sub}\nolimits_{w\left(
P\right)  }^{w\left(  Q\right)  }\left(  \operatorname*{adj}\widetilde{A}%
\right)  \right)  \right) \nonumber\\
&  =\varepsilon\left(  \left(  -1\right)  ^{\sum P+\sum Q}\left(
\det\widetilde{A}\right)  ^{\left\vert Q\right\vert -1}\det\left(
\operatorname*{sub}\nolimits_{w\left(  \widetilde{Q}\right)  }^{w\left(
\widetilde{P}\right)  }\widetilde{A}\right)  \right) \nonumber\\
&  =\left(  -1\right)  ^{\sum P+\sum Q}\left(  \varepsilon\left(
\det\widetilde{A}\right)  \right)  ^{\left\vert Q\right\vert -1}%
\varepsilon\left(  \det\left(  \operatorname*{sub}\nolimits_{w\left(
\widetilde{Q}\right)  }^{w\left(  \widetilde{P}\right)  }\widetilde{A}\right)
\right)  \label{pf.thm.jacobi-complement.1}%
\end{align}
(since $\varepsilon$ is a $\mathbb{K}$-algebra homomorphism).

\begin{vershort}
The definition of $\widetilde{P}$ yields $\widetilde{P}=\left\{
1,2,\ldots,n\right\}  \setminus P$. Hence,
\begin{align*}
\left\vert \widetilde{P}\right\vert  &  =\underbrace{\left\vert \left\{
1,2,\ldots,n\right\}  \right\vert }_{=n}-\left\vert P\right\vert
\ \ \ \ \ \ \ \ \ \ \left(  \text{since }P\subseteq\left\{  1,2,\ldots
,n\right\}  \right) \\
&  =n-\left\vert P\right\vert .
\end{align*}
Similarly, $\left\vert \widetilde{Q}\right\vert =n-\left\vert Q\right\vert $.
Notice that $\left\vert \widetilde{P}\right\vert =n-\underbrace{\left\vert
P\right\vert }_{=\left\vert Q\right\vert }=n-\left\vert Q\right\vert $.
\end{vershort}

\begin{verlong}
The definition of $\widetilde{P}$ yields $\widetilde{P}=\left\{
1,2,\ldots,n\right\}  \setminus P\subseteq\left\{  1,2,\ldots,n\right\}  $. In
other words, $\widetilde{P}$ is a subset of $\left\{  1,2,\ldots,n\right\}  $.
The same argument (applied to $Q$ instead of $P$) yields that $\widetilde{Q}$
is a subset of $\left\{  1,2,\ldots,n\right\}  $.

The set $P$ is a subset of $\left\{  1,2,\ldots,n\right\}  $, and thus is
finite (since $\left\{  1,2,\ldots,n\right\}  $ is finite). The set
$\widetilde{P}$ is a subset of $\left\{  1,2,\ldots,n\right\}  $, and thus is
finite (since $\left\{  1,2,\ldots,n\right\}  $ is finite).

Furthermore,%
\begin{align*}
\left\vert \underbrace{\widetilde{P}}_{=\left\{  1,2,\ldots,n\right\}
\setminus P}\right\vert  &  =\left\vert \left\{  1,2,\ldots,n\right\}
\setminus P\right\vert =\underbrace{\left\vert \left\{  1,2,\ldots,n\right\}
\right\vert }_{=n}-\left\vert P\right\vert \ \ \ \ \ \ \ \ \ \ \left(
\text{since }P\subseteq\left\{  1,2,\ldots,n\right\}  \right) \\
&  =n-\left\vert P\right\vert .
\end{align*}
The same argument (applied to $Q$ instead of $P$) yields $\left\vert
\widetilde{Q}\right\vert =n-\left\vert Q\right\vert $. Notice that $\left\vert
\widetilde{P}\right\vert =n-\underbrace{\left\vert P\right\vert }_{=\left\vert
Q\right\vert }=n-\left\vert Q\right\vert $. Thus, $n-\left\vert Q\right\vert
=\left\vert \widetilde{P}\right\vert \in\mathbb{N}$ (since $\widetilde{P}$ is
a finite set).
\end{verlong}

But Remark \ref{rmk.w(I).in} (applied to $I=P$) yields $w\left(  P\right)
\in\left\{  1,2,\ldots,n\right\}  ^{\left\vert P\right\vert }=\left\{
1,2,\ldots,n\right\}  ^{\left\vert Q\right\vert }$ (since $\left\vert
P\right\vert =\left\vert Q\right\vert $). Also, Remark \ref{rmk.w(I).in}
(applied to $I=Q$) yields $w\left(  Q\right)  \in\left\{  1,2,\ldots
,n\right\}  ^{\left\vert Q\right\vert }$. Furthermore, Remark
\ref{rmk.w(I).in} (applied to $I=\widetilde{P}$) yields $w\left(
\widetilde{P}\right)  \in\left\{  1,2,\ldots,n\right\}  ^{\left\vert
\widetilde{P}\right\vert }=\left\{  1,2,\ldots,n\right\}  ^{n-\left\vert
Q\right\vert }$ (since $\left\vert \widetilde{P}\right\vert =n-\left\vert
Q\right\vert $). Finally, Remark \ref{rmk.w(I).in} (applied to
$I=\widetilde{Q}$) yields $w\left(  \widetilde{Q}\right)  \in\left\{
1,2,\ldots,n\right\}  ^{\left\vert \widetilde{Q}\right\vert }=\left\{
1,2,\ldots,n\right\}  ^{n-\left\vert Q\right\vert }$ (since $\left\vert
\widetilde{Q}\right\vert =n-\left\vert Q\right\vert $).

Recall that $\operatorname*{adj}\widetilde{A}\in\left(  \mathbb{K}\left[
t\right]  \right)  ^{n\times n}$. Furthermore,%
\[
\varepsilon^{n\times n}\left(  \operatorname*{adj}\underbrace{\widetilde{A}%
}_{=tI_{n}+A}\right)  =\varepsilon^{n\times n}\left(  \operatorname*{adj}%
\left(  tI_{n}+A\right)  \right)  =\operatorname*{adj}A
\]
(by Proposition \ref{prop.tInA.ev} \textbf{(b)}).

We have $w\left(  P\right)  \in\left\{  1,2,\ldots,n\right\}  ^{\left\vert
Q\right\vert }$ and $w\left(  Q\right)  \in\left\{  1,2,\ldots,n\right\}
^{\left\vert Q\right\vert }$. Hence,%
\[
\operatorname*{sub}\nolimits_{w\left(  P\right)  }^{w\left(  Q\right)
}\left(  \operatorname*{adj}\widetilde{A}\right)  \in\left(  \mathbb{K}\left[
t\right]  \right)  ^{\left\vert Q\right\vert \times\left\vert Q\right\vert }.
\]
Thus, Proposition \ref{prop.functor.det} \textbf{(a)} (applied to
$\mathbb{K}\left[  t\right]  $, $\mathbb{K}$, $\varepsilon$, $\left\vert
Q\right\vert $ and $\operatorname*{sub}\nolimits_{w\left(  P\right)
}^{w\left(  Q\right)  }\left(  \operatorname*{adj}\widetilde{A}\right)  $
instead of $\mathbb{L}$, $\mathbb{M}$, $f$, $n$ and $A$) yields%
\begin{align}
\varepsilon\left(  \det\left(  \operatorname*{sub}\nolimits_{w\left(
P\right)  }^{w\left(  Q\right)  }\left(  \operatorname*{adj}\widetilde{A}%
\right)  \right)  \right)   &  =\det\left(  \underbrace{\varepsilon
^{\left\vert Q\right\vert \times\left\vert Q\right\vert }\left(
\operatorname*{sub}\nolimits_{w\left(  P\right)  }^{w\left(  Q\right)
}\left(  \operatorname*{adj}\widetilde{A}\right)  \right)  }%
_{\substack{=\operatorname*{sub}\nolimits_{w\left(  P\right)  }^{w\left(
Q\right)  }\left(  \varepsilon^{n\times n}\left(  \operatorname*{adj}%
\widetilde{A}\right)  \right)  \\\text{(by Lemma \ref{lem.functor.subij}
(applied to }\mathbb{K}\left[  t\right]  \text{, }\mathbb{K}\text{,
}\varepsilon\text{, }n\text{, }\operatorname*{adj}\widetilde{A}\text{,
}\left\vert Q\right\vert \text{, }\left\vert Q\right\vert \text{,}\\w\left(
P\right)  \text{ and }w\left(  Q\right)  \text{ instead of }\mathbb{L}\text{,
}\mathbb{M}\text{, }f\text{, }m\text{, }A\text{, }u\text{, }v\text{,
}\mathbf{i}\text{ and }\mathbf{j}\text{))}}}\right) \nonumber\\
&  =\det\left(  \operatorname*{sub}\nolimits_{w\left(  P\right)  }^{w\left(
Q\right)  }\left(  \underbrace{\varepsilon^{n\times n}\left(
\operatorname*{adj}\widetilde{A}\right)  }_{=\operatorname*{adj}A}\right)
\right) \nonumber\\
&  =\det\left(  \operatorname*{sub}\nolimits_{w\left(  P\right)  }^{w\left(
Q\right)  }\left(  \operatorname*{adj}A\right)  \right)  .
\label{pf.thm.jacobi-complement.3}%
\end{align}
Comparing this with (\ref{pf.thm.jacobi-complement.1}), we obtain%
\begin{align}
&  \det\left(  \operatorname*{sub}\nolimits_{w\left(  P\right)  }^{w\left(
Q\right)  }\left(  \operatorname*{adj}A\right)  \right) \nonumber\\
&  =\left(  -1\right)  ^{\sum P+\sum Q}\left(  \varepsilon\left(
\det\widetilde{A}\right)  \right)  ^{\left\vert Q\right\vert -1}%
\varepsilon\left(  \det\left(  \operatorname*{sub}\nolimits_{w\left(
\widetilde{Q}\right)  }^{w\left(  \widetilde{P}\right)  }\widetilde{A}\right)
\right)  . \label{pf.thm.jacobi-complement.4}%
\end{align}

Recall that $\varepsilon^{n\times n}\left(  \underbrace{\widetilde{A}%
}_{=tI_{n}+A}\right)  =\varepsilon^{n\times n}\left(  tI_{n}+A\right)  =A$ (by
Proposition \ref{prop.tInA.ev} \textbf{(c)}).

On the other hand, $w\left(  \widetilde{P}\right)  \in\left\{  1,2,\ldots
,n\right\}  ^{n-\left\vert Q\right\vert }$ and $w\left(  \widetilde{Q}\right)
\in\left\{  1,2,\ldots,n\right\}  ^{n-\left\vert Q\right\vert }$. Hence,
\[
\operatorname*{sub}\nolimits_{w\left(  \widetilde{Q}\right)  }^{w\left(
\widetilde{P}\right)  }\widetilde{A}\in\left(  \mathbb{K}\left[  t\right]
\right)  ^{\left(  n-\left\vert Q\right\vert \right)  \times\left(
n-\left\vert Q\right\vert \right)  }.
\]
Hence, Proposition \ref{prop.functor.det} \textbf{(a)} (applied to
$\mathbb{K}\left[  t\right]  $, $\mathbb{K}$, $\varepsilon$, $n-\left\vert
Q\right\vert $ and $\operatorname*{sub}\nolimits_{w\left(  \widetilde{Q}%
\right)  }^{w\left(  \widetilde{P}\right)  }\widetilde{A}$ instead of
$\mathbb{L}$, $\mathbb{M}$, $f$, $n$ and $A$) yields%
\begin{align}
\varepsilon\left(  \det\left(  \operatorname*{sub}\nolimits_{w\left(
\widetilde{Q}\right)  }^{w\left(  \widetilde{P}\right)  }\widetilde{A}\right)
\right)   &  =\det\left(  \underbrace{\varepsilon^{\left(  n-\left\vert
Q\right\vert \right)  \times\left(  n-\left\vert Q\right\vert \right)
}\left(  \operatorname*{sub}\nolimits_{w\left(  \widetilde{Q}\right)
}^{w\left(  \widetilde{P}\right)  }\widetilde{A}\right)  }%
_{\substack{=\operatorname*{sub}\nolimits_{w\left(  \widetilde{Q}\right)
}^{w\left(  \widetilde{P}\right)  }\left(  \varepsilon^{n\times n}\left(
\widetilde{A}\right)  \right)  \\\text{(by Lemma \ref{lem.functor.subij}
(applied to }\mathbb{K}\left[  t\right]  \text{, }\mathbb{K}\text{,
}\varepsilon\text{, }n\text{, }\widetilde{A}\text{, }n-\left\vert Q\right\vert
\text{, }n-\left\vert Q\right\vert \text{,}\\w\left(  \widetilde{Q}\right)
\text{ and }w\left(  \widetilde{P}\right)  \text{ instead of }\mathbb{L}%
\text{, }\mathbb{M}\text{, }f\text{, }m\text{, }A\text{, }u\text{, }v\text{,
}\mathbf{i}\text{ and }\mathbf{j}\text{))}}}\right) \nonumber\\
&  =\det\left(  \operatorname*{sub}\nolimits_{w\left(  \widetilde{Q}\right)
}^{w\left(  \widetilde{P}\right)  }\left(  \underbrace{\varepsilon^{n\times
n}\left(  \widetilde{A}\right)  }_{=A}\right)  \right) \nonumber\\
&  =\det\left(  \operatorname*{sub}\nolimits_{w\left(  \widetilde{Q}\right)
}^{w\left(  \widetilde{P}\right)  }A\right)  .
\label{pf.thm.jacobi-complement.6}%
\end{align}
Also, $\varepsilon\left(  \det\underbrace{\widetilde{A}}_{=tI_{n}+A}\right)
=\varepsilon\left(  \det\left(  tI_{n}+A\right)  \right)  =\det A$ (by
Proposition \ref{prop.tInA.ev} \textbf{(a)}).

Now, (\ref{pf.thm.jacobi-complement.4}) becomes%
\begin{align*}
&  \det\left(  \operatorname*{sub}\nolimits_{w\left(  P\right)  }^{w\left(
Q\right)  }\left(  \operatorname*{adj}A\right)  \right) \\
&  =\left(  -1\right)  ^{\sum P+\sum Q}\left(  \underbrace{\varepsilon\left(
\det\widetilde{A}\right)  }_{=\det A}\right)  ^{\left\vert Q\right\vert
-1}\underbrace{\varepsilon\left(  \det\left(  \operatorname*{sub}%
\nolimits_{w\left(  \widetilde{Q}\right)  }^{w\left(  \widetilde{P}\right)
}\widetilde{A}\right)  \right)  }_{\substack{=\det\left(  \operatorname*{sub}%
\nolimits_{w\left(  \widetilde{Q}\right)  }^{w\left(  \widetilde{P}\right)
}A\right)  \\\text{(by (\ref{pf.thm.jacobi-complement.6}))}}}\\
&  =\left(  -1\right)  ^{\sum P+\sum Q}\left(  \det A\right)  ^{\left\vert
Q\right\vert -1}\det\left(  \operatorname*{sub}\nolimits_{w\left(
\widetilde{Q}\right)  }^{w\left(  \widetilde{P}\right)  }A\right)  .
\end{align*}
This proves Theorem \ref{thm.jacobi-complement}.
\end{proof}

\subsection{Another application of the $tI_{n}+A$ strategy}

The strategy that we have used to prove Theorem \ref{thm.adj.adj(AB)}, Theorem
\ref{thm.adj.adjadj} and Theorem \ref{thm.jacobi-complement} (namely,
replacing a matrix $A\in\mathbb{K}^{n\times n}$ by the matrix $tI_{n}%
+A\in\left(  \mathbb{K}\left[  t\right]  \right)  ^{n\times n}$, whose
determinant is a regular element of $\mathbb{K}\left[  t\right]  $; and then
applying the homomorphism $\varepsilon$ to get back to $A$) has many
applications; not all of them concern the adjugate of a matrix. As an example
of such an application, let us prove a neat property of commuting matrices:

\begin{theorem}
\label{thm.det.AS+B}Let $n\in\mathbb{N}$. Let $A$, $B$ and $S$ be three
$n\times n$-matrices such that $AB=BA$. Then,%
\[
\det\left(  AS+B\right)  =\det\left(  SA+B\right)  .
\]

\end{theorem}

Again, we start by showing a particular case of this theorem:

\begin{lemma}
\label{lem.det.AS+B.reg}Let $n\in\mathbb{N}$. Let $A$, $B$ and $S$ be three
$n\times n$-matrices such that $AB=BA$. Assume that the element $\det A$ of
$\mathbb{K}$ is regular. Then,%
\[
\det\left(  AS+B\right)  =\det\left(  SA+B\right)  .
\]

\end{lemma}

\begin{proof}
[Proof of Lemma \ref{lem.det.AS+B.reg}.]Define two $n\times n$-matrices $X$
and $Y$ by $X=AS+B$ and $Y=SA+B$. Comparing%
\[
\underbrace{X}_{=AS+B}A=\left(  AS+B\right)  A=ASA+BA
\]
with%
\[
A\underbrace{Y}_{=SA+B}=A\left(  SA+B\right)  =ASA+\underbrace{AB}%
_{=BA}=ASA+BA,
\]
we obtain $XA=AY$. Now, comparing%
\[
\det\left(  \underbrace{XA}_{=AY}\right)  =\det\left(  AY\right)  =\det
A\cdot\det Y
\]
with%
\[
\det\left(  XA\right)  =\det X\cdot\det A=\det A\cdot\det X,
\]
we obtain $\det A\cdot\det X=\det A\cdot\det Y$. Lemma
\ref{lem.regular.cancel} (applied to $\mathbb{K}$, $\det A$, $\det X$ and
$\det Y$ instead of $\mathbb{A}$, $a$, $b$ and $c$) thus yields $\det X=\det
Y$ (since the element $\det A$ of $\mathbb{K}$ is regular). In view of
$X=AS+B$ and $Y=SA+B$, this rewrites as $\det\left(  AS+B\right)  =\det\left(
SA+B\right)  $. This proves Lemma \ref{lem.det.AS+B.reg}.
\end{proof}

\begin{proof}
[Proof of Theorem \ref{thm.det.AS+B}.]Define the $\mathbb{K}$-algebra
homomorphism $\varepsilon:\mathbb{K}\left[  t\right]  \rightarrow\mathbb{K}$
as in Proposition \ref{prop.tInA.ev}. Thus, $\varepsilon$ is a ring
homomorphism. Hence, Proposition \ref{prop.functor.clas} \textbf{(a)} (applied
to $\mathbb{L}=\mathbb{K}\left[  t\right]  $, $\mathbb{M}=\mathbb{K}$ and
$m=n$) shows that the map $\varepsilon^{n\times n}:\left(  \mathbb{K}\left[
t\right]  \right)  ^{n\times n}\rightarrow\mathbb{K}^{n\times n}$ is a
homomorphism of additive groups.

Recall that every $n\times n$-matrix in $\mathbb{K}^{n\times n}$ can be
considered as a matrix in $\left(  \mathbb{K}\left[  t\right]  \right)
^{n\times n}$. In other words, for each $F\in\mathbb{K}^{n\times n}$, we can
consider $F$ as a matrix in $\left(  \mathbb{K}\left[  t\right]  \right)
^{n\times n}$; therefore, $\varepsilon^{n\times n}\left(  F\right)  $ is
well-defined. We have
\begin{equation}
\varepsilon^{n\times n}\left(  F\right)  =F\ \ \ \ \ \ \ \ \ \ \text{for every
}F\in\mathbb{K}^{n\times n}. \label{pf.thm.det.AS+B.FF}%
\end{equation}

\begin{vershort}
\noindent(In fact, the proof of (\ref{pf.thm.det.AS+B.FF}) is identical with
the proof of (\ref{pf.lem.adj.poly.-A.short.eF}) we gave above.)
\end{vershort}

\begin{verlong}
\noindent(In fact, the proof of (\ref{pf.thm.det.AS+B.FF}) is identical with
the proof of (\ref{pf.lem.adj.poly.-A.eF}) we gave above.)
\end{verlong}

Let $\widetilde{A}$ be the matrix $tI_{n}+A\in\left(  \mathbb{K}\left[
t\right]  \right)  ^{n\times n}$. Thus, $\widetilde{A}=tI_{n}+A$. Applying the
map $\varepsilon^{n\times n}$ to both sides of this equality, we find
$\varepsilon^{n\times n}\left(  \widetilde{A}\right)  =\varepsilon^{n\times
n}\left(  tI_{n}+A\right)  =A$ (by Proposition \ref{prop.tInA.ev} \textbf{(c)}).

Corollary \ref{cor.tInA.reg} shows that the element $\det\left(
tI_{n}+A\right)  $ of $\mathbb{K}\left[  t\right]  $ is regular. In other
words, the element $\det\widetilde{A}$ of $\mathbb{K}\left[  t\right]  $ is
regular (since $\widetilde{A}=tI_{n}+A$).

Let us consider the matrix $S\in\mathbb{K}^{n\times n}$ as a matrix in
$\left(  \mathbb{K}\left[  t\right]  \right)  ^{n\times n}$ (since every
$n\times n$-matrix in $\mathbb{K}^{n\times n}$ can be considered as a matrix
in $\left(  \mathbb{K}\left[  t\right]  \right)  ^{n\times n}$).

Similarly, let us consider the matrix $B\in\mathbb{K}^{n\times n}$ as a matrix
in $\left(  \mathbb{K}\left[  t\right]  \right)  ^{n\times n}$. Then,%
\begin{align*}
\underbrace{\widetilde{A}}_{=tI_{n}+A}B  &  =\left(  tI_{n}+A\right)
B=t\underbrace{I_{n}B}_{=B=BI_{n}}+\underbrace{AB}_{=BA}=\underbrace{tBI_{n}%
}_{=B\cdot tI_{n}}+\,BA\\
&  =B\cdot tI_{n}+BA=B\underbrace{\left(  tI_{n}+A\right)  }_{=\widetilde{A}%
}=B\widetilde{A}.
\end{align*}
Hence, Lemma \ref{lem.det.AS+B.reg} (applied to $\mathbb{K}\left[  t\right]  $
and $\widetilde{A}$ instead of $\mathbb{K}$ and $A$) yields
\begin{equation}
\det\left(  \widetilde{A}S+B\right)  =\det\left(  S\widetilde{A}+B\right)  .
\label{pf.thm.det.AS+B.1}%
\end{equation}

Proposition \ref{prop.functor.det} \textbf{(a)} (applied to $\mathbb{K}\left[
t\right]  $, $\mathbb{K}$, $\varepsilon$ and $\widetilde{A}S+B$ instead of
$\mathbb{L}$, $\mathbb{M}$, $f$ and $A$) yields
\[
\varepsilon\left(  \det\left(  \widetilde{A}S+B\right)  \right)  =\det\left(
\varepsilon^{n\times n}\left(  \widetilde{A}S+B\right)  \right)  .
\]
In view of%
\begin{align*}
\varepsilon^{n\times n}\left(  \widetilde{A}S+B\right)   &
=\underbrace{\varepsilon^{n\times n}\left(  \widetilde{A}S\right)
}_{\substack{=\varepsilon^{n\times n}\left(  \widetilde{A}\right)
\cdot\varepsilon^{n\times n}\left(  S\right)  \\\text{(by Proposition
\ref{prop.functor.clas} \textbf{(b)}}\\\text{(applied to }\mathbb{K}\left[
t\right]  \text{, }\mathbb{K}\text{, }\varepsilon\text{, }n\text{, }n\text{,
}\widetilde{A}\text{ and }S\\\text{instead of }\mathbb{L}\text{, }%
\mathbb{M}\text{, }f\text{, }m\text{, }p\text{, }A\text{ and }B\text{))}%
}}+\,\varepsilon^{n\times n}\left(  B\right) \\
&  \ \ \ \ \ \ \ \ \ \ \ \ \ \ \ \ \ \ \ \ \left(
\begin{array}
[c]{c}%
\text{since the map }\varepsilon^{n\times n}\text{ is a homomorphism}\\
\text{of additive groups}%
\end{array}
\right) \\
&  =\underbrace{\varepsilon^{n\times n}\left(  \widetilde{A}\right)  }%
_{=A}\cdot\underbrace{\varepsilon^{n\times n}\left(  S\right)  }%
_{\substack{=S\\\text{(by (\ref{pf.thm.det.AS+B.FF})}\\\text{(applied to
}F=S\text{))}}}+\underbrace{\varepsilon^{n\times n}\left(  B\right)
}_{\substack{=B\\\text{(by (\ref{pf.thm.det.AS+B.FF})}\\\text{(applied to
}F=B\text{))}}}\\
&  =AS+B,
\end{align*}
this becomes%
\begin{align}
\varepsilon\left(  \det\left(  \widetilde{A}S+B\right)  \right)   &
=\det\left(  \underbrace{\varepsilon^{n\times n}\left(  \widetilde{A}%
S+B\right)  }_{=AS+B}\right) \nonumber\\
&  =\det\left(  AS+B\right)  . \label{pf.thm.det.AS+B.2}%
\end{align}

\begin{vershort}
Similarly,%
\[
\varepsilon\left(  \det\left(  S\widetilde{A}+B\right)  \right)  =\det\left(
SA+B\right)  .
\]

\end{vershort}

\begin{verlong}
Proposition \ref{prop.functor.det} \textbf{(a)} (applied to $\mathbb{K}\left[
t\right]  $, $\mathbb{K}$, $\varepsilon$ and $S\widetilde{A}+B$ instead of
$\mathbb{L}$, $\mathbb{M}$, $f$ and $A$) yields
\[
\varepsilon\left(  \det\left(  S\widetilde{A}+B\right)  \right)  =\det\left(
\varepsilon^{n\times n}\left(  S\widetilde{A}+B\right)  \right)  .
\]
In view of%
\begin{align*}
\varepsilon^{n\times n}\left(  S\widetilde{A}+B\right)   &
=\underbrace{\varepsilon^{n\times n}\left(  S\widetilde{A}\right)
}_{\substack{=\varepsilon^{n\times n}\left(  S\right)  \cdot\varepsilon
^{n\times n}\left(  \widetilde{A}\right)  \\\text{(by Proposition
\ref{prop.functor.clas} \textbf{(b)}}\\\text{(applied to }\mathbb{K}\left[
t\right]  \text{, }\mathbb{K}\text{, }\varepsilon\text{, }n\text{, }n\text{,
}S\text{ and }\widetilde{A}\\\text{instead of }\mathbb{L}\text{, }%
\mathbb{M}\text{, }f\text{, }m\text{, }p\text{, }A\text{ and }B\text{))}%
}}+\varepsilon^{n\times n}\left(  B\right) \\
&  \ \ \ \ \ \ \ \ \ \ \left(
\begin{array}
[c]{c}%
\text{since the map }\varepsilon^{n\times n}\text{ is a homomorphism}\\
\text{of additive groups}%
\end{array}
\right) \\
&  =\underbrace{\varepsilon^{n\times n}\left(  S\right)  }%
_{\substack{=S\\\text{(by (\ref{pf.thm.det.AS+B.FF})}\\\text{(applied to
}F=S\text{))}}}\cdot\underbrace{\varepsilon^{n\times n}\left(  \widetilde{A}%
\right)  }_{=A}+\underbrace{\varepsilon^{n\times n}\left(  B\right)
}_{\substack{=B\\\text{(by (\ref{pf.thm.det.AS+B.FF})}\\\text{(applied to
}F=B\text{))}}}\\
&  =SA+B,
\end{align*}
this becomes%
\[
\varepsilon\left(  \det\left(  S\widetilde{A}+B\right)  \right)  =\det\left(
\underbrace{\varepsilon^{n\times n}\left(  S\widetilde{A}+B\right)  }%
_{=SA+B}\right)  =\det\left(  SA+B\right)  .
\]

\end{verlong}

Comparing this with%
\[
\varepsilon\left(  \underbrace{\det\left(  S\widetilde{A}+B\right)
}_{\substack{=\det\left(  \widetilde{A}S+B\right)  \\\text{(by
(\ref{pf.thm.det.AS+B.1}))}}}\right)  =\varepsilon\left(  \det\left(
\widetilde{A}S+B\right)  \right)  =\det\left(  AS+B\right)
\ \ \ \ \ \ \ \ \ \ \left(  \text{by (\ref{pf.thm.det.AS+B.2})}\right)  ,
\]
we obtain $\det\left(  AS+B\right)  =\det\left(  SA+B\right)  $. This proves
Theorem \ref{thm.det.AS+B}.
\end{proof}

\subsection{Another application of the strategy: block matrices}

The same strategy (replacing $A\in\mathbb{K}^{n\times n}$ by $tI_{n}%
+A\in\left(  \mathbb{K}\left[  t\right]  \right)  ^{n\times n}$) turns out to
be useful in proving a formula for determinants of block matrices with a
certain property.

We will use \cite[Definition 6.89]{detnotes} in this section. Roughly
speaking, this definition says that if $n$, $n^{\prime}$, $m$ and $m^{\prime}$
are four nonnegative integers, and if $A\in\mathbb{K}^{n\times m}$,
$B\in\mathbb{K}^{n\times m^{\prime}}$, $C\in\mathbb{K}^{n^{\prime}\times m}$
and $D\in\mathbb{K}^{n^{\prime}\times m^{\prime}}$ are four matrices, then
$\left(
\begin{array}
[c]{cc}%
A & B\\
C & D
\end{array}
\right)  $ shall denote the $\left(  n+n^{\prime}\right)  \times\left(
m+m^{\prime}\right)  $-matrix obtained by \textquotedblleft gluing the
matrices $A$, $B$, $C$ and $D$ together\textquotedblright\ in the way the
notation suggests (i.e., the matrix $B$ is glued to the right edge of $A$, and
then the matrices $C$ and $D$ are glued to the bottom edges of $A$ and $B$,
respectively). For example, if $n=2$, $n^{\prime}=2$, $m=2$ and $m^{\prime}%
=2$, and if%
\begin{align}
A  &  =\left(
\begin{array}
[c]{cc}%
a_{1,1} & a_{1,2}\\
a_{2,1} & a_{2,2}%
\end{array}
\right)  ,\ \ \ \ \ \ \ \ \ \ B=\left(
\begin{array}
[c]{cc}%
b_{1,1} & b_{1,2}\\
b_{2,1} & b_{2,2}%
\end{array}
\right)  ,\label{eq.block2x2.ex2x2.1}\\
C  &  =\left(
\begin{array}
[c]{cc}%
c_{1,1} & c_{1,2}\\
c_{2,1} & c_{2,2}%
\end{array}
\right)  ,\ \ \ \ \ \ \ \ \ \ \text{and }D=\left(
\begin{array}
[c]{cc}%
d_{1,1} & d_{1,2}\\
d_{2,1} & d_{2,2}%
\end{array}
\right)  , \label{eq.block2x2.ex2x2.2}%
\end{align}
then%
\[
\left(
\begin{array}
[c]{cc}%
A & B\\
C & D
\end{array}
\right)  =\left(
\begin{array}
[c]{cccc}%
a_{1,1} & a_{1,2} & b_{1,1} & b_{1,2}\\
a_{2,1} & a_{2,2} & b_{2,1} & b_{2,2}\\
c_{1,1} & c_{1,2} & d_{1,1} & d_{1,2}\\
c_{2,1} & c_{2,2} & d_{2,1} & d_{2,2}%
\end{array}
\right)  .
\]
There are more general versions of this \textquotedblleft gluing
operation\textquotedblright\ that allow for more than four matrices; but we
will only concern ourselves with the case of four matrices.

We are aiming to prove the following theorem:

\begin{theorem}
\label{thm.block2x2.detAD-CB}Let $n\in\mathbb{N}$. Let $A$, $B$, $C$ and $D$
be four $n\times n$-matrices such that $AC=CA$. Then, the $\left(  2n\right)
\times\left(  2n\right)  $-matrix $\left(
\begin{array}
[c]{cc}%
A & B\\
C & D
\end{array}
\right)  $ satisfies%
\[
\det\left(
\begin{array}
[c]{cc}%
A & B\\
C & D
\end{array}
\right)  =\det\left(  AD-CB\right)  .
\]

\end{theorem}

Theorem \ref{thm.block2x2.detAD-CB} appears, e.g., in \cite[(14)]{Silvest}.
Our proof of this theorem will closely follow \cite[proof of Lemma 2]%
{Silvest}. We will use the following obvious lemma:

\begin{lemma}
\label{lem.functor.block2x2}Let $\mathbb{L}$ and $\mathbb{M}$ be two
commutative rings. Let $f:\mathbb{L}\rightarrow\mathbb{M}$ be any map. Let
$n$, $n^{\prime}$, $m$ and $m^{\prime}$ be four nonnegative integers. Let
$A\in\mathbb{L}^{n\times m}$, $B\in\mathbb{L}^{n\times m^{\prime}}$,
$C\in\mathbb{L}^{n^{\prime}\times m}$ and $D\in\mathbb{L}^{n^{\prime}\times
m^{\prime}}$ be four matrices. Then,
\[
f^{\left(  n+n^{\prime}\right)  \times\left(  m+m^{\prime}\right)  }\left(
\left(
\begin{array}
[c]{cc}%
A & B\\
C & D
\end{array}
\right)  \right)  =\left(
\begin{array}
[c]{cc}%
f^{n\times m}\left(  A\right)  & f^{n\times m^{\prime}}\left(  B\right) \\
f^{n^{\prime}\times m}\left(  C\right)  & f^{n^{\prime}\times m^{\prime}%
}\left(  D\right)
\end{array}
\right)  .
\]

\end{lemma}

\begin{example}
For this example, set $n=2$ and $n^{\prime}=2$ and $m=2$ and $m^{\prime}=2$,
and let the $2\times2$-matrices $A$, $B$, $C$ and $D$ be given by
(\ref{eq.block2x2.ex2x2.1}) and (\ref{eq.block2x2.ex2x2.2}). Then, Lemma
\ref{lem.functor.block2x2} says that%
\[
f^{4\times4}\left(  \left(
\begin{array}
[c]{cc}%
A & B\\
C & D
\end{array}
\right)  \right)  =\left(
\begin{array}
[c]{cc}%
f^{2\times2}\left(  A\right)  & f^{2\times2}\left(  B\right) \\
f^{2\times2}\left(  C\right)  & f^{2\times2}\left(  D\right)
\end{array}
\right)  .
\]
Both the left and the right hand side of this equality are easily seen to
equal
\[
\left(
\begin{array}
[c]{cccc}%
f\left(  a_{1,1}\right)  & f\left(  a_{1,2}\right)  & f\left(  b_{1,1}\right)
& f\left(  b_{1,2}\right) \\
f\left(  a_{2,1}\right)  & f\left(  a_{2,2}\right)  & f\left(  b_{2,1}\right)
& f\left(  b_{2,2}\right) \\
f\left(  c_{1,1}\right)  & f\left(  c_{1,2}\right)  & f\left(  d_{1,1}\right)
& f\left(  d_{1,2}\right) \\
f\left(  c_{2,1}\right)  & f\left(  c_{2,2}\right)  & f\left(  d_{2,1}\right)
& f\left(  d_{2,2}\right)
\end{array}
\right)  .
\]

\end{example}

\begin{verlong}
\begin{proof}
[Proof of Lemma \ref{lem.functor.block2x2}.]Write the $n\times m$-matrix $A$
in the form $A=\left(  a_{i,j}\right)  _{1\leq i\leq n,\ 1\leq j\leq m}$.

Write the $n\times m^{\prime}$-matrix $B$ in the form $B=\left(
b_{i,j}\right)  _{1\leq i\leq n,\ 1\leq j\leq m^{\prime}}$.

Write the $n^{\prime}\times m$-matrix $C$ in the form $C=\left(
c_{i,j}\right)  _{1\leq i\leq n^{\prime},\ 1\leq j\leq m}$.

Write the $n^{\prime}\times m^{\prime}$-matrix $D$ in the form $D=\left(
d_{i,j}\right)  _{1\leq i\leq n^{\prime},\ 1\leq j\leq m^{\prime}}$.

Now, we have $A=\left(  a_{i,j}\right)  _{1\leq i\leq n,\ 1\leq j\leq m}$ and
$B=\left(  b_{i,j}\right)  _{1\leq i\leq n,\ 1\leq j\leq m^{\prime}}$ and
\newline$C=\left(  c_{i,j}\right)  _{1\leq i\leq n^{\prime},\ 1\leq j\leq m}$
and $D=\left(  d_{i,j}\right)  _{1\leq i\leq n^{\prime},\ 1\leq j\leq
m^{\prime}}$. Thus, the definition of the block matrix $\left(
\begin{array}
[c]{cc}%
A & B\\
C & D
\end{array}
\right)  $\ \ \ \ \footnote{or, more precisely, the equality \cite[(436)]%
{detnotes} (applied to $\mathbb{L}$ instead of $\mathbb{K}$)} yields%
\[
\left(
\begin{array}
[c]{cc}%
A & B\\
C & D
\end{array}
\right)  =\left(
\begin{cases}
a_{i,j}, & \text{if }i\leq n\text{ and }j\leq m;\\
b_{i,j-m}, & \text{if }i\leq n\text{ and }j>m;\\
c_{i-n,j}, & \text{if }i>n\text{ and }j\leq m;\\
d_{i-n,j-m}, & \text{if }i>n\text{ and }j>m
\end{cases}
\right)  _{1\leq i\leq n+n^{\prime},\ 1\leq j\leq m+m^{\prime}}.
\]
Hence, the definition of the map $f^{\left(  n+n^{\prime}\right)
\times\left(  m+m^{\prime}\right)  }$ yields%
\begin{align}
&  f^{\left(  n+n^{\prime}\right)  \times\left(  m+m^{\prime}\right)  }\left(
\left(
\begin{array}
[c]{cc}%
A & B\\
C & D
\end{array}
\right)  \right) \nonumber\\
&  =\left(  f\left(
\begin{cases}
a_{i,j}, & \text{if }i\leq n\text{ and }j\leq m;\\
b_{i,j-m}, & \text{if }i\leq n\text{ and }j>m;\\
c_{i-n,j}, & \text{if }i>n\text{ and }j\leq m;\\
d_{i-n,j-m}, & \text{if }i>n\text{ and }j>m
\end{cases}
\right)  \right)  _{1\leq i\leq n+n^{\prime},\ 1\leq j\leq m+m^{\prime}%
}\nonumber\\
&  =\left(
\begin{cases}
f\left(  a_{i,j}\right)  , & \text{if }i\leq n\text{ and }j\leq m;\\
f\left(  b_{i,j-m}\right)  , & \text{if }i\leq n\text{ and }j>m;\\
f\left(  c_{i-n,j}\right)  , & \text{if }i>n\text{ and }j\leq m;\\
f\left(  d_{i-n,j-m}\right)  , & \text{if }i>n\text{ and }j>m
\end{cases}
\right)  _{1\leq i\leq n+n^{\prime},\ 1\leq j\leq m+m^{\prime}}
\label{pf.lem.functor.block2x2.L=}%
\end{align}
(since each $i\in\left\{  1,2,\ldots,n+n^{\prime}\right\}  $ and $j\in\left\{
1,2,\ldots,m+m^{\prime}\right\}  $ satisfy \newline$f\left(
\begin{cases}
a_{i,j}, & \text{if }i\leq n\text{ and }j\leq m;\\
b_{i,j-m}, & \text{if }i\leq n\text{ and }j>m;\\
c_{i-n,j}, & \text{if }i>n\text{ and }j\leq m;\\
d_{i-n,j-m}, & \text{if }i>n\text{ and }j>m
\end{cases}
\right)  =%
\begin{cases}
f\left(  a_{i,j}\right)  , & \text{if }i\leq n\text{ and }j\leq m;\\
f\left(  b_{i,j-m}\right)  , & \text{if }i\leq n\text{ and }j>m;\\
f\left(  c_{i-n,j}\right)  , & \text{if }i>n\text{ and }j\leq m;\\
f\left(  d_{i-n,j-m}\right)  , & \text{if }i>n\text{ and }j>m
\end{cases}
$).

On the other hand, the definition of the map $f^{n\times m}$ yields
$f^{n\times m}\left(  A\right)  =\left(  f\left(  a_{i,j}\right)  \right)
_{1\leq i\leq n,\ 1\leq j\leq m}$ (since $A=\left(  a_{i,j}\right)  _{1\leq
i\leq n,\ 1\leq j\leq m}$). The definition of the map $f^{n\times m^{\prime}}$
yields $f^{n\times m^{\prime}}\left(  B\right)  =\left(  f\left(
b_{i,j}\right)  \right)  _{1\leq i\leq n,\ 1\leq j\leq m^{\prime}}$ (since
$B=\left(  b_{i,j}\right)  _{1\leq i\leq n,\ 1\leq j\leq m^{\prime}}$). The
definition of the map $f^{n^{\prime}\times m}$ yields $f^{n^{\prime}\times
m}\left(  C\right)  =\left(  f\left(  c_{i,j}\right)  \right)  _{1\leq i\leq
n^{\prime},\ 1\leq j\leq m}$ (since $C=\left(  c_{i,j}\right)  _{1\leq i\leq
n^{\prime},\ 1\leq j\leq m}$). The definition of the map $f^{n^{\prime}\times
m^{\prime}}$ yields $f^{n^{\prime}\times m^{\prime}}\left(  D\right)  =\left(
f\left(  d_{i,j}\right)  \right)  _{1\leq i\leq n^{\prime},\ 1\leq j\leq
m^{\prime}}$ (since $D=\left(  d_{i,j}\right)  _{1\leq i\leq n^{\prime
},\ 1\leq j\leq m^{\prime}}$).

Now, we have%
\begin{align*}
f^{n\times m}\left(  A\right)   &  =\left(  f\left(  a_{i,j}\right)  \right)
_{1\leq i\leq n,\ 1\leq j\leq m}\ \ \ \ \ \ \ \ \ \ \text{and}\\
f^{n\times m^{\prime}}\left(  B\right)   &  =\left(  f\left(  b_{i,j}\right)
\right)  _{1\leq i\leq n,\ 1\leq j\leq m^{\prime}}%
\ \ \ \ \ \ \ \ \ \ \text{and}\\
f^{n^{\prime}\times m}\left(  C\right)   &  =\left(  f\left(  c_{i,j}\right)
\right)  _{1\leq i\leq n^{\prime},\ 1\leq j\leq m}%
\ \ \ \ \ \ \ \ \ \ \text{and}\\
f^{n^{\prime}\times m^{\prime}}\left(  D\right)   &  =\left(  f\left(
d_{i,j}\right)  \right)  _{1\leq i\leq n^{\prime},\ 1\leq j\leq m^{\prime}}.
\end{align*}
Hence, the definition of the block matrix $\left(
\begin{array}
[c]{cc}%
f^{n\times m}\left(  A\right)  & f^{n\times m^{\prime}}\left(  B\right) \\
f^{n^{\prime}\times m}\left(  C\right)  & f^{n^{\prime}\times m^{\prime}%
}\left(  D\right)
\end{array}
\right)  $\ \ \ \ \footnote{or, more precisely, the equality \cite[(436)]%
{detnotes} (applied to $\mathbb{M}$, $f^{n\times m}\left(  A\right)  $,
$f\left(  a_{i,j}\right)  $, $f^{n\times m^{\prime}}\left(  B\right)  $,
$f\left(  b_{i,j}\right)  $, $f^{n^{\prime}\times m}\left(  C\right)  $,
$f\left(  c_{i,j}\right)  $, $f^{n^{\prime}\times m^{\prime}}\left(  D\right)
$ and $f\left(  d_{i,j}\right)  $ instead of $\mathbb{K}$, $A$, $a_{i,j}$,
$B$, $b_{i,j}$, $C$, $c_{i,j}$, $D$ and $d_{i,j}$)} yields%
\begin{align*}
&  \left(
\begin{array}
[c]{cc}%
f^{n\times m}\left(  A\right)  & f^{n\times m^{\prime}}\left(  B\right) \\
f^{n^{\prime}\times m}\left(  C\right)  & f^{n^{\prime}\times m^{\prime}%
}\left(  D\right)
\end{array}
\right) \\
&  =\left(
\begin{cases}
f\left(  a_{i,j}\right)  , & \text{if }i\leq n\text{ and }j\leq m;\\
f\left(  b_{i,j-m}\right)  , & \text{if }i\leq n\text{ and }j>m;\\
f\left(  c_{i-n,j}\right)  , & \text{if }i>n\text{ and }j\leq m;\\
f\left(  d_{i-n,j-m}\right)  , & \text{if }i>n\text{ and }j>m
\end{cases}
\right)  _{1\leq i\leq n+n^{\prime},\ 1\leq j\leq m+m^{\prime}}.
\end{align*}
Comparing this with (\ref{pf.lem.functor.block2x2.L=}), we obtain%
\[
f^{\left(  n+n^{\prime}\right)  \times\left(  m+m^{\prime}\right)  }\left(
\left(
\begin{array}
[c]{cc}%
A & B\\
C & D
\end{array}
\right)  \right)  =\left(
\begin{array}
[c]{cc}%
f^{n\times m}\left(  A\right)  & f^{n\times m^{\prime}}\left(  B\right) \\
f^{n^{\prime}\times m}\left(  C\right)  & f^{n^{\prime}\times m^{\prime}%
}\left(  D\right)
\end{array}
\right)  .
\]
This proves Lemma \ref{lem.functor.block2x2}.
\end{proof}
\end{verlong}

Next, let us recall a result from \cite{detnotes} (a version of the Schur
complement theorem):

\begin{proposition}
\label{prop.block2x2.VB+WD}Let $n\in\mathbb{N}$ and $m\in\mathbb{N}$. Let
$A\in\mathbb{K}^{n\times n}$, $B\in\mathbb{K}^{n\times m}$, $C\in
\mathbb{K}^{m\times n}$ and $D\in\mathbb{K}^{m\times m}$. Furthermore, let
$W\in\mathbb{K}^{m\times m}$ and $V\in\mathbb{K}^{m\times n}$ be such that
$VA=-WC$. Then,%
\[
\det W\cdot\det\left(
\begin{array}
[c]{cc}%
A & B\\
C & D
\end{array}
\right)  =\det A\cdot\det\left(  VB+WD\right)  .
\]

\end{proposition}

Proposition \ref{prop.block2x2.VB+WD} appears (with proof) in \cite[Exercise
6.35]{detnotes}, so we will not prove it here.

Let us next prove the particular case of Theorem \ref{thm.block2x2.detAD-CB}
in which we assume $\det A$ to be regular:

\begin{lemma}
\label{lem.block2x2.detAD-CB.reg}Let $n\in\mathbb{N}$. Let $A$, $B$, $C$ and
$D$ be four $n\times n$-matrices such that $AC=CA$. Assume that the element
$\det A$ of $\mathbb{K}$ is regular. Then, the $\left(  2n\right)
\times\left(  2n\right)  $-matrix $\left(
\begin{array}
[c]{cc}%
A & B\\
C & D
\end{array}
\right)  $ satisfies%
\[
\det\left(
\begin{array}
[c]{cc}%
A & B\\
C & D
\end{array}
\right)  =\det\left(  AD-CB\right)  .
\]

\end{lemma}

\begin{proof}
[Proof of Lemma \ref{lem.block2x2.detAD-CB.reg}.]The matrix $\left(
\begin{array}
[c]{cc}%
A & B\\
C & D
\end{array}
\right)  $ is an $\left(  n+n\right)  \times\left(  n+n\right)  $-matrix (by
its definition), i.e., a $\left(  2n\right)  \times\left(  2n\right)  $-matrix
(since $n+n=2n$).

We have $CA=-\left(  -A\right)  C$ (since $-\left(  -A\right)  C=AC=CA$).
Thus, Proposition \ref{prop.block2x2.VB+WD} (applied to $V=C$ and $W=-A$)
yields%
\begin{align}
\det\left(  -A\right)  \cdot\det\left(
\begin{array}
[c]{cc}%
A & B\\
C & D
\end{array}
\right)   &  =\det A\cdot\det\left(  \underbrace{CB+\left(  -A\right)
D}_{=CB-AD}\right) \nonumber\\
&  =\det A\cdot\det\left(  CB-AD\right)  .
\label{pf.lem.block2x2.detAD-CB.reg.1}%
\end{align}
But (\ref{pf.lem.adj.adjadj.reg.lam1}) (applied to $-1$ and $CB-AD$ instead of
$\lambda$ and $C$) yields
\begin{equation}
\det\left(  \left(  -1\right)  \left(  CB-AD\right)  \right)  =\left(
-1\right)  ^{n}\det\left(  CB-AD\right)  .
\label{pf.lem.block2x2.detAD-CB.reg.2}%
\end{equation}
Also, (\ref{pf.lem.adj.adjadj.reg.lam1}) (applied to $-1$ and $-A$ instead of
$\lambda$ and $C$) yields $\det\left(  \left(  -1\right)  \left(  -A\right)
\right)  =\left(  -1\right)  ^{n}\det\left(  -A\right)  $. In view of $\left(
-1\right)  \left(  -A\right)  =A$, this rewrites as $\det A=\left(  -1\right)
^{n}\det\left(  -A\right)  $. Hence,%
\begin{align*}
&  \underbrace{\det A}_{=\left(  -1\right)  ^{n}\det\left(  -A\right)  }%
\cdot\det\left(
\begin{array}
[c]{cc}%
A & B\\
C & D
\end{array}
\right) \\
&  =\left(  -1\right)  ^{n}\underbrace{\det\left(  -A\right)  \cdot\det\left(
\begin{array}
[c]{cc}%
A & B\\
C & D
\end{array}
\right)  }_{\substack{=\det A\cdot\det\left(  CB-AD\right)  \\\text{(by
(\ref{pf.lem.block2x2.detAD-CB.reg.1}))}}}\\
&  =\left(  -1\right)  ^{n}\det A\cdot\det\left(  CB-AD\right)  =\det
A\cdot\underbrace{\left(  -1\right)  ^{n}\det\left(  CB-AD\right)
}_{\substack{=\det\left(  \left(  -1\right)  \left(  CB-AD\right)  \right)
\\\text{(by (\ref{pf.lem.block2x2.detAD-CB.reg.2}))}}}\\
&  =\det A\cdot\det\left(  \underbrace{\left(  -1\right)  \left(
CB-AD\right)  }_{=AD-CB}\right)  =\det A\cdot\det\left(  AD-CB\right)  .
\end{align*}
Lemma \ref{lem.regular.cancel} (applied to $\mathbb{K}$, $\det A$,
$\det\left(
\begin{array}
[c]{cc}%
A & B\\
C & D
\end{array}
\right)  $ and $\det\left(  AD-CB\right)  $ instead of $\mathbb{A}$, $a$, $b$
and $c$) thus yields $\det\left(
\begin{array}
[c]{cc}%
A & B\\
C & D
\end{array}
\right)  =\det\left(  AD-CB\right)  $ (since the element $\det A$ of
$\mathbb{K}$ is regular). This proves Lemma \ref{lem.block2x2.detAD-CB.reg}.
\end{proof}

We are now ready to prove Theorem \ref{thm.block2x2.detAD-CB}:

\begin{proof}
[Proof of Theorem \ref{thm.block2x2.detAD-CB}.]The matrix $\left(
\begin{array}
[c]{cc}%
A & B\\
C & D
\end{array}
\right)  $ is a $\left(  2n\right)  \times\left(  2n\right)  $-matrix. (This
is proven in the same way as in our proof of Lemma
\ref{lem.block2x2.detAD-CB.reg}.)

Define the $\mathbb{K}$-algebra homomorphism $\varepsilon:\mathbb{K}\left[
t\right]  \rightarrow\mathbb{K}$ as in Proposition \ref{prop.tInA.ev}. Thus,
$\varepsilon$ is a ring homomorphism. Hence, Proposition
\ref{prop.functor.clas} \textbf{(a)} (applied to $\mathbb{L}=\mathbb{K}\left[
t\right]  $, $\mathbb{M}=\mathbb{K}$ and $m=n$) shows that the map
$\varepsilon^{n\times n}:\left(  \mathbb{K}\left[  t\right]  \right)
^{n\times n}\rightarrow\mathbb{K}^{n\times n}$ is a homomorphism of additive groups.

Recall that every $n\times n$-matrix in $\mathbb{K}^{n\times n}$ can be
considered as a matrix in $\left(  \mathbb{K}\left[  t\right]  \right)
^{n\times n}$. In other words, for each $F\in\mathbb{K}^{n\times n}$, we can
consider $F$ as a matrix in $\left(  \mathbb{K}\left[  t\right]  \right)
^{n\times n}$; therefore, $\varepsilon^{n\times n}\left(  F\right)  $ is
well-defined. We have
\begin{equation}
\varepsilon^{n\times n}\left(  F\right)  =F\ \ \ \ \ \ \ \ \ \ \text{for every
}F\in\mathbb{K}^{n\times n}. \label{pf.thm.block2x2.detAD-CB.FF}%
\end{equation}

\begin{vershort}
\noindent(In fact, the proof of (\ref{pf.thm.block2x2.detAD-CB.FF}) is
identical with the proof of (\ref{pf.lem.adj.poly.-A.short.eF}) we gave above.)
\end{vershort}

\begin{verlong}
\noindent(In fact, the proof of (\ref{pf.thm.block2x2.detAD-CB.FF}) is
identical with the proof of (\ref{pf.lem.adj.poly.-A.eF}) we gave above.)
\end{verlong}

Let $\widetilde{A}$ be the matrix $tI_{n}+A\in\left(  \mathbb{K}\left[
t\right]  \right)  ^{n\times n}$. Thus, $\widetilde{A}=tI_{n}+A$. Applying the
map $\varepsilon^{n\times n}$ to both sides of this equality, we find
$\varepsilon^{n\times n}\left(  \widetilde{A}\right)  =\varepsilon^{n\times
n}\left(  tI_{n}+A\right)  =A$ (by Proposition \ref{prop.tInA.ev} \textbf{(c)}).

Corollary \ref{cor.tInA.reg} shows that the element $\det\left(
tI_{n}+A\right)  $ of $\mathbb{K}\left[  t\right]  $ is regular. In other
words, the element $\det\widetilde{A}$ of $\mathbb{K}\left[  t\right]  $ is
regular (since $\widetilde{A}=tI_{n}+A$).

Let us consider the matrix $B\in\mathbb{K}^{n\times n}$ as a matrix in
$\left(  \mathbb{K}\left[  t\right]  \right)  ^{n\times n}$ (since every
$n\times n$-matrix in $\mathbb{K}^{n\times n}$ can be considered as a matrix
in $\left(  \mathbb{K}\left[  t\right]  \right)  ^{n\times n}$). Similarly,
let us consider the matrices $C$ and $D$ as matrices in $\left(
\mathbb{K}\left[  t\right]  \right)  ^{n\times n}$.

\begin{vershort}
Notice that (\ref{pf.thm.block2x2.detAD-CB.FF}) (applied to $F=B$) yields
$\varepsilon^{n\times n}\left(  B\right)  =B$. Similarly, $\varepsilon
^{n\times n}\left(  C\right)  =C$ and $\varepsilon^{n\times n}\left(
D\right)  =D$.
\end{vershort}

\begin{verlong}
Notice that (\ref{pf.thm.block2x2.detAD-CB.FF}) (applied to $F=B$) yields
$\varepsilon^{n\times n}\left(  B\right)  =B$. Also,
(\ref{pf.thm.block2x2.detAD-CB.FF}) (applied to $F=C$) yields $\varepsilon
^{n\times n}\left(  C\right)  =C$. Also, (\ref{pf.thm.block2x2.detAD-CB.FF})
(applied to $F=D$) yields $\varepsilon^{n\times n}\left(  D\right)  =D$.
\end{verlong}

Now,%
\begin{align*}
\underbrace{\widetilde{A}}_{=tI_{n}+A}C  &  =\left(  tI_{n}+A\right)
C=t\underbrace{I_{n}C}_{=C=CI_{n}}+\underbrace{AC}_{=CA}=\underbrace{tCI_{n}%
}_{=C\cdot tI_{n}}+\,CA\\
&  =C\cdot tI_{n}+CA=C\underbrace{\left(  tI_{n}+A\right)  }_{=\widetilde{A}%
}=C\widetilde{A}.
\end{align*}
Thus, Lemma \ref{lem.block2x2.detAD-CB.reg} (applied to $\mathbb{K}\left[
t\right]  $ and $\widetilde{A}$ instead of $\mathbb{K}$ and $A$) yields
\[
\det\left(
\begin{array}
[c]{cc}%
\widetilde{A} & B\\
C & D
\end{array}
\right)  =\det\left(  \widetilde{A}D-CB\right)
\]
(since the element $\det\widetilde{A}$ of $\mathbb{K}\left[  t\right]  $ is
regular). Applying the map $\varepsilon$ to both sides of this equality, we
find%
\begin{equation}
\varepsilon\left(  \det\left(
\begin{array}
[c]{cc}%
\widetilde{A} & B\\
C & D
\end{array}
\right)  \right)  =\varepsilon\left(  \det\left(  \widetilde{A}D-CB\right)
\right)  . \label{pf.thm.block2x2.detAD-CB.2}%
\end{equation}
But Proposition \ref{prop.functor.det} \textbf{(a)} (applied to $\mathbb{K}%
\left[  t\right]  $, $\mathbb{K}$, $\varepsilon$ and $\widetilde{A}D-CB$
instead of $\mathbb{L}$, $\mathbb{M}$, $f$ and $A$) yields
\[
\varepsilon\left(  \det\left(  \widetilde{A}D-CB\right)  \right)  =\det\left(
\varepsilon^{n\times n}\left(  \widetilde{A}D-CB\right)  \right)  .
\]
In view of%
\begin{align*}
\varepsilon^{n\times n}\left(  \widetilde{A}D-CB\right)   &
=\underbrace{\varepsilon^{n\times n}\left(  \widetilde{A}D\right)
}_{\substack{=\varepsilon^{n\times n}\left(  \widetilde{A}\right)
\cdot\varepsilon^{n\times n}\left(  D\right)  \\\text{(by Proposition
\ref{prop.functor.clas} \textbf{(b)}}\\\text{(applied to }\mathbb{K}\left[
t\right]  \text{, }\mathbb{K}\text{, }\varepsilon\text{, }n\text{, }n\text{,
}\widetilde{A}\text{ and }D\\\text{instead of }\mathbb{L}\text{, }%
\mathbb{M}\text{, }f\text{, }m\text{, }p\text{, }A\text{ and }B\text{))}%
}}-\underbrace{\varepsilon^{n\times n}\left(  CB\right)  }%
_{\substack{=\varepsilon^{n\times n}\left(  C\right)  \cdot\varepsilon
^{n\times n}\left(  B\right)  \\\text{(by Proposition \ref{prop.functor.clas}
\textbf{(b)}}\\\text{(applied to }\mathbb{K}\left[  t\right]  \text{,
}\mathbb{K}\text{, }\varepsilon\text{, }n\text{, }n\text{, }C\text{ and
}B\\\text{instead of }\mathbb{L}\text{, }\mathbb{M}\text{, }f\text{, }m\text{,
}p\text{, }A\text{ and }B\text{))}}}\\
&  \ \ \ \ \ \ \ \ \ \ \left(
\begin{array}
[c]{c}%
\text{since the map }\varepsilon^{n\times n}\text{ is a homomorphism}\\
\text{of additive groups}%
\end{array}
\right) \\
&  =\underbrace{\varepsilon^{n\times n}\left(  \widetilde{A}\right)  }%
_{=A}\cdot\underbrace{\varepsilon^{n\times n}\left(  D\right)  }%
_{=D}-\underbrace{\varepsilon^{n\times n}\left(  C\right)  }_{=C}%
\cdot\underbrace{\varepsilon^{n\times n}\left(  B\right)  }_{\substack{=B}}\\
&  =AD-CB,
\end{align*}
this becomes%
\begin{align}
\varepsilon\left(  \det\left(  \widetilde{A}D-CB\right)  \right)   &
=\det\left(  \underbrace{\varepsilon^{n\times n}\left(  \widetilde{A}%
D-CB\right)  }_{=AD-CB}\right) \nonumber\\
&  =\det\left(  AD-CB\right)  . \label{pf.thm.block2x2.detAD-CB.3}%
\end{align}

But Proposition \ref{prop.functor.det} \textbf{(a)} (applied to $\mathbb{K}%
\left[  t\right]  $, $\mathbb{K}$, $\varepsilon$, $n+n$ and $\left(
\begin{array}
[c]{cc}%
\widetilde{A} & B\\
C & D
\end{array}
\right)  $ instead of $\mathbb{L}$, $\mathbb{M}$, $f$, $n$ and $A$) yields%
\begin{equation}
\varepsilon\left(  \det\left(
\begin{array}
[c]{cc}%
\widetilde{A} & B\\
C & D
\end{array}
\right)  \right)  =\det\left(  \varepsilon^{\left(  n+n\right)  \times\left(
n+n\right)  }\left(  \left(
\begin{array}
[c]{cc}%
\widetilde{A} & B\\
C & D
\end{array}
\right)  \right)  \right)  . \label{pf.thm.block2x2.detAD-CB.4}%
\end{equation}

On the other hand, Lemma \ref{lem.functor.block2x2} (applied to $\mathbb{K}%
\left[  t\right]  $, $\mathbb{K}$, $\varepsilon$, $n$, $n$, $n$ and
$\widetilde{A}$ instead of $\mathbb{L}$, $\mathbb{M}$, $f$, $n^{\prime}$, $m$,
$m^{\prime}$ and $A$) yields
\[
\varepsilon^{\left(  n+n\right)  \times\left(  n+n\right)  }\left(  \left(
\begin{array}
[c]{cc}%
\widetilde{A} & B\\
C & D
\end{array}
\right)  \right)  =\left(
\begin{array}
[c]{cc}%
\varepsilon^{n\times n}\left(  \widetilde{A}\right)  & \varepsilon^{n\times
n}\left(  B\right) \\
\varepsilon^{n\times n}\left(  C\right)  & \varepsilon^{n\times n}\left(
D\right)
\end{array}
\right)  =\left(
\begin{array}
[c]{cc}%
A & B\\
C & D
\end{array}
\right)
\]
(since $\varepsilon^{n\times n}\left(  \widetilde{A}\right)  =A$ and
$\varepsilon^{n\times n}\left(  B\right)  =B$ and $\varepsilon^{n\times
n}\left(  C\right)  =C$ and $\varepsilon^{n\times n}\left(  D\right)  =D$).
Taking determinants on both sides of this equality, we find%
\[
\det\left(  \varepsilon^{\left(  n+n\right)  \times\left(  n+n\right)
}\left(  \left(
\begin{array}
[c]{cc}%
\widetilde{A} & B\\
C & D
\end{array}
\right)  \right)  \right)  =\det\left(
\begin{array}
[c]{cc}%
A & B\\
C & D
\end{array}
\right)  .
\]
Hence,%
\begin{align*}
\det\left(
\begin{array}
[c]{cc}%
A & B\\
C & D
\end{array}
\right)   &  =\det\left(  \varepsilon^{\left(  n+n\right)  \times\left(
n+n\right)  }\left(  \left(
\begin{array}
[c]{cc}%
\widetilde{A} & B\\
C & D
\end{array}
\right)  \right)  \right) \\
&  =\varepsilon\left(  \det\left(
\begin{array}
[c]{cc}%
\widetilde{A} & B\\
C & D
\end{array}
\right)  \right)  \ \ \ \ \ \ \ \ \ \ \left(  \text{by
(\ref{pf.thm.block2x2.detAD-CB.4})}\right) \\
&  =\varepsilon\left(  \det\left(  \widetilde{A}D-CB\right)  \right)
\ \ \ \ \ \ \ \ \ \ \left(  \text{by (\ref{pf.thm.block2x2.detAD-CB.2}%
)}\right) \\
&  =\det\left(  AD-CB\right)  \ \ \ \ \ \ \ \ \ \ \left(  \text{by
(\ref{pf.thm.block2x2.detAD-CB.3})}\right)  .
\end{align*}
This completes the proof of Theorem \ref{thm.block2x2.detAD-CB}.
\end{proof}

Theorem \ref{thm.block2x2.detAD-CB} has an analogue for $m\times m$ block
matrices (see \cite[Theorem 1]{Silvest} or \cite[Chapter II, Lemma 5.30 and
Exercise 26]{LomQui16}), which too is proved using our \textquotedblleft%
$tI_{n}+A$\textquotedblright\ trick.

\subsection{The trace of the adjugate}

The following neat result follows so easily from Theorem \ref{thm.adj.poly}
and Theorem \ref{thm.TCH} that it would be strange not to mention it:

\begin{theorem}
\label{thm.adj.Tr}Let $n\in\mathbb{N}$. Let $A\in\mathbb{K}^{n\times n}$. For
every $j\in\mathbb{Z}$, define an element $c_{j}\in\mathbb{K}$ by
$c_{j}=\left[  t^{n-j}\right]  \chi_{A}$. Then,
\[
\operatorname*{Tr}\left(  \operatorname*{adj}A\right)  =\left(  -1\right)
^{n-1}c_{n-1}=\left(  -1\right)  ^{n-1}\left[  t^{1}\right]  \chi_{A}.
\]

\end{theorem}

In other words, the trace of the adjugate $\operatorname*{adj}A$ of an
$n\times n$-matrix is the coefficient of $t$ in the characteristic polynomial
$\chi_{A}$.

\begin{proof}
[Proof of Theorem \ref{thm.adj.Tr}.]The definition of $c_{n-1}$ yields
$c_{n-1}=\left[  t^{n-\left(  n-1\right)  }\right]  \chi_{A}=\left[
t^{1}\right]  \chi_{A}$ (since $n-\left(  n-1\right)  =1$).

It is easy to see that Theorem \ref{thm.adj.Tr} holds for $n=0$%
\ \ \ \ \footnote{\textit{Proof.} Assume that $n=0$. Thus, $1>0=n$. But
Corollary \ref{cor.chiA} \textbf{(a)} yields that $\chi_{A}\in\mathbb{K}%
\left[  t\right]  $ is a polynomial of degree $\leq n$ in $t$. Hence, $\left[
t^{m}\right]  \chi_{A}=0$ for every integer $m>n$. Applying this to $m=1$, we
obtain $\left[  t^{1}\right]  \chi_{A}=0$ (since $1>n$). Also,
$\operatorname*{adj}A$ is an $n\times n$-matrix, and thus a $0\times0$-matrix
(since $n=0$). Hence, $\operatorname*{Tr}\left(  \operatorname*{adj}A\right)
=0$ (since the trace of a $0\times0$-matrix is $0$). Comparing this with
$\left(  -1\right)  ^{n-1}\underbrace{c_{n-1}}_{=\left[  t^{1}\right]
\chi_{A}=0}=0$, we obtain $\operatorname*{Tr}\left(  \operatorname*{adj}%
A\right)  =\left(  -1\right)  ^{n-1}\underbrace{c_{n-1}}_{=\left[
t^{1}\right]  \chi_{A}}=\left(  -1\right)  ^{n-1}\left[  t^{1}\right]
\chi_{A}$. Hence, we have proven Theorem \ref{thm.adj.Tr} under the assumption
that $n=0$.}. Thus, for the rest of this proof, we can WLOG assume that we
don't have $n=0$. Assume this. Hence, $n\neq0$, so that $n\geq1$ (since
$n\in\mathbb{N}$). Therefore, $n-1\in\mathbb{N}$.

Thus, Theorem \ref{thm.TCH} (applied to $k=n-1$) yields%
\[
\left(  n-1\right)  c_{n-1}+\sum_{i=1}^{n-1}\operatorname*{Tr}\left(
A^{i}\right)  c_{n-1-i}=0.
\]
Subtracting $\left(  n-1\right)  c_{n-1}$ from both sides of this equation, we
obtain%
\begin{equation}
\sum_{i=1}^{n-1}\operatorname*{Tr}\left(  A^{i}\right)  c_{n-1-i}=-\left(
n-1\right)  c_{n-1}. \label{pf.thm.adj.Tr.1}%
\end{equation}

But Theorem \ref{thm.adj.poly} yields%
\[
\operatorname*{adj}A=\left(  -1\right)  ^{n-1}\sum_{i=0}^{n-1}c_{n-1-i}A^{i}.
\]
Applying the map $\operatorname*{Tr}:\mathbb{K}^{n\times n}\rightarrow
\mathbb{K}$ to both sides of this equality, we obtain%
\begin{align}
\operatorname*{Tr}\left(  \operatorname*{adj}A\right)   &  =\operatorname*{Tr}%
\left(  \left(  -1\right)  ^{n-1}\sum_{i=0}^{n-1}c_{n-1-i}A^{i}\right)
\nonumber\\
&  =\left(  -1\right)  ^{n-1}\sum_{i=0}^{n-1}c_{n-1-i}\operatorname*{Tr}%
\left(  A^{i}\right)  \label{pf.thm.adj.Tr.2}%
\end{align}
(since the map $\operatorname*{Tr}:\mathbb{K}^{n\times n}\rightarrow
\mathbb{K}$ is $\mathbb{K}$-linear). But $n-1\geq0$ (since $n\geq1$);
therefore, $0\in\left\{  0,1,\ldots,n-1\right\}  $. Hence, we can split off
the addend for $i=0$ from the sum $\sum_{i=0}^{n-1}c_{n-1-i}\operatorname*{Tr}%
\left(  A^{i}\right)  $. We thus obtain
\begin{align*}
\sum_{i=0}^{n-1}c_{n-1-i}\operatorname*{Tr}\left(  A^{i}\right)   &
=\underbrace{c_{n-1-0}}_{=c_{n-1}}\operatorname*{Tr}\left(  \underbrace{A^{0}%
}_{=I_{n}}\right)  +\sum_{i=1}^{n-1}\underbrace{c_{n-1-i}\operatorname*{Tr}%
\left(  A^{i}\right)  }_{=\operatorname*{Tr}\left(  A^{i}\right)  c_{n-1-i}}\\
&  =c_{n-1}\underbrace{\operatorname*{Tr}\left(  I_{n}\right)  }%
_{=n}+\underbrace{\sum_{i=1}^{n-1}\operatorname*{Tr}\left(  A^{i}\right)
c_{n-1-i}}_{\substack{=-\left(  n-1\right)  c_{n-1}\\\text{(by
(\ref{pf.thm.adj.Tr.1}))}}}=c_{n-1}n+\left(  -\left(  n-1\right)
c_{n-1}\right) \\
&  =\underbrace{\left(  n-\left(  n-1\right)  \right)  }_{=1}c_{n-1}=c_{n-1}.
\end{align*}
Hence, (\ref{pf.thm.adj.Tr.2}) becomes%
\[
\operatorname*{Tr}\left(  \operatorname*{adj}A\right)  =\left(  -1\right)
^{n-1}\underbrace{\sum_{i=0}^{n-1}c_{n-1-i}\operatorname*{Tr}\left(
A^{i}\right)  }_{=c_{n-1}}=\left(  -1\right)  ^{n-1}\underbrace{c_{n-1}%
}_{=\left[  t^{1}\right]  \chi_{A}}=\left(  -1\right)  ^{n-1}\left[
t^{1}\right]  \chi_{A}.
\]
This proves Theorem \ref{thm.adj.Tr}.
\end{proof}

\subsection{Yet another application to block matrices}

Let us show one further formula for determinants of certain block matrices
that can be proved using our \textquotedblleft replace $A$ by $tI_{n}%
+A$\textquotedblright\ strategy.

We will again use \cite[Definition 6.89]{detnotes} in this section. We shall
furthermore use the following notation:

\begin{definition}
\label{def.sol.det.rk1upd.scalar}If $B$ is any $1\times1$-matrix, then
$\operatorname*{ent}B$ will denote the $\left(  1,1\right)  $-th entry of $B$.
(This entry is, of course, the only entry of $B$. Thus, the $1\times1$-matrix
$B$ satisfies $B=\left(
\begin{array}
[c]{c}%
\operatorname*{ent}B
\end{array}
\right)  $.)
\end{definition}

We now claim the following:

\begin{theorem}
\label{thm.block.puqv}Let $n\in\mathbb{N}$ and $m\in\mathbb{N}$. Let
$A\in\mathbb{K}^{n\times n}$ and $D\in\mathbb{K}^{m\times m}$ be two square
matrices. Let $p\in\mathbb{K}^{n\times1}$ and $q\in\mathbb{K}^{m\times1}$ be
two column vectors. Let $v\in\mathbb{K}^{1\times m}$ and $u\in\mathbb{K}%
^{1\times n}$ be two row vectors. Then,%
\[
\det\left(
\begin{array}
[c]{cc}%
A & pv\\
qu & D
\end{array}
\right)  =\det A\cdot\det D-\operatorname*{ent}\left(  u\left(
\operatorname*{adj}A\right)  p\right)  \cdot\operatorname*{ent}\left(
v\left(  \operatorname*{adj}D\right)  q\right)  .
\]

\end{theorem}

\begin{example}
Let us see what Theorem \ref{thm.block.puqv} says in the case when $n=2$ and
$m=2$. Indeed, let $n=2$ and $m=2$ and%
\begin{align*}
A  &  =\left(
\begin{array}
[c]{cc}%
a_{1,1} & a_{1,2}\\
a_{2,1} & a_{2,2}%
\end{array}
\right)  \ \ \ \ \ \ \ \ \ \ \text{and}\ \ \ \ \ \ \ \ \ \ D=\left(
\begin{array}
[c]{cc}%
d_{1,1} & d_{1,2}\\
d_{2,1} & d_{2,2}%
\end{array}
\right)  \ \ \ \ \ \ \ \ \ \ \text{and}\\
p  &  =\left(
\begin{array}
[c]{c}%
p_{1}\\
p_{2}%
\end{array}
\right)  \ \ \ \ \ \ \ \ \ \ \text{and}\ \ \ \ \ \ \ \ \ \ q=\left(
\begin{array}
[c]{c}%
q_{1}\\
q_{2}%
\end{array}
\right)  \ \ \ \ \ \ \ \ \ \ \text{and}\\
v  &  =\left(
\begin{array}
[c]{cc}%
v_{1} & v_{2}%
\end{array}
\right)  \ \ \ \ \ \ \ \ \ \ \text{and}\ \ \ \ \ \ \ \ \ \ u=\left(
\begin{array}
[c]{cc}%
u_{1} & u_{2}%
\end{array}
\right)  .
\end{align*}
Then,
\[
pv=\left(
\begin{array}
[c]{cc}%
p_{1}v_{1} & p_{1}v_{2}\\
p_{2}v_{1} & p_{2}v_{2}%
\end{array}
\right)  \ \ \ \ \ \ \ \ \ \ \text{and}\ \ \ \ \ \ \ \ \ \ qu=\left(
\begin{array}
[c]{cc}%
q_{1}u_{1} & q_{1}u_{2}\\
q_{2}u_{1} & q_{2}u_{2}%
\end{array}
\right)  .
\]
Hence,%
\[
\left(
\begin{array}
[c]{cc}%
A & pv\\
qu & D
\end{array}
\right)  =\left(
\begin{array}
[c]{cccc}%
a_{1,1} & a_{1,2} & p_{1}v_{1} & p_{1}v_{2}\\
a_{2,1} & a_{2,2} & p_{2}v_{1} & p_{2}v_{2}\\
q_{1}u_{1} & q_{1}u_{2} & d_{1,1} & d_{1,2}\\
q_{2}u_{1} & q_{2}u_{2} & d_{2,1} & d_{2,2}%
\end{array}
\right)  .
\]
Hence, the claim of Theorem \ref{thm.block.puqv} rewrites as follows in our
case:
\begin{align*}
&  \det\left(
\begin{array}
[c]{cccc}%
a_{1,1} & a_{1,2} & p_{1}v_{1} & p_{1}v_{2}\\
a_{2,1} & a_{2,2} & p_{2}v_{1} & p_{2}v_{2}\\
q_{1}u_{1} & q_{1}u_{2} & d_{1,1} & d_{1,2}\\
q_{2}u_{1} & q_{2}u_{2} & d_{2,1} & d_{2,2}%
\end{array}
\right) \\
&  =\det\left(
\begin{array}
[c]{cc}%
a_{1,1} & a_{1,2}\\
a_{2,1} & a_{2,2}%
\end{array}
\right)  \cdot\det\left(
\begin{array}
[c]{cc}%
d_{1,1} & d_{1,2}\\
d_{2,1} & d_{2,2}%
\end{array}
\right) \\
&  \ \ \ \ \ \ \ \ \ \ -\operatorname*{ent}\left(  \left(
\begin{array}
[c]{cc}%
u_{1} & u_{2}%
\end{array}
\right)  \left(  \operatorname*{adj}\left(
\begin{array}
[c]{cc}%
a_{1,1} & a_{1,2}\\
a_{2,1} & a_{2,2}%
\end{array}
\right)  \right)  \left(
\begin{array}
[c]{c}%
p_{1}\\
p_{2}%
\end{array}
\right)  \right) \\
&  \ \ \ \ \ \ \ \ \ \ \ \ \ \ \ \ \ \ \ \ \cdot\operatorname*{ent}\left(
\left(
\begin{array}
[c]{cc}%
v_{1} & v_{2}%
\end{array}
\right)  \left(  \operatorname*{adj}\left(
\begin{array}
[c]{cc}%
d_{1,1} & d_{1,2}\\
d_{2,1} & d_{2,2}%
\end{array}
\right)  \right)  \left(
\begin{array}
[c]{c}%
q_{1}\\
q_{2}%
\end{array}
\right)  \right)  .
\end{align*}

\end{example}

In order to prove Theorem \ref{thm.block.puqv}, we will need one standard
result (known as the \textit{matrix determinant lemma}):

\begin{theorem}
\label{thm.sol.det.rk1upd.claim}Let $n\in\mathbb{N}$. Let $u$ be a column
vector with $n$ entries, and let $v$ be a row vector with $n$ entries. (Thus,
$uv$ is an $n\times n$-matrix, whereas $vu$ is a $1\times1$-matrix.) Let $A$
be an $n\times n$-matrix. Then,%
\[
\det\left(  A+uv\right)  =\det A+\operatorname*{ent}\left(  v\left(
\operatorname*{adj}A\right)  u\right)  .
\]

\end{theorem}

See \cite[Theorem 7.262]{detnotes} for a proof of Theorem
\ref{thm.sol.det.rk1upd.claim}.

We will furthermore need the following three trivial lemmas:

\begin{lemma}
\label{lem.ent.scalp}Let $C\in\mathbb{K}^{1\times1}$ be an $1\times1$-matrix.
Let $\lambda\in\mathbb{K}$. Then, $\operatorname*{ent}\left(  \lambda
C\right)  =\lambda\operatorname*{ent}C$.
\end{lemma}

\begin{verlong}
\begin{proof}
[Proof of Lemma \ref{lem.ent.scalp}.]Since $C$ is a $1\times1$-matrix, we have
$C=\left(
\begin{array}
[c]{c}%
\operatorname*{ent}C
\end{array}
\right)  $ (since $\operatorname*{ent}C$ denotes the $\left(  1,1\right)  $-th
entry of $C$). Hence, $\lambda C=\lambda\left(
\begin{array}
[c]{c}%
\operatorname*{ent}C
\end{array}
\right)  =\left(
\begin{array}
[c]{c}%
\lambda\operatorname*{ent}C
\end{array}
\right)  $. Now, the definition of $\operatorname*{ent}\left(  \lambda
C\right)  $ yields
\begin{align*}
\operatorname*{ent}\left(  \lambda C\right)   &  =\left(  \text{the }\left(
1,1\right)  \text{-th entry of }\lambda C\right) \\
&  =\lambda\operatorname*{ent}C\ \ \ \ \ \ \ \ \ \ \left(  \text{since
}\lambda C=\left(
\begin{array}
[c]{c}%
\lambda\operatorname*{ent}C
\end{array}
\right)  \right)  .
\end{align*}
This proves Lemma \ref{lem.ent.scalp}.
\end{proof}
\end{verlong}

\begin{lemma}
\label{lem.ent.prod}Let $B$ and $C$ be two $1\times1$-matrices. Then,
$\operatorname*{ent}\left(  BC\right)  =\operatorname*{ent}B\cdot
\operatorname*{ent}C$.
\end{lemma}

\begin{verlong}
\begin{proof}
[Proof of Lemma \ref{lem.ent.prod}.]Since $B$ is a $1\times1$-matrix, we have
$B=\left(
\begin{array}
[c]{c}%
\operatorname*{ent}B
\end{array}
\right)  $ (since $\operatorname*{ent}B$ denotes the $\left(  1,1\right)  $-th
entry of $B$). Similarly, $C=\left(
\begin{array}
[c]{c}%
\operatorname*{ent}C
\end{array}
\right)  $. Thus,
\[
BC=\left(
\begin{array}
[c]{c}%
\operatorname*{ent}B
\end{array}
\right)  \left(
\begin{array}
[c]{c}%
\operatorname*{ent}C
\end{array}
\right)  =\left(
\begin{array}
[c]{c}%
\sum_{k=1}^{1}\operatorname*{ent}B\cdot\operatorname*{ent}C
\end{array}
\right)
\]
(by the definition of the product of two matrices). Now, the definition of
$\operatorname*{ent}\left(  BC\right)  $ yields%
\begin{align*}
\operatorname*{ent}\left(  BC\right)   &  =\left(  \text{the }\left(
1,1\right)  \text{-th entry of }BC\right) \\
&  =\sum_{k=1}^{1}\operatorname*{ent}B\cdot\operatorname*{ent}%
C\ \ \ \ \ \ \ \ \ \ \left(  \text{since }BC=\left(
\begin{array}
[c]{c}%
\sum_{k=1}^{1}\operatorname*{ent}B\cdot\operatorname*{ent}C
\end{array}
\right)  \right) \\
&  =\operatorname*{ent}B\cdot\operatorname*{ent}C.
\end{align*}
This proves Lemma \ref{lem.ent.prod}.
\end{proof}
\end{verlong}

\begin{lemma}
\label{lem.ent.map}Let $\mathbb{L}$ and $\mathbb{M}$ be rings. Let
$f:\mathbb{L}\rightarrow\mathbb{M}$ be any map. Let $B\in\mathbb{L}^{1\times
1}$. Then, $\operatorname*{ent}\left(  f^{1\times1}\left(  B\right)  \right)
=f\left(  \operatorname*{ent}B\right)  $.
\end{lemma}

\begin{verlong}
\begin{proof}
[Proof of Lemma \ref{lem.ent.map}.]Since $B$ is a $1\times1$-matrix, we have
$B=\left(
\begin{array}
[c]{c}%
\operatorname*{ent}B
\end{array}
\right)  $ (since $\operatorname*{ent}B$ denotes the $\left(  1,1\right)  $-th
entry of $B$). Thus, $f^{1\times1}\left(  B\right)  =f^{1\times1}\left(
\left(
\begin{array}
[c]{c}%
\operatorname*{ent}B
\end{array}
\right)  \right)  =\left(
\begin{array}
[c]{c}%
f\left(  \operatorname*{ent}B\right)
\end{array}
\right)  $ (by the definition of $f^{1\times1}$). Now, the definition of
$\operatorname*{ent}\left(  f^{1\times1}\left(  B\right)  \right)  $ yields%
\begin{align*}
\operatorname*{ent}\left(  f^{1\times1}\left(  B\right)  \right)   &  =\left(
\text{the }\left(  1,1\right)  \text{-th entry of }f^{1\times1}\left(
B\right)  \right) \\
&  =f\left(  \operatorname*{ent}B\right)  \ \ \ \ \ \ \ \ \ \ \left(
\text{since }f^{1\times1}\left(  B\right)  =\left(
\begin{array}
[c]{c}%
f\left(  \operatorname*{ent}B\right)
\end{array}
\right)  \right)  .
\end{align*}
This proves Lemma \ref{lem.ent.map}.
\end{proof}
\end{verlong}

Finally, we will need a simple property of regular elements in a commutative ring:

\begin{lemma}
\label{lem.reg.power}Let $\mathbb{A}$ be a commutative ring. Let $a$ be a
regular element of $\mathbb{A}$. Let $m\in\mathbb{N}$. Then, $a^{m}$ is a
regular element of $\mathbb{A}$.
\end{lemma}

\begin{proof}
[Proof of Lemma \ref{lem.reg.power}.]The element $a$ is regular. In other
words,%
\begin{equation}
\text{every }x\in\mathbb{A}\text{ satisfying }ax=0\text{ satisfies }x=0
\label{pf.lem.reg.power.ass}%
\end{equation}
(by the definition of \textquotedblleft regular\textquotedblright).

Now, let $x\in\mathbb{A}$ satisfy $a^{m}x=0$. We shall show that $x=0$.
Indeed, we shall first prove that%
\begin{equation}
a^{m-i}x=0\ \ \ \ \ \ \ \ \ \ \text{for each }i\in\left\{  0,1,\ldots
,m\right\}  \text{.} \label{pf.lem.reg.power.1}%
\end{equation}

[\textit{Proof of (\ref{pf.lem.reg.power.1}):} We proceed by induction on $i$:

\textit{Induction base:} We have $a^{m-0}x=a^{m}x=0$ (by assumption). Hence,
(\ref{pf.lem.reg.power.1}) holds for $i=0$.

\textit{Induction step:} Let $j\in\left\{  1,2,\ldots,m\right\}  $. Assume
that (\ref{pf.lem.reg.power.1}) holds for $i=j-1$. We must show that
(\ref{pf.lem.reg.power.1}) holds for $i=j$ as well.

We have assumed that (\ref{pf.lem.reg.power.1}) holds for $i=j-1$. In other
words, $a^{m-\left(  j-1\right)  }x=0$. In other words, $a^{m-j+1}x=0$ (since
$m-\left(  j-1\right)  =m-j+1$). In other words, $aa^{m-j}x=0$ (since
$a^{m-j+1}=aa^{m-j}$). Hence, (\ref{pf.lem.reg.power.ass}) (applied to
$a^{m-j}x$ instead of $x$) yields $a^{m-j}x=0$. In other words,
(\ref{pf.lem.reg.power.1}) holds for $i=j$. This completes the induction step.
Thus, (\ref{pf.lem.reg.power.1}) is proved by induction.] \medskip

Now, (\ref{pf.lem.reg.power.1}) (applied to $i=m$) yields $a^{m-m}x=0$. Since
$\underbrace{a^{m-m}}_{=a^{0}=1}x=x$, this rewrites as $x=0$.

Forget that we fixed $x$. We thus have shown that every $x\in\mathbb{A}$
satisfying $a^{m}x=0$ satisfies $x=0$. In other words, the element $a^{m}$ of
$\mathbb{A}$ is regular (by the definition of \textquotedblleft
regular\textquotedblright). This proves Lemma \ref{lem.reg.power}.
\end{proof}

Now, we can approach the proof of Theorem \ref{thm.block.puqv} using the same
technique as various theorems proved above. We begin by proving it in the case
when $\det A$ is regular:

\begin{lemma}
\label{lem.block.puqv-reg}Let $n\in\mathbb{N}$ and $m\in\mathbb{N}$. Let
$A\in\mathbb{K}^{n\times n}$ and $D\in\mathbb{K}^{m\times m}$ be two square
matrices. Assume that the element $\det A$ of $\mathbb{K}$ is regular. Let
$p\in\mathbb{K}^{n\times1}$ and $q\in\mathbb{K}^{m\times1}$ be two column
vectors. Let $v\in\mathbb{K}^{1\times m}$ and $u\in\mathbb{K}^{1\times n}$ be
two row vectors. Then,%
\[
\det\left(
\begin{array}
[c]{cc}%
A & pv\\
qu & D
\end{array}
\right)  =\det A\cdot\det D-\operatorname*{ent}\left(  u\left(
\operatorname*{adj}A\right)  p\right)  \cdot\operatorname*{ent}\left(
v\left(  \operatorname*{adj}D\right)  q\right)  .
\]

\end{lemma}

\begin{proof}
[Proof of Lemma \ref{lem.block.puqv-reg}.]Set $\lambda=\det A$. Thus, the
element $\lambda$ of $\mathbb{K}$ is regular (since the element $\det A$ of
$\mathbb{K}$ is regular). Hence, Lemma \ref{lem.reg.power} (applied to
$\mathbb{A}=\mathbb{K}$ and $a=\lambda$) shows that $\lambda^{m}$ is a regular
element of $\mathbb{K}$.

It is furthermore easy to see that%
\begin{equation}
\lambda\operatorname*{adj}\left(  \lambda D\right)  =\lambda^{m}%
\operatorname*{adj}D. \label{pf.lem.block.puqv-reg.ladjlD}%
\end{equation}

[\textit{Proof of (\ref{pf.lem.block.puqv-reg.ladjlD}):} If $m=0$, then
(\ref{pf.lem.block.puqv-reg.ladjlD}) holds for trivial
reasons\footnote{\textit{Proof.} Assume that $m=0$. The matrices
$\lambda\operatorname*{adj}\left(  \lambda D\right)  $ and $\lambda
^{m}\operatorname*{adj}D$ are $m\times m$-matrices (since $D$ is an $m\times
m$-matrix). In other words, the matrices $\lambda\operatorname*{adj}\left(
\lambda D\right)  $ and $\lambda^{m}\operatorname*{adj}D$ are $0\times
0$-matrices (since $m=0$). Hence, these two matrices $\lambda
\operatorname*{adj}\left(  \lambda D\right)  $ and $\lambda^{m}%
\operatorname*{adj}D$ are equal (since there exists only one $0\times
0$-matrix, and therefore any two $0\times0$-matrices are equal). In other
words, $\lambda\operatorname*{adj}\left(  \lambda D\right)  =\lambda
^{m}\operatorname*{adj}D$. Thus, we have proved
(\ref{pf.lem.block.puqv-reg.ladjlD}) under the assumption that $m=0$.}. Thus,
for the rest of this proof, we WLOG assume that we don't have $m=0$. Hence,
$m$ is a positive integer (since $m\in\mathbb{N}$). Therefore, Proposition
\ref{prop.adj.lamA} (applied to $m$ and $D$ instead of $n$ and $A$) yields
$\operatorname*{adj}\left(  \lambda D\right)  =\lambda^{m-1}%
\operatorname*{adj}D$. Hence, $\lambda\underbrace{\operatorname*{adj}\left(
\lambda D\right)  }_{=\lambda^{m-1}\operatorname*{adj}D}=\underbrace{\lambda
\lambda^{m-1}}_{=\lambda^{m}}\operatorname*{adj}D=\lambda^{m}%
\operatorname*{adj}D$. This proves (\ref{pf.lem.block.puqv-reg.ladjlD}).]
\medskip

Define two matrices $W\in\mathbb{K}^{m\times m}$ and $V\in\mathbb{K}^{m\times
n}$ by%
\[
W=\lambda I_{m}\ \ \ \ \ \ \ \ \ \ \text{and}%
\ \ \ \ \ \ \ \ \ \ V=-qu\operatorname*{adj}A.
\]
Then,%
\begin{align*}
\underbrace{V}_{=-qu\operatorname*{adj}A}A  &
=-qu\underbrace{\operatorname*{adj}A\cdot A}_{\substack{=\det A\cdot
I_{n}\\\text{(by Theorem \ref{thm.adj.inverse})}}}=-qu\underbrace{\det
A}_{\substack{=\lambda\\\text{(since }\lambda=\det A\text{)}}}\cdot
\,I_{n}=-qu\lambda I_{n}\\
&  =-\lambda quI_{n}=-\lambda qu=-Wqu
\end{align*}
(since $-\underbrace{W}_{=\lambda I_{m}}qu=-\lambda I_{m}qu=-\lambda qu$).
Hence, Proposition \ref{prop.block2x2.VB+WD} (applied to $B=pv$ and $C=qu$)
yields
\[
\det W\cdot\det\left(
\begin{array}
[c]{cc}%
A & pv\\
qu & D
\end{array}
\right)  =\det A\cdot\det\left(  Vpv+WD\right)  .
\]
In view of%
\begin{align*}
\det\underbrace{W}_{=\lambda I_{m}}  &  =\det\left(  \lambda I_{m}\right)
=\lambda^{m}\underbrace{\det\left(  I_{m}\right)  }_{=1}%
\ \ \ \ \ \ \ \ \ \ \left(
\begin{array}
[c]{c}%
\text{by (\ref{pf.lem.adj.adjadj.reg.lam1}), applied to }m\text{ and }I_{m}\\
\text{instead of }n\text{ and }C
\end{array}
\right) \\
&  =\lambda^{m}%
\end{align*}
and $\det A=\lambda$, we can rewrite this as%
\begin{equation}
\lambda^{m}\cdot\det\left(
\begin{array}
[c]{cc}%
A & pv\\
qu & D
\end{array}
\right)  =\lambda\cdot\det\left(  Vpv+WD\right)  .
\label{pf.lem.block.puqv-reg.3}%
\end{equation}

Let us define two $1\times1$-matrices $B$ and $C$ by
\[
B=u\left(  \operatorname*{adj}A\right)  p\in\mathbb{K}^{1\times1}%
\ \ \ \ \ \ \ \ \ \ \text{and}\ \ \ \ \ \ \ \ \ \ C=v\cdot\operatorname*{adj}%
\left(  \lambda D\right)  \cdot\left(  -q\right)  \in\mathbb{K}^{1\times1}.
\]

Next, we observe that%
\begin{align}
&  \underbrace{V}_{=-qu\operatorname*{adj}A}pv+\underbrace{W}_{=\lambda I_{m}%
}D\nonumber\\
&  =\underbrace{\left(  -qu\operatorname*{adj}A\right)  pv}_{=\left(
-q\right)  u\left(  \operatorname*{adj}A\right)  pv}+\,\lambda
\underbrace{I_{m}D}_{=D}=\left(  -q\right)  \underbrace{u\left(
\operatorname*{adj}A\right)  p}_{\substack{=B\\\text{(since }B=u\left(
\operatorname*{adj}A\right)  p\text{)}}}v+\lambda D=\left(  -q\right)
Bv+\lambda D\nonumber\\
&  =\lambda D+\left(  -q\right)  Bv=\lambda D+\left(  -q\right)  \left(
Bv\right)  . \label{pf.lem.block.puqv-reg.4}%
\end{align}
Note that $\lambda D\in\mathbb{K}^{m\times m}$ and $-q\in\mathbb{K}^{m\times
1}$ and $Bv\in\mathbb{K}^{1\times m}$. Hence, Theorem
\ref{thm.sol.det.rk1upd.claim} (applied to $m$, $-q$, $Bv$ and $D$ instead of
$n$, $u$, $v$ and $A$) yields%
\begin{align*}
&  \det\left(  \lambda D+\left(  -q\right)  \left(  Bv\right)  \right) \\
&  =\underbrace{\det\left(  \lambda D\right)  }_{\substack{=\lambda^{m}\det
D\\\text{(by (\ref{pf.lem.adj.adjadj.reg.lam1}), applied to }m\text{ and
}D\\\text{instead of }n\text{ and }A\text{)}}}+\operatorname*{ent}\left(
\underbrace{\left(  Bv\right)  \cdot\operatorname*{adj}\left(  \lambda
D\right)  \cdot\left(  -q\right)  }_{=B\left(  v\cdot\operatorname*{adj}%
\left(  \lambda D\right)  \cdot\left(  -q\right)  \right)  }\right) \\
&  =\lambda^{m}\det D+\operatorname*{ent}\left(  B\underbrace{\left(
v\cdot\operatorname*{adj}\left(  \lambda D\right)  \cdot\left(  -q\right)
\right)  }_{\substack{=C\\\text{(since }C=v\cdot\operatorname*{adj}\left(
\lambda D\right)  \cdot\left(  -q\right)  \text{)}}}\right)  =\lambda^{m}\det
D+\underbrace{\operatorname*{ent}\left(  BC\right)  }%
_{\substack{=\operatorname*{ent}B\cdot\operatorname*{ent}C\\\text{(by Lemma
\ref{lem.ent.prod})}}}\\
&  =\lambda^{m}\det D+\operatorname*{ent}B\cdot\operatorname*{ent}C.
\end{align*}
In view of (\ref{pf.lem.block.puqv-reg.4}), we can rewrite this as%
\[
\det\left(  Vpv+WD\right)  =\lambda^{m}\det D+\operatorname*{ent}%
B\cdot\operatorname*{ent}C.
\]
Thus, (\ref{pf.lem.block.puqv-reg.3}) becomes%
\begin{align}
\lambda^{m}\cdot\det\left(
\begin{array}
[c]{cc}%
A & pv\\
qu & D
\end{array}
\right)   &  =\lambda\cdot\underbrace{\det\left(  Vpv+WD\right)  }%
_{=\lambda^{m}\det D+\operatorname*{ent}B\cdot\operatorname*{ent}C}\nonumber\\
&  =\lambda\cdot\left(  \lambda^{m}\det D+\operatorname*{ent}B\cdot
\operatorname*{ent}C\right) \nonumber\\
&  =\lambda\lambda^{m}\det D+\underbrace{\lambda\operatorname*{ent}%
B\cdot\operatorname*{ent}C}_{=\lambda\operatorname*{ent}C\cdot
\operatorname*{ent}B}\nonumber\\
&  =\lambda\lambda^{m}\det D+\lambda\operatorname*{ent}C\cdot
\operatorname*{ent}B\nonumber\\
&  =\lambda\lambda^{m}\det D-\left(  -\lambda\right)  \operatorname*{ent}%
C\cdot\operatorname*{ent}B. \label{pf.lem.block.puqv-reg.6}%
\end{align}
However, Lemma \ref{lem.ent.scalp} (applied to $-\lambda$ instead of $\lambda
$) yields $\operatorname*{ent}\left(  \left(  -\lambda\right)  C\right)
=\left(  -\lambda\right)  \operatorname*{ent}C$. Thus,%
\begin{align*}
\left(  -\lambda\right)  \operatorname*{ent}C  &  =\operatorname*{ent}\left(
\left(  -\lambda\right)  \underbrace{C}_{=v\cdot\operatorname*{adj}\left(
\lambda D\right)  \cdot\left(  -q\right)  }\right)  =\operatorname*{ent}%
\left(  \underbrace{\left(  -\lambda\right)  v\cdot\operatorname*{adj}\left(
\lambda D\right)  \cdot\left(  -q\right)  }_{\substack{=\lambda v\cdot
\operatorname*{adj}\left(  \lambda D\right)  \cdot q\\=v\cdot\lambda
\operatorname*{adj}\left(  \lambda D\right)  \cdot q}}\right) \\
&  =\operatorname*{ent}\left(  v\cdot\underbrace{\lambda\operatorname*{adj}%
\left(  \lambda D\right)  }_{\substack{=\lambda^{m}\operatorname*{adj}%
D\\\text{(by (\ref{pf.lem.block.puqv-reg.ladjlD}))}}}\cdot q\right)
=\operatorname*{ent}\left(  \underbrace{v\cdot\left(  \lambda^{m}%
\operatorname*{adj}D\right)  \cdot q}_{=\lambda^{m}\cdot v\left(
\operatorname*{adj}D\right)  q}\right) \\
&  =\operatorname*{ent}\left(  \lambda^{m}\cdot v\left(  \operatorname*{adj}%
D\right)  q\right)  =\lambda^{m}\operatorname*{ent}\left(  v\left(
\operatorname*{adj}D\right)  q\right)
\end{align*}
(by Lemma \ref{lem.ent.scalp}, applied to $\lambda^{m}$ and $v\left(
\operatorname*{adj}D\right)  q$ instead of $\lambda$ and $C$). Thus,
(\ref{pf.lem.block.puqv-reg.6}) becomes%
\begin{align*}
\lambda^{m}\cdot\det\left(
\begin{array}
[c]{cc}%
A & pv\\
qu & D
\end{array}
\right)   &  =\lambda\lambda^{m}\det D-\underbrace{\left(  -\lambda\right)
\operatorname*{ent}C}_{=\lambda^{m}\operatorname*{ent}\left(  v\left(
\operatorname*{adj}D\right)  q\right)  }\cdot\operatorname*{ent}B\\
&  =\lambda\lambda^{m}\det D-\lambda^{m}\operatorname*{ent}\left(  v\left(
\operatorname*{adj}D\right)  q\right)  \cdot\operatorname*{ent}B\\
&  =\lambda^{m}\cdot\left(  \lambda\det D-\operatorname*{ent}\left(  v\left(
\operatorname*{adj}D\right)  q\right)  \cdot\operatorname*{ent}B\right)  .
\end{align*}
Since $\lambda^{m}$ is a regular element of $\mathbb{K}$, we can thus conclude
that%
\[
\det\left(
\begin{array}
[c]{cc}%
A & pv\\
qu & D
\end{array}
\right)  =\lambda\det D-\operatorname*{ent}\left(  v\left(
\operatorname*{adj}D\right)  q\right)  \cdot\operatorname*{ent}B
\]
(by Lemma \ref{lem.regular.cancel}, applied to $\mathbb{A}=\mathbb{K}$ and
$a=\lambda^{m}$ and $b=\det\left(
\begin{array}
[c]{cc}%
A & pv\\
qu & D
\end{array}
\right)  $ and $c=\lambda\det D-\operatorname*{ent}\left(  v\left(
\operatorname*{adj}D\right)  q\right)  \cdot\operatorname*{ent}B$). In view of
$\lambda=\det A$ and $B=u\left(  \operatorname*{adj}A\right)  p$, we can
rewrite this as
\begin{align*}
\det\left(
\begin{array}
[c]{cc}%
A & pv\\
qu & D
\end{array}
\right)   &  =\det A\cdot\det D-\operatorname*{ent}\left(  v\left(
\operatorname*{adj}D\right)  q\right)  \cdot\operatorname*{ent}\left(
u\left(  \operatorname*{adj}A\right)  p\right) \\
&  =\det A\cdot\det D-\operatorname*{ent}\left(  u\left(  \operatorname*{adj}%
A\right)  p\right)  \cdot\operatorname*{ent}\left(  v\left(
\operatorname*{adj}D\right)  q\right)  .
\end{align*}
This proves Lemma \ref{lem.block.puqv-reg}.
\end{proof}

We can now derive Theorem \ref{thm.block.puqv} from Lemma
\ref{lem.block.puqv-reg} by the same recipe as before:

\begin{proof}
[Proof of Theorem \ref{thm.block.puqv}.]Define the $\mathbb{K}$-algebra
homomorphism $\varepsilon:\mathbb{K}\left[  t\right]  \rightarrow\mathbb{K}$
as in Proposition \ref{prop.tInA.ev}. Thus, $\varepsilon$ is a ring homomorphism.

Recall that every $n\times n$-matrix in $\mathbb{K}^{n\times n}$ can be
considered as a matrix in $\left(  \mathbb{K}\left[  t\right]  \right)
^{n\times n}$. More generally, every $k\times\ell$-matrix in $\mathbb{K}%
^{k\times\ell}$ (for any nonnegative integers $k$ and $\ell$) can be
considered as a matrix in $\left(  \mathbb{K}\left[  t\right]  \right)
^{k\times\ell}$. In other words, for each $F\in\mathbb{K}^{k\times\ell}$, we
can consider $F$ as a matrix in $\left(  \mathbb{K}\left[  t\right]  \right)
^{k\times\ell}$; therefore, $\varepsilon^{k\times\ell}\left(  F\right)  $ is
well-defined. We have
\begin{equation}
\varepsilon^{k\times\ell}\left(  F\right)  =F\ \ \ \ \ \ \ \ \ \ \text{for
every }F\in\mathbb{K}^{k\times\ell}. \label{pf.thm.block.puqv.FF}%
\end{equation}

\begin{vershort}
\noindent(In fact, the proof of (\ref{pf.thm.block.puqv.FF}) is analogous to
the proof of (\ref{pf.lem.adj.poly.-A.short.eF}) we gave above.)
\end{vershort}

\begin{verlong}
\noindent\footnote{\textit{Proof of (\ref{pf.thm.block.puqv.FF}):} Let
$F\in\mathbb{K}^{k\times\ell}$. Write the $k\times\ell$-matrix $F$ in the form
$F=\left(  f_{i,j}\right)  _{1\leq i\leq k,\ 1\leq j\leq\ell}$. Then,
$f_{i,j}\in\mathbb{K}$ for every $\left(  i,j\right)  \in\left\{
1,2,\ldots,n\right\}  ^{2}$ (since $F\in\mathbb{K}^{k\times\ell}$). Hence,%
\begin{equation}
\varepsilon\left(  f_{i,j}\right)  =f_{i,j}\ \ \ \ \ \ \ \ \ \ \text{for every
}\left(  i,j\right)  \in\left\{  1,2,\ldots,n\right\}  ^{2}%
\end{equation}
(by (\ref{pf.lem.adj.poly.-A.eu}) (applied to $u=f_{i,j}$)).
\par
Applying the map $\varepsilon^{k\times\ell}$ to both sides of the equality
$F=\left(  f_{i,j}\right)  _{1\leq i\leq k,\ 1\leq j\leq\ell}$, we obtain%
\begin{align*}
\varepsilon^{k\times\ell}\left(  F\right)   &  =\varepsilon^{k\times\ell
}\left(  \left(  f_{i,j}\right)  _{1\leq i\leq k,\ 1\leq j\leq\ell}\right)
=\left(  \underbrace{\varepsilon\left(  f_{i,j}\right)  }_{\substack{=f_{i,j}%
\\\text{(by (\ref{pf.lem.adj.poly.-A.eF.pf.1}))}}}\right)  _{1\leq i\leq
k,\ 1\leq j\leq\ell}\ \ \ \ \ \ \ \ \ \ \left(  \text{by the definition of the
map }\varepsilon^{k\times\ell}\right) \\
&  =\left(  f_{i,j}\right)  _{1\leq i\leq k,\ 1\leq j\leq\ell}=F.
\end{align*}
This proves (\ref{pf.thm.block.puqv.FF}).}
\end{verlong}

Let $\widetilde{A}$ be the matrix $tI_{n}+A\in\left(  \mathbb{K}\left[
t\right]  \right)  ^{n\times n}$. Thus, $\widetilde{A}=tI_{n}+A$. Applying the
map $\varepsilon^{n\times n}$ to both sides of this equality, we find
$\varepsilon^{n\times n}\left(  \widetilde{A}\right)  =\varepsilon^{n\times
n}\left(  tI_{n}+A\right)  =A$ (by Proposition \ref{prop.tInA.ev} \textbf{(c)}).

Corollary \ref{cor.tInA.reg} shows that the element $\det\left(
tI_{n}+A\right)  $ of $\mathbb{K}\left[  t\right]  $ is regular. In other
words, the element $\det\widetilde{A}$ of $\mathbb{K}\left[  t\right]  $ is
regular (since $\widetilde{A}=tI_{n}+A$).

Let us consider the matrix $A\in\mathbb{K}^{n\times n}$ as a matrix in
$\left(  \mathbb{K}\left[  t\right]  \right)  ^{n\times n}$ (since every
$n\times n$-matrix in $\mathbb{K}^{n\times n}$ can be considered as a matrix
in $\left(  \mathbb{K}\left[  t\right]  \right)  ^{n\times n}$). Similarly,
let us consider the matrices $D\in\mathbb{K}^{m\times m}$, $p\in
\mathbb{K}^{n\times1}$, $q\in\mathbb{K}^{m\times1}$, $v\in\mathbb{K}^{1\times
m}$ and $u\in\mathbb{K}^{1\times n}$ as matrices in $\left(  \mathbb{K}\left[
t\right]  \right)  ^{m\times m}$, $\left(  \mathbb{K}\left[  t\right]
\right)  ^{n\times1}$, $\left(  \mathbb{K}\left[  t\right]  \right)
^{m\times1}$, $\left(  \mathbb{K}\left[  t\right]  \right)  ^{1\times m}$ and
$\left(  \mathbb{K}\left[  t\right]  \right)  ^{1\times n}$, respectively.

\begin{vershort}
Notice that (\ref{pf.thm.block.puqv.FF}) (applied to $k=n$ and $\ell=n$ and
$F=A$) yields $\varepsilon^{n\times n}\left(  A\right)  =A$. Similarly,
$\varepsilon^{m\times m}\left(  D\right)  =D$ and $\varepsilon^{n\times
1}\left(  p\right)  =p$ and $\varepsilon^{m\times1}\left(  q\right)  =q$ and
$\varepsilon^{1\times m}\left(  v\right)  =v$ and $\varepsilon^{1\times
n}\left(  u\right)  =u$ and $\varepsilon^{n\times m}\left(  pv\right)  =pv$
and $\varepsilon^{m\times n}\left(  qu\right)  =qu$.
\end{vershort}

\begin{verlong}
Notice that (\ref{pf.thm.block.puqv.FF}) (applied to $F=A$) yields
$\varepsilon^{n\times n}\left(  A\right)  =A$. Also,
(\ref{pf.thm.block.puqv.FF}) (applied to $k=m$ and $\ell=m$ and $F=D$) yields
$\varepsilon^{m\times m}\left(  D\right)  =D$. Also,
(\ref{pf.thm.block.puqv.FF}) (applied to $k=n$ and $\ell=1$ and $F=p$) yields
$\varepsilon^{n\times1}\left(  p\right)  =p$. Also,
(\ref{pf.thm.block.puqv.FF}) (applied to $k=m$ and $\ell=1$ and $F=q$) yields
$\varepsilon^{m\times1}\left(  q\right)  =q$. Also,
(\ref{pf.thm.block.puqv.FF}) (applied to $k=1$ and $\ell=m$ and $F=v$) yields
$\varepsilon^{1\times m}\left(  v\right)  =v$. Also, Also,
(\ref{pf.thm.block.puqv.FF}) (applied to $k=1$ and $\ell=n$ and $F=u$) yields
$\varepsilon^{1\times n}\left(  u\right)  =u$. Also,
(\ref{pf.thm.block.puqv.FF}) (applied to $k=n$ and $\ell=m$ and $F=pv$) yields
$\varepsilon^{n\times m}\left(  pv\right)  =pv$. Also,
(\ref{pf.thm.block.puqv.FF}) (applied to $k=m$ and $\ell=n$ and $F=qu$) yields
$\varepsilon^{m\times n}\left(  qu\right)  =qu$.
\end{verlong}

In the proof of Theorem \ref{thm.adj.adjadj}, we have already shown that
\[
\varepsilon^{n\times n}\left(  \operatorname*{adj}\widetilde{A}\right)
=\operatorname*{adj}A\ \ \ \ \ \ \ \ \ \ \text{and}%
\ \ \ \ \ \ \ \ \ \ \varepsilon\left(  \det\widetilde{A}\right)  =\det A.
\]

Next, we claim that%
\begin{equation}
\varepsilon\left(  \operatorname*{ent}\left(  u\left(  \operatorname*{adj}%
\widetilde{A}\right)  p\right)  \right)  =\operatorname*{ent}\left(  u\left(
\operatorname*{adj}A\right)  p\right)  . \label{pf.thm.block.puqv.ent1}%
\end{equation}

\begin{vershort}
[\textit{Proof of (\ref{pf.thm.block.puqv.ent1}):} Define the $1\times
1$-matrix $B=u\left(  \operatorname*{adj}\widetilde{A}\right)  p\in\left(
\mathbb{K}\left[  t\right]  \right)  ^{1\times1}$. Then, Lemma
\ref{lem.ent.map} (applied to $\mathbb{L}=\mathbb{K}\left[  t\right]  $ and
$\mathbb{M}=\mathbb{K}$ and $f=\varepsilon$) yields $\operatorname*{ent}%
\left(  \varepsilon^{1\times1}\left(  B\right)  \right)  =\varepsilon\left(
\operatorname*{ent}B\right)  $.

However, from $B=u\left(  \operatorname*{adj}\widetilde{A}\right)  p=u\left(
\left(  \operatorname*{adj}\widetilde{A}\right)  p\right)  $, we obtain%
\begin{align*}
\varepsilon^{1\times1}\left(  B\right)   &  =\varepsilon^{1\times1}\left(
u\left(  \left(  \operatorname*{adj}\widetilde{A}\right)  p\right)  \right) \\
&  =\underbrace{\varepsilon^{1\times n}\left(  u\right)  }_{=u}\cdot
\underbrace{\varepsilon^{n\times1}\left(  \left(  \operatorname*{adj}%
\widetilde{A}\right)  p\right)  }_{\substack{=\varepsilon^{n\times n}\left(
\operatorname*{adj}\widetilde{A}\right)  \cdot\varepsilon^{n\times1}\left(
p\right)  \\\text{(by Theorem \ref{prop.functor.clas} \textbf{(c)})}%
}}\ \ \ \ \ \ \ \ \ \ \left(  \text{by Theorem \ref{prop.functor.clas}
\textbf{(c)}}\right) \\
&  =u\cdot\underbrace{\varepsilon^{n\times n}\left(  \operatorname*{adj}%
\widetilde{A}\right)  }_{=\operatorname*{adj}A}\cdot\underbrace{\varepsilon
^{n\times1}\left(  p\right)  }_{=p}=u\left(  \operatorname*{adj}A\right)  p.
\end{align*}
In view of this, we can rewrite the equality $\operatorname*{ent}\left(
\varepsilon^{1\times1}\left(  B\right)  \right)  =\varepsilon\left(
\operatorname*{ent}B\right)  $ (which we have proved in the previous
paragraph) as
\[
\operatorname*{ent}\left(  u\left(  \operatorname*{adj}A\right)  p\right)
=\varepsilon\left(  \operatorname*{ent}B\right)  =\varepsilon\left(
\operatorname*{ent}\left(  u\left(  \operatorname*{adj}\widetilde{A}\right)
p\right)  \right)
\]
(since $B=u\left(  \operatorname*{adj}\widetilde{A}\right)  p$). This proves
(\ref{pf.thm.block.puqv.ent1}).] \medskip
\end{vershort}

\begin{verlong}
[\textit{Proof of (\ref{pf.thm.block.puqv.ent1}):} Define the $1\times
1$-matrix $B=u\left(  \operatorname*{adj}\widetilde{A}\right)  p\in\left(
\mathbb{K}\left[  t\right]  \right)  ^{1\times1}$. Then, Lemma
\ref{lem.ent.map} (applied to $\mathbb{L}=\mathbb{K}\left[  t\right]  $ and
$\mathbb{M}=\mathbb{K}$ and $f=\varepsilon$) yields%
\begin{equation}
\operatorname*{ent}\left(  \varepsilon^{1\times1}\left(  B\right)  \right)
=\varepsilon\left(  \operatorname*{ent}B\right)  .
\label{pf.thm.block.puqv.ent1.pf.1}%
\end{equation}

However, applying the map $\varepsilon^{1\times1}$ to both sides of the
equality $B=u\left(  \operatorname*{adj}\widetilde{A}\right)  p=u\left(
\left(  \operatorname*{adj}\widetilde{A}\right)  p\right)  $, we obtain%
\begin{align*}
\varepsilon^{1\times1}\left(  B\right)   &  =\varepsilon^{1\times1}\left(
u\left(  \left(  \operatorname*{adj}\widetilde{A}\right)  p\right)  \right)
=\underbrace{\varepsilon^{1\times n}\left(  u\right)  }_{=u}\cdot
\underbrace{\varepsilon^{n\times1}\left(  \left(  \operatorname*{adj}%
\widetilde{A}\right)  p\right)  }_{\substack{=\varepsilon^{n\times n}\left(
\operatorname*{adj}\widetilde{A}\right)  \cdot\varepsilon^{n\times1}\left(
p\right)  \\\text{(by Theorem \ref{prop.functor.clas} \textbf{(c)}%
,}\\\text{applied to }\mathbb{K}\left[  t\right]  \text{, }\mathbb{K}\text{,
}\varepsilon\text{, }n\text{, }n\text{, }1\text{, }\operatorname*{adj}%
\widetilde{A}\text{ and }p\\\text{instead of }\mathbb{L}\text{, }%
\mathbb{M}\text{, }f\text{, }n\text{, }m\text{, }p\text{, }A\text{ and
}B\text{)}}}\\
&  \ \ \ \ \ \ \ \ \ \ \ \ \ \ \ \ \ \ \ \ \left(
\begin{array}
[c]{c}%
\text{by Theorem \ref{prop.functor.clas} \textbf{(c)},}\\
\text{applied to }\mathbb{K}\left[  t\right]  \text{, }\mathbb{K}\text{,
}\varepsilon\text{, }1\text{, }n\text{, }1\text{, }u\text{ and }\left(
\operatorname*{adj}\widetilde{A}\right)  p\\
\text{instead of }\mathbb{L}\text{, }\mathbb{M}\text{, }f\text{, }n\text{,
}m\text{, }p\text{, }A\text{ and }B
\end{array}
\right) \\
&  =u\cdot\underbrace{\varepsilon^{n\times n}\left(  \operatorname*{adj}%
\widetilde{A}\right)  }_{=\operatorname*{adj}A}\cdot\underbrace{\varepsilon
^{n\times1}\left(  p\right)  }_{=p}=u\left(  \operatorname*{adj}A\right)  p.
\end{align*}
In view of this, we can rewrite (\ref{pf.thm.block.puqv.ent1.pf.1}) as
\[
\operatorname*{ent}\left(  u\left(  \operatorname*{adj}A\right)  p\right)
=\varepsilon\left(  \operatorname*{ent}B\right)  =\varepsilon\left(
\operatorname*{ent}\left(  u\left(  \operatorname*{adj}\widetilde{A}\right)
p\right)  \right)
\]
(since $B=u\left(  \operatorname*{adj}\widetilde{A}\right)  p$). This proves
(\ref{pf.thm.block.puqv.ent1}).] \medskip
\end{verlong}

Furthermore, we have $\varepsilon\left(  \lambda\right)  =\lambda$ for each
$\lambda\in\mathbb{K}$ (by the definition of $\varepsilon$). Applying this to
$\lambda=\det D$, we find%
\[
\varepsilon\left(  \det D\right)  =\det D.
\]

Also, $v\left(  \operatorname*{adj}D\right)  q\in\mathbb{K}^{1\times1}$ and
thus $\operatorname*{ent}\left(  v\left(  \operatorname*{adj}D\right)
q\right)  \in\mathbb{K}$. Recall again that we have $\varepsilon\left(
\lambda\right)  =\lambda$ for each $\lambda\in\mathbb{K}$. Applying this to
$\lambda=\operatorname*{ent}\left(  v\left(  \operatorname*{adj}D\right)
q\right)  $, we obtain%
\begin{equation}
\varepsilon\left(  \operatorname*{ent}\left(  v\left(  \operatorname*{adj}%
D\right)  q\right)  \right)  =\operatorname*{ent}\left(  v\left(
\operatorname*{adj}D\right)  q\right)  \label{pf.thm.block.puqv.ent2}%
\end{equation}
(since $\operatorname*{ent}\left(  v\left(  \operatorname*{adj}D\right)
q\right)  \in\mathbb{K}$).

Now, Lemma \ref{lem.block.puqv-reg} (applied to $\mathbb{K}\left[  t\right]  $
and $\widetilde{A}$ instead of $\mathbb{K}$ and $A$) yields
\[
\det\left(
\begin{array}
[c]{cc}%
\widetilde{A} & pv\\
qu & D
\end{array}
\right)  =\det\widetilde{A}\cdot\det D-\operatorname*{ent}\left(  u\left(
\operatorname*{adj}\widetilde{A}\right)  p\right)  \cdot\operatorname*{ent}%
\left(  v\left(  \operatorname*{adj}D\right)  q\right)
\]
(since the element $\det\widetilde{A}$ of $\mathbb{K}\left[  t\right]  $ is
regular). Applying the map $\varepsilon$ to both sides of this equality, we
find%
\begin{align}
&  \varepsilon\left(  \det\left(
\begin{array}
[c]{cc}%
\widetilde{A} & pv\\
qu & D
\end{array}
\right)  \right) \nonumber\\
&  =\varepsilon\left(  \det\widetilde{A}\cdot\det D-\operatorname*{ent}\left(
u\left(  \operatorname*{adj}\widetilde{A}\right)  p\right)  \cdot
\operatorname*{ent}\left(  v\left(  \operatorname*{adj}D\right)  q\right)
\right) \nonumber\\
&  =\underbrace{\varepsilon\left(  \det\widetilde{A}\right)  }_{=\det A}%
\cdot\underbrace{\varepsilon\left(  \det D\right)  }_{=\det D}%
-\underbrace{\varepsilon\left(  \operatorname*{ent}\left(  u\left(
\operatorname*{adj}\widetilde{A}\right)  p\right)  \right)  }%
_{\substack{=\operatorname*{ent}\left(  u\left(  \operatorname*{adj}A\right)
p\right)  \\\text{(by (\ref{pf.thm.block.puqv.ent1}))}}}\cdot
\underbrace{\varepsilon\left(  \operatorname*{ent}\left(  v\left(
\operatorname*{adj}D\right)  q\right)  \right)  }%
_{\substack{=\operatorname*{ent}\left(  v\left(  \operatorname*{adj}D\right)
q\right)  \\\text{(by (\ref{pf.thm.block.puqv.ent2}))}}}\nonumber\\
&  \ \ \ \ \ \ \ \ \ \ \ \ \ \ \ \ \ \ \ \ \left(  \text{since }%
\varepsilon\text{ is a ring homomorphism}\right) \nonumber\\
&  =\det A\cdot\det D-\operatorname*{ent}\left(  u\left(  \operatorname*{adj}%
A\right)  p\right)  \cdot\operatorname*{ent}\left(  v\left(
\operatorname*{adj}D\right)  q\right)  . \label{pf.thm.block.puqv.5}%
\end{align}

However, $\left(
\begin{array}
[c]{cc}%
\widetilde{A} & pv\\
qu & D
\end{array}
\right)  $ is an $\left(  n+m\right)  \times\left(  n+m\right)  $-matrix in
$\left(  \mathbb{K}\left[  t\right]  \right)  ^{\left(  n+m\right)
\times\left(  n+m\right)  }$. Thus, Proposition \ref{prop.functor.det}
\textbf{(a)} (applied to $\mathbb{K}\left[  t\right]  $, $\mathbb{K}$,
$\varepsilon$, $n+m$ and $\left(
\begin{array}
[c]{cc}%
\widetilde{A} & pv\\
qu & D
\end{array}
\right)  $ instead of $\mathbb{L}$, $\mathbb{M}$, $f$, $n$ and $A$) yields
\begin{equation}
\varepsilon\left(  \det\left(
\begin{array}
[c]{cc}%
\widetilde{A} & pv\\
qu & D
\end{array}
\right)  \right)  =\det\left(  \varepsilon^{\left(  n+m\right)  \times\left(
n+m\right)  }\left(
\begin{array}
[c]{cc}%
\widetilde{A} & pv\\
qu & D
\end{array}
\right)  \right)  . \label{pf.thm.block.puqv.6}%
\end{equation}

On the other hand, Lemma \ref{lem.functor.block2x2} (applied to $\mathbb{K}%
\left[  t\right]  $, $\mathbb{K}$, $\varepsilon$, $m$, $n$, $m$,
$\widetilde{A}$, $pv$ and $qu$ instead of $\mathbb{L}$, $\mathbb{M}$, $f$,
$n^{\prime}$, $m$, $m^{\prime}$, $A$, $B$ and $C$) yields
\[
\varepsilon^{\left(  n+m\right)  \times\left(  n+m\right)  }\left(  \left(
\begin{array}
[c]{cc}%
\widetilde{A} & pv\\
qu & D
\end{array}
\right)  \right)  =\left(
\begin{array}
[c]{cc}%
\varepsilon^{n\times n}\left(  \widetilde{A}\right)  & \varepsilon^{n\times
m}\left(  pv\right) \\
\varepsilon^{m\times n}\left(  qu\right)  & \varepsilon^{m\times m}\left(
D\right)
\end{array}
\right)  =\left(
\begin{array}
[c]{cc}%
A & pv\\
qu & D
\end{array}
\right)
\]
(since $\varepsilon^{n\times n}\left(  \widetilde{A}\right)  =A$ and
$\varepsilon^{n\times m}\left(  pv\right)  =pv$ and $\varepsilon^{m\times
n}\left(  qu\right)  =qu$ and $\varepsilon^{m\times m}\left(  D\right)  =D$).
Taking determinants on both sides of this equality, we find%
\[
\det\left(  \varepsilon^{\left(  n+m\right)  \times\left(  n+m\right)
}\left(  \left(
\begin{array}
[c]{cc}%
\widetilde{A} & pv\\
qu & D
\end{array}
\right)  \right)  \right)  =\det\left(
\begin{array}
[c]{cc}%
A & pv\\
qu & D
\end{array}
\right)  .
\]
Hence,%
\begin{align*}
\det\left(
\begin{array}
[c]{cc}%
A & pv\\
qu & D
\end{array}
\right)   &  =\det\left(  \varepsilon^{\left(  n+m\right)  \times\left(
n+m\right)  }\left(  \left(
\begin{array}
[c]{cc}%
\widetilde{A} & pv\\
qu & D
\end{array}
\right)  \right)  \right) \\
&  =\varepsilon\left(  \det\left(
\begin{array}
[c]{cc}%
\widetilde{A} & pv\\
qu & D
\end{array}
\right)  \right)  \ \ \ \ \ \ \ \ \ \ \left(  \text{by
(\ref{pf.thm.block.puqv.6})}\right) \\
&  =\det A\cdot\det D-\operatorname*{ent}\left(  u\left(  \operatorname*{adj}%
A\right)  p\right)  \cdot\operatorname*{ent}\left(  v\left(
\operatorname*{adj}D\right)  q\right)  \ \ \ \ \ \ \ \ \ \ \left(  \text{by
(\ref{pf.thm.block.puqv.5})}\right)  .
\end{align*}
This completes the proof of Theorem \ref{thm.block.puqv}.
\end{proof}

Note that Theorem \ref{thm.block.puqv} generalizes the following known fact
(e.g., \cite[Exercise 6.60 \textbf{(a)}]{detnotes}):

\begin{corollary}
\label{cor.block.bordered}Let $n\in\mathbb{N}$. Let $u\in\mathbb{K}^{n\times
1}$ be a column vector with $n$ entries, and let $v\in\mathbb{K}^{1\times n}$
be a row vector with $n$ entries. (Thus, $uv$ is an $n\times n$-matrix,
whereas $vu$ is a $1\times1$-matrix.) Let $h\in\mathbb{K}$. Let $H$ be the
$1\times1$-matrix $\left(
\begin{array}
[c]{c}%
h
\end{array}
\right)  \in\mathbb{K}^{1\times1}$. Let $A\in\mathbb{K}^{n\times n}$ be an
$n\times n$-matrix. Then,%
\[
\det\left(
\begin{array}
[c]{cc}%
A & u\\
v & H
\end{array}
\right)  =h\det A-\operatorname*{ent}\left(  v\left(  \operatorname*{adj}%
A\right)  u\right)  .
\]

\end{corollary}

\begin{proof}
[Proof of Corollary \ref{cor.block.bordered} (sketched).]Consider the
$1\times1$ identity matrix $I_{1}=\left(
\begin{array}
[c]{c}%
1
\end{array}
\right)  $. Then, $I_{1}$ is both a row vector and a column vector, and we
have $u=uI_{1}$ and $v=I_{1}v$. Moreover, the adjugate of any $1\times
1$-matrix is $\left(
\begin{array}
[c]{c}%
1
\end{array}
\right)  $ (since the determinant of a $0\times0$-matrix is defined to be
$1$). Thus, in particular, $\operatorname*{adj}H=\left(
\begin{array}
[c]{c}%
1
\end{array}
\right)  =I_{1}$. Furthermore, from $H=\left(
\begin{array}
[c]{c}%
h
\end{array}
\right)  $, we obtain $\det H=h$. Now, from $u=uI_{1}$ and $v=I_{1}v$, we
obtain
\begin{align*}
\det\left(
\begin{array}
[c]{cc}%
A & u\\
v & H
\end{array}
\right)   &  =\det\left(
\begin{array}
[c]{cc}%
A & uI_{1}\\
I_{1}v & H
\end{array}
\right) \\
&  =\det A\cdot\underbrace{\det H}_{=h}-\operatorname*{ent}\left(  v\left(
\operatorname*{adj}A\right)  u\right)  \cdot\operatorname*{ent}\left(
I_{1}\underbrace{\left(  \operatorname*{adj}H\right)  }_{=I_{1}}I_{1}\right)
\\
&  \ \ \ \ \ \ \ \ \ \ \ \ \ \ \ \ \ \ \ \ \left(
\begin{array}
[c]{c}%
\text{by Theorem \ref{thm.block.puqv}, applied to }1\text{, }H\text{,
}u\text{, }I_{1}\text{, }I_{1}\text{ and }v\\
\text{instead of }m\text{, }D\text{, }p\text{, }q\text{, }v\text{ and }u
\end{array}
\right) \\
&  =\underbrace{\det A\cdot h}_{=h\det A}-\operatorname*{ent}\left(  v\left(
\operatorname*{adj}A\right)  u\right)  \cdot\operatorname*{ent}%
\underbrace{\left(  I_{1}\cdot I_{1}\cdot I_{1}\right)  }_{=I_{1}}\\
&  =h\det A-\operatorname*{ent}\left(  v\left(  \operatorname*{adj}A\right)
u\right)  \cdot\underbrace{\operatorname*{ent}\left(  I_{1}\right)  }_{=1}\\
&  =h\det A-\operatorname*{ent}\left(  v\left(  \operatorname*{adj}A\right)
u\right)  .
\end{align*}
This proves Corollary \ref{cor.block.bordered}.
\end{proof}

Another particular case of Theorem \ref{thm.block.puqv} is the following:

\begin{corollary}
\label{cor.block.01}Let $n$ and $m$ be two positive integers. Let
$A\in\mathbb{K}^{n\times n}$ and $D\in\mathbb{K}^{m\times m}$ be two square
matrices. Let $B$ be the $n\times m$-matrix whose $\left(  n,1\right)  $-th
entry is $1$ and whose all other entries are $0$. Let $C$ be the $m\times
n$-matrix whose $\left(  1,n\right)  $-th entry is $1$ and whose all other
entries are $0$. Then,%
\[
\det\left(
\begin{array}
[c]{cc}%
A & B\\
C & D
\end{array}
\right)  =\det A\cdot\det D-\det\left(  A_{\sim n,\sim n}\right)  \cdot
\det\left(  D_{\sim1,\sim1}\right)  .
\]
(Recall that we are using the notations from Definition
\ref{def.submatrix.minor}.)
\end{corollary}

\begin{example}
Let us see what Corollary \ref{cor.block.01} says in the case when $n=2$ and
$m=3$. Indeed, let $n=2$ and $m=3$ and%
\[
A=\left(
\begin{array}
[c]{cc}%
a_{1,1} & a_{1,2}\\
a_{2,1} & a_{2,2}%
\end{array}
\right)  \ \ \ \ \ \ \ \ \ \ \text{and}\ \ \ \ \ \ \ \ \ \ D=\left(
\begin{array}
[c]{ccc}%
d_{1,1} & d_{1,2} & d_{1,3}\\
d_{2,1} & d_{2,2} & d_{2,3}\\
d_{3,1} & d_{3,2} & d_{3,3}%
\end{array}
\right)  .
\]
Then, the matrices $B$ and $C$ defined in Corollary \ref{cor.block.01} are%
\[
B=\left(
\begin{array}
[c]{ccc}%
0 & 0 & 0\\
1 & 0 & 0
\end{array}
\right)  \ \ \ \ \ \ \ \ \ \ \text{and}\ \ \ \ \ \ \ \ \ \ C=\left(
\begin{array}
[c]{cc}%
0 & 1\\
0 & 0\\
0 & 0
\end{array}
\right)  .
\]
Hence,%
\[
\left(
\begin{array}
[c]{cc}%
A & B\\
C & D
\end{array}
\right)  =\left(
\begin{array}
[c]{ccccc}%
a_{1,1} & a_{1,2} & 0 & 0 & 0\\
a_{2,1} & a_{2,2} & 1 & 0 & 0\\
0 & 1 & d_{1,1} & d_{1,2} & d_{1,3}\\
0 & 0 & d_{2,1} & d_{2,2} & d_{2,3}\\
0 & 0 & d_{3,1} & d_{3,2} & d_{3,3}%
\end{array}
\right)  .
\]
Therefore, the claim of Corollary \ref{cor.block.01} rewrites as follows in
our case:
\begin{align*}
&  \det\left(
\begin{array}
[c]{ccccc}%
a_{1,1} & a_{1,2} & 0 & 0 & 0\\
a_{2,1} & a_{2,2} & 1 & 0 & 0\\
0 & 1 & d_{1,1} & d_{1,2} & d_{1,3}\\
0 & 0 & d_{2,1} & d_{2,2} & d_{2,3}\\
0 & 0 & d_{3,1} & d_{3,2} & d_{3,3}%
\end{array}
\right) \\
&  =\det\left(
\begin{array}
[c]{cc}%
a_{1,1} & a_{1,2}\\
a_{2,1} & a_{2,2}%
\end{array}
\right)  \cdot\det\left(
\begin{array}
[c]{ccc}%
d_{1,1} & d_{1,2} & d_{1,3}\\
d_{2,1} & d_{2,2} & d_{2,3}\\
d_{3,1} & d_{3,2} & d_{3,3}%
\end{array}
\right)  -\det\left(
\begin{array}
[c]{c}%
a_{1,1}%
\end{array}
\right)  \cdot\det\left(
\begin{array}
[c]{cc}%
d_{2,2} & d_{2,3}\\
d_{3,2} & d_{3,3}%
\end{array}
\right)  .
\end{align*}

\end{example}

\begin{verlong}
Our proof of Corollary \ref{cor.block.01} will rely on two lemmas:

\begin{lemma}
\label{lem.entvBu}Let $n\in\mathbb{N}$ and $m\in\mathbb{N}$. Let $u=\left(
\begin{array}
[c]{c}%
u_{1}\\
u_{2}\\
\vdots\\
u_{m}%
\end{array}
\right)  \in\mathbb{K}^{m\times1}$ be a column vector, and let $v=\left(
\begin{array}
[c]{cccc}%
v_{1} & v_{2} & \cdots & v_{n}%
\end{array}
\right)  \in\mathbb{K}^{1\times n}$ be a row vector. Let $B=\left(
b_{i,j}\right)  _{1\leq i\leq n,\ 1\leq j\leq m}\in\mathbb{K}^{n\times m}$ be
an $n\times m$-matrix. Then,%
\[
\operatorname*{ent}\left(  vBu\right)  =\sum_{i=1}^{n}\ \ \sum_{j=1}^{m}%
u_{j}v_{i}b_{i,j}.
\]

\end{lemma}

Lemma \ref{lem.entvBu} is precisely \cite[Proposition 7.263]{detnotes} (except
that the column vector $\left(
\begin{array}
[c]{c}%
u_{1}\\
u_{2}\\
\vdots\\
u_{m}%
\end{array}
\right)  $ is written as $\left(  u_{1},u_{2},\ldots,u_{m}\right)  ^{T}$, and
that the row vector $\left(
\begin{array}
[c]{cccc}%
v_{1} & v_{2} & \cdots & v_{n}%
\end{array}
\right)  $ is written as $\left(  v_{1},v_{2},\ldots,v_{n}\right)  $ in
\cite[Proposition 7.263]{detnotes}). Thus, we will not prove it here.

\begin{lemma}
\label{lem.uv}Let $n\in\mathbb{N}$ and $m\in\mathbb{N}$. Let $u=\left(
\begin{array}
[c]{c}%
u_{1}\\
u_{2}\\
\vdots\\
u_{m}%
\end{array}
\right)  \in\mathbb{K}^{m\times1}$ be a column vector, and let $v=\left(
\begin{array}
[c]{cccc}%
v_{1} & v_{2} & \cdots & v_{n}%
\end{array}
\right)  \in\mathbb{K}^{1\times n}$ be a row vector. Then,%
\[
uv=\left(  u_{i}v_{j}\right)  _{1\leq i\leq m,\ 1\leq j\leq n}.
\]

\end{lemma}

\begin{proof}
[Proof of Lemma \ref{lem.uv}.]We have $u=\left(
\begin{array}
[c]{c}%
u_{1}\\
u_{2}\\
\vdots\\
u_{m}%
\end{array}
\right)  =\left(  u_{i}\right)  _{1\leq i\leq m,\ 1\leq j\leq1}$ and
$v=\left(
\begin{array}
[c]{cccc}%
v_{1} & v_{2} & \cdots & v_{n}%
\end{array}
\right)  =\left(  v_{j}\right)  _{1\leq i\leq1,\ 1\leq j\leq n}$. Hence, the
definition of the product of two matrices yields%
\[
uv=\left(  \underbrace{\sum_{k=1}^{1}u_{i}v_{j}}_{=u_{i}v_{j}}\right)  _{1\leq
i\leq m,\ 1\leq j\leq n}=\left(  u_{i}v_{j}\right)  _{1\leq i\leq m,\ 1\leq
j\leq n}.
\]
This proves Lemma \ref{lem.uv}.
\end{proof}
\end{verlong}

\begin{proof}
[Proof of Corollary \ref{cor.block.01} (sketched).]What follows is by far not
the easiest proof of Corollary \ref{cor.block.01}, but it puts the corollary
in the context of Theorem \ref{thm.block.puqv}.

We let $p\in\mathbb{K}^{n\times1}$ be the column vector whose $n$-th entry is
$1$ and whose all other entries are $0$.

We let $q\in\mathbb{K}^{m\times1}$ be the column vector whose $1$-st entry is
$1$ and whose all other entries are $0$.

We let $v\in\mathbb{K}^{1\times m}$ be the row vector whose $1$-st entry is
$1$ and whose all other entries are $0$.

We let $u\in\mathbb{K}^{1\times n}$ be the row vector whose $n$-th entry is
$1$ and whose all other entries are $0$.

Thus,%
\begin{align*}
p  &  =\left(
\begin{array}
[c]{c}%
0\\
0\\
\vdots\\
0\\
1
\end{array}
\right)  \in\mathbb{K}^{n\times1},\ \ \ \ \ \ \ \ \ \ q=\left(
\begin{array}
[c]{c}%
1\\
0\\
0\\
\vdots\\
0
\end{array}
\right)  \in\mathbb{K}^{m\times1},\\
v  &  =\left(
\begin{array}
[c]{ccccc}%
1 & 0 & 0 & \cdots & 0
\end{array}
\right)  \in\mathbb{K}^{1\times m},\ \ \ \ \ \ \ \ \ \ u=\left(
\begin{array}
[c]{ccccc}%
0 & 0 & \cdots & 0 & 1
\end{array}
\right)  \in\mathbb{K}^{1\times n}.
\end{align*}

\begin{vershort}
Now, it is easy to see (using just the definition of a product of two
matrices) that the following four claims hold:
\end{vershort}

\begin{verlong}
Now, it is easy to see that the following four claims hold\footnote{Indeed,
Claim 1 and Claim 2 follow easily from Lemma \ref{lem.uv}, whereas Claim 3 and
Claim 4 follow easily from Lemma \ref{lem.entvBu}.}:
\end{verlong}

\begin{statement}
\textit{Claim 1:} We have $qu=C$.
\end{statement}

\begin{statement}
\textit{Claim 2:} We have $pv=B$.
\end{statement}

\begin{statement}
\textit{Claim 3:} Every $m\times m$-matrix $Y$ satisfies%
\begin{equation}
\operatorname*{ent}\left(  vYq\right)  =\left(  \text{the }\left(  1,1\right)
\text{-th entry of }Y\right)  . \label{pf.cor.block.01.entvYq=}%
\end{equation}

\end{statement}

\begin{statement}
\textit{Claim 4:} Every $n\times n$-matrix $X$ satisfies%
\begin{equation}
\operatorname*{ent}\left(  uXp\right)  =\left(  \text{the }\left(  n,n\right)
\text{-th entry of }X\right)  . \label{pf.cor.block.01.entuXp=}%
\end{equation}

\end{statement}

Now, Theorem \ref{thm.block.puqv} yields%
\[
\det\left(
\begin{array}
[c]{cc}%
A & pv\\
qu & D
\end{array}
\right)  =\det A\cdot\det D-\operatorname*{ent}\left(  u\left(
\operatorname*{adj}A\right)  p\right)  \cdot\operatorname*{ent}\left(
v\left(  \operatorname*{adj}D\right)  q\right)  .
\]
In view of $pv=B$ and $qu=C$ and%
\begin{align*}
\operatorname*{ent}\left(  u\left(  \operatorname*{adj}A\right)  p\right)   &
=\left(  \text{the }\left(  n,n\right)  \text{-th entry of }%
\operatorname*{adj}A\right) \\
&  \ \ \ \ \ \ \ \ \ \ \ \ \ \ \ \ \ \ \ \ \left(  \text{by
(\ref{pf.cor.block.01.entuXp=}), applied to }X=\operatorname*{adj}A\right) \\
&  =\underbrace{\left(  -1\right)  ^{n+n}}_{\substack{=\left(  -1\right)
^{2n}=1\\\text{(since }2n\text{ is even)}}}\det\left(  A_{\sim n,\sim
n}\right)  \ \ \ \ \ \ \ \ \ \ \left(  \text{by the definition of an
adjugate}\right) \\
&  =\det\left(  A_{\sim n,\sim n}\right)
\end{align*}
and%
\begin{align*}
\operatorname*{ent}\left(  v\left(  \operatorname*{adj}D\right)  q\right)   &
=\left(  \text{the }\left(  1,1\right)  \text{-th entry of }%
\operatorname*{adj}D\right) \\
&  \ \ \ \ \ \ \ \ \ \ \ \ \ \ \ \ \ \ \ \ \left(  \text{by
(\ref{pf.cor.block.01.entvYq=}), applied to }Y=\operatorname*{adj}D\right) \\
&  =\underbrace{\left(  -1\right)  ^{1+1}}_{=1}\det\left(  D_{\sim1,\sim
1}\right)  \ \ \ \ \ \ \ \ \ \ \left(  \text{by the definition of an
adjugate}\right) \\
&  =\det\left(  D_{\sim1,\sim1}\right)  ,
\end{align*}
we can rewrite this as%
\[
\det\left(
\begin{array}
[c]{cc}%
A & B\\
C & D
\end{array}
\right)  =\det A\cdot\det D-\det\left(  A_{\sim n,\sim n}\right)  \cdot
\det\left(  D_{\sim1,\sim1}\right)  .
\]
This proves Corollary \ref{cor.block.01}.
\end{proof}

\section{\label{sec.further}Further remarks}

We end by mentioning some related results about traces of matrices.

\subsection{Almkvist's theorem}

It is well-known from linear algebra (see, e.g., \cite[8.16 and, implicitly,
8.18 (c)]{Axler25}) that when $\mathbb{K}$ is a field, any nilpotent $n\times
n$-matrix $A\in\mathbb{K}^{n\times n}$ satisfies $A^{n}=0$ and
$\operatorname*{Tr}A=0$. For general commutative rings $\mathbb{K}$, neither
of these two claims is true; for example, the $1\times1$-matrix $A=\left(
\begin{array}
[c]{c}%
\overline{2}%
\end{array}
\right)  $ over $\mathbb{K}=\mathbb{Z}/8$ is nilpotent, but neither $A^{1}$
nor $\operatorname*{Tr}A$ vanishes. Thus, we cannot expect Corollary
\ref{cor.nilpcrit} to have a straightforward converse. However, some things in
this direction can be said. In 1973, Gert Almkvist proved (as part of a more
extensive result \cite[Theorem 1.7 (i)]{Almkvi73}) the following fact:

\begin{theorem}
\label{thm.almkvist}Let $n\in\mathbb{N}$ and $k\in\mathbb{N}$. Let
$A\in\mathbb{K}^{n\times n}$ be a matrix such that $A^{k+1}=0$. Then,%
\[
\left(  \operatorname*{Tr}A\right)  ^{nk+1}=0.
\]

\end{theorem}

This does not easily follow from our above methods. However, Almkvist gave an
elegant proof using exterior algebra, which was nicely exposed by Zeilberger
in \cite{Zeilbe} (and again by Lombardi and Quitt\'{e} in \cite[Chapter II,
Exercise 2 part 2b]{LomQui16}). With some work, this proof can be rewritten to
avoid any use of exterior algebra; as its core is a simple identity combining
determinants and traces. To state it, we need two pieces of notation:

\begin{itemize}
\item If $m\in\mathbb{N}$ and $i_{1},i_{2},\ldots,i_{n}\in\mathbb{N}$ satisfy
$i_{1}+i_{2}+\cdots+i_{n}=m$, then the \emph{multinomial coefficient}
$\dbinom{m}{i_{1},i_{2},\ldots,i_{n}}$ shall be defined as the number
$\dfrac{m!}{i_{1}!i_{2}!\cdots i_{n}!}$. This number is well-known to be an
integer (see \cite[Exercise 3.1 and Remark 3.10]{detnotes}).

\item If $B$ is a matrix and $j$ is a positive integer, then
$\operatorname{row}_{j}B$ shall denote the $j$-th row of the matrix $B$
(assuming, of course, that $B$ has at least $j$ rows).
\end{itemize}

\begin{verlong}
We recall a well-known property of matrix multiplication (see, e.g.,
\cite[\S 3C, Exercise 8]{Axler25}):

\begin{lemma}
\label{lem.row.rowAB}Let $A$ and $B$ be two matrices such that the product
$AB$ is well-defined (i.e., the matrix $A$ has as many columns as $B$ has
rows). Let $j$ be a positive integer such that the matrix $A$ has at least $j$
rows. Then,%
\[
\operatorname{row}_{j}\left(  AB\right)  =\left(  \operatorname{row}%
_{j}A\right)  \cdot B.
\]

\end{lemma}
\end{verlong}

Now we are ready to state the identity crucial for proving Theorem
\ref{thm.almkvist}:

\begin{proposition}
\label{prop.almkvist-lem}Let $n\in\mathbb{N}$. Let $A\in\mathbb{K}^{n\times
n}$ be a matrix. For any $n$ integers $i_{1},i_{2},\ldots,i_{n}\in\mathbb{N}$,
we let $M_{i_{1},i_{2},\ldots,i_{n}}$ denote the $n\times n$-matrix whose rows
are
\[
\operatorname*{row}\nolimits_{1}\left(  A^{i_{1}}\right)
,\ \operatorname*{row}\nolimits_{2}\left(  A^{i_{2}}\right)  ,\ \ldots
,\ \operatorname*{row}\nolimits_{n}\left(  A^{i_{n}}\right)  .
\]

Then, for each $m\in\mathbb{N}$, we have%
\[
\left(  \operatorname*{Tr}A\right)  ^{m}=\sum_{\substack{\left(  i_{1}%
,i_{2},\ldots,i_{n}\right)  \in\mathbb{N}^{n};\\i_{1}+i_{2}+\cdots+i_{n}%
=m}}\dbinom{m}{i_{1},i_{2},\ldots,i_{n}}\det\left(  M_{i_{1},i_{2}%
,\ldots,i_{n}}\right)  .
\]

\end{proposition}

\begin{vershort}
Once proved, Proposition \ref{prop.almkvist-lem} implies Theorem
\ref{thm.almkvist} almost immediately, since any $n$-tuple $\left(
i_{1},i_{2},\ldots,i_{n}\right)  \in\mathbb{N}^{n}$ satisfying $i_{1}%
+i_{2}+\cdots+i_{n}=m=nk+1$ will contain at least one entry $i_{j}\geq k+1$
(by the pigeonhole principle, if you wish). Proposition
\ref{prop.almkvist-lem}, in turn, can be proved by induction on $m$ using the
following three lemmas:
\end{vershort}

\begin{verlong}
Once proved, Proposition \ref{prop.almkvist-lem} will imply Theorem
\ref{thm.almkvist} almost immediately, since any $n$-tuple $\left(
i_{1},i_{2},\ldots,i_{n}\right)  \in\mathbb{N}^{n}$ satisfying $i_{1}%
+i_{2}+\cdots+i_{n}=m=nk+1$ will contain at least one entry $i_{j}\geq k+1$
(see below for the details). Proposition \ref{prop.almkvist-lem}, in turn, can
be easily derived by induction on $m$ using the following three lemmas:
\end{verlong}

\begin{noncompile}
The following lemma and its proof, while fully written up, is NOT USED in this work:

\begin{lemma}
\label{lem.almkvist-lem2}Let $n\in\mathbb{N}$. Let $A$ and $B$ be two matrices
in $\mathbb{K}^{n\times n}$. For each $j\in\left\{  1,2,\ldots,n\right\}  $,
let $B_{j}^{\prime}$ be the $n\times n$-matrix obtained from $B$ by replacing
the $j$-th row of $B$ by the $j$-th row of $AB$. Then,%
\[
\sum_{j=1}^{n}\det\left(  B_{j}^{\prime}\right)  =\left(  \operatorname*{Tr}%
A\right)  \cdot\det B.
\]

\end{lemma}



Our proof of Lemma \ref{lem.almkvist-lem2} shall rely on a simple
combinatorial lemma:

\begin{lemma}
\label{lem.perm.diff-fix-pot}Let $P$ be any set. Let $\sigma$ be a permutation
of $P$. Assume that $\sigma\neq\operatorname*{id}$. Let $p\in P$. Then, there
exists some $k\in P\setminus\left\{  p\right\}  $ such that $\sigma\left(
k\right)  \neq k$.
\end{lemma}

\begin{proof}
[Proof of Lemma \ref{lem.perm.diff-fix-pot}.]Assume the contrary. Thus,
\begin{equation}
\sigma\left(  k\right)  =k\ \ \ \ \ \ \ \ \ \ \text{for each }k\in
P\setminus\left\{  p\right\}  . \label{pf.lem.perm.diff-fix-pot.1}%
\end{equation}

But we have $\sigma\neq\operatorname*{id}$. In other words, there exists some
$x\in P$ such that $\sigma\left(  x\right)  \neq\operatorname*{id}\left(
x\right)  $. Consider this $x$. Hence, $\sigma\left(  x\right)  \neq
\operatorname*{id}\left(  x\right)  =x$. If we had $x\in P\setminus\left\{
p\right\}  $, then (\ref{pf.lem.perm.diff-fix-pot.1}) (applied to $k=x$) would
yield $\sigma\left(  x\right)  =x$, which would contradict $\sigma\left(
x\right)  \neq x$. Hence, we cannot have $x\in P\setminus\left\{  p\right\}
$. Thus, $x\notin P\setminus\left\{  p\right\}  $, so that $x\in
P\setminus\left(  P\setminus\left\{  p\right\}  \right)  \subseteq\left\{
p\right\}  $. In other words, $x=p$.

Now, $\sigma\left(  x\right)  \neq x=p$. Hence, $\sigma\left(  x\right)  \in
P\setminus\left\{  p\right\}  $. Therefore, (\ref{pf.lem.perm.diff-fix-pot.1})
(applied to $k=\sigma\left(  x\right)  $) yields $\sigma\left(  \sigma\left(
x\right)  \right)  =\sigma\left(  x\right)  $. But the map $\sigma$ is a
permutation, thus has an inverse $\sigma^{-1}$. Applying this inverse
$\sigma^{-1}$ to both sides of the equality $\sigma\left(  \sigma\left(
x\right)  \right)  =\sigma\left(  x\right)  $, we obtain $\sigma\left(
x\right)  =x$. This contradicts $\sigma\left(  x\right)  \neq x$. This
contradiction shows that our assumption was false. Thus, Lemma
\ref{lem.perm.diff-fix-pot} is proved.
\end{proof}

\begin{proof}
[Proof of Lemma \ref{lem.almkvist-lem2}.]Write the $n\times n$-matrix $A$ as
$A=\left(  a_{i,j}\right)  _{1\leq i\leq n,\ 1\leq j\leq n}$. Then, by the
definition of the trace, we have%
\begin{equation}
\operatorname*{Tr}A=a_{1,1}+a_{2,2}+\cdots+a_{n,n}=\sum_{j=1}^{n}a_{j,j}.
\label{pf.lem.almkvist-lem2.tr}%
\end{equation}

Let $p\in\left\{  1,2,\ldots,n\right\}  $. Then, $B_{p}^{\prime}$ is the
$n\times n$-matrix obtained from $B$ by replacing the $p$-th row of $B$ by the
$p$-th row of $AB$ (since this is how $B_{p}^{\prime}$ was defined). Thus,%
\begin{equation}
\operatorname*{row}\nolimits_{p}\left(  B_{p}^{\prime}\right)
=\operatorname*{row}\nolimits_{p}\left(  AB\right)
\label{pf.lem.almkvist-lem2.ro1}%
\end{equation}
and%
\begin{equation}
\operatorname*{row}\nolimits_{k}\left(  B_{p}^{\prime}\right)
=\operatorname*{row}\nolimits_{k}B\ \ \ \ \ \ \ \ \ \ \text{for each }%
k\in\left\{  1,2,\ldots,n\right\}  \setminus\left\{  p\right\}  .
\label{pf.lem.almkvist-lem2.ro2}%
\end{equation}

Let $X$ be the $n\times n$-matrix obtained from the identity matrix $I_{n}$ by
replacing the $p$-th row of $I_{n}$ by the $p$-th row of $A$. Thus,%
\begin{equation}
\operatorname*{row}\nolimits_{p}X=\operatorname*{row}\nolimits_{p}A
\label{pf.lem.almkvist-lem2.ro3}%
\end{equation}
and%
\begin{equation}
\operatorname*{row}\nolimits_{k}X=\operatorname*{row}\nolimits_{k}\left(
I_{n}\right)  \ \ \ \ \ \ \ \ \ \ \text{for each }k\in\left\{  1,2,\ldots
,n\right\}  \setminus\left\{  p\right\}  . \label{pf.lem.almkvist-lem2.ro4}%
\end{equation}

Write the $n\times n$-matrix $X$ as $X=\left(  x_{i,j}\right)  _{1\leq i\leq
n,\ 1\leq j\leq n}$. Then, its $p$-th row is given by $\operatorname{row}%
_{p}X=\left(
\begin{array}
[c]{cccc}%
x_{p,1} & x_{p,2} & \cdots & x_{p,n}%
\end{array}
\right)  $. Hence, for each $q\in\left\{  1,2,\ldots,n\right\}  $, we have%
\[
\left(  \text{the }q\text{-th entry of }\operatorname{row}_{p}X\right)
=x_{p,q}%
\]
and therefore%
\begin{align}
x_{p,q}  &  =\left(  \text{the }q\text{-th entry of }\operatorname{row}%
_{p}X\right) \nonumber\\
&  =\left(  \text{the }q\text{-th entry of }\operatorname{row}_{p}A\right)
\ \ \ \ \ \ \ \ \ \ \left(  \text{by (\ref{pf.lem.almkvist-lem2.ro3})}\right)
\nonumber\\
&  =a_{p,q} \label{pf.lem.almkvist-lem2.xpq=}%
\end{align}
(since the vector $\operatorname{row}_{p}A$ is the $p$-th row of $A=\left(
a_{i,j}\right)  _{1\leq i\leq n,\ 1\leq j\leq n}$, and thus its $q$-th entry
is $a_{p,q}$). Moreover, if $k\in\left\{  1,2,\ldots,n\right\}  \setminus
\left\{  p\right\}  $ and $q\in\left\{  1,2,\ldots,n\right\}  $ are arbitrary,
then%
\begin{align}
x_{k,q}  &  =\left(  \text{the }q\text{-th entry of }\operatorname{row}%
_{k}X\right) \nonumber\\
&  \ \ \ \ \ \ \ \ \ \ \ \ \ \ \ \ \ \ \ \ \left(
\begin{array}
[c]{c}%
\text{since }X=\left(  x_{i,j}\right)  _{1\leq i\leq n,\ 1\leq j\leq n}\text{
and}\\
\text{thus }\operatorname{row}_{k}X=\left(
\begin{array}
[c]{cccc}%
x_{k,1} & x_{k,2} & \cdots & x_{k,n}%
\end{array}
\right)  \text{,}\\
\text{so that the }q\text{-th entry of }\operatorname{row}_{k}X\text{ is
}x_{k,q}%
\end{array}
\right) \nonumber\\
&  =\left(  \text{the }q\text{-th entry of }\operatorname{row}_{k}\left(
I_{n}\right)  \right)  \ \ \ \ \ \ \ \ \ \ \left(  \text{since
(\ref{pf.lem.almkvist-lem2.ro4}) yields }\operatorname*{row}\nolimits_{k}%
X=\operatorname*{row}\nolimits_{k}\left(  I_{n}\right)  \right) \nonumber\\
&  =\left(  \text{the }\left(  k,q\right)  \text{-th entry of }I_{n}\right)
\ \ \ \ \ \ \ \ \ \ \left(  \text{since }\operatorname{row}_{k}\left(
I_{n}\right)  \text{ is the }k\text{-th row of the matrix }I_{n}\right)
\nonumber\\
&  =%
\begin{cases}
1, & \text{if }k=q;\\
0, & \text{if }k\neq q
\end{cases}
\label{pf.lem.almkvist-lem2.xkq=}%
\end{align}
(by the definition of the identity matrix $I_{n}$).

Now, we claim the following:

\begin{statement}
\textit{Claim 1:} We have $\det X=a_{p,p}$.
\end{statement}

\begin{proof}
[Proof of Claim 1.]Recall that $X=\left(  x_{i,j}\right)  _{1\leq i\leq
n,\ 1\leq j\leq n}$. Thus, (\ref{eq.det}) (applied to $X$ and $x_{i,j}$
instead of $A$ and $a_{i,j}$) shows that%
\begin{align}
\det X  &  =\sum_{\sigma\in S_{n}}\left(  -1\right)  ^{\sigma}\prod_{i=1}%
^{n}x_{i,\sigma\left(  i\right)  }\nonumber\\
&  =\left(  -1\right)  ^{\operatorname*{id}}\prod_{i=1}^{n}%
x_{i,\operatorname*{id}\left(  i\right)  }+\sum_{\substack{\sigma\in
S_{n};\\\sigma\neq\operatorname*{id}}}\left(  -1\right)  ^{\sigma}\prod
_{i=1}^{n}x_{i,\sigma\left(  i\right)  } \label{pf.lem.almkvist-lem2.c1.pf.1}%
\end{align}
(here, we have split off the addend for $\sigma=\operatorname*{id}$ from the
sum). Note that the identity permutation $\operatorname*{id}\in S_{n}$ has
sign $\left(  -1\right)  ^{\operatorname*{id}}=1$.

Now, let $\sigma\in S_{n}$ be such that $\sigma\neq\operatorname*{id}$. We
shall show that $\prod_{i=1}^{n}x_{i,\sigma\left(  i\right)  }=0$. Indeed,
Lemma \ref{lem.perm.diff-fix-pot} (applied to $P=\left\{  1,2,\ldots
,n\right\}  $) shows that there exists some $k\in\left\{  1,2,\ldots
,n\right\}  \setminus\left\{  p\right\}  $ such that $\sigma\left(  k\right)
\neq k$ (since $\sigma\in S_{n}$ is a permutation of $\left\{  1,2,\ldots
,n\right\}  $ satisfying $\sigma\neq\operatorname*{id}$). Consider this $k$.
Applying (\ref{pf.lem.almkvist-lem2.xkq=}) to $q=\sigma\left(  k\right)  $, we
obtain%
\[
x_{k,\sigma\left(  k\right)  }=%
\begin{cases}
1, & \text{if }k=\sigma\left(  k\right)  ;\\
0, & \text{if }k\neq\sigma\left(  k\right)
\end{cases}
\ \ =0\ \ \ \ \ \ \ \ \ \ \left(  \text{because }k\neq\sigma\left(  k\right)
\text{ (since }\sigma\left(  k\right)  \neq k\text{)}\right)  .
\]
But $x_{k,\sigma\left(  k\right)  }$ is a factor of the product $\prod
_{i=1}^{n}x_{i,\sigma\left(  i\right)  }$. Hence, the product $\prod_{i=1}%
^{n}x_{i,\sigma\left(  i\right)  }$ has a factor equal to $0$ (namely, the
factor $x_{k,\sigma\left(  k\right)  }=0$). Therefore, this whole product must
be $0$. In other words, $\prod_{i=1}^{n}x_{i,\sigma\left(  i\right)  }=0$.

Forget that we fixed $\sigma$. We thus have shown that if $\sigma\in S_{n}$ is
any permutation satisfying $\sigma\neq\operatorname*{id}$, then%
\begin{equation}
\prod_{i=1}^{n}x_{i,\sigma\left(  i\right)  }=0.
\label{pf.lem.almkvist-lem2.c1.pf.2}%
\end{equation}

Next, let us compute the product $\prod_{i=1}^{n}x_{i,\operatorname*{id}%
\left(  i\right)  }$. We have%
\begin{equation}
\prod_{i=1}^{n}\underbrace{x_{i,\operatorname*{id}\left(  i\right)  }%
}_{\substack{=x_{i,i}\\\text{(since }\operatorname*{id}\left(  i\right)
=i\text{)}}}=\prod_{i=1}^{n}x_{i,i}=x_{p,p}\cdot\prod_{i\in\left\{
1,2,\ldots,n\right\}  \setminus\left\{  p\right\}  }x_{i,i}
\label{pf.lem.almkvist-lem2.c1.pf.3}%
\end{equation}
(here, we have split off the factor for $i=p$ from the product). But
(\ref{pf.lem.almkvist-lem2.xpq=}) (applied to $q=p$) yields $x_{p,p}=a_{p,p}$.
Moreover, for each $i\in\left\{  1,2,\ldots,n\right\}  \setminus\left\{
p\right\}  $, we have%
\begin{align}
x_{i,i}  &  =%
\begin{cases}
1, & \text{if }i=i;\\
0, & \text{if }i\neq i
\end{cases}
\ \ \ \ \ \ \ \ \ \ \left(  \text{by (\ref{pf.lem.almkvist-lem2.xkq=}),
applied to }k=i\text{ and }q=i\right) \nonumber\\
&  =1 \label{pf.lem.almkvist-lem2.c1.pf.4}%
\end{align}
(since $i=i$). Hence, (\ref{pf.lem.almkvist-lem2.c1.pf.3}) becomes%
\[
\prod_{i=1}^{n}x_{i,\operatorname*{id}\left(  i\right)  }=\underbrace{x_{p,p}%
}_{=a_{p,p}}\cdot\prod_{i\in\left\{  1,2,\ldots,n\right\}  \setminus\left\{
p\right\}  }\underbrace{x_{i,i}}_{\substack{=1\\\text{(by
(\ref{pf.lem.almkvist-lem2.c1.pf.4}))}}}=a_{p,p}\cdot\underbrace{\prod
_{i\in\left\{  1,2,\ldots,n\right\}  \setminus\left\{  p\right\}  }1}%
_{=1}=a_{p,p}.
\]

Now, (\ref{pf.lem.almkvist-lem2.c1.pf.1}) becomes%
\[
\det X=\underbrace{\left(  -1\right)  ^{\operatorname*{id}}}_{=1}%
\underbrace{\prod_{i=1}^{n}x_{i,\operatorname*{id}\left(  i\right)  }%
}_{=a_{p,p}}+\sum_{\substack{\sigma\in S_{n};\\\sigma\neq\operatorname*{id}%
}}\left(  -1\right)  ^{\sigma}\underbrace{\prod_{i=1}^{n}x_{i,\sigma\left(
i\right)  }}_{\substack{=0\\\text{(by (\ref{pf.lem.almkvist-lem2.c1.pf.2}))}%
}}=a_{p,p}+\underbrace{\sum_{\substack{\sigma\in S_{n};\\\sigma\neq
\operatorname*{id}}}\left(  -1\right)  ^{\sigma}0}_{=0}=a_{p,p}.
\]
This proves Claim 1.
\end{proof}

\begin{statement}
\textit{Claim 2:} We have $B_{p}^{\prime}=XB$.
\end{statement}

\begin{proof}
[Proof of Claim 2.]We must prove that the two $n\times n$-matrices
$B_{p}^{\prime}$ and $XB$ are equal. Clearly, it suffices to show that each
row of the former matrix equals the corresponding row of the latter. In other
words, it suffices to show that each $k\in\left\{  1,2,\ldots,n\right\}  $
satisfies $\operatorname{row}_{k}\left(  B_{p}^{\prime}\right)
=\operatorname{row}_{k}\left(  XB\right)  $.

So let us do this. Fix $k\in\left\{  1,2,\ldots,n\right\}  $. We must show
that $\operatorname{row}_{k}\left(  B_{p}^{\prime}\right)  =\operatorname{row}%
_{k}\left(  XB\right)  $.

Comparing%
\begin{align*}
\operatorname{row}_{p}\left(  B_{p}^{\prime}\right)   &  =\operatorname{row}%
_{p}\left(  AB\right)  \ \ \ \ \ \ \ \ \ \ \left(  \text{by
(\ref{pf.lem.almkvist-lem2.ro1})}\right) \\
&  =\left(  \operatorname{row}_{p}A\right)  \cdot B\ \ \ \ \ \ \ \ \ \ \left(
\text{by Lemma \ref{lem.row.rowAB}, applied to }j=p\right)
\end{align*}
with%
\begin{align*}
\operatorname{row}_{p}\left(  XB\right)   &  =\left(  \operatorname{row}%
_{p}X\right)  \cdot B\ \ \ \ \ \ \ \ \ \ \left(  \text{by Lemma
\ref{lem.row.rowAB}, applied to }p\text{ and }X\text{ instead of }j\text{ and
}A\right) \\
&  =\left(  \operatorname{row}_{p}A\right)  \cdot B\ \ \ \ \ \ \ \ \ \ \left(
\text{by (\ref{pf.lem.almkvist-lem2.ro3})}\right)  ,
\end{align*}
we obtain $\operatorname{row}_{p}\left(  B_{p}^{\prime}\right)
=\operatorname{row}_{p}\left(  XB\right)  $. If $k=p$, then this already
achieves our goal of proving that $\operatorname{row}_{k}\left(  B_{p}%
^{\prime}\right)  =\operatorname{row}_{k}\left(  XB\right)  $. Hence, for the
rest of this proof, we can WLOG assume that $k\neq p$.

Assume this. Thus, $k\in\left\{  1,2,\ldots,n\right\}  \setminus\left\{
p\right\}  $ (since $k\neq p$). Hence,
\begin{align*}
\operatorname{row}_{k}\left(  B_{p}^{\prime}\right)   &  =\operatorname{row}%
_{k}B\ \ \ \ \ \ \ \ \ \ \left(  \text{by (\ref{pf.lem.almkvist-lem2.ro2}%
)}\right) \\
&  =\operatorname{row}_{k}\left(  I_{n}B\right)  \ \ \ \ \ \ \ \ \ \ \left(
\text{since }B=I_{n}B\right) \\
&  =\left(  \operatorname{row}_{k}\left(  I_{n}\right)  \right)  \cdot
B\ \ \ \ \ \ \ \ \ \ \left(  \text{by Lemma \ref{lem.row.rowAB}, applied to
}k\text{ and }I_{n}\text{ instead of }j\text{ and }A\right)
\end{align*}
and%
\begin{align*}
\operatorname{row}_{k}\left(  XB\right)   &  =\left(  \operatorname{row}%
_{k}X\right)  \cdot B\ \ \ \ \ \ \ \ \ \ \left(  \text{by Lemma
\ref{lem.row.rowAB}, applied to }k\text{ and }X\text{ instead of }j\text{ and
}A\right) \\
&  =\left(  \operatorname{row}_{k}\left(  I_{n}\right)  \right)  \cdot B\\
&  \ \ \ \ \ \ \ \ \ \ \ \ \ \ \ \ \ \ \ \ \left(  \text{since }%
\operatorname{row}_{k}X=\operatorname{row}_{k}\left(  I_{n}\right)  \text{ (by
(\ref{pf.lem.almkvist-lem2.ro4}), since }k\in\left\{  1,2,\ldots,n\right\}
\setminus\left\{  p\right\}  \text{)}\right)  .
\end{align*}
Comparing these two equalities, we obtain $\operatorname{row}_{k}\left(
B_{p}^{\prime}\right)  =\operatorname{row}_{k}\left(  XB\right)  $, which is
precisely what we needed to show. Thus, the proof of Claim 2 is complete.
\end{proof}

Now, \cite[Theorem 6.23]{detnotes} (applied to $X$ instead of $A$) shows that
$\det\left(  XB\right)  =\left(  \det X\right)  \cdot\det B$. But Claim 2
shows that $B_{p}^{\prime}=XB$. Hence,%
\begin{equation}
\det\left(  B_{p}^{\prime}\right)  =\det\left(  XB\right)
=\underbrace{\left(  \det X\right)  }_{\substack{=a_{p,p}\\\text{(by Claim
1)}}}\cdot\,\det B=a_{p,p}\det B. \label{pf.lem.almkvist-lem2.at}%
\end{equation}

Forget that we fixed $p$. We thus have proved (\ref{pf.lem.almkvist-lem2.at})
for each $p\in\left\{  1,2,\ldots,n\right\}  $. Now,%
\[
\sum_{j=1}^{n}\underbrace{\det\left(  B_{j}^{\prime}\right)  }%
_{\substack{=a_{j,j}\det B\\\text{(by (\ref{pf.lem.almkvist-lem2.at}), applied
to }p=j\text{)}}}=\underbrace{\sum_{j=1}^{n}a_{j,j}}%
_{\substack{=\operatorname*{Tr}A\\\text{(by (\ref{pf.lem.almkvist-lem2.tr}))}%
}}\det B=\left(  \operatorname*{Tr}A\right)  \cdot\det B.
\]
This proves Lemma \ref{lem.almkvist-lem2}.
\end{proof}
\end{noncompile}

\begin{lemma}
\label{lem.almkvist-lem2BA}Let $n\in\mathbb{N}$. Let $A$ and $B$ be two
matrices in $\mathbb{K}^{n\times n}$. For each $j\in\left\{  1,2,\ldots
,n\right\}  $, let $B_{j}^{\prime}$ be the $n\times n$-matrix obtained from
$B$ by replacing the $j$-th row of $B$ by the $j$-th row of $BA$. Then,%
\[
\sum_{j=1}^{n}\det\left(  B_{j}^{\prime}\right)  =\left(  \operatorname*{Tr}%
A\right)  \cdot\det B.
\]

\end{lemma}

\begin{vershort}
\begin{proof}
[Proof idea for Lemma \ref{lem.almkvist-lem2BA}.]Hint: Laplace expansion (more
precisely, \cite[Theorem 6.82 \textbf{(a)} and \textbf{(b)} and Proposition
6.96 \textbf{(b)}]{detnotes}). More specifically, expand each $\det\left(
B_{j}^{\prime}\right)  $ along the $j$-th row, showing that%
\[
\det\left(  B_{j}^{\prime}\right)  =\sum_{k=1}^{n}\ \ \sum_{q=1}^{n}\left(
-1\right)  ^{j+q}B_{j,k}A_{k,q}\det\left(  B_{\sim j,\sim q}\right)
\ \ \ \ \ \ \ \ \ \ \text{for each }j\in\left\{  1,2,\ldots,n\right\}
\]
(where $X_{i,j}$ denotes the $\left(  i,j\right)  $-th entry of any matrix
$X$). Then, sum over all $j$, and recombine the resulting triple sums into
Laplace expansions (this time along columns).
\end{proof}
\end{vershort}

\begin{verlong}
\begin{proof}
[Proof of Lemma \ref{lem.almkvist-lem2BA}.]Recall Definition
\ref{def.submatrix.minor}. For any matrix $X$ and any positive integers $i$
and $j$, we let $X_{i,j}$ denote the $\left(  i,j\right)  $-th entry of $X$
(assuming that such an entry exists). We will use the following three basic
properties of determinants:

\begin{itemize}
\item For any $n\times n$-matrix $X$ and any $p\in\left\{  1,2,\ldots
,n\right\}  $, we have%
\begin{equation}
\det X=\sum_{q=1}^{n}\left(  -1\right)  ^{p+q}X_{p,q}\det\left(  X_{\sim
p,\sim q}\right)  . \label{pf.lem.almkvist-lem2BA.lap-row}%
\end{equation}
(This is \cite[Theorem 6.82 \textbf{(a)}]{detnotes}, and is known as Laplace
expansion along the $p$-th row.)

\item For any $n\times n$-matrix $X$ and any $q\in\left\{  1,2,\ldots
,n\right\}  $, we have%
\[
\det X=\sum_{p=1}^{n}\left(  -1\right)  ^{p+q}X_{p,q}\det\left(  X_{\sim
p,\sim q}\right)  .
\]
(This is \cite[Theorem 6.82 \textbf{(b)}]{detnotes}, and is known as Laplace
expansion along the $q$-th column.) In other words, for any $n\times n$-matrix
$X$ and any $q\in\left\{  1,2,\ldots,n\right\}  $, we have%
\begin{equation}
\sum_{p=1}^{n}\left(  -1\right)  ^{p+q}X_{p,q}\det\left(  X_{\sim p,\sim
q}\right)  =\det X. \label{pf.lem.almkvist-lem2BA.lap-col}%
\end{equation}

\item For any $n\times n$-matrix $X$ and any $q,r\in\left\{  1,2,\ldots
,n\right\}  $ satisfying $q\neq r$, we have%
\[
0=\sum_{p=1}^{n}\left(  -1\right)  ^{p+q}X_{p,r}\det\left(  X_{\sim p,\sim
q}\right)  .
\]
(This is \cite[Proposition 6.96 \textbf{(b)}]{detnotes}, and follows easily
from Laplace expansion.) In other words, for any $n\times n$-matrix $X$ and
any $q,r\in\left\{  1,2,\ldots,n\right\}  $ satisfying $q\neq r$, we have%
\begin{equation}
\sum_{p=1}^{n}\left(  -1\right)  ^{p+q}X_{p,r}\det\left(  X_{\sim p,\sim
q}\right)  =0. \label{pf.lem.almkvist-lem2BA.lap-col0}%
\end{equation}

\end{itemize}

Let us also note that, by the definition of the trace of a matrix, we have%
\begin{equation}
\operatorname*{Tr}A=A_{1,1}+A_{2,2}+\cdots+A_{n,n}=\sum_{k=1}^{n}A_{k,k}.
\label{pf.lem.almkvist-lem2BA.Tr}%
\end{equation}

Fix $j\in\left\{  1,2,\ldots,n\right\}  $. Then, the $j$-th row of the matrix
$B_{j}^{\prime}$ equals the $j$-th row of $BA$ (by the definition of
$B_{j}^{\prime}$). In other words,%
\begin{equation}
\left(  B_{j}^{\prime}\right)  _{j,q}=\left(  BA\right)  _{j,q}%
\ \ \ \ \ \ \ \ \ \ \text{for each }q\in\left\{  1,2,\ldots,n\right\}
\label{pf.lem.almkvist-lem2BA.1}%
\end{equation}
(since the $j$-th row of an $n\times n$-matrix $X$ consists of the entries
$X_{j,q}$ for $q\in\left\{  1,2,\ldots,n\right\}  $).

Moreover, the matrix $B_{j}^{\prime}$ agrees with the matrix $B$ in all its
rows except for the $j$-th row (by the definition of $B_{j}^{\prime}$). Thus,
if we remove the $j$-th rows from the matrices $B_{j}^{\prime}$ and $B$, then
these matrices become equal. Hence, the submatrix $\left(  B_{j}^{\prime
}\right)  _{\sim j,\sim q}$ of $B_{j}^{\prime}$ equals the corresponding
submatrix $B_{\sim j,\sim q}$ of $B$ (since these submatrices are obtained by
removing the $j$-th row and the $q$-th column from $B_{j}^{\prime}$ and $B$,
and as we just saw, the removal of the $j$-th rows renders the matrices
$B_{j}^{\prime}$ and $B$ equal). In other words, $\left(  B_{j}^{\prime
}\right)  _{\sim j,\sim q}=B_{\sim j,\sim q}$.

Now, (\ref{pf.lem.almkvist-lem2BA.lap-row}) (applied to $p=j$ and
$X=B_{j}^{\prime}$) yields%
\begin{align}
\det\left(  B_{j}^{\prime}\right)   &  =\sum_{q=1}^{n}\left(  -1\right)
^{j+q}\underbrace{\left(  B_{j}^{\prime}\right)  _{j,q}}_{\substack{=\left(
BA\right)  _{j,q}\\\text{(by (\ref{pf.lem.almkvist-lem2BA.1}))}}%
}\det\underbrace{\left(  \left(  B_{j}^{\prime}\right)  _{\sim j,\sim
q}\right)  }_{\substack{=B_{\sim j,\sim q}}}\nonumber\\
&  =\sum_{q=1}^{n}\left(  -1\right)  ^{j+q}\underbrace{\left(  BA\right)
_{j,q}}_{\substack{=\sum_{k=1}^{n}B_{j,k}A_{k,q}\\\text{(by the definition of
the product}\\\text{of two matrices)}}}\det\left(  B_{\sim j,\sim q}\right)
\nonumber\\
&  =\sum_{q=1}^{n}\left(  -1\right)  ^{j+q}\sum_{k=1}^{n}B_{j,k}A_{k,q}%
\det\left(  B_{\sim j,\sim q}\right) \nonumber\\
&  =\sum_{k=1}^{n}\ \ \sum_{q=1}^{n}\left(  -1\right)  ^{j+q}B_{j,k}%
A_{k,q}\det\left(  B_{\sim j,\sim q}\right)  .
\label{pf.lem.almkvist-lem2BA.d}%
\end{align}

Forget that we fixed $j$. We thus have proved (\ref{pf.lem.almkvist-lem2BA.d})
for each $j\in\left\{  1,2,\ldots,n\right\}  $. Now, summing
(\ref{pf.lem.almkvist-lem2BA.d}) over all $j\in\left\{  1,2,\ldots,n\right\}
$, we obtain%
\begin{align}
\sum_{j=1}^{n}\det\left(  B_{j}^{\prime}\right)   &  =\sum_{j=1}^{n}%
\ \ \sum_{k=1}^{n}\ \ \sum_{q=1}^{n}\left(  -1\right)  ^{j+q}B_{j,k}%
A_{k,q}\det\left(  B_{\sim j,\sim q}\right) \nonumber\\
&  =\sum_{k=1}^{n}\ \ \sum_{q=1}^{n}A_{k,q}\sum_{j=1}^{n}\left(  -1\right)
^{j+q}B_{j,k}\det\left(  B_{\sim j,\sim q}\right)  .
\label{pf.lem.almkvist-lem2BA.5}%
\end{align}
But each $k\in\left\{  1,2,\ldots,n\right\}  $ satisfies%
\begin{align}
&  \sum_{q=1}^{n}A_{k,q}\sum_{j=1}^{n}\left(  -1\right)  ^{j+q}B_{j,k}%
\det\left(  B_{\sim j,\sim q}\right) \nonumber\\
&  =\sum_{q=1}^{n}A_{k,q}\sum_{p=1}^{n}\left(  -1\right)  ^{p+q}B_{p,k}%
\det\left(  B_{\sim p,\sim q}\right) \nonumber\\
&  \ \ \ \ \ \ \ \ \ \ \ \ \ \ \ \ \ \ \ \ \left(
\begin{array}
[c]{c}%
\text{here, we have renamed the}\\
\text{summation index }j\text{ as }p\text{ in the inner sum}%
\end{array}
\right) \nonumber\\
&  =A_{k,k}\underbrace{\sum_{p=1}^{n}\left(  -1\right)  ^{p+k}B_{p,k}%
\det\left(  B_{\sim p,\sim k}\right)  }_{\substack{=\det B\\\text{(by
(\ref{pf.lem.almkvist-lem2BA.lap-col}), applied to }X=B\text{ and }%
q=k\text{)}}}\nonumber\\
&  \ \ \ \ \ \ \ \ \ \ +\sum_{\substack{q\in\left\{  1,2,\ldots,n\right\}
;\\q\neq k}}A_{k,q}\underbrace{\sum_{p=1}^{n}\left(  -1\right)  ^{p+q}%
B_{p,k}\det\left(  B_{\sim p,\sim q}\right)  }_{\substack{=0\\\text{(by
(\ref{pf.lem.almkvist-lem2BA.lap-col0}), applied to }X=B\text{ and
}r=k\text{)}}}\nonumber\\
&  \ \ \ \ \ \ \ \ \ \ \ \ \ \ \ \ \ \ \ \ \left(
\begin{array}
[c]{c}%
\text{here, we have split off the addend}\\
\text{for }q=k\text{ from the outer sum}%
\end{array}
\right) \nonumber\\
&  =A_{k,k}\det B+\underbrace{\sum_{\substack{q\in\left\{  1,2,\ldots
,n\right\}  ;\\q\neq k}}A_{k,q}0}_{=0}=A_{k,k}\det B.
\label{pf.lem.almkvist-lem2BA.f}%
\end{align}
Thus, (\ref{pf.lem.almkvist-lem2BA.5}) becomes%
\begin{align*}
\sum_{j=1}^{n}\det\left(  B_{j}^{\prime}\right)   &  =\sum_{k=1}%
^{n}\ \ \underbrace{\sum_{q=1}^{n}A_{k,q}\sum_{j=1}^{n}\left(  -1\right)
^{j+q}B_{j,k}\det\left(  B_{\sim j,\sim q}\right)  }_{\substack{=A_{k,k}\det
B\\\text{(by (\ref{pf.lem.almkvist-lem2BA.f}))}}}\\
&  =\underbrace{\sum_{k=1}^{n}A_{k,k}}_{\substack{=\operatorname*{Tr}%
A\\\text{(by (\ref{pf.lem.almkvist-lem2BA.Tr}))}}}\det B=\left(
\operatorname*{Tr}A\right)  \cdot\det B.
\end{align*}
This proves Lemma \ref{lem.almkvist-lem2BA}.
\end{proof}
\end{verlong}

\begin{lemma}
\label{lem.multinom.rec}Let $m\in\mathbb{N}$ and $i_{1},i_{2},\ldots,i_{n}%
\in\mathbb{N}$ be such that $i_{1}+i_{2}+\cdots+i_{n}=m>0$. Then,
\[
\dbinom{m}{i_{1},i_{2},\ldots,i_{n}}=\sum_{\substack{j\in\left\{
1,2,\ldots,n\right\}  ;\\i_{j}\geq1}}\dbinom{m-1}{i_{1},i_{2},\ldots
,i_{j-1},i_{j}-1,i_{j+1},i_{j+2},\ldots,i_{n}}.
\]

\end{lemma}

\begin{vershort}
\begin{proof}
[Proof idea for Lemma \ref{lem.multinom.rec}.]This is an easy computation
using the fact that $\left(  k-1\right)  !=k!/k$ for each positive integer $k$.
\end{proof}
\end{vershort}

\begin{verlong}
\begin{proof}
[Proof of Lemma \ref{lem.multinom.rec}.]For each $j\in\left\{  1,2,\ldots
,n\right\}  $ that satisfies $i_{j}<1$, we have%
\begin{equation}
i_{j}=0 \label{pf.lem.multinom.rec.0}%
\end{equation}
(because from $i_{1},i_{2},\ldots,i_{n}\in\mathbb{N}$, we obtain $i_{j}%
\in\mathbb{N}=\left\{  0,1,2,\ldots\right\}  $, and thus the only way we can
have $i_{j}<1$ is by having $i_{j}=0$). Next, by assumption, we have%
\begin{align}
m  &  =i_{1}+i_{2}+\cdots+i_{n}=\sum_{j\in\left\{  1,2,\ldots,n\right\}
}i_{j}=\sum_{\substack{j\in\left\{  1,2,\ldots,n\right\}  ;\\i_{j}\geq1}%
}i_{j}+\sum_{\substack{j\in\left\{  1,2,\ldots,n\right\}  ;\\i_{j}%
<1}}\underbrace{i_{j}}_{\substack{=0\\\text{(by (\ref{pf.lem.multinom.rec.0}%
))}}}\nonumber\\
&  \ \ \ \ \ \ \ \ \ \ \ \ \ \ \ \ \ \ \ \ \left(
\begin{array}
[c]{c}%
\text{since each }j\in\left\{  1,2,\ldots,n\right\}  \text{ satisfies}\\
\text{either }i_{j}\geq1\text{ or }i_{j}<1\text{ (but not both)}%
\end{array}
\right) \nonumber\\
&  =\sum_{\substack{j\in\left\{  1,2,\ldots,n\right\}  ;\\i_{j}\geq1}%
}i_{j}+\underbrace{\sum_{\substack{j\in\left\{  1,2,\ldots,n\right\}
;\\i_{j}<1}}0}_{=0}=\sum_{\substack{j\in\left\{  1,2,\ldots,n\right\}
;\\i_{j}\geq1}}i_{j}. \label{pf.lem.multinom.rec.2}%
\end{align}

Now we claim:

\begin{statement}
\textit{Claim 1:} Let $j\in\left\{  1,2,\ldots,n\right\}  $ be such that
$i_{j}\geq1$. Then,%
\[
\dbinom{m-1}{i_{1},i_{2},\ldots,i_{j-1},i_{j}-1,i_{j+1},i_{j+2},\ldots,i_{n}%
}=\dfrac{i_{j}}{m}\cdot\dbinom{m}{i_{1},i_{2},\ldots,i_{n}}.
\]

\end{statement}

\begin{proof}
[Proof of Claim 1.]It is well-known that each positive integer $k$ satisfies
$\left(  k-1\right)  !=k!/k$. Hence, $\left(  m-1\right)  !=m!/m$ (since
$m>0$) and $\left(  i_{j}-1\right)  !=i_{j}!/i_{j}$ (since $i_{j}\geq1>0$).
Now, the definition of $\dbinom{m-1}{i_{1},i_{2},\ldots,i_{j-1},i_{j}%
-1,i_{j+1},i_{j+2},\ldots,i_{n}}$ yields%
\begin{align}
&  \dbinom{m-1}{i_{1},i_{2},\ldots,i_{j-1},i_{j}-1,i_{j+1},i_{j+2}%
,\ldots,i_{n}}\nonumber\\
&  =\dfrac{\left(  m-1\right)  !}{i_{1}!i_{2}!\cdots i_{j-1}!\left(
i_{j}-1\right)  !i_{j+1}!i_{j+2}!\cdots i_{n}!}\nonumber\\
&  =\dfrac{m!/m}{i_{1}!i_{2}!\cdots i_{j-1}!\left(  i_{j}!/i_{j}\right)
i_{j+1}!i_{j+2}!\cdots i_{n}!}\nonumber\\
&  \ \ \ \ \ \ \ \ \ \ \ \ \ \ \ \ \ \ \ \ \left(  \text{since }\left(
m-1\right)  !=m!/m\text{ and }\left(  i_{j}-1\right)  !=i_{j}!/i_{j}\right)
\nonumber\\
&  =\dfrac{i_{j}}{m}\cdot\dfrac{m!}{i_{1}!i_{2}!\cdots i_{j-1}!i_{j}%
!i_{j+1}!i_{j+2}!\cdots i_{n}!}\nonumber\\
&  =\dfrac{i_{j}}{m}\cdot\dfrac{m!}{i_{1}!i_{2}!\cdots i_{n}!}
\label{pf.lem.multinom.rec.c1.pf.2}%
\end{align}
(since $i_{1}!i_{2}!\cdots i_{j-1}!i_{j}!i_{j+1}!i_{j+2}!\cdots i_{n}%
!=i_{1}!i_{2}!\cdots i_{n}!$).

However, the definition of $\dbinom{m}{i_{1},i_{2},\ldots,i_{n}}$ yields
$\dbinom{m}{i_{1},i_{2},\ldots,i_{n}}=\dfrac{m!}{i_{1}!i_{2}!\cdots i_{n}!}$.
In view of this, we can rewrite (\ref{pf.lem.multinom.rec.c1.pf.2}) as%
\[
\dbinom{m-1}{i_{1},i_{2},\ldots,i_{j-1},i_{j}-1,i_{j+1},i_{j+2},\ldots,i_{n}%
}=\dfrac{i_{j}}{m}\cdot\dbinom{m}{i_{1},i_{2},\ldots,i_{n}}.
\]
This proves Claim 1.
\end{proof}

Now,%
\begin{align*}
&  \sum_{\substack{j\in\left\{  1,2,\ldots,n\right\}  ;\\i_{j}\geq
1}}\underbrace{\dbinom{m-1}{i_{1},i_{2},\ldots,i_{j-1},i_{j}-1,i_{j+1}%
,i_{j+2},\ldots,i_{n}}}_{\substack{=\dfrac{i_{j}}{m}\cdot\dbinom{m}%
{i_{1},i_{2},\ldots,i_{n}}\\\text{(by Claim 1)}}}\\
&  =\sum_{\substack{j\in\left\{  1,2,\ldots,n\right\}  ;\\i_{j}\geq1}%
}\dfrac{i_{j}}{m}\cdot\dbinom{m}{i_{1},i_{2},\ldots,i_{n}}=\dfrac{1}{m}%
\dbinom{m}{i_{1},i_{2},\ldots,i_{n}}\cdot\underbrace{\sum_{\substack{j\in
\left\{  1,2,\ldots,n\right\}  ;\\i_{j}\geq1}}i_{j}}_{\substack{=m\\\text{(by
(\ref{pf.lem.multinom.rec.2}))}}}\\
&  =\dfrac{1}{m}\dbinom{m}{i_{1},i_{2},\ldots,i_{n}}\cdot m=\dbinom{m}%
{i_{1},i_{2},\ldots,i_{n}}.
\end{align*}
This proves Lemma \ref{lem.multinom.rec}.
\end{proof}
\end{verlong}

\begin{lemma}
\label{lem.almkvist-lem3}Let $n\in\mathbb{N}$. Let $A\in\mathbb{K}^{n\times
n}$ be a matrix. For any $n$ integers $i_{1},i_{2},\ldots,i_{n}\in\mathbb{N}$,
we let $M_{i_{1},i_{2},\ldots,i_{n}}$ denote the $n\times n$-matrix whose rows
are
\[
\operatorname*{row}\nolimits_{1}\left(  A^{i_{1}}\right)
,\ \operatorname*{row}\nolimits_{2}\left(  A^{i_{2}}\right)  ,\ \ldots
,\ \operatorname*{row}\nolimits_{n}\left(  A^{i_{n}}\right)  .
\]

Let $i_{1},i_{2},\ldots,i_{n}\in\mathbb{N}$. Then,%
\[
\sum_{j=1}^{n}\det\left(  M_{i_{1},i_{2},\ldots,i_{j-1},i_{j}+1,i_{j+1}%
,i_{j+2},\ldots,i_{n}}\right)  =\left(  \operatorname*{Tr}A\right)  \cdot
\det\left(  M_{i_{1},i_{2},\ldots,i_{n}}\right)  .
\]

\end{lemma}

\begin{vershort}
\begin{proof}
[Proof idea for Lemma \ref{lem.almkvist-lem3}.]Apply Lemma
\ref{lem.almkvist-lem2BA} to $B=M_{i_{1},i_{2},\ldots,i_{n}}$, and show that
each of the matrices $B_{j}^{\prime}$ is precisely $M_{i_{1},i_{2}%
,\ldots,i_{j-1},i_{j}+1,i_{j+1},i_{j+2},\ldots,i_{n}}$.
\end{proof}
\end{vershort}

\begin{verlong}
\begin{proof}
[Proof of Lemma \ref{lem.almkvist-lem3}.]Define the $n\times n$-matrix
$B:=M_{i_{1},i_{2},\ldots,i_{n}}\in\mathbb{K}^{n\times n}$. For each
$j\in\left\{  1,2,\ldots,n\right\}  $, let $B_{j}^{\prime}$ be the $n\times
n$-matrix obtained from $B$ by replacing the $j$-th row of $B$ by the $j$-th
row of $BA$. Then, Lemma \ref{lem.almkvist-lem2BA} yields
\begin{equation}
\sum_{j=1}^{n}\det\left(  B_{j}^{\prime}\right)  =\left(  \operatorname*{Tr}%
A\right)  \cdot\det B. \label{pf.lem.almkvist-lem3.1}%
\end{equation}

Now, we claim the following:

\begin{statement}
\textit{Claim 1:} For each $j\in\left\{  1,2,\ldots,n\right\}  $, we have
$B_{j}^{\prime}=M_{i_{1},i_{2},\ldots,i_{j-1},i_{j}+1,i_{j+1},i_{j+2}%
,\ldots,i_{n}}$.
\end{statement}

\begin{proof}
[Proof of Claim 1.]Let $j\in\left\{  1,2,\ldots,n\right\}  $. Then,
$B_{j}^{\prime}$ is defined as the $n\times n$-matrix obtained from $B$ by
replacing the $j$-th row of $B$ by the $j$-th row of $BA$. Hence,
\begin{equation}
\operatorname*{row}\nolimits_{j}\left(  B_{j}^{\prime}\right)
=\operatorname*{row}\nolimits_{j}\left(  BA\right)  ,
\label{pf.lem.almkvist-lem3.c1.pf.1}%
\end{equation}
and%
\begin{equation}
\operatorname{row}_{k}\left(  B_{j}^{\prime}\right)  =\operatorname{row}%
_{k}B\ \ \ \ \ \ \ \ \ \ \text{for each }k\in\left\{  1,2,\ldots,n\right\}
\setminus\left\{  j\right\}  . \label{pf.lem.almkvist-lem3.c1.pf.2}%
\end{equation}

However, $B=M_{i_{1},i_{2},\ldots,i_{n}}$ is the $n\times n$-matrix whose rows
are
\[
\operatorname*{row}\nolimits_{1}\left(  A^{i_{1}}\right)
,\ \operatorname*{row}\nolimits_{2}\left(  A^{i_{2}}\right)  ,\ \ldots
,\ \operatorname*{row}\nolimits_{n}\left(  A^{i_{n}}\right)
\]
(by the definition of $M_{i_{1},i_{2},\ldots,i_{n}}$). Thus,%
\begin{equation}
\operatorname{row}_{k}B=\operatorname{row}_{k}\left(  A^{i_{k}}\right)
\ \ \ \ \ \ \ \ \ \ \text{for each }k\in\left\{  1,2,\ldots,n\right\}  .
\label{pf.lem.almkvist-lem3.c1.pf.4}%
\end{equation}

Now, let $\left(  p_{1},p_{2},\ldots,p_{n}\right)  $ denote the $n$-tuple
$\left(  i_{1},i_{2},\ldots,i_{j-1},i_{j}+1,i_{j+1},i_{j+2},\ldots
,i_{n}\right)  \in\mathbb{N}^{n}$. Then,
\begin{equation}
p_{j}=i_{j}+1, \label{pf.lem.almkvist-lem3.c1.pf.5}%
\end{equation}
and%
\begin{equation}
p_{k}=i_{k}\ \ \ \ \ \ \ \ \ \ \text{for each }k\in\left\{  1,2,\ldots
,n\right\}  \setminus\left\{  j\right\}  .
\label{pf.lem.almkvist-lem3.c1.pf.6}%
\end{equation}
From $\left(  p_{1},p_{2},\ldots,p_{n}\right)  =\left(  i_{1},i_{2}%
,\ldots,i_{j-1},i_{j}+1,i_{j+1},i_{j+2},\ldots,i_{n}\right)  $, we also
obtain
\begin{equation}
M_{p_{1},p_{2},\ldots,p_{n}}=M_{i_{1},i_{2},\ldots,i_{j-1},i_{j}%
+1,i_{j+1},i_{j+2},\ldots,i_{n}}. \label{pf.lem.almkvist-lem3.c1.pf.7}%
\end{equation}

Furthermore, $M_{p_{1},p_{2},\ldots,p_{n}}$ is the $n\times n$-matrix whose
rows are
\[
\operatorname*{row}\nolimits_{1}\left(  A^{p_{1}}\right)
,\ \operatorname*{row}\nolimits_{2}\left(  A^{p_{2}}\right)  ,\ \ldots
,\ \operatorname*{row}\nolimits_{n}\left(  A^{p_{n}}\right)
\]
(by the definition of $M_{p_{1},p_{2},\ldots,p_{n}}$). Thus,%
\begin{equation}
\operatorname{row}_{k}\left(  M_{p_{1},p_{2},\ldots,p_{n}}\right)
=\operatorname{row}_{k}\left(  A^{p_{k}}\right)  \ \ \ \ \ \ \ \ \ \ \text{for
each }k\in\left\{  1,2,\ldots,n\right\}  .
\label{pf.lem.almkvist-lem3.c1.pf.8}%
\end{equation}

From (\ref{pf.lem.almkvist-lem3.c1.pf.1}), we obtain%
\[
\operatorname*{row}\nolimits_{j}\left(  B_{j}^{\prime}\right)
=\operatorname*{row}\nolimits_{j}\left(  BA\right)  =\left(
\operatorname{row}_{j}B\right)  \cdot A
\]
(by Lemma \ref{lem.row.rowAB}, applied to $B$ and $A$ instead of $A$ and $B$).
In view of%
\[
\operatorname{row}_{j}B=\operatorname{row}_{j}\left(  A^{i_{j}}\right)
\ \ \ \ \ \ \ \ \ \ \left(  \text{by (\ref{pf.lem.almkvist-lem3.c1.pf.4}),
applied to }k=j\right)  ,
\]
we can rewrite this as%
\[
\operatorname*{row}\nolimits_{j}\left(  B_{j}^{\prime}\right)
=\operatorname{row}_{j}\left(  A^{i_{j}}\right)  \cdot A.
\]
Comparing this with%
\begin{align*}
\operatorname{row}_{j}\left(  M_{p_{1},p_{2},\ldots,p_{n}}\right)   &
=\operatorname{row}_{j}\left(  A^{p_{j}}\right)  \ \ \ \ \ \ \ \ \ \ \left(
\text{by (\ref{pf.lem.almkvist-lem3.c1.pf.8}), applied to }k=j\right) \\
&  =\operatorname{row}_{j}\left(  A^{i_{j}+1}\right)
\ \ \ \ \ \ \ \ \ \ \left(  \text{by (\ref{pf.lem.almkvist-lem3.c1.pf.5}%
)}\right) \\
&  =\operatorname{row}_{j}\left(  A^{i_{j}}A\right)
\ \ \ \ \ \ \ \ \ \ \left(  \text{since }A^{i_{j}+1}=A^{i_{j}}A\right) \\
&  =\operatorname{row}_{j}\left(  A^{i_{j}}\right)  \cdot
A\ \ \ \ \ \ \ \ \ \ \left(
\begin{array}
[c]{c}%
\text{by Lemma \ref{lem.row.rowAB}, applied to }A^{i_{j}}\text{ and }A\\
\text{instead of }A\text{ and }B
\end{array}
\right)  ,
\end{align*}
we obtain%
\begin{equation}
\operatorname*{row}\nolimits_{j}\left(  B_{j}^{\prime}\right)
=\operatorname{row}_{j}\left(  M_{p_{1},p_{2},\ldots,p_{n}}\right)  .
\label{pf.lem.almkvist-lem3.c1.pf.9}%
\end{equation}

We have to prove that $B_{j}^{\prime}=M_{i_{1},i_{2},\ldots,i_{j-1}%
,i_{j}+1,i_{j+1},i_{j+2},\ldots,i_{n}}$. In view of
(\ref{pf.lem.almkvist-lem3.c1.pf.7}), this means that we have to prove that
$B_{j}^{\prime}=M_{p_{1},p_{2},\ldots,p_{n}}$. In other words, we must prove
that the two $n\times n$-matrices $B_{j}^{\prime}$ and $M_{p_{1},p_{2}%
,\ldots,p_{n}}$ are equal. Clearly, it suffices to show that each row of the
former matrix equals the corresponding row of the latter matrix. In other
words, it suffices to show that each $k\in\left\{  1,2,\ldots,n\right\}  $
satisfies%
\begin{equation}
\operatorname*{row}\nolimits_{k}\left(  B_{j}^{\prime}\right)
=\operatorname{row}_{k}\left(  M_{p_{1},p_{2},\ldots,p_{n}}\right)  .
\label{pf.lem.almkvist-lem3.c1.pf.gl}%
\end{equation}

So let us prove this. Let $k\in\left\{  1,2,\ldots,n\right\}  $. We must prove
(\ref{pf.lem.almkvist-lem3.c1.pf.gl}). If $k=j$, then this follows immediately
from (\ref{pf.lem.almkvist-lem3.c1.pf.9}). Thus, we WLOG assume that $k\neq
j$. Hence, $k\in\left\{  1,2,\ldots,n\right\}  \setminus\left\{  j\right\}  $,
so that $p_{k}=i_{k}$ (by (\ref{pf.lem.almkvist-lem3.c1.pf.6})), and thus
$i_{k}=p_{k}$. Furthermore, from $k\in\left\{  1,2,\ldots,n\right\}
\setminus\left\{  j\right\}  $, we obtain
\begin{align*}
\operatorname{row}_{k}\left(  B_{j}^{\prime}\right)   &  =\operatorname{row}%
_{k}B\ \ \ \ \ \ \ \ \ \ \left(  \text{by (\ref{pf.lem.almkvist-lem3.c1.pf.2}%
)}\right) \\
&  =\operatorname{row}_{k}\left(  A^{i_{k}}\right)
\ \ \ \ \ \ \ \ \ \ \left(  \text{by (\ref{pf.lem.almkvist-lem3.c1.pf.4}%
)}\right) \\
&  =\operatorname{row}_{k}\left(  A^{p_{k}}\right)
\ \ \ \ \ \ \ \ \ \ \left(  \text{since }i_{k}=p_{k}\right) \\
&  =\operatorname{row}_{k}\left(  M_{p_{1},p_{2},\ldots,p_{n}}\right)
\ \ \ \ \ \ \ \ \ \ \left(  \text{by (\ref{pf.lem.almkvist-lem3.c1.pf.8}%
)}\right)  .
\end{align*}
Thus, (\ref{pf.lem.almkvist-lem3.c1.pf.gl}) is proved. As we explained, this
completes the proof of Claim 1.
\end{proof}

Now, Claim 1 shows that $B_{j}^{\prime}=M_{i_{1},i_{2},\ldots,i_{j-1}%
,i_{j}+1,i_{j+1},i_{j+2},\ldots,i_{n}}$ for each $j\in\left\{  1,2,\ldots
,n\right\}  $. Thus, we can rewrite (\ref{pf.lem.almkvist-lem3.1}) as%
\begin{align*}
\sum_{j=1}^{n}\det\left(  M_{i_{1},i_{2},\ldots,i_{j-1},i_{j}+1,i_{j+1}%
,i_{j+2},\ldots,i_{n}}\right)   &  =\left(  \operatorname*{Tr}A\right)
\cdot\det\underbrace{B}_{=M_{i_{1},i_{2},\ldots,i_{n}}}\\
&  =\left(  \operatorname*{Tr}A\right)  \cdot\det\left(  M_{i_{1},i_{2}%
,\ldots,i_{n}}\right)  .
\end{align*}
This proves Lemma \ref{lem.almkvist-lem3}.
\end{proof}
\end{verlong}

\begin{vershort}
\begin{proof}
[Proof idea for Proposition \ref{prop.almkvist-lem}.]Induct on $m$. In the
induction step (from $m-1$ to $m$), start by substituting Lemma
\ref{lem.multinom.rec} into the right hand side. Interchange the summation
signs and substitute $i_{j}$ for $i_{j}-1$. Then use Lemma
\ref{lem.almkvist-lem3} to simplify the inner sum $\sum_{j=1}^{n}\det\left(
M_{i_{1},i_{2},\ldots,i_{j-1},i_{j}+1,i_{j+1},i_{j+2},\ldots,i_{n}}\right)  $.
\end{proof}
\end{vershort}

\begin{verlong}
\begin{proof}
[Proof of Proposition \ref{prop.almkvist-lem}.]We induct on $m$.

\textit{Base case:} Let us prove Proposition \ref{prop.almkvist-lem} for
$m=0$. Indeed, there is only one $n$-tuple $\left(  i_{1},i_{2},\ldots
,i_{n}\right)  \in\mathbb{N}^{n}$ satisfying $i_{1}+i_{2}+\cdots+i_{n}=0$,
namely the $n$-tuple $\left(  0,0,\ldots,0\right)  $ (since a sum of
nonnegative integers cannot be $0$ unless all its summands are $0$). Thus, the
sum%
\[
\sum_{\substack{\left(  i_{1},i_{2},\ldots,i_{n}\right)  \in\mathbb{N}%
^{n};\\i_{1}+i_{2}+\cdots+i_{n}=0}}\dbinom{0}{i_{1},i_{2},\ldots,i_{n}}%
\det\left(  M_{i_{1},i_{2},\ldots,i_{n}}\right)
\]
has only one addend, and simplifies as follows:%
\begin{align}
&  \sum_{\substack{\left(  i_{1},i_{2},\ldots,i_{n}\right)  \in\mathbb{N}%
^{n};\\i_{1}+i_{2}+\cdots+i_{n}=0}}\dbinom{0}{i_{1},i_{2},\ldots,i_{n}}%
\det\left(  M_{i_{1},i_{2},\ldots,i_{n}}\right) \nonumber\\
&  =\underbrace{\dbinom{0}{0,0,\ldots,0}}_{=1}\det\left(  M_{0,0,\ldots
,0}\right) \nonumber\\
&  =\det\left(  M_{0,0,\ldots,0}\right)  . \label{pf.prop.almkvist-lem.IB.2}%
\end{align}

But the matrix $M_{0,0,\ldots,0}$ is the $n\times n$-matrix whose rows are
\[
\operatorname*{row}\nolimits_{1}\left(  A^{0}\right)  ,\ \operatorname*{row}%
\nolimits_{2}\left(  A^{0}\right)  ,\ \ldots,\ \operatorname*{row}%
\nolimits_{n}\left(  A^{0}\right)
\]
(by the definition of $M_{0,0,\ldots,0}$). In other words, it is the $n\times
n$-matrix whose rows equal the respective rows of the identity matrix $A^{0}$.
Hence, this matrix $M_{0,0,\ldots,0}$ simply equals $A^{0}$. That is,
$M_{0,0,\ldots,0}=A^{0}=I_{n}$. Therefore, $\det\left(  M_{0,0,\ldots
,0}\right)  =\det\left(  I_{n}\right)  =1$. Comparing this with $\left(
\operatorname*{Tr}A\right)  ^{0}=1$, we obtain
\[
\left(  \operatorname*{Tr}A\right)  ^{0}=\det\left(  M_{0,0,\ldots,0}\right)
=\sum_{\substack{\left(  i_{1},i_{2},\ldots,i_{n}\right)  \in\mathbb{N}%
^{n};\\i_{1}+i_{2}+\cdots+i_{n}=0}}\dbinom{0}{i_{1},i_{2},\ldots,i_{n}}%
\det\left(  M_{i_{1},i_{2},\ldots,i_{n}}\right)
\]
(by (\ref{pf.prop.almkvist-lem.IB.2})). This proves Proposition
\ref{prop.almkvist-lem} for $m=0$. Thus, the base case is complete.

\textit{Induction step:} Let $m$ be a positive integer. Assume (as the
induction hypothesis) that Proposition \ref{prop.almkvist-lem} holds for $m-1$
instead of $m$. We must prove that Proposition \ref{prop.almkvist-lem} holds
for $m$ as well.

We have assumed that Proposition \ref{prop.almkvist-lem} holds for $m-1$
instead of $m$. In other words, we have%
\begin{equation}
\left(  \operatorname*{Tr}A\right)  ^{m-1}=\sum_{\substack{\left(  i_{1}%
,i_{2},\ldots,i_{n}\right)  \in\mathbb{N}^{n};\\i_{1}+i_{2}+\cdots+i_{n}%
=m-1}}\dbinom{m-1}{i_{1},i_{2},\ldots,i_{n}}\det\left(  M_{i_{1},i_{2}%
,\ldots,i_{n}}\right)  . \label{pf.prop.almkvist-lem.IH}%
\end{equation}

Now,%
\begin{align}
&  \sum_{\substack{\left(  i_{1},i_{2},\ldots,i_{n}\right)  \in\mathbb{N}%
^{n};\\i_{1}+i_{2}+\cdots+i_{n}=m}}\underbrace{\dbinom{m}{i_{1},i_{2}%
,\ldots,i_{n}}}_{\substack{=\sum_{\substack{j\in\left\{  1,2,\ldots,n\right\}
;\\i_{j}\geq1}}\dbinom{m-1}{i_{1},i_{2},\ldots,i_{j-1},i_{j}-1,i_{j+1}%
,i_{j+2},\ldots,i_{n}}\\\text{(by Lemma \ref{lem.multinom.rec})}}}\det\left(
M_{i_{1},i_{2},\ldots,i_{n}}\right) \nonumber\\
&  =\underbrace{\sum_{\substack{\left(  i_{1},i_{2},\ldots,i_{n}\right)
\in\mathbb{N}^{n};\\i_{1}+i_{2}+\cdots+i_{n}=m}}\ \ \sum_{\substack{j\in
\left\{  1,2,\ldots,n\right\}  ;\\i_{j}\geq1}}}_{=\sum_{j\in\left\{
1,2,\ldots,n\right\}  }\ \ \sum_{\substack{\left(  i_{1},i_{2},\ldots
,i_{n}\right)  \in\mathbb{N}^{n};\\i_{1}+i_{2}+\cdots+i_{n}=m;\\i_{j}\geq1}%
}}\dbinom{m-1}{i_{1},i_{2},\ldots,i_{j-1},i_{j}-1,i_{j+1},i_{j+2},\ldots
,i_{n}}\det\left(  M_{i_{1},i_{2},\ldots,i_{n}}\right) \nonumber\\
&  =\sum_{j\in\left\{  1,2,\ldots,n\right\}  }\ \ \sum_{\substack{\left(
i_{1},i_{2},\ldots,i_{n}\right)  \in\mathbb{N}^{n};\\i_{1}+i_{2}+\cdots
+i_{n}=m;\\i_{j}\geq1}}\dbinom{m-1}{i_{1},i_{2},\ldots,i_{j-1},i_{j}%
-1,i_{j+1},i_{j+2},\ldots,i_{n}}\det\left(  M_{i_{1},i_{2},\ldots,i_{n}%
}\right)  . \label{pf.prop.almkvist-lem.4}%
\end{align}

However, for each $j\in\left\{  1,2,\ldots,n\right\}  $, we have%
\begin{align}
&  \sum_{\substack{\left(  i_{1},i_{2},\ldots,i_{n}\right)  \in\mathbb{N}%
^{n};\\i_{1}+i_{2}+\cdots+i_{n}=m;\\i_{j}\geq1}}\dbinom{m-1}{i_{1}%
,i_{2},\ldots,i_{j-1},i_{j}-1,i_{j+1},i_{j+2},\ldots,i_{n}}\det\left(
M_{i_{1},i_{2},\ldots,i_{n}}\right) \nonumber\\
&  =\sum_{\substack{\left(  i_{1},i_{2},\ldots,i_{n}\right)  \in\mathbb{N}%
^{n};\\i_{1}+i_{2}+\cdots+i_{n}=m;\\i_{j}-1\in\mathbb{N}}}\dbinom{m-1}%
{i_{1},i_{2},\ldots,i_{j-1},i_{j}-1,i_{j+1},i_{j+2},\ldots,i_{n}}%
\det\underbrace{\left(  M_{i_{1},i_{2},\ldots,i_{n}}\right)  }%
_{\substack{=M_{i_{1},i_{2},\ldots,i_{j-1},i_{j},i_{j+1},i_{j+2},\ldots,i_{n}%
}\\=M_{i_{1},i_{2},\ldots,i_{j-1},\left(  i_{j}-1\right)  +1,i_{j+1}%
,i_{j+2},\ldots,i_{n}}\\\text{(since }i_{j}=\left(  i_{j}-1\right)
+1\text{)}}}\nonumber\\
&  \ \ \ \ \ \ \ \ \ \ \ \ \ \ \ \ \ \ \ \ \left(
\begin{array}
[c]{c}%
\text{here, we have replaced the condition \textquotedblleft}i_{j}%
\geq1\text{\textquotedblright\ under the}\\
\text{summation sign by the equivalent condition \textquotedblleft}i_{j}%
-1\in\mathbb{N}\text{\textquotedblright}%
\end{array}
\right) \nonumber\\
&  =\sum_{\substack{\left(  i_{1},i_{2},\ldots,i_{n}\right)  \in\mathbb{N}%
^{n};\\\left(  i_{1}+i_{2}+\cdots+i_{n}\right)  -1=m-1;\\i_{j}-1\in\mathbb{N}%
}}\dbinom{m-1}{i_{1},i_{2},\ldots,i_{j-1},i_{j}-1,i_{j+1},i_{j+2},\ldots
,i_{n}}\det\left(  M_{i_{1},i_{2},\ldots,i_{j-1},\left(  i_{j}-1\right)
+1,i_{j+1},i_{j+2},\ldots,i_{n}}\right) \nonumber\\
&  \ \ \ \ \ \ \ \ \ \ \ \ \ \ \ \ \ \ \ \ \left(
\begin{array}
[c]{c}%
\text{here, we have replaced the condition \textquotedblleft}i_{1}%
+i_{2}+\cdots+i_{n}=m\text{\textquotedblright}\\
\text{under the summation sign by the equivalent}\\
\text{condition \textquotedblleft}\left(  i_{1}+i_{2}+\cdots+i_{n}\right)
-1=m-1\text{\textquotedblright}\\
\text{(these two conditions are equivalent because the latter}\\
\text{can be obtained from the former by subtracting }-1\text{)}%
\end{array}
\right) \nonumber\\
&  =\sum_{\substack{\left(  i_{1},i_{2},\ldots,i_{n}\right)  \in\mathbb{N}%
^{n};\\i_{1}+i_{2}+\cdots+i_{j-1}+\left(  i_{j}-1\right)  +i_{j+1}%
+i_{j+2}+\cdots+i_{n}=m-1;\\i_{j}-1\in\mathbb{N}}}\dbinom{m-1}{i_{1}%
,i_{2},\ldots,i_{j-1},i_{j}-1,i_{j+1},i_{j+2},\ldots,i_{n}}\nonumber\\
&  \ \ \ \ \ \ \ \ \ \ \det\left(  M_{i_{1},i_{2},\ldots,i_{j-1},\left(
i_{j}-1\right)  +1,i_{j+1},i_{j+2},\ldots,i_{n}}\right) \nonumber\\
&  \ \ \ \ \ \ \ \ \ \ \ \ \ \ \ \ \ \ \ \ \left(
\begin{array}
[c]{c}%
\text{here, we have rewritten the expression }\left(  i_{1}+i_{2}+\cdots
+i_{n}\right)  -1\\
\text{under the summation sign}\\
\text{as }i_{1}+i_{2}+\cdots+i_{j-1}+\left(  i_{j}-1\right)  +i_{j+1}%
+i_{j+2}+\cdots+i_{n}%
\end{array}
\right) \nonumber\\
&  =\sum_{\substack{\left(  k_{1},k_{2},\ldots,k_{n}\right)  \in\mathbb{N}%
^{n};\\k_{1}+k_{2}+\cdots+k_{n}=m-1}}\dbinom{m-1}{k_{1},k_{2},\ldots,k_{n}%
}\det\left(  M_{k_{1},k_{2},\ldots,k_{j-1},k_{j}+1,k_{j+1},k_{j+2}%
,\ldots,k_{n}}\right) \nonumber\\
&  \ \ \ \ \ \ \ \ \ \ \ \ \ \ \ \ \ \ \ \ \left(
\begin{array}
[c]{c}%
\text{here, we have substituted }\left(  k_{1},k_{2},\ldots,k_{n}\right) \\
\text{for }\left(  i_{1},i_{2},\ldots,i_{j-1},i_{j}-1,i_{j+1},i_{j+2}%
,\ldots,i_{n}\right)  \text{ in our sum}\\
\text{(because the condition \textquotedblleft}i_{j}-1\in\mathbb{N}%
\text{\textquotedblright\ ensures}\\
\text{that }\left(  i_{1},i_{2},\ldots,i_{j-1},i_{j}-1,i_{j+1},i_{j+2}%
,\ldots,i_{n}\right)  \in\mathbb{N}^{n}\text{)}%
\end{array}
\right) \nonumber\\
&  =\sum_{\substack{\left(  i_{1},i_{2},\ldots,i_{n}\right)  \in\mathbb{N}%
^{n};\\i_{1}+i_{2}+\cdots+i_{n}=m-1}}\dbinom{m-1}{i_{1},i_{2},\ldots,i_{n}%
}\det\left(  M_{i_{1},i_{2},\ldots,i_{j-1},i_{j}+1,i_{j+1},i_{j+2}%
,\ldots,i_{n}}\right)  \label{pf.prop.almkvist-lem.5}%
\end{align}
(here, we have renamed the summation index $\left(  k_{1},k_{2},\ldots
,k_{n}\right)  $ as $\left(  i_{1},i_{2},\ldots,i_{n}\right)  $).

Thus, our above computation (\ref{pf.prop.almkvist-lem.4}) becomes%
\begin{align*}
&  \sum_{\substack{\left(  i_{1},i_{2},\ldots,i_{n}\right)  \in\mathbb{N}%
^{n};\\i_{1}+i_{2}+\cdots+i_{n}=m}}\dbinom{m}{i_{1},i_{2},\ldots,i_{n}}%
\det\left(  M_{i_{1},i_{2},\ldots,i_{n}}\right) \\
&  =\underbrace{\sum_{j\in\left\{  1,2,\ldots,n\right\}  }}_{=\sum_{j=1}^{n}%
}\ \ \underbrace{\sum_{\substack{\left(  i_{1},i_{2},\ldots,i_{n}\right)
\in\mathbb{N}^{n};\\i_{1}+i_{2}+\cdots+i_{n}=m;\\i_{j}\geq1}}\dbinom
{m-1}{i_{1},i_{2},\ldots,i_{j-1},i_{j}-1,i_{j+1},i_{j+2},\ldots,i_{n}}%
\det\left(  M_{i_{1},i_{2},\ldots,i_{n}}\right)  }_{\substack{=\sum
_{\substack{\left(  i_{1},i_{2},\ldots,i_{n}\right)  \in\mathbb{N}^{n}%
;\\i_{1}+i_{2}+\cdots+i_{n}=m-1}}\dbinom{m-1}{i_{1},i_{2},\ldots,i_{n}}%
\det\left(  M_{i_{1},i_{2},\ldots,i_{j-1},i_{j}+1,i_{j+1},i_{j+2},\ldots
,i_{n}}\right)  \\\text{(by (\ref{pf.prop.almkvist-lem.5}))}}}\\
&  =\sum_{j=1}^{n}\ \ \sum_{\substack{\left(  i_{1},i_{2},\ldots,i_{n}\right)
\in\mathbb{N}^{n};\\i_{1}+i_{2}+\cdots+i_{n}=m-1}}\dbinom{m-1}{i_{1}%
,i_{2},\ldots,i_{n}}\det\left(  M_{i_{1},i_{2},\ldots,i_{j-1},i_{j}%
+1,i_{j+1},i_{j+2},\ldots,i_{n}}\right) \\
&  =\sum_{\substack{\left(  i_{1},i_{2},\ldots,i_{n}\right)  \in\mathbb{N}%
^{n};\\i_{1}+i_{2}+\cdots+i_{n}=m-1}}\dbinom{m-1}{i_{1},i_{2},\ldots,i_{n}%
}\underbrace{\sum_{j=1}^{n}\det\left(  M_{i_{1},i_{2},\ldots,i_{j-1}%
,i_{j}+1,i_{j+1},i_{j+2},\ldots,i_{n}}\right)  }_{\substack{=\left(
\operatorname*{Tr}A\right)  \cdot\det\left(  M_{i_{1},i_{2},\ldots,i_{n}%
}\right)  \\\text{(by Lemma \ref{lem.almkvist-lem3})}}}\\
&  =\sum_{\substack{\left(  i_{1},i_{2},\ldots,i_{n}\right)  \in\mathbb{N}%
^{n};\\i_{1}+i_{2}+\cdots+i_{n}=m-1}}\dbinom{m-1}{i_{1},i_{2},\ldots,i_{n}%
}\left(  \operatorname*{Tr}A\right)  \cdot\det\left(  M_{i_{1},i_{2}%
,\ldots,i_{n}}\right) \\
&  =\left(  \operatorname*{Tr}A\right)  \cdot\underbrace{\sum
_{\substack{\left(  i_{1},i_{2},\ldots,i_{n}\right)  \in\mathbb{N}^{n}%
;\\i_{1}+i_{2}+\cdots+i_{n}=m-1}}\dbinom{m-1}{i_{1},i_{2},\ldots,i_{n}}%
\det\left(  M_{i_{1},i_{2},\ldots,i_{n}}\right)  }_{\substack{=\left(
\operatorname*{Tr}A\right)  ^{m-1}\\\text{(by (\ref{pf.prop.almkvist-lem.IH}%
))}}}\\
&  =\left(  \operatorname*{Tr}A\right)  \cdot\left(  \operatorname*{Tr}%
A\right)  ^{m-1}=\left(  \operatorname*{Tr}A\right)  ^{m}.
\end{align*}
In other words,%
\[
\left(  \operatorname*{Tr}A\right)  ^{m}=\sum_{\substack{\left(  i_{1}%
,i_{2},\ldots,i_{n}\right)  \in\mathbb{N}^{n};\\i_{1}+i_{2}+\cdots+i_{n}%
=m}}\dbinom{m}{i_{1},i_{2},\ldots,i_{n}}\det\left(  M_{i_{1},i_{2}%
,\ldots,i_{n}}\right)  .
\]
Thus, Proposition \ref{prop.almkvist-lem} holds for our $m$. This completes
the induction step, and with it the proof of Proposition
\ref{prop.almkvist-lem}.
\end{proof}
\end{verlong}

\begin{vershort}
\begin{proof}
[Proof idea for Theorem \ref{thm.almkvist}.]Apply Proposition
\ref{prop.almkvist-lem} to $m=nk+1$. Observe that each $n$-tuple $\left(
i_{1},i_{2},\ldots,i_{n}\right)  \in\mathbb{N}^{n}$ satisfying $i_{1}%
+i_{2}+\cdots+i_{n}=nk+1$ has at least one entry $i_{j}\geq k+1$, and thus the
matrix $M_{i_{1},i_{2},\ldots,i_{n}}$ has a zero row (since $A^{k+1}=0$
entails $A^{i_{j}}=0$).
\end{proof}
\end{vershort}

\begin{verlong}
\begin{proof}
[Proof of Theorem \ref{thm.almkvist}.]For any $n$ integers $i_{1},i_{2}%
,\ldots,i_{n}\in\mathbb{N}$, we define an $n\times n$-matrix $M_{i_{1}%
,i_{2},\ldots,i_{n}}$ as in Proposition \ref{prop.almkvist-lem}. We claim the following:

\begin{statement}
\textit{Claim 1:} Let $i_{1},i_{2},\ldots,i_{n}\in\mathbb{N}$ be such that
$i_{1}+i_{2}+\cdots+i_{n}=nk+1$. Then, $\det\left(  M_{i_{1},i_{2}%
,\ldots,i_{n}}\right)  =0$.
\end{statement}

\begin{proof}
[Proof of Claim 1.]If each $j\in\left\{  1,2,\ldots,n\right\}  $ satisfied the
inequality $i_{j}\leq k$, then the sum of these $n$ inequalities would be
saying that%
\[
i_{1}+i_{2}+\cdots+i_{n}\leq\underbrace{k+k+\cdots+k}_{n\text{ times}}=nk,
\]
which would contradict $i_{1}+i_{2}+\cdots+i_{n}=nk+1>nk$. Hence, not every
$j\in\left\{  1,2,\ldots,n\right\}  $ can satisfy the inequality $i_{j}\leq
k$. In other words, there exists some $j\in\left\{  1,2,\ldots,n\right\}  $
that satisfies $i_{j}>k$. Consider this $j$.

From $i_{j}>k$, we obtain $i_{j}\geq k+1$ (since $i_{j}$ and $k$ are
integers). Hence, $A^{i_{j}}=\underbrace{A^{k+1}}_{=0}A^{i_{j}-\left(
k+1\right)  }=0$. Therefore, $\operatorname{row}_{j}\left(  A^{i_{j}}\right)
=\operatorname{row}_{j}0=0$.

But $M_{i_{1},i_{2},\ldots,i_{n}}$ is defined as the $n\times n$-matrix whose
rows are
\[
\operatorname*{row}\nolimits_{1}\left(  A^{i_{1}}\right)
,\ \operatorname*{row}\nolimits_{2}\left(  A^{i_{2}}\right)  ,\ \ldots
,\ \operatorname*{row}\nolimits_{n}\left(  A^{i_{n}}\right)  .
\]
Thus, in particular, the $j$-th row of this matrix $M_{i_{1},i_{2}%
,\ldots,i_{n}}$ is $\operatorname{row}_{j}\left(  A^{i_{j}}\right)  $. In
other words, the $j$-th row of this matrix $M_{i_{1},i_{2},\ldots,i_{n}}$ is
$0$ (since $\operatorname{row}_{j}\left(  A^{i_{j}}\right)  =0$). In other
words, the $j$-th row of the matrix $M_{i_{1},i_{2},\ldots,i_{n}}$ consists of
zeroes. Therefore, the determinant of this matrix equals $0$ (because if a row
of a square matrix consists of zeroes, then the determinant of this matrix
equals $0$). In other words, $\det\left(  M_{i_{1},i_{2},\ldots,i_{n}}\right)
=0$. This proves Claim 1.
\end{proof}

Now, Proposition \ref{prop.almkvist-lem} (applied to $m=nk+1$) yields%
\begin{align*}
\left(  \operatorname*{Tr}A\right)  ^{nk+1}  &  =\sum_{\substack{\left(
i_{1},i_{2},\ldots,i_{n}\right)  \in\mathbb{N}^{n};\\i_{1}+i_{2}+\cdots
+i_{n}=nk+1}}\dbinom{nk+1}{i_{1},i_{2},\ldots,i_{n}}\underbrace{\det\left(
M_{i_{1},i_{2},\ldots,i_{n}}\right)  }_{\substack{=0\\\text{(by Claim 1)}}}\\
&  =\sum_{\substack{\left(  i_{1},i_{2},\ldots,i_{n}\right)  \in\mathbb{N}%
^{n};\\i_{1}+i_{2}+\cdots+i_{n}=nk+1}}\dbinom{nk+1}{i_{1},i_{2},\ldots,i_{n}%
}0=0.
\end{align*}
This proves Theorem \ref{thm.almkvist}.
\end{proof}
\end{verlong}

We leave it to the reader to prove the following \textquotedblleft
supplemental law\textquotedblright\ to Theorem \ref{thm.almkvist}:

\begin{theorem}
\label{thm.almkvist-1}Let $n\in\mathbb{N}$ and $k\in\mathbb{N}$. Let
$A\in\mathbb{K}^{n\times n}$ be a matrix such that $A^{k+1}=0$. Then,%
\[
\left(  \operatorname*{Tr}A\right)  ^{nk}=\dfrac{\left(  nk\right)  !}{k!^{n}%
}\left(  \det A\right)  ^{k}.
\]

\end{theorem}

\begin{verlong}
\begin{proof}
[Proof of Theorem \ref{thm.almkvist-1}.]In the following proof, the expression
\textquotedblleft$k,k,\ldots,k$\textquotedblright\ shall always mean
\textquotedblleft$\underbrace{k,k,\ldots,k}_{n\text{ times}}$%
\textquotedblright. Note that $\left(  k,k,\ldots,k\right)  $ is an $n$-tuple
$\left(  i_{1},i_{2},\ldots,i_{n}\right)  \in\mathbb{N}^{n}$ satisfying
$i_{1}+i_{2}+\cdots+i_{n}=nk$ (since $k+k+\cdots+k=nk$).

For any $n$ integers $i_{1},i_{2},\ldots,i_{n}\in\mathbb{N}$, we define an
$n\times n$-matrix $M_{i_{1},i_{2},\ldots,i_{n}}$ as in Proposition
\ref{prop.almkvist-lem}. We claim the following:

\begin{statement}
\textit{Claim 1:} Let $i_{1},i_{2},\ldots,i_{n}\in\mathbb{N}$ be such that
$i_{1}+i_{2}+\cdots+i_{n}=nk$ and $\left(  i_{1},i_{2},\ldots,i_{n}\right)
\neq\left(  k,k,\ldots,k\right)  $. Then, $\det\left(  M_{i_{1},i_{2}%
,\ldots,i_{n}}\right)  =0$.
\end{statement}

\begin{proof}
[Proof of Claim 1.]We shall first show that there exists some $j\in\left\{
1,2,\ldots,n\right\}  $ that satisfies $i_{j}>k$.

Indeed, assume the contrary. Thus, there is no such $j$. In other words, each
$j\in\left\{  1,2,\ldots,n\right\}  $ satisfies $i_{j}\leq k$. Moreover, at
least one $j\in\left\{  1,2,\ldots,k\right\}  $ must satisfy $i_{j}\neq k$
(since $\left(  i_{1},i_{2},\ldots,i_{n}\right)  \neq\left(  k,k,\ldots
,k\right)  $) and therefore $i_{j}<k$ (since the preceding sentence yields
$i_{j}\leq k$).

Thus we have shown that

\begin{itemize}
\item each $j\in\left\{  1,2,\ldots,n\right\}  $ satisfies $i_{j}\leq k$;

\item at least one $j\in\left\{  1,2,\ldots,k\right\}  $ satisfies $i_{j}<k$.
\end{itemize}

In other words, the weak inequality $i_{j}\leq k$ is satisfied for each
$j\in\left\{  1,2,\ldots,n\right\}  $, and at least one of these $n$
inequalities is strict. Thus, summing these $n$ inequalities together, we
obtain the strict inequality%
\[
i_{1}+i_{2}+\cdots+i_{n}<\underbrace{k+k+\cdots+k}_{n\text{ times}}=nk.
\]
This contradicts $i_{1}+i_{2}+\cdots+i_{n}=nk$. This contradiction shows that
our assumption was false.

Hence, we have shown that there exists some $j\in\left\{  1,2,\ldots
,n\right\}  $ that satisfies $i_{j}>k$. Consider this $j$.

From $i_{j}>k$, we obtain $i_{j}\geq k+1$ (since $i_{j}$ and $k$ are
integers). Hence, $A^{i_{j}}=\underbrace{A^{k+1}}_{=0}A^{i_{j}-\left(
k+1\right)  }=0$. Therefore, $\operatorname{row}_{j}\left(  A^{i_{j}}\right)
=\operatorname{row}_{j}0=0$.

But $M_{i_{1},i_{2},\ldots,i_{n}}$ is defined as the $n\times n$-matrix whose
rows are
\[
\operatorname*{row}\nolimits_{1}\left(  A^{i_{1}}\right)
,\ \operatorname*{row}\nolimits_{2}\left(  A^{i_{2}}\right)  ,\ \ldots
,\ \operatorname*{row}\nolimits_{n}\left(  A^{i_{n}}\right)  .
\]
Thus, in particular, the $j$-th row of this matrix $M_{i_{1},i_{2}%
,\ldots,i_{n}}$ is $\operatorname{row}_{j}\left(  A^{i_{j}}\right)  $. In
other words, the $j$-th row of this matrix $M_{i_{1},i_{2},\ldots,i_{n}}$ is
$0$ (since $\operatorname{row}_{j}\left(  A^{i_{j}}\right)  =0$). In other
words, the $j$-th row of the matrix $M_{i_{1},i_{2},\ldots,i_{n}}$ consists of
zeroes. Therefore, the determinant of this matrix equals $0$ (because if a row
of a square matrix consists of zeroes, then the determinant of this matrix
equals $0$). In other words, $\det\left(  M_{i_{1},i_{2},\ldots,i_{n}}\right)
=0$. This proves Claim 1.
\end{proof}

\begin{statement}
\textit{Claim 2:} We have $\det\left(  M_{k,k,\ldots,k}\right)  =\left(  \det
A\right)  ^{k}$.
\end{statement}

\begin{proof}
[Proof of Claim 2.]Recall that $M_{k,k,\ldots,k}$ is defined as the $n\times
n$-matrix whose rows are
\[
\operatorname*{row}\nolimits_{1}\left(  A^{k}\right)  ,\ \operatorname*{row}%
\nolimits_{2}\left(  A^{k}\right)  ,\ \ldots,\ \operatorname*{row}%
\nolimits_{n}\left(  A^{k}\right)  .
\]
In other words, it is the $n\times n$-matrix whose rows are precisely the
corresponding rows of $A^{k}$. Thus, it is simply the $n\times n$-matrix
$A^{k}$. That is, we have $M_{k,k,\ldots,k}=A^{k}$. Hence,%
\[
\det\left(  M_{k,k,\ldots,k}\right)  =\det\left(  A^{k}\right)  =\left(  \det
A\right)  ^{k}%
\]
(by \cite[Corollary 6.25 \textbf{(b)}]{detnotes}, applied to $B=A$). This
proves Claim 2.
\end{proof}

\begin{statement}
\textit{Claim 3:} We have
\[
\dbinom{nk}{k,k,\ldots,k}=\dfrac{\left(  nk\right)  !}{k!^{n}}.
\]

\end{statement}

\begin{proof}
[Proof of Claim 3.]The definition of $\dbinom{nk}{k,k,\ldots,k}$ yields
$\dbinom{nk}{k,k,\ldots,k}=\dfrac{\left(  nk\right)  !}{k!k!\cdots k!}%
=\dfrac{\left(  nk\right)  !}{k!^{n}}$ (since $k!k!\cdots k!=k!^{n}$). This
proves Claim 3.
\end{proof}

Now, Proposition \ref{prop.almkvist-lem} (applied to $m=nk$) yields%
\begin{align*}
\left(  \operatorname*{Tr}A\right)  ^{nk}  &  =\sum_{\substack{\left(
i_{1},i_{2},\ldots,i_{n}\right)  \in\mathbb{N}^{n};\\i_{1}+i_{2}+\cdots
+i_{n}=nk}}\dbinom{nk}{i_{1},i_{2},\ldots,i_{n}}\det\left(  M_{i_{1}%
,i_{2},\ldots,i_{n}}\right) \\
&  =\underbrace{\dbinom{nk}{k,k,\ldots,k}}_{\substack{=\dfrac{\left(
nk\right)  !}{k!^{n}}\\\text{(by Claim 3)}}}\underbrace{\det\left(
M_{k,k,\ldots,k}\right)  }_{\substack{=\left(  \det A\right)  ^{k}\\\text{(by
Claim 2)}}}\\
&  \ \ \ \ \ \ \ \ \ \ +\sum_{\substack{\left(  i_{1},i_{2},\ldots
,i_{n}\right)  \in\mathbb{N}^{n};\\i_{1}+i_{2}+\cdots+i_{n}=nk;\\\left(
i_{1},i_{2},\ldots,i_{n}\right)  \neq\left(  k,k,\ldots,k\right)  }%
}\dbinom{nk}{i_{1},i_{2},\ldots,i_{n}}\underbrace{\det\left(  M_{i_{1}%
,i_{2},\ldots,i_{n}}\right)  }_{\substack{=0\\\text{(by Claim 1)}}}\\
&  \ \ \ \ \ \ \ \ \ \ \ \ \ \ \ \ \ \ \ \ \left(
\begin{array}
[c]{c}%
\text{here, we have split off the addend for }\left(  i_{1},i_{2},\ldots
,i_{n}\right)  =\left(  k,k,\ldots,k\right) \\
\text{from the sum (since }\left(  k,k,\ldots,k\right)  \text{ is an
}n\text{-tuple }\left(  i_{1},i_{2},\ldots,i_{n}\right)  \in\mathbb{N}^{n}\\
\text{satisfying }i_{1}+i_{2}+\cdots+i_{n}=nk\text{)}%
\end{array}
\right) \\
&  =\dfrac{\left(  nk\right)  !}{k!^{n}}\left(  \det A\right)  ^{k}%
+\underbrace{\sum_{\substack{\left(  i_{1},i_{2},\ldots,i_{n}\right)
\in\mathbb{N}^{n};\\i_{1}+i_{2}+\cdots+i_{n}=nk;\\\left(  i_{1},i_{2}%
,\ldots,i_{n}\right)  \neq\left(  k,k,\ldots,k\right)  }}\dbinom{nk}%
{i_{1},i_{2},\ldots,i_{n}}0}_{=0}\\
&  =\dfrac{\left(  nk\right)  !}{k!^{n}}\left(  \det A\right)  ^{k}.
\end{align*}
This proves Theorem \ref{thm.almkvist-1}.
\end{proof}
\end{verlong}

\begin{corollary}
\label{cor.almkvist-2}Let $n\in\mathbb{N}$ and $k\in\mathbb{N}$. Let
$A\in\mathbb{K}^{n\times n}$ be a matrix such that $A^{k+1}=0$. Then,%
\[
\left(  \operatorname*{Tr}A\right)  ^{nk}A=0.
\]

\end{corollary}

\begin{verlong}
\begin{proof}
[Proof of Corollary \ref{cor.almkvist-2}.]Theorem \ref{thm.almkvist-1} yields%
\[
\left(  \operatorname*{Tr}A\right)  ^{nk}=\dfrac{\left(  nk\right)  !}{k!^{n}%
}\left(  \det A\right)  ^{k}.
\]
Thus,%
\begin{equation}
\left(  \operatorname*{Tr}A\right)  ^{nk}A=\dfrac{\left(  nk\right)  !}%
{k!^{n}}\left(  \det A\right)  ^{k}A.\label{pf.cor.almkvist-2.1}%
\end{equation}
However, \cite[Corollary 6.25 \textbf{(b)}]{detnotes} (applied to $B=A$)
yields $\det\left(  A^{k}\right)  =\left(  \det A\right)  ^{k}$. Furthermore,
Theorem \ref{thm.adj.inverse} (applied to $A^{k}$ instead of $A$) yields
\[
A^{k}\cdot\operatorname*{adj}\left(  A^{k}\right)  =\operatorname*{adj}\left(
A^{k}\right)  \cdot A^{k}=\underbrace{\det\left(  A^{k}\right)  }_{=\left(
\det A\right)  ^{k}}\cdot\,I_{n}=\left(  \det A\right)  ^{k}\cdot I_{n}.
\]
Thus, $\det\left(  A^{k}\right)  \cdot I_{n}=\operatorname*{adj}\left(
A^{k}\right)  \cdot A^{k}$. Finally,%
\[
\left(  \det A\right)  ^{k}\underbrace{A}_{=I_{n}A}=\underbrace{\left(  \det
A\right)  ^{k}\cdot I_{n}}_{=\operatorname*{adj}\left(  A^{k}\right)  \cdot
A^{k}}A=\operatorname*{adj}\left(  A^{k}\right)  \cdot\underbrace{A^{k}%
A}_{=A^{k+1}=0}=0.
\]
Thus, (\ref{pf.cor.almkvist-2.1}) becomes%
\[
\left(  \operatorname*{Tr}A\right)  ^{nk}A=\dfrac{\left(  nk\right)  !}%
{k!^{n}}\underbrace{\left(  \det A\right)  ^{k}A}_{=0}=0.
\]
This proves Corollary \ref{cor.almkvist-2}.
\end{proof}
\end{verlong}

\subsection{$\operatorname*{Tr}\left(  A^{p}\right)  $ in characteristic $p$}

Another folklore property of traces (see, e.g.,
\url{https://math.stackexchange.com/questions/3044716/} ) is the following:

\begin{theorem}
\label{thm.TrAp}Let $p$ be a prime number. Assume that $p=0$ in $\mathbb{K}$.

Let $n\in\mathbb{N}$. Let $A\in\mathbb{K}^{n\times n}$ be a matrix. Then,
\[
\operatorname*{Tr}\left(  A^{p}\right)  =\left(  \operatorname*{Tr}A\right)
^{p}.
\]

\end{theorem}

There are easy and direct ways to prove this using a cyclic group action (see
\cite[\S 2]{Steinl17} for a proof in the particular case $\mathbb{K}%
=\mathbb{Z}/p$, which is easily adapted to the general case\footnote{Note that
this particular case takes the simpler form $\operatorname*{Tr}\left(
A^{p}\right)  =\operatorname*{Tr}A$ (in $\mathbb{Z}/p$), since Fermat's Little
Theorem ensures that $b^{p}=b$ for each $b\in\mathbb{Z}/p$. In the general
case, Fermat's Little Theorem no longer applies, so the Freshman's Dream
formula $\left(  u+v\right)  ^{p}=u^{p}+v^{p}$ must be used instead.}). But it
can also be derived from Corollary \ref{cor.TCH.c1} or from Proposition
\ref{prop.almkvist-lem}. We leave the details to the reader.

\subsection{Another formula for derivations of $\det A$}

Besides Theorem \ref{thm.deriv.ddet}, there is another formula for the image
of a determinant under a derivation:

\begin{theorem}
\label{thm.deriv.ddet2}Let $\mathbb{L}$ be a commutative $\mathbb{K}$-algebra.
Let $f:\mathbb{L}\rightarrow\mathbb{L}$ be a $\mathbb{K}$-derivation. Let
$n\in\mathbb{N}$. Let $A=\left(  a_{i,j}\right)  _{1\leq i\leq n,\ 1\leq j\leq
n}\in\mathbb{L}^{n\times n}$. For each $k\in\left\{  1,2,\ldots,n\right\}  $,
let $A_{k}^{\prime}$ be the matrix $\left(  f^{\delta_{k,i}}\left(
a_{i,j}\right)  \right)  _{1\leq i\leq n,\ 1\leq j\leq n}\in\mathbb{L}%
^{n\times n}$ (where $\delta_{k,i}$ is the Kronecker delta, i.e., the number
$1$ if $k=i$ and the number $0$ if $k\neq i$). Then,%
\[
f\left(  \det A\right)  =\sum_{k=1}^{n}\det A_{k}^{\prime}.
\]

\end{theorem}

We leave the (easy) proof to the reader.

\begin{noncompile}
[Additional material that could find its way into this paper:

-- The $\operatorname*{Tr}\left(  A^{p}\right)  \equiv\left(
\operatorname*{Tr}A\right)  ^{p}$ congruence.

-- Arnold's congruences (require $\psi$-structure on $\mathbb{K}$).

-- Another formula for derivations on determinants: $f\left(  \det\left(
a_{i,j}\right)  \right)  =\sum_{k=1}^{n}\det\left(  f^{\delta_{k,i}}%
a_{i,j}\right)  $.

-- Proof of Theorem \ref{thm.almkvist-1}.]

See also todo1.
\end{noncompile}

\end{document}